\documentclass[11pt]{article} 
\usepackage{amsmath, amssymb}
\usepackage{amsthm} 
\usepackage{url}
\usepackage{graphicx}
\usepackage{here}
\usepackage{mathrsfs}
\usepackage{comment}
\usepackage{fancyhdr}
\usepackage{enumitem}
\usepackage{tikz} 
\usetikzlibrary{intersections, calc, through}

\usepackage[
  setpagesize=false,
  bookmarks=true,
  bookmarksdepth=tocdepth,
  bookmarksnumbered=true,
  linkcolor=blue,
  citecolor=red,
  colorlinks=true,
  pdftitle={},
  pdfsubject={},
  pdfauthor={},
  pdfkeywords={}
]{hyperref}

\theoremstyle{plain} 
\newtheorem{theorem}{\indent\sc Theorem}[section]
\newtheorem{lemma}[theorem]{\indent\sc Lemma}

\newtheorem{proposition}[theorem]{\indent\sc Proposition}

\newtheorem{conditions}[theorem]{\indent\sc Conditions}
\newtheorem{fact}[theorem]{\indent\sc Fact}
\theoremstyle{definition} 

\newtheorem{remark}[theorem]{\indent\sc Remark}
\newtheorem{example}[theorem]{\indent\sc Example}

\newtheorem{assertion}[theorem]{\indent\sc Assertion}
\newtheorem{problem}[theorem]{\indent\sc Problem}
\newtheorem{question}[theorem]{\indent\sc Question}
\newtheorem{conjecture}[theorem]{\indent\sc Conjecture}
\newtheorem{definitiontheorem}[theorem]{\indent\sc Definition-Theorem}

\makeatletter
    
    \@addtoreset{equation}{section}

\def\address#1#2{\begingroup
\noindent\parbox[t]{7.8cm}{%
\small{\scshape\ignorespaces#1}\par\vskip1ex
\noindent\small{\itshape E-mail address}%
\/: #2\par\vskip4ex}\hfill%
\endgroup}%
\makeatother

\title{An extension of the Floquet-Bloch theory to nilpotent groups and its applications} 
\author{%
\textbf{Atsushi Katsuda}$^{*}$\\
\emph{Dedicated to Professor Toshikazu Sunada}
}
\begin{document}
\markright{An extension of the Floquet-Bloch theory to nilpotent groups}
\pagestyle{myheadings}
\maketitle
\begin{abstract}
We develop a generalized Floquet-Bloch theory for discrete torsion-free nilpotent groups by exploiting their Malcev completions. Our main result is a branching formula that relates finite-dimensional representations of a discrete nilpotent lattice to infinite-dimensional unitary representations of its simply connected nilpotent Lie group. This generalization enables to extend the following two classical asymptotic problems (i) a Chebotarev density analogue for prime closed geodesics on compact negatively curved manifolds with nilpotent covers, and (ii) long time asymptotic expansions of heat kernels on such coverings to the nilpotent setting.  

As a by-product, we derive a semi-classical expansion for the Harper operator, presenting an alternative to mathematical justification of Wilkinson's formula by Helffer-Sj\"ostrand.  

We conclude by proposing several avenues for future work: extensions to general hyperbolic flows and noncompact manifolds (in particular knot complements and related quasi-morphisms), connections to modified Riemann-Hilbert problems and opers. Furthermore, we give a brief comment on asymptotic behavior of knot invariants and infinite extensions in number theory.  
\end{abstract}

\footnotetext{
$^{*}$This work was supported by JSPS KAKENHI Grant Number JP24K06715 and the Research Institute for Mathematical Sciences, an International Joint 
Usage/Research Center located in Kyoto University.}

\tableofcontents

\section{Introduction} 

The Floquet-Bloch theory is a widely used tool for investigating materials with periodic structures. For instance, it can demonstrate that the spectrum of periodic Schr\"odinger operators exhibits band structures. 
Within this context, this theory has been applied to the following problems over several decades in the case of abelian extensions:

\begin{description} 
\item[{\rm (1)}] A geometric analog of the Chebotarev density theorem for prime closed geodesics in a compact Riemannian manifold with negative curvature. 
\item[{\rm (2)}] A long time asymptotic expansion of the heat kernels on covering manifolds of compact Riemannian manifolds.\end{description} 

In this paper, we establish our version of the generalized Floquet-Bloch theory for discrete nilpotent groups and apply it to the above problems for nilpotent extension. 
Since the 1960s, primarily through the work of Glimm, direct handling of these groups has been considered difficult due to their non-type I property and the complexity of the structure of their unitary duals. 

One of our contributions is to relate finite-dimensional representations of a discrete torsion-free nilpotent group and infinite-dimensional representations of its Malcev completion, the simply connected nilpotent Lie group containing the former discrete group as its lattice. We obtain a branching formula of restriction to the lattice of the above infinite-dimensional representaions, providing approximations of a finite-dimensional representation by infinite-dimensional ones and vice versa.

As a by-product of our methods, we offer an alternative mathematical explanation for the semi-classical asymptotic expansion formula for the Harper operator, originally derived by Wilkinson using WKB analysis and later justified mathematically by Helffer-Sj\"ostrand using semi-classical localizations. Our approach may offer advantages for further applications, particularly in various asymptotic problems.

In the final chapter, we propose several future research directions, including extensions to more general hyperbolic flows, noncompact manifolds, especially knot complements and quasi-morphisms, and their connections to modified Riemann-Hilbert problems and opers. We also briefly comment on the asymptotics of knot invariants such as Kashaev's invariant, also known as the volume conjecture, and infinite extensions in number theory, focusing on the Lang-Trotter conjecture.

\subsection{Basic example of usual Floquet-Bloch theory of discrete abelian groups and spectrum of periodic operators}

Our primary concern in this paper is extending the Floquet-Bloch theory to nilpotent groups.

To compare with our main results, we first recall the usual Floquet-Bloch theory applied to the simplest case: The periodic (i.e., $\mathbb{Z}$-invariant) differential operator $P = -\frac{d^2}{dx^2}+V(x)$, where $V(x+1)=V(x)$. This operator acts on $L^2(\mathbb{R}) = L^2(\mathbb{R},\mathbb{C})$ and exhibits band structures of its spectrum Spec($P$), meaning Spec($P$) is the union of closed intervals.

This proof is achieved as follows: Take a complex-valued function $f(x) \in L^2(\mathbb{R})$ and fix $x \in \mathbb{R}$. Define the function $f^x:\mathbb{Z} \to \mathbb{C}$ by $f^x(n) = f(x+n)$.
\paragraph{Abelian model}   
\begin{description}
\item[{\rm A-1}](Decomposition of $L^2(\mathbb{R})$ using Fourier transform over $\mathbb{Z}$) We shall decompose $L^2(\mathbb{R})$ as
\begin{equation}L^2(\mathbb{R}) = L^2([0,1) \times \mathbb{Z}) \simeq L^2([0,1)) \otimes L^2(\mathbb{Z}) = \int_{[0,1)}^\oplus\mathcal{H}_ada
\label{decompositionl2R}
\end{equation}
where the right hand side is the direct integral of the spaces 
\begin{equation}
\mathcal{H}_a:= \{f:\mathbb{R}\to\mathbb{C}|f(x+1) = e^{2\pi\sqrt{-1}a}f(x)\}.  \label{decompositionHa}
\end{equation}

Take the Fourier transform over $\mathbb{Z}$ as 
\begin{equation}\mathcal{F}(f^x)(a) = \widehat{f^x}(a) = \mbox{``$\displaystyle \int_\mathbb{Z}f^x(n)e^{2\pi\sqrt{-1}an}dn$''}= \sum_{n =-\infty}^{\infty}f^x(n)e^{2\pi\sqrt{-1}an} \quad \mbox{for}\quad a \in [0,1). \label{zfourier}\end{equation}
where the integral in the middle, surrounded by the double quotation
marks, is a formal one, and its precise definition is the rightmost sum.

Note that $\widehat{f^x}(a)$ is an element of $\mathcal{H}_a$ as a function of $x$. The inversion formula gives 
\begin{align} \mathcal{F}^{-1}(\widehat{f^x})(m) &= \int_{[0,1)}\widehat{f^x}(a)e^{-2\pi\sqrt{-1}am}da \notag \\ &= \sum_{n =-\infty}^{\infty}\int_{[0,1)}f^x(n)e^{2\pi\sqrt{-1}a(n-m)}da \notag \\ &= \sum_{n =-\infty}^{\infty}f^x(n)\delta_{nm} = f^x(m) = f(x+m). \label{zinversefourier}\end{align}
If we put $m=0$, then we have the following decomposition formula, which explains (\ref{decompositionl2R}).
\begin{equation}
f(x) = \int_{[0,1)}\widehat{f^x}(a)da
\label{decompositionfourier}\end{equation}

The above theorem is the critical lemma for the usual Floquet-Bloch theory, and the domain $[0,1)$ of the integral is sometimes called the Brillouin zone or the reciprocal lattice primitive cell in the physics literature.

\item[{\rm A-2}] (Spectral analysis of decomposed operators)
Since $\mathcal{H}_a$ is an invariant subspace for $P$, take a restriction $P^a := P|_{\mathcal{H}_a}$. Then $P^a$ has discrete spectrum 
\[   \lambda_0(a) \leq \lambda_1(a) \leq \cdots \nearrow \infty.
\] 
The whole spectrum Spec($P$) of $P$ is given by 
\[ \cup_{a\in [0,1)}\cup_{k=0}^\infty\{\lambda_k(a)\}= \cup_{k=0}^\infty[\cup_{a\in [0,1)}\{\lambda_k(a)\}].\]

Note that $\cup_{a\in [0,1)}\{\lambda_k(a)\}$ is a closed interval since $\lambda_k(a)$ is continuous with respect to $a$ and closedness of the spectrum. This fact implies that Spec($P$) has a band structure.

Furthermore, in the research for condensed matter or solid state physics, the real problems are to investigate the detailed behavior of Spec($P^a$) with respect to $a$ in the Brillouin zone. Examples of characteristic keywords related to these are the Fermi surface, flat bands, and the Dirac cone. 

To obtain geometric analogs of the Chebotarev density theorem or long time asymptotic expansions of the heat kernels for abelian extensions mentioned in abstract, it is also necessary to such investgation, especially, in slightly different situation but essentially corresponding to, the knowedge of the behavior of the lowest eigenvalues $\lambda_0(a)$ of $P^a$ in the vicinity of $a=0$, which corresponds the trivial character.  In a later Chapter \ref{outline}, we shall explain them briefly as motivation for the content of this paper.
\end{description}

\subsection{Strategy of noncommutative Floquet-Bloch theory for discrete nilpotent groups}

We begin by briefly explaining the geometric framework typically envisioned for Floquet-Bloch theory and the challenges it presents. In the following subsections, we outline the necessary modifications for theoretical justification and, building on these, describe the practical implementation of our version of noncommutative Floquet-Bloch theory.

We note that the content related to Heisenberg groups, particularly in Chapter \ref{dHeisenberg}, is already included in the previous version \cite{Katsuda0}. We inherit all of that content except the final Chapter in this paper.

Let $M$ be a compact manifold and $\Gamma$ a discrete group. Let $\Pi: X \to M = X/\Gamma$ be a normal covering. Our aim is to understand the structure of the spectrum of a $\Gamma$-periodic operator $P$ on $X$, compatible with the covering map $\Pi$. As in the abelian case discussed in the previous section, our strategy is divided into the following two steps:

\paragraph{Nilpotent version I}  
\begin{description}
\item[{\rm NI-1}](Decomposition of $L^2(X)$) If $\Gamma$ is infinite group, then $X$ is noncompact and thus, the above operator $P$ acting to $L^2(X)$ may have continuous spectrum, it may cause difficulties for detailed analysis, thus we need to decompose $L^2(X)$ as a space ``$L^2(M) \otimes L^2(\Gamma )$''. Its rigorous formulation of the latter space is a space of $L^2$-sections $L^2(E_R)$ of the flat vector bundle $E_R$ associated to the right regular representation $R$ of $\Gamma$, a.k.a. local system associated to $R$, where $E_R$ is defined as follows:
\[   E_R \simeq X \times L^2(\Gamma )/\sim, \hspace{12pt}  (p,\varphi) \sim (\gamma p, R(\gamma^{-1})\varphi )\quad \mbox{for}\quad \gamma \in \Gamma.\] 

The right regular representation $R$ usually admits several direct integral decompositions into irreducible representations as $R \simeq \int_{\widehat{\Gamma}}^\oplus \pi d\pi$, 
which implies a decomposition of $E_R$ by the flat vector bundle associated to $\pi$ as 
\begin{equation} E_R \simeq \int_{\widehat{\Gamma}}^\oplus E_\pi d\pi. \label{decopintoirreducible}
\end{equation}
where $\widehat{\Gamma}$ is the unitary dual, which is the set of equivalence classes of irreducible unitary representations of $\Gamma$.
In the case $\Gamma = \mathbb{Z}$, this decomposition is essentially the same as the previous formula (\ref{decompositionl2R}). Moreover, if the group is type I, which is equivalent to a finite extension of an abelian group provided $\Gamma$ is a discrete group by the result of Thoma \cite{Thoma}, this formula (\ref{decopintoirreducible}) is sufficient for our purpose.
 However, in non type I case, including 
discrete nilpotent groups, it is not possible to single
 out a canonical decomposition among them, and the unitary dual $\widehat{\Gamma}$ is not countably separated as a Borel space. These are not suitable for practical computation. 
\item[{\rm NI-2}](Spectral analysis of decomposed operators) We need to analyze the spectral structure of individual operators and investigate their behavior as a total of them. 
Note that this is rather trivial in part (A-2) of the previous section. 
A model argument of the following part (NS-1) in the abelian case will appear in Step 1 (F1) in Section \ref{Infiniteabelian}, and clearly, the difficulty about the part (NS-2) appears only in the noncommutative extension.  

Our strategies for the nilpotent case are explained as follows:
\begin{description}
\item[{\rm NS-1}] One of ingredients is the perturbation theory.  
Since each operator is coming from one operator $P$ on the covering space and the domains are different from each other, 
the setting here is different from the usual one of the theory, namely operators are varied but their domains are the same fixed one. To adjust our setting to the usual one, we need a kind of ``Gauge transform''. It can be performed using the Lie integral or Chen's iterated integral, see Chapter \ref{reductionpathintegral}.
\item[{\rm NS-2}] Decomposed representations are sometimes infinite-dimensional but otherwise finite. In the former case, the arguments are sometimes meaningful only at formal levels because of convergence problems. On the other hand, in the latter case, the dimensions of each representation are varied, which makes the perturbations difficult. We will briefly explain the main tools in the next section to remedy them. It should be noted that there is another merit of these considerations, practically. As often happens, formal computations in infinite dimensions are easier than those in finite but high dimensions. 
\end{description}
\end{description}

\subsection{Main ingredients for noncommutative Floquet-Bloch theory for discrete nilpotent groups}
Since discrete nilpotent groups are not type I, it is generally believed that this kind of problem is difficult. Indeed, if we wish to know all information about the irreducible unitary representations of a non-type I group completely or at least at the same level as type I groups, then it is sometimes said to be beyond the abilities of human beings. 
Thus, we could only extract partial information from their representations. There is a vast amount of research in this direction, which is summarized in the recent book \cite{Bekka4}. In this book, there are guidelines to treat such groups as follows: Since it is tough to consider all irreducible unitary representations, the following material may be substituted for them
\begin{description}
\item finite-dimensional irreducible unitary representations;
\item factor representations;
\item primitive ideals.
\end{description}
We use the first as the main ingredient and the second as a supplement.

First, we define the Fourier transform $\pi(f)$ of $f \in L^1(\Gamma )$ with respect to a representation of $\pi$ of $\Gamma$ by 
\begin{equation}\pi(f) = \mathcal{F}(f)(\pi) = \sum_{\gamma \in \Gamma}\pi(\gamma)f(\gamma), \label{pifourier}\end{equation}
which is clearly a generalization of the usual one (\ref{zfourier}). Hereafter, we distinguish the letter for representations of the Heisenberg groups and the general nilpotent groups; we use $\rho$ for the former and $\pi$ for the latter.

Then, a form of ``well-defined'' Fourier inversion formula that is practically useful is defined as follows:

\paragraph{Nilpotent version II}  
\begin{description}
\item[{\rm NII-1}](Decomposition of $L^2(X)$) 
The first one generalizes the above formula (\ref{decompositionfourier}). It is a decomposition of the right regular representations $R$ of discrete nilpotent groups into finite-dimensional unitary representations $\pi_{\mbox{\tiny fin}}$'s in a suitable sense, which the Fourier inversion formula generalizes Pytlik's theorem \cite{Pytlik} and \eqref{Fourier1} for the discrete Heisenberg group. One of the reasons why it is enough to consider only finite-dimensional representations seems to be the fact that these are dense in the unitary dual with respect to the Fell topology or more strongly uniform topology as explained later. Moreover, there are other advantages of these representations that they are classified by Howe (see Theorem \ref{Howefinite}).

\begin{theorem}{{\rm ($=$ Theorem \ref{newPytlik}, Fourier Inversion formula)}}\label{introPytlik} For finitely generated discrete torsion-free nilpotent group $\Gamma$, 
there exists a positive, finitely additive measure $\mu$ on the unitary dual $\widehat{\Gamma}$ of $\Gamma$, supported by the set $\widehat{\Gamma}_{\rm fin}$ of (equivalence classes of) finite dimensional representation of $\Gamma$, such that for $f \in L^1(\Gamma )$, the function 
\[ \pi_{\mbox{\tiny fin}} \mapsto \frac{1}{\mbox{{\rm dim}}\; \pi_{\mbox{\tiny fin}}}\mbox{Tr}(\pi_{\mbox{\tiny fin}} (f))
\]
is $\mu$-integrable on $\hat{\Gamma}$ and 
\begin{equation}     f(0) = \int_{\widehat{\Gamma}_{\rm fin}}\frac{1}{\mbox{{\rm dim}}\; \pi_{\mbox{\tiny fin}}}\mbox{Tr}(\pi_{\mbox{\tiny fin}} (f))d\mu(\pi_{\mbox{\tiny fin}} ). \label{nilpotentPytlik1}
\end{equation}
\end{theorem}

As already mentioned, this theorem gives an essentially generalization of Pytlik's theorem \cite{Pytlik} for the discrete Heisenberg group.  Although it is stated in the form of the Plancherel theorem there, as it easily implies the Fourier inversion formula \eqref{Fourier1}.
This theorem is obtained from the following combination of arguments in Chapter \ref{Representationsdiscretenilpotent}, namely in Section \ref{Step0-1}, \cite{Bekka2}, \cite{Johnston} about decomposition into factor representations of $\Gamma$, in Section \ref{nilpotentfujiwara} \cite{Fujiwara1}, \cite{Fujiwara2} about decomposion of monomial representation of nilpotent Lie groups $G$ after taking Malcev completion of $\Gamma$, which is the (uniquely determined) simply connected nilpotent Lie group $G$ containing $\Gamma$ as a lattice, in Section \ref{Step0-5} \cite{Bekka1} about a branching formula from $G$ to its restriction $\Gamma$, finally, further decomposion can be performed by our additional consideration on the case satisfying ``rationality condition'' in Section \ref{Step0-6}.

\item[{\rm NII-2}](Spectral analysis of decomposed operators)
To realize the arguments in (NS-2) in the previous section, we can utilize a relationship between finite-dimensional irreducible unitary representations of a finitely generated torsion-free discrete nilpotent group $\Gamma$ and infinite-dimensional irreducible unitary representations of the Malcev completion $G$ of $\Gamma$.

Here, we explain the above result in the case $\Gamma = \mbox{Heis}_3(\mathbb{Z})$, as it is not necessary to use several terminologies in representation theory, and we believe our novelty in this paper essentially appears in this case. 
For a general torsion-free nilpotent group $\Gamma$, see Section \ref{nilpotentexplanation} for the statement and Chapter \ref{Representationsdiscretenilpotent} for the proof in detail.

First, we recall the definition of a finite-dimensional irreducible unitary representation $\rho_{{\rm \tiny{fin}},x}$ of the discrete Heisenberg group $\Gamma = \mbox{Heis}_3(\mathbb{Z})$. We denote an element $n \in \Gamma$ as $3 \times  3$-matrix as
\begin{equation}
n = [n_1,n_2,n_3] := \left( \begin{array}{ccc} 1 & n_3 & n_1 \\ 0 & 1 & n_2 \\ 0 & 0 & 1\end{array}\right) \in \Gamma.\label{3heis}
\end{equation}

For every $x = (x_1,x_2,x_3) \in  \widehat{\Gamma}_{\rm fin} := \widehat{X} = (\mathbb{Q}\cap [0,1]) \times [0,1] \times [0,1]$, with $x_1 = p/q$ ($p$ and $q$ are relatively prime), let us define a finite-dimensional representation $\rho_{{\rm \tiny{fin}},x}$ of $\Gamma $, acting on the Hilbert space $\mathcal{H}_x$ of dimension $q$, which will be realized as a space $\mathcal{H}_x$ of all complex, periodic sequence $\{\varphi (k)\}_{k\in\mathbb{Z}}$ with period $q$ as
\begin{equation} (\rho_{{\rm \tiny{fin}},x}(n)\varphi )(k) = \exp \left(\frac{2\pi\sqrt{-1}}{q}(n_3x_3+n_2x_2+n_1qx_1+kn_2qx_1)\right)\varphi (k+n_3) \label{finite rep}
\end{equation}

Furthermore, there is another description (cf. \cite{Davidson}) of $\rho_{{\rm \tiny{fin}},x}$ as follows: For  $\gamma = \exp (2\pi\sqrt{-1}p/q)= \exp (2\pi\sqrt{-1}x_1)$, let  
\begin{equation}
w = [1,0,0],\quad v= [0,1,0],\quad u=[0,0,1] \in \Gamma \label{uvw}
\end{equation}
and \begin{equation}\alpha=\exp\left(\frac{2\pi\sqrt{-1}}{q}x_3\right),\quad \beta = \exp\left(\frac{2\pi\sqrt{-1}}{q}x_2\right) \in U(1). \label{introfluctuation}\end{equation}

 Then we have $\rho_{{\rm \tiny{fin}},x}(w) = \gamma$ and
\begin{equation}  \rho_{{\rm \tiny{fin}},x}(u) = \alpha\left(\begin{array}{ccccc}0 & 0 & \cdots & 0 & 1 \\ 1 & 0 & 0 &\cdots & 0 \\ 0 & 1 & 0 & \cdots & 0 \\ \vdots&\ddots &\ddots & \ddots & \vdots \\ 0 &\cdots &0 & 1& 0\end{array}\right), \quad \rho_{{\rm \tiny{fin}},x}(v) = \beta\left(\begin{array}{cccc}1 & 0 & \cdots & 0 \\0 & \gamma & \cdots & 0 \\ \vdots & \ddots & \ddots & \vdots \\ 0 & \cdots & 0 & \gamma^{q-1}\end{array}\right), \label{intromatrixrep}
\end{equation}
Here, we identify $\mathcal{H}_x$ with $\mathbb{C}^q$ where the above matrices act.       
Moreover, the following holds:
\begin{align}
&{}\mbox{pair $(\alpha^q,\beta^q)$ is uniquely determined in $U(1) \times U(1)$ by $\rho_{{\rm \tiny{fin}},x}$} \notag \\  &{}\mbox{and this pair determines $\rho_{{\rm \tiny{fin}},x}$ up to unitary equivalence.} \label{intronewequivalence}
\end{align}

Next, we also recall the definition of the irreducible unitary representation $\rho_h$ of the real Heisenberg group $G = \mbox{Heis}_3(\mathbb{R})$ and $h \in \mathbb{R}\setminus\{0\}$, which is called the Schr\"odinder representation. 
 
It is defined as, for $f \in L^2(\mathbb{R})$ and 
\begin{equation}
 \gamma = [x,y,z] :=  \left(\begin{array}{ccc} 1 & z & x \\ 0 & 1 & y \\ 0 & 0 & 1\end{array}\right)\in G, \label{matrixelementrep}
\end{equation}
\begin{equation} (\rho_h(\gamma)f)(s) = (\rho_h((z,y,x))f)(s) = e^{2\pi \sqrt{-1}h(z+sy)}f(s+x) \label{oldschrep}.
\end{equation}
which is unitary equivalent to 
\begin{equation}
    (\rho_h(\gamma)f)(s) = e^{2\pi \sqrt{-1}(hz+\sqrt{h}sy)}f(s+\sqrt{h}x)=e^{2\pi \sqrt{-1}(hz+\sqrt{h}sy)}e^{\sqrt{h}x\frac{d}{ds}}f(s).   \label{newschrep}
\end{equation}

Then, our result shows that if $h = p/q \in \mathbb{Q}$ ($p,q$ are relatively prime), then the restriction of $\rho_h$ to $\Gamma ={\rm Heis}_3(\mathbb{Z})$  decomposed as follows:

\begin{theorem}{{\rm ($=$ Theorem \ref{discretetoLie})}}\label{introdiscretetoLie} If $h= x_1 =p/q \in \mathbb{Q} \cap [0,1]$, then 
\begin{equation} \rho_h\left|_{\Gamma}\right. \simeq {\int_0^1}^\oplus\!{\int_0^1}^\oplus\rho_{{\rm \tiny{fin}},(x_1, \{qx_1x_2\}, x_3)}dm(x_2)dm(x_3),\label{introdiscretetoHeisenbergLie}
\end{equation}
where $\{a\}$ is the fractional part $a-[a]$ of $a \in \mathbb{R}$ and $dm$ is the Lebesgue measure on $[0,1]$
\end{theorem}

Combining with (\ref{introfluctuation}), (\ref{intromatrixrep}), it gives an approximation of the left-hand side by integrands in the right-hand side in the above formula (\ref{introdiscretetoHeisenbergLie}). It should be noted that this approximation is stronger than that in the Fell topology of the unitary dual in \cite{Lubotzky} in general. The reason is that the former case gives a uniform approximation in the whole representation space $L^2(\mathbb{R})$, i.e. an approximation in the norm topology, by contrast, the latter gives only a collection of uniform approximations for each finite-dimensional subspaces of $L^2(\mathbb{R})$, i.e. an approximation in the strong topology.  In fact, the latter property is rather easy to prove in the case of discrete nilpotent groups or, more generally, residually finite or amenable groups, as explained in \cite{Lubotzky}.

To elucidate its availability, we mention the mathematical justification of the following (formal) expansion:
\begin{align}h_\theta &:= -2\cos \left(\sqrt{\theta}\frac{d}{\sqrt{-1}ds}\right)-2\cos (\sqrt{\theta}s)  \notag \\ &= -4 + \left(-\frac{d^2}{ds^2}+s^2\right)\theta + O(\theta^2) \label{introformal}.
\end{align}
This expansion gives the Wilkinson formula (\ref{introeigen1}) at the formal level.  Note that the left-hand side $h_\theta$ is a bounded operator; on the other hand, the second term $-\frac{d^2}{ds^2}+s^2$ of the right-hand side, the harmonic oscillator, is an unbounded operator. This fact is the reason why we need justification. It will be explained in a later section \ref{Hofstadterbutterfly} and Chapter \ref{Anotherproofwilkinson} as an application of the above theorem, contrasting with the former arguments by Helffer-Sj\"ostrand \cite{Helffer}, \cite{Helffer01}.
\end{description}

\subsection{Practical information for the usage of nilpotent Floquet-Bloch theory}\label{Introprac}
Here, we provide a step-by-step summary of our theory for practical use. 

\paragraph{Nilpotent version III}  
\begin{description}
\item[{\rm NIII-1}](Decomposition of $L^2(X)$) 
\begin{itemize}
\item By Theorem \ref{introPytlik}, we can decompose the regular representation into finite-dimensional unitary representations $\pi$ with the (finitely additive) measure $\mu$ on the set $\widehat{\Gamma}_{\rm fin}$ as a direct integral. Thus, we need to compute finite-dimensional representations $\pi$, the subset $\widehat{\Gamma}_{\rm fin}$ of the unitary dual, and the Plancherel measure $\mu$ concretely. As for representations $\pi$, Howe \cite{Howe} already classified irreducible finite-dimensional representations of $\Gamma$ in \cite{Howe} (cf.Theorem \ref{Howefinite} in later.)
\item However, instead of doing this, it seems better to approximate them by infinite-dimensional irreducible representations of the Malcev completion $G$ of $\Gamma$ than consider individual finite-dimensional representations that appeared in the above classification for the reason explained before. Especially for practical computation. The same principle also apply to  $\mu$ and $\widehat{\Gamma}_{\rm fin}$.

\item Then, the relatively standard unitary representation theory of the nilpotent Lie group is available. We summarize several beneficial results in Appendices A and B, in which the following theorems play such roles: concrete computations of each representation (Theorem \ref{nilpLieirredrep}), the character formulas (Theorems \ref{nilpotenttrace}, \ref{nilpLietrace}, \ref{nilpLiecharcter}), the Plancherel measure in the Fourier inversion formula (Theorem \ref{FourierinversionLie}). For more concrete procedure, see Section \ref{constracthypoelliptic}.

\item In other words, formal computations using infinite-dimensional representations are justified in principle, and we can feel free to use them.
In later sections, examples of the above arguments will appear in the proofs of Theorems \ref{theorem-geod} and \ref{conj-heat}.

\item More detailed explanation of the theoretical base of the above arguments will be given in Remark \ref{ExplanationofPractical}.
\end{itemize}

\item[{\rm NIII-2}](Spectral analysis of decomposed operators)
\begin{itemize}

\item We need to connect the representation theory to
Floquet-Bloch theory, namely, investigating $\Gamma$-invariant operators acting on $L^2(X)$ (or equivalently, operators acting on $\Gamma(E_R)$, we utilize the line integral, and the de-Rham Hodge theorem for the abelian case and the Lie integral or $\pi_1$ de Rham theorem, their harmonic theory.
Please refer to Section \ref{Infiniteabelian} for the former and Chapter \ref{reductionpathintegral} for the latter cases.
\item To investigate the spectral structure, we first make formal computation   
based on infinite-dimensional irreducible representations of $G$ or its restriction to $\Gamma$, then to make rigorous, replace them with the proofs based on their approximation by finite-dimensional unitary representations of $\Gamma$, which is explained in Chapter \ref{asymnilpotent} in some detail.
\end{itemize}
\end{description}

\subsection{Geometric analogue of the Chebotarev density theorem}\label{IntroChebotarev}
Here, we explain one of the applications that initially motivated us to start this project.  
However, we would like to note the following aspects:
\begin{description}
\item[{\rm (1)}] The main results in this and the following sections are not immediate consequences of the results in the previous subsections. Namely, we need several extra items explained later.
However, if we assume our nilpotent Floquet-Bloch theory, some knowledge of Chen's iterated integrals, and the arguments for the case of abelian extensions, then obtaining the results in this and the following subsections is not so difficult, except that some more tedious computations are required than in the abelian case.  Thus, these results themselves are likely conceivable for experts, which is also applicable to several future problems in the first half of the last chapter.
 
\item[{\rm (2)}] We hope these results themselves attract some interest to people working in related fields, such as geometry, dynamical systems, and probability. However, we believe the scope of our methods, noncommutative Floquet-Bloch theory for nilpotent groups, and our novel method for justification of (formal) asymptotic formulas, somewhat similar to the Wilkinson formula (\ref{introeigen1}), would have broader applications. We propose such candidates in the last chapter for future directions.   
\end{description}

Let $M$ be a compact Riemannian manifold of negative curvature, and let $\pi (x)$ denote the number of prime closed geodesics on $M$ whose length is at most $x$. Celebrated results of Selberg \cite{Selberg}, Huber \cite{Huber}, Margulis \cite{Margulis}, \cite{Margulis2}, Parry and Pollicott \cite{Parry1} asserts that

\begin{theorem} \label{primegeodesictheorem}\[ \pi (x) \sim \frac{e^{hx}}{hx},
\]
Where $h > 0$ is the topological entropy of the geodesic flow on $M$.
\end{theorem}

The above result is called the prime geodesic theorem or the prime orbit theorem,  a geometric analog of the prime number theorem.
 Furthermore, several results in geometry are analogs of the Dirichlet density theorem for arithmetic progressions or the Chebotarev density theorem for algebraic extensions of number fields. 

Let us formulate problems precisely. Taking into account that there is one-to-one correspondence among the following triplet,  closed geodesics in $M$, free homotopy classes of closed curves in $M$, and conjugacy classes of each element in the fundamental group $\pi_1(M)$ of $M$, we consider a surjective homomorphism 
\[\Phi : \pi_1(M) \to \Gamma\] 
for a finitely generated discrete group $\Gamma$.

For a conjugacy class $\alpha$ in $\Gamma$, let $\pi (x,\Phi,\alpha )$ denote the number of prime closed geodesics $\gamma$ on $M$ whose length are at most $x$ and satisfy $\Phi ([\gamma ]) \subset \alpha$, where $[\gamma ]$ denotes the corresponding conjugacy class of $\gamma$ in $\pi_1(M)$. When $\Phi$ is a canonical projection to some quotient group $\Gamma$ of $\pi_1(M)$, we simply denote $\pi(x, \alpha) = \pi(x,\Phi,\alpha)$. 
Then, our problems are stated as follows:

\begin{problem}
What is the asymptotic behavior of $\pi (x,\Phi, \alpha )$ as $x \to \infty$?
\end{problem}

When $\Gamma$ is the trivial group, $\pi(x,\Phi,\alpha ) = \pi (x)$, and the answer to the above problem is given by the prime geodesic theorem. First answer of the above problem is given when $\Gamma$ is a finite quotient group of $\pi_1(M)$ by Sunada \cite{Sunada1}, Parry and Pollicott \cite{Parry2} and Adachi and Sunada \cite{Adachi} independently as follows:

\begin{theorem}\label{finitegeodesic} \[ \pi (x,\alpha ) \sim \frac{{}^{\tiny{\#}} \alpha}{{}^{\tiny{\#}} \Gamma}\frac{e^{hx}}{hx},
\]
where ${}^{\tiny{\#}} \alpha$ and ${}^{\tiny{\#}} \Gamma$ denote the cardinals of $\alpha$ and $\Gamma$ respectively.
\end{theorem}
This theorem can be considered a geometric analog of the (usual) Chebotarev density theorem for finite extensions of number fields.

As for infinite extensions, the situations are rather different in contrast with number theory, which will be explained in Section \ref{InfiniteextensionsNT}.
In geometry, the first answer is provided in the simplest case that $\Gamma$ is an infinite abelian quotient group of $\pi_1(M)$. In the case when $M$ is  a compact Riemann surface of genus $g$ with constant negative curvature $-1$ and $\Gamma = H_1(M,\mathbb{Z})$, the following asymptotic result is obtained by Phillips and Sarnak \cite{Phillips}:

\begin{theorem}\label{abel-geod}\begin{equation} \pi(x,\alpha ) \sim \frac{Ce^x}{x^{g+1}}\left(1 + \frac{c_1}{x} + \frac{c_2}{x^2} + \cdots \right) \label{Sarnak}
\end{equation}
where 
\begin{equation} C = \Big(\frac{{\rm vol}(M)}{2\pi}\Big)^{b/2}\frac{1}{{\rm vol}(\mbox{J}(M))} = (g-1)^g, \label{Sarnakconstant}
\end{equation}
where ${\rm vol}(M)$, ${\rm vol}(\mbox{J}(M))$ and $b$ are the volumes of $M$ and the Jacobi torus $\mbox{J}(M)$, and the first Betti number of $M$ respectively
\end{theorem} 
Note that ${\rm vol}(M) = 2\pi(g-1)$ by the Gauss-Bonnet formula and ${\rm vol}(\mbox{J}(M)) =1$ by the computation of the Riemannian matrix in the above case.
The leading term in the above theorem was also obtained by Katsuda and Sunada \cite{Katsuda1} independently. The genus $g$ is half of the first Betti number $b$, the rank of the first integral homology group $H_1(M,\mathbb{Z})$ of $M$.

This kind of asymptotic expansion is generalized to the case of prime closed geodesics in certain compact Riemannian manifolds with variable negative curvature by Anantharaman \cite{Anantharaman1}, Pollicott and Sharp \cite{Pollicott2}, and Kotani \cite{Kotani1}.

As for leading term asymptotics, more generally, prime closed orbits of weakly mixing Anosov flows on compact Riemannian manifolds by Lalley \cite{Lalley}, Pollicott \cite{Pollicott1}, Katsuda and Sunada \cite{Katsuda2}, and Sharp \cite{Sharp1}. Furthermore, there are generalizations to the central limit theorems and the large deviations \cite{Lalley}, \cite{Babillot}, \cite{Anantharaman2}. For additional information, see \cite{Parry3}, \cite{Sharp2}.

To compare conjectures below, we recall a result in \cite{Katsuda2}. 
Let $\varphi_t$ be a smooth, transitive, and weakly mixing Anosov flow on a compact manifold $N$, and $Z$ be the vector field generating the flow $\varphi_t$. 
We denote by $h$ the topological entropy of the flow and by $m$ a (unique) invariant probability measure on $N$ of maximal entropy. 

Given a surjective homomorphism $\Phi:\pi_1(N) \to H$ which factors as $\Phi = \Phi_2\circ\Phi_1$ through $\Phi_1:\pi_1(N) \to H_1(N,\mathbb{Z})$ and $\Phi_2: H_1(N,\mathbb{Z}) \to H$ onto an abelian group $H$, we set, similarly as before, for each $\alpha \in H$ and positive number $x$,
\[ \pi(x, \alpha) = {}^\sharp\{ \mathfrak{p}| \mbox{closed orbits with}\; \Phi([\mathfrak{p}]) \subset \alpha \; \mbox{and}\; \ell(\mathfrak{p}) < x \},
\]
where $[\mathfrak{p}]$ denotes the conjugacy class in $\pi_1(N)$ corresponding to the free homopy class of $\mathfrak{p}$ and $\ell(\mathfrak{p})$ the least period of $\mathfrak{p}$.

We define the winding cycle $\Psi$, which is a linear functional on the space of closed one forms on $N$, by
\[   \Psi(\omega ) = \int_N \langle \omega, Z \rangle dm.
\]
Since $\Phi(\mbox{exact forms}) = 0$, the linear functional $\Psi$ yields a homology class in
$H_1(N, \mathbb{R}) = \mbox{Hom}(H^1(N, \mathbb{R}), \mathbb{R})$. The ergodicity of the flow leads to the equality
\[ \Psi(\omega )  = \lim_{t\to \infty} \frac{1}{t}\int_0^t\langle \omega, Z \rangle(\varphi_\tau)d\tau \quad \mbox{a.e.}\;\tau.,
\]
When the orbits are traveling. The central limit theorem (cf. Denker and Philipp \cite{Denker})
guarantees the existence of the limit
\begin{equation}
\delta(\omega, \omega)= \lim_{t\to \infty}\frac{1}{t}\int_X dm(x) \left(\int_0^t\langle \omega, Z \rangle(\varphi_\tau)d\tau-t\Psi(\omega)\right)^2,  \label{denkerdelta}
\end{equation}
which yields a positive semi-definite quadratic form on $H^1(N, \mathbb{R})$. We call $\delta$ the covariance form. 
As shown in \cite{Katsuda2}, $\delta$ is positive definite on $\mbox{Ker}\;\Psi$, and hence gives rise to a Euclidean metric on $\mbox{Ker}\;\Psi$. 

Consider the character group $\widehat{H}$ of $H$. 
The tangent space $T_\textbf{1}\widehat{H}$ at the trivial character $\textbf{1}$ is identified with the dual $H^\dagger = \mbox{Hom}(H,\mathbb{R})$, which is also identified, naturally, with a subspace in $H^1(N, \mathbb{R})$. 
Therefore, if $\Psi$ vanishes on $H^\dagger$, the covariance form induces a flat metric on the group $\widehat{H}$.

We denote by $\mbox{vol}(\widehat{H})$ the volume with respect to the metric.
Then, we recall the following theorem in \cite{Katsuda2}. 

\begin{theorem}\label{KatsudaSunadaIHES} If $\Phi$ vanishes on the dual $H^\dagger$, then
\begin{equation}
\pi(x,\alpha) \sim C\frac{e^{hx}}{x^{(b/2)+1}}, \label{Anosovabel}
\end{equation}
where $b = {\rm rank}\; H$ and 
\[C = \frac{1}{(2\pi)^{b/2}h({\rm vol}(\widehat{H})}.\]
\end{theorem}

The next case of our target is a discrete nilpotent group $\Gamma$. This generalization seems to be natural for the following reasons:
If $\Gamma$ is an abelian group, then it is a quotient of the integral first homology group $H_1(M, \mathbb{Z})$, since $H_1(M, \mathbb{Z})$ is a covering transformation group of the maximal abelian covering of $M$. 
By the Hurewicz theorem, it is isomorphic to the abelianization
\[\Gamma \simeq \pi_1(M)/[\pi_1(M),\pi_1(M)],\]
 of $\pi_1(M)$, where $[\pi_1(M),\pi_1(M)]$ denotes the commutator subgroup of $\pi_1(M)$. 
An example of nilpotent group naturally arises as $\Gamma \simeq \pi_1(M)/[\pi_1(M),[\pi_1(M),\pi_1(M)]]$ where $[\pi_1(M),[\pi_1(M),\pi_1(M)]]$ denotes the double commutator subgroup of $\pi_1(M)$. 

Here, we propose the following two conjectures. By taking a finite covering if necessary, we may assume $\Gamma$ is torsion-free.

\begin{conjecture}\label{conj-flow} Let $\varphi_t$ be a smooth, transitive and weakly mixing Anosov flow on a compact
manifold $N$, $\Gamma$ be a finitely generated torsion-free nilpotent group, $\Phi:\pi_1(N) \to \Gamma$ be a surjective homomorphism and $\alpha$ be a conjugacy class of a central element of $\Gamma$. Let $H = \Gamma/[\Gamma,\Gamma]$ and $b$ is the rank of $H$. If the winding cycle $\Psi$ vanishes on the dual $H^\dagger$, then, we have
\[    \pi (x,\Phi, \alpha ) \sim \frac{Ce^{hx}}{x^{1+d/2}},
\]
where 
\[  C = \frac{1}{(2\pi)^{
b/2}h({\rm vol}(\widehat{H}))}\zeta_H(d/2)
\]
where $\zeta_H (d/2)$ the special value at $s = d/2$ of the spectral zeta function $\zeta_H (s)$ of a hypo-elliptic operator $H$ related to some irreducible unitary representations of the Malcev completion $G$ of $\Gamma$, which is the simply connected nilpotent Lie group $G$ containing $\Gamma$ as a lattice subgroup, and $d$ is the polynomial growth order of $\Gamma$. 
\end{conjecture}

Note that for a finitely generated discrete nilpotent group $\Gamma$, the polynomial growth order $d$ of $\Gamma$ can be defined as follows: 
Take a finite set $S$ of generators of $\Gamma$ and let $\omega (k)$ be the number of elements of $\Gamma$ whose word length with respect to $S$ is less than or equal to $k$. Then it is known that there are constants $C_1, C_2 > 0$ and positive integer $d$ such that 
\[   C_1k^d \leq \omega (k) \leq C_2k^d
\]
and  $d$ is independent to the choice of $S$ and called the polynomial growth order of $\Gamma$. 

We notice that an example of a surjective homomorphism $\Phi$ from $\pi_1(N)$ to $\Gamma$ is given in the case when $N$ is the unit tangent bundle $UM$ of a Riemann surface $M$ of genus $g \geq 2$. This homomorphism is obtained as follows: Recall that
\begin{align}
\pi_1(UM) &= \left\langle a_1,b_1, \ldots, a_g, b_g, t \mid \prod_{i=1}^g[a_i,b_i]t^{2g-2} = [a_j, t] = [b_j,t]= e, \; j=1,\ldots, g \right\rangle, \label{piunittanngent} \\ 
\pi_1(M) &= \left\langle a_1,\ldots,a_g,b_1,\ldots,b_g ; \prod_{i=1}^g[a_i,b_i] =1 \right\rangle, \label{presentationsurfcegroup} 
\end{align}
where $[a_i,b_i] = a_ib_i{a_i}^{-1}{b_i}^{-1}$. Then $\Phi$ is written as the composition $\Phi := \Phi_2\circ\Phi_1$ of a map $\Phi_1$ from $\pi_1(N)$ onto the free group $F_g$ of rank $g$ generated by $u_i$, $i = 1,\ldots, g$, defined by 
\begin{equation}\Phi_1(a_i) = u_i, \quad  \Phi_1(b_i) = 1\; (\mbox{identity element}), \quad i = 1,\ldots, g, \quad \Phi_1(t) =1. \label{pitofreetonil}  \end{equation} 
and the canonical quotient map $\Phi_2:F_g \to \Gamma$.

In the case when $\varphi_t$ is the geodesic flow on a unit tangent bundle $N=UM$ of a compact manifold $M$ with negative curvature, the above conjecture is to be strengthened to the asymptotic expansions as follows. Note that in this case, the winding cycle $\Psi$ vanishes on $H^1(N,\mathbb{R})$ and the covariance form  $\delta$ can also be expressed as, for harmonic one forms $\omega,\eta$ on $M$,
\[ \delta (\pi^\ast\omega,\pi^\ast\eta) = \frac{1}{(h/2){\rm vol}\;(M)}\int_M \langle \omega,\eta \rangle 
\]
where $\pi:N=UM \to M$ is the canonical projection (cf. Proposition 1.3 in \cite{Katsuda2}).

\begin{conjecture}\label{conj-geod} Let $M$ be a compact Riemann manifold with negative curvature, $\Gamma$ be a finitely generated torsion-free nilpotent group, $\Phi:\pi_1(M) \to \Gamma$ be a surjective homomorphism and $\alpha$ be a conjugacy class of a central element of $\Gamma$. Then, we have
\begin{equation}    \pi (x,\Phi, \alpha ) \sim \frac{Ce^{hx}}{x^{1+d/2}}\left(1 + \frac{c_1}{x} + \frac{c_2}{x^2} + \cdots \right), \label{asymtoticsnilpotentgeodesics}
\end{equation}
where the constant $C$ is the same as in Conjecture \ref{conj-flow} with $N = UM$ and the constants $c_1, c_2, \ldots$ are expressed in terms of $M$ related to the Lie integrals or Chen's iterated integrals and are explicitly computable.
\end{conjecture}

\begin{remark}
\begin{description}
\item[{\rm (i)}]  One might wonder whether these results follow simply by iterating the abelian case arguments-- after all, a nilpotent group is built from successive abelian extensions. In reality, however, each extension carries nontrivial twisting, and this added complexity is reflected in the fact that nilpotent Lie algebras have been classified only up to dimension six.

\item[{\rm (ii)}] Conjectures of this kind were announced as a theorem in \cite{Katsudainverse}. Unfortunately, the proof outline there contains critical gaps that undermine its validity, so we advise readers to disregard that account. In this note, we provide a fully corrected and detailed proof for the case where 
$M$ is a compact Riemann surface of constant negative curvature. 

\item[{\rm (iii)}] The main theorems presented here and in the next section were first announced in \cite{KatsudaRIMS}, although their applications illustrate rather than drive the paper's core results. An earlier preprint \cite{Katsuda0} treated the three-dimensional Heisenberg group $\mbox{Heis}_3(\mathbb{Z})$.  All of those findings, except for the final chapter, have been incorporated into this work. The discussion of variable negative curvature from the final chapter is postponed to future research (see Section \ref{variable}), as those arguments remain preliminary.
\end{description}
\end{remark}

It seems possible to obtain full conjectures by combining several ideas from dynamical systems developed over the past decades. We plan to treat them in separate papers.

\begin{theorem}\label{theorem-geod} Let $M$ be a compact Riemann surface with constant negative curvature $-1$, $\Gamma$ be a finitely generated torsion-free nilpotent group, $\Phi:\pi_1(M) \to \Gamma$ be a surjective homomorphism and $\alpha$ be a conjugacy class of a central element of $\Gamma$. Then, we have
\[    \pi (x,\Phi, \alpha ) \sim \frac{Ce^{hx}}{x^{1+d/2}}\left(1 + \frac{c_1}{x} + \frac{c_2}{x^2} + \cdots \right),
\]
Here, the constant $C$ is the same as Conjecture \ref{conj-geod}
and the constants $c_1, c_2, \ldots$ are expressed in terms of $M$ related to the Lie integrals or Chen's iterated integrals and are explicitly computable.
\end{theorem}

Moreover, we notice that the explicit numerical values $\zeta_H(d/2)$ are not generally known.

However, we can give them in the following two examples. 

First example is the three dimensional discrete Heisenberg group ${\rm Heis}_3(\mathbb{Z})$
\begin{theorem}\label{heisenberg-geod} Let $M$ be a compact Riemann surface with the constant negative curvature $-1$ of genus $g$ and $\Gamma$ be ${\rm Heis}_3(\mathbb{Z})$ and $\Phi:\pi_1(M) \to \Gamma$ be the standard surjective homomorphism defined by 
\begin{align}\Phi(a_1) &= u := \left( \begin{array}{ccc} 1 & 1 & 0 \\ 0 & 1 & 0 \\ 0 & 0 & 1\end{array}\right),\quad \Phi(a_2) = v := \left( \begin{array}{ccc} 1 & 0 & 0 \\ 0 & 1 & 1 \\ 0 & 0 & 1\end{array}\right),  \notag \\ \Phi(a_i) &= \Phi(b_j) = e := \left( \begin{array}{ccc} 1 & 0 & 0 \\ 0 & 1 & 0 \\ 0 & 0 & 1\end{array}\right)\quad\mbox{for}\quad 3 \leq i \leq g, 1 \leq j \leq g, \label{surfacetoheisenberg}
\end{align}
where $a_1,a_2, \ldots, b_g,\ldots, b_g$ are elements in $\pi_1(M)$ as described in (\ref{presentationsurfcegroup}). 

For a conjugacy class $\alpha$ of a central element of $\Gamma$, we have the following asymptotic expansion:
\[    \pi (x,\Phi,\alpha) \sim \frac{Ce^x}{x^3}\left(1 +\frac{c_1}{x} +\frac{c_2}{x^2} + \cdots \right),
\]
where 
\begin{equation}
C =  \left(\frac{{\rm vol}\; M}{2\pi\|\omega_1\|_{L^2(M)}\|\omega_2\|_{L^2(M)}}\right)^2\zeta_H(2) = \frac{\pi^2(g-1)^2}{8}. \label{heisenbergC}
\end{equation} 
Moreover, the explicit expression of $c_1$ is available in Theorem \ref{sectiontenheisenberg-geod} later.

\end{theorem}

In this case, the polynomial growth order $d$ is $4$, and thus, the exponent of the denominator of the leading term is  $3 = 1+ 4/2$. 
Moreover, the latter equality of (\ref{heisenbergC}) is derived from the fact that the hypo-elliptic operator $H$ in Conjecture \ref{conj-flow} is essentially the harmonic oscillator $H = -\frac{d^2}{ds^2} + s^2$.
The special value $\zeta_H(d/2)=\zeta_H(2)$ can be expressed by the value $\zeta_H(2) = \frac34\zeta(2) = \frac{\pi^2}{8}$ where $\zeta(2) = \frac{\pi^2}{6}$ is the value of Riemann zeta function $\zeta(s)$ at $s=2$.   

If a conjugacy class $\alpha$ is not coming from central elements, then the asymptotic behavior of $\pi (x,\Phi,\alpha)$ can be reduced to the analysis for the case of abelian groups as follows:

\begin{proposition}\label{noncenter} Let $M$ be a compact Riemann surface with the constant negative curvature $-1$ and $\Gamma$ be ${\rm Heis}_3(\mathbb{Z})$ and $\Phi:\pi_1(M) \to \Gamma$ be a surjective homomorphism. For a conjugacy class $\alpha$ of a non-central element of $\Gamma$, we have the following asymptotic expansion:
\begin{equation}
\pi (x,\Phi,\alpha) \sim \frac{Ce^x}{x^2}\left(1 + \frac{c_1}{x} + \frac{c_2}{x^2} + \cdots \right). \label{nonasym}
\end{equation}
where
\begin{equation}
C = \frac{{\rm vol}\; M}{2\pi\|\omega_1\|_{L^2(M)}\|\omega_2\|_{L^2(M)}} = g-1 \label{noncenterCheisenberg}
\end{equation}
Moreover, an algorithm for computations of the other constants $c_1, \ldots$ can be given.
\end{proposition}

One could imagine (conjectural) corresponding results for general nilpotent cases for the non-central conjugacy class; however, the author must be more transparent about their realities.

The next example is that $\Gamma$ is the Engle group $E(\mathbb{Z})$. It is a uniform lattice of the Engel-Lie group $E(\mathbb{R})$ whose Lie algebra $\mbox{Lie}(E(\mathbb{R}))$ is generated by elements $W,X,Y,Z$ with relations 
\[   [W,X]=Y, [W,Y]=Z,
\]
Moreover, the other Lie brackets of generators vanish. Note that this group is the simplest example of nilpotent groups of step three, while the Heisenberg groups are nilpotent groups of step two. 
Moreover, the hypo-elliptic operator $H$ for $E(\mathbb{Z})$ is essentially the quartic operator $-\frac{d^2}{ds^2} + s^4$,  up to some ambiguity of geometric quantities, whose spectral zeta function $\zeta_H$ is investigated in \cite{Voros}.  
In particular, special values of $\zeta_H$ at integers are expressed as certain integrals. 
In our case, we need a value $\zeta_H(7/2)$, and its integral expression is also obtained by generalizing the method in \cite{Voros}. 

\begin{theorem}\label{engel-geod} Let $M$ be a compact Riemann surface with the constant negative curvature $-1$ and $\Gamma$ be ${\rm E}(\mathbb{Z})$ and $\Phi:\pi_1(M) \to \Gamma$ be a surjective homomorphism. For a conjugacy class $\alpha$ of a central element of $\Gamma$, we have the following asymptotic expansions
\[    \pi (x,\Phi,\alpha) \sim \frac{Ce^x}{x^{9/2}}\left(1 +\frac{c_1}{x} +\frac{c_2}{x^2} + \cdots \right),
\]
where the constant $C$ is written by some integral related to the Bessel functions (see Example \ref{Engelquartic}) and an algorithm for computations of the other constants $c_1, \ldots$ can be given.
\end{theorem}

For further examples, see Section \ref{Step0-9}.

\subsection{Long time asymptotics of the heat kernels on covering manifolds}\label{IntroHeat}
Next, we consider the long-time asymptotics of the heat kernels on $\Gamma$-covering manifolds.  We can also apply our extension of the Floquet-Bloch analysis.  In the case when $\Gamma$ is an abelian group, the similarity of the analysis of the prime closed geodesics and the heat kernels was pointed out by Sunada in the late 1990s.  When $\Gamma$ is nilpotent, Boulanger \cite{Boulanger2} reduces the asymptotics of the prime closed geodesics for manifolds with constant negative curvature to that of the heat kernels.   

 Let $M$ be a compact Riemannian manifold or a finite unoriented graph. For a discrete group $\Gamma$, take a normal covering $\pi:X \to M$ with the covering transformation group $\Gamma$. 
We denote by $k_X(t,p,q)$ the heat kernel (resp. the transition probability of simple random walks) of $X$ when $M$ is a compact Riemannian manifold (resp. a finite graph). 
For the brevity of the description of results, we assume that $M$ is not a bipartite graph (for bipartite graphs, see \cite{Kotani3} for the abelian case, and generalizations to a nilpotent case are also treated similarly).

\begin{problem}
What is the asymptotic behavior of $k_X(t,p,q)$ as $t \to \infty$?
\end{problem}

The following results are known when $\Gamma$ is an abelian group.

\begin{theorem}\label{abel-heat} If $\Gamma$ is an abelian group of rank $d$, then  we have
\begin{equation} k_X(t,p,q) \sim \frac{C}{t^{d/2}}\left(1 + \frac{c_1}{t} + \frac{c_2}{t^2} + \cdots \right),
\label{Kotaniheat}\end{equation}
where $C, c_1, c_2, \ldots$ are constants depending on the geometry of $M$.
\end{theorem}
  
This result is due to Kotani and Sunada \cite{Kotani3}. There, for example, the 
A constant $C$ is written as 
\[     C = \frac{{\rm vol}(M)^{d/2}({\rm vol}{\rm Alb}^\Gamma)^{d/2}}{(4\pi)^{d/2}\sharp ({\rm Tor})}
\]
where  ${\rm vol}({\rm Alb}^\Gamma)$ is the Albanese torus associated to $\Gamma$ defined in \cite{Kotani3} and $\sharp ({\rm Tor})$ is the order of torsion subgroup of $\Gamma$ and the explicit expressions of the other constants are also written in \cite{Kotani3}. Note that the Albanese map is a generalization of the Abel-Jacobi map in algebraic geometry; thus, this expression is essentially of the same nature as appeared in the constant expressed by the Jacobi torus for several asymptotic formulas.

If $M$ is compact Riemann surface of genus $g \geq 2$ with constant curvature $-1$ and $\Gamma = H_1(M,\mathbb{Z})$, then $\mbox{Alb}^\Gamma$ is the dual torus of the Jacobi torus $\mbox{J}(M)$.
Then we have ${\rm vol}(\mbox{Alb}^\Gamma) = ({\rm vol}(\mbox{J}(M))^{-1}$.

Several authors obtain the leading term of the above theorem \cite{Guivarch}, \cite{Jacod}, \cite{Kramli}, and \cite{Sinai} independently.
 In particular, Kotani, Shirai, and Sunada \cite{Kotani2} focus 
on the geometric nature of $C$ as above, which is inspired by the corresponding case for prime closed geodesics \cite{Phillips}, \cite{Katsuda1} as in Theorem \ref{abel-geod}. There are also several results for non-symmetric random walks (cf. \cite{Ishiwata2}, \cite{Ishiwata3}).  

In the case of nilpotent groups, we have the following theorem for heat kernels, which corresponds to Conjecture \ref{conj-geod},

\begin{theorem}\label{conj-heat} Let $\Gamma$ be a finitely generated torsion-free nilpotent group, $H = \Gamma/[\Gamma,\Gamma]$ and $b$ is the rank of $H$. Then,

\[    k_X(t,p,q) \sim \frac{C}{t^{d/2}}\left(1 + \frac{c_1}{t} + \frac{c_2}{t^2} + \cdots \right),
\]
where 
\begin{equation}  C = \frac{1}{(2\pi)^{b/2}h({\rm vol}(\widehat{H}))}\zeta_H(d/2) \label{Cheat}
\end{equation}
and the constants $c_1, c_2, \ldots$ are expressed in terms of $M$ related to the Lie integrals or Chen's iterated integrals and are explicitly computable.
\end{theorem}

It should be noted that in the case when $M$ is a finite graph, the asymptotics for the leading term are already obtained by a combination of Alexopoulos \cite{Alexopoulos}, \cite{Alexopoulos2}, \cite{Alexopoulos3} and Ishiwata \cite{Ishiwata} for general nilpotent groups $\Gamma$.
However, their method does not give the geometric nature of the leading coefficient $C$. 
They compare $k_X(t,p,q)$ to the heat kernel of stratified nilpotent Lie groups and then derive the asymptotics, namely the existence of $C$, by simple scaling arguments of the latter.

In this regard, the explicit formulas of the heat kernel of the real Heisenberg group ${\rm Heis}_3(\mathbb{R})$ are known (cf. \cite{Calin}) and thus, as a consequence, we can compute $C$ in this case by combining with their results. 

On the other hand, one advantage of our methods is that we can compute $C$ without referring to the above explicit formulas. This coincidence will be explained in Remark \ref{compariontoexplicitformula}. 
Furthermore, to the author's knowledge, there are no examples of nilpotent Lie groups of steps greater than two whose heat kernels are described explicitly. For example, on page 205, line 8 from the below in \cite{Calin}, there is a comment ``Until now, there are no explicit formulas for Laplacians and sub-Laplacians on nilpotent Lie groups, step greater than three (including the Engel group).'' Note that the phrase ``greater than three '' would be corrected as ``greater than two'' since the Engel group is a three-step nilpotent group. 
There are several related results (cf. \cite{Breuillard}, \cite{Crepel}, \cite{Diaconis}, \cite{Hough}, \cite{Raugi}, \cite{Tanaka}). Moreover, a few extensions to non-symmetric versions of \cite{Ishiwata} are done in \cite{Ishiwata4}, \cite{Ishiwata5}, \cite{Benard}. Moreover, an asymptotic expansion formula is also obtained in \cite{Namba}, which is weaker than Theorem \ref{conj-heat} if it is restricted to the special case that the random walk is symmetric. 

In the case when $\Gamma = {\rm Heis}_3(\mathbb{Z})$ or $E(\mathbb{Z})$, we also obtain the following results corresponding to Theorem \ref{heisenberg-geod} and Theorem \ref{engel-geod}.

\begin{theorem}\label{heisenberg-heat} Let $\Gamma$ be ${\rm Heis}_3(\mathbb{Z})$. 
\[k_X(t,p,q) \sim \frac{C}{t^2}\left(1 +\frac{c_1}{t} + \frac{c_2}{t^2} + \cdots \right),
\]
where the constant $C$ is the same as Theorem \ref{heisenberg-geod} and the constants $c_1, c_2, \ldots$ are written in the quantities of $M$ related to the Lie integrals or Chen's iterated integrals and explicitly computable. 
Moreover, the explicit expression of the constant $c_1$ will be given in Theorem \ref{leadingsecond}.
\end{theorem}

For the Engel group ${\rm E}(\mathbb{Z})$, we have the followings:
\begin{theorem}\label{engle-heat} Let $\Gamma$ be ${\rm E}(\mathbb{Z})$. 
\[k_X(t,p,q) \sim \frac{C}{t^{7/2}}\left(1 +\frac{c_1}{t} + \frac{c_2}{t^2} + \cdots \right),
\]
where $C$ can be written by some integral related to the Bessel functions (see Example \ref{Engelquartic}).
\end{theorem}
and the constants $c_1, c_2, \ldots$ are expressed in terms of $M$ related to the Lie integrals or Chen's iterated integrals and are explicitly computable.
Moreover, an algorithm for computing the other constants $c_1, \ldots$ can be provided.

Three further examples in Section \ref{Step0-9} are also computable by similar arguments, but we leave them for the reader's exercises.

\subsection{The spectrum of the Harper operator, Hofstadter butterfly}\label{Hofstadterbutterfly}

Finally, we shall explain a relation to our methods in the case of the Heisenberg group and the analysis of the discrete magnetic Laplacian or the Harper operator on the square lattice $\mathbb{Z}^2$. It should be noted that their relations are already pointed out by \cite{Beguin}. The celebrated  Hofstadter butterfly expresses the spectrum of the latter operators as shown in the following figure:

\begin{figure}[H]
\centering
\includegraphics[width=50mm]{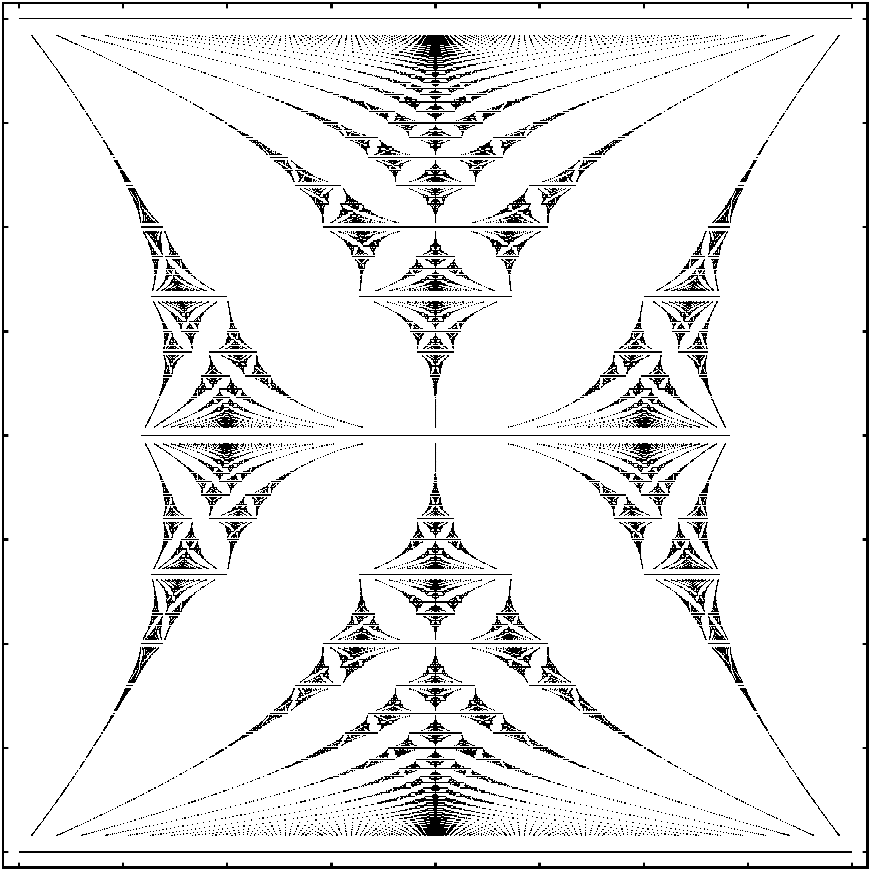}
\caption{{\footnotesize Hofstadter's butterfly (created by Hisashi Naito). The vertical axis represents the value of $\theta$ corresponding to the strength of a magnetic field or the magnetic flux density in the interval $[0,2\pi]$, and the horizontal axis represents the range of the spectrum of $H_\theta$, which is a subset of the closed interval $[-4,4]$.}}
\end{figure}
 
The Harper operator $H_\theta : \ell^2(\mathbb{Z}^2) \to \ell^2(\mathbb{Z}^2)$ is defined as follows:
\begin{equation}
(H_\theta u)(m,n) = u(m+1,n) + u(m-1,n) +e^{\sqrt{-1}m\theta}u(m,n+1) + e^{-\sqrt{-1}m\theta}u(m,n-1) \label{introharperfirst}
\end{equation}
This operator is a discrete analog of the Laplacian (with a shift by $4\mbox{Id}$, where $\mbox{Id}$ is the identity operator) on the plane under a constant magnetic field, which is the Hamiltonian of Landau quantization.

The relation to our extension of Floquet-Bloch theory will be explained in the following paragraphs, and our method for justification of the Wilkinson formula (\ref{introeigen1}) could be considered as the simplest model of the arguments in sections \ref{reductionpathintegral}, \ref{sectionasym}, \ref{asymnilpotent}, but including some essential points there. 

 
First, we note that the analysis of the Harper operator corresponds to the case where the base space $M$ is the graph $\mathcal{G} = (V, E)$, such that the set $V$ of vertices has only one element, $p$. The set $E$ of edges has two loops $u, v$ at the vertex $p$, namely if we realize $\mathcal{G}$ as a one-dimensional complex, $\mathcal{G}$ is the bouquet of two ${S^1}\text{'}s$, i.e., one point sum of two circles. If $\Gamma = {\rm Heis}_3(\mathbb{Z})$, then the $\Gamma$-covering space $X$ is the Cayley graph of $\Gamma$ with generators $\{u,v\}$,
which is described by the following figures:
\begin{figure}[H]
\centering
\includegraphics[width=80mm]{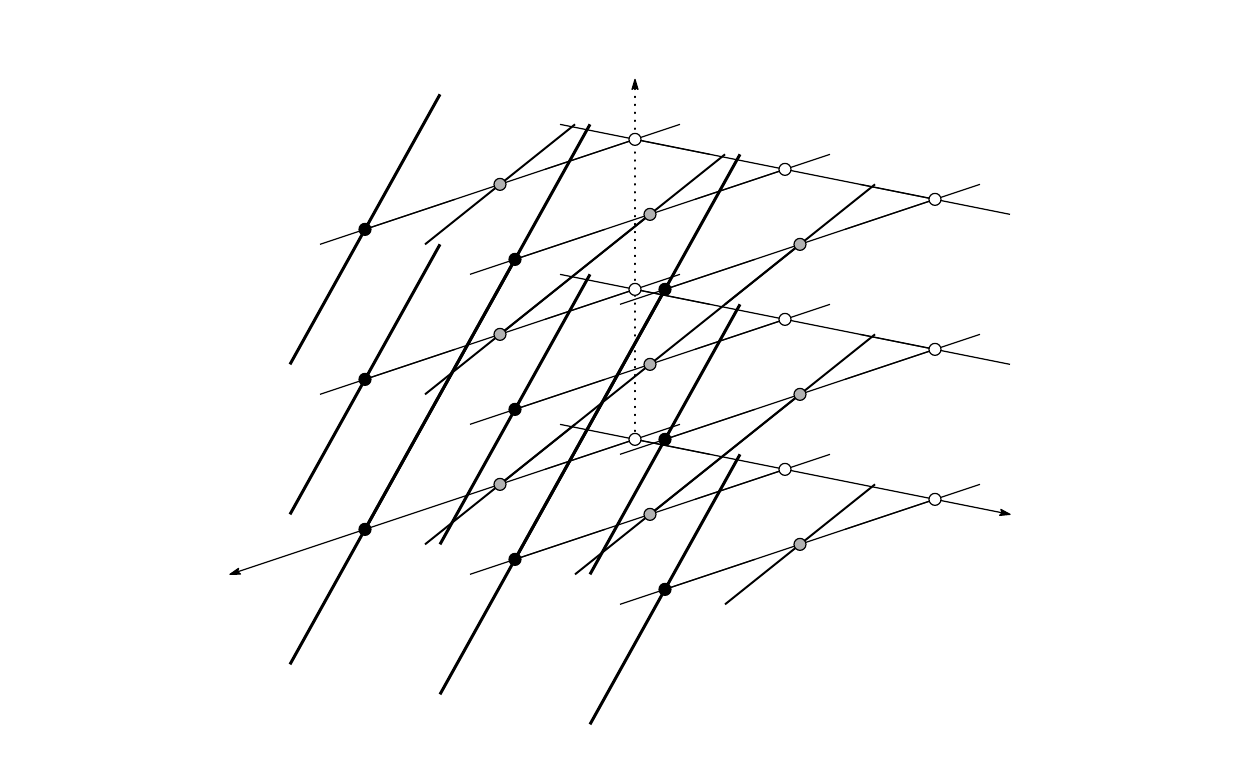}
\caption{{\footnotesize The Cayley graph $X$ of $({\rm Heis}_3(\mathbb{Z}), \{u,v\})$ created by Satoshi Ishiwata: This is a regular graph of degree four embedded in $\mathbb{R}^3$. Solid lines express the edges of the graph, but the dotted line represents the $z$-axis of the orthogonal coordinates of $\mathbb{R}^3$, which does not mean the edges.}}
\end{figure}
Moreover, the square lattice $\mathbb{Z}^2$ appears as an intermediate space of
coverings $\varpi_1:X \to \mathbb{Z}^2$ and $\varpi_2:\mathbb{Z}^2 \to M$ of $\pi = \varpi_2\circ \varpi_1:X \to M$.  For $\theta \in \mathbb{Q}$, the Harper operator $H_\theta$ is a lift (with a shift by $4\mbox{Id}$ with the identity operator $\mbox{Id}$) of the twist discrete Laplacian $\Delta_{\rho_{{\rm \tiny{fin}},x}}$ on $M$ with $x = (0,0,\theta/2\pi) \in \widehat{X} = \widehat{\Gamma}_{\rm fin}$, where $\rho_{{\rm \tiny{fin}},x}$ are finite dimensional irreducible unitary representation appeared in Theorem \ref{Pytlik} and $H_\theta$ (or $H_\theta- 4\mbox{Id}$ acting on $\ell^2(\mathbb{Z}^2)$) can be decomposed as  a direct integral of $\Delta_{\rho_{{\rm \tiny{fin}},x}}$, $x \in \widehat{X}$ in accordance with Theorem \ref{discretetoLie}. These are described by,
\hspace{-20pt}
\begin{figure}[H]
\centering
\includegraphics[width=120mm]{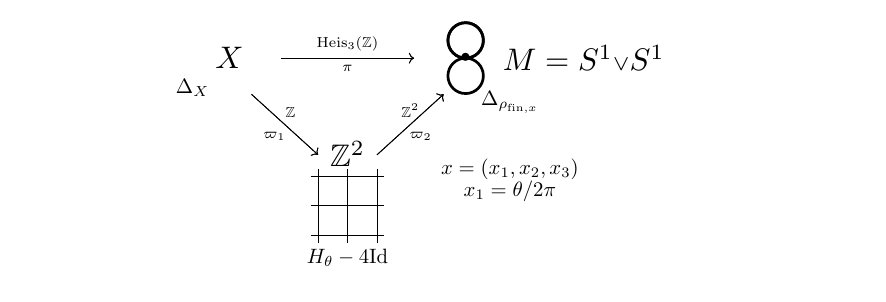}
\caption{{\footnotesize Relation between twisted Laplacian $\Delta_{\rho_{\tiny{{\rm fin}}, x}}$ and Harper operator $H_{\theta}$ in terms of covering of graphs.}}
\end{figure}
Although the structure of the spectrum of $H_\theta$ is far from a complete understanding, it has already been extensively studied (cf.\cite{Satija}). For example, if the parameter $\theta$ is a rational number, then the spectrum of $H_\theta$ is a union of finite intervals, namely it has band structure, and if $\theta$ is an irrational number; it was conjectured that the spectrum is a Cantor set, which was known as M. Kac's ``Ten Martini Problem'' and was finally settled by Avila and Jitomirskaya \cite{Avila}. 
  Our interest here is the following semi-classical asymptotic expansion formula of the spectrum of $H_\theta$ as $\theta \to 0$, where we regard $\theta$ as a semi-classical parameters (3.9 in \cite{Rammal}, similar formulas (6.9 in \cite{Wilkinson} and 6.3.1 in \cite{Helffer}), 
\begin{equation} E_n = -4 + (2n+1)\theta + O(\theta^2) \qquad n= 0,1,2,\ldots \label{introeigen1} \end{equation}
This formula can be used to approximate the Landau level by making the mesh progressively finer (see Remark \ref{meshapproximation} for details).

This is first obtained by Wilkinson \cite{Wilkinson} by WKB arguments. Rammal and Bellissard \cite{Rammal} also derived this formula by a similar method.  Formal derivation of the above expansion is not difficult as follows: It is known that $H_\theta$ is unitary equivalent to the following operator $h_\theta$ acting on $L^2(\mathbb{R})$,
\begin{equation}h_\theta = -2\cos \left(\sqrt{\theta}\frac{d}{\sqrt{-1}ds}\right)-2\cos (\sqrt{\theta}s).  \label{introhtheta}\end{equation}
Note that this expression $h_\theta$ is closely related to the Schr\"odinger representation of the real Heisenberg group ${\rm Heis}_3(\mathbb{R})$.

By the (formal) Taylor expansion formula, this operator can be expressed as 
\begin{equation}
      h_\theta = -4 + \left(-\frac{d^2}{ds^2}+s^2\right)\theta + O(\theta^2) \label{introformal}.
\end{equation}
The coefficient of the linear part in $\theta$ is the harmonic oscillator $\widetilde{\mathcal{H}} := -\frac{d^2}{ds^2}+s^2$ whose eigenvalues are $2n + 1, n=0, 1, 2 ,\ldots$, which implies the above expansion (\ref{introeigen1}). However, (\ref{introformal}) is a form in which a bounded operator $h_\theta$ is approximated by an unbounded operator $\widetilde{\mathcal{H}}$. We need to justify the above arguments mathematically. Note that there are a few other reasons for the necessity of mathematical justifications (see Chapter \ref{Anotherproofwilkinson}).

It is first given by Helffer and Sj\"{o}strand \cite{Helffer} using semi-classical analysis. Their method is described very roughly as follows: If the domain $L^2(\mathbb{R})$ of both operators $h_\theta$ and $\mathcal{H}$ are exhausted by a sequence of common invariant finite-dimensional spaces $V_k$, $k = 1, 2, \ldots$, then the expansion (\ref{introformal}) can be considered as a limit of the Taylor expansion of the restriction $h_\theta\left|_{V_k}\right.$ of $h_\theta$. However, there is no such sequence. Instead of $V_k$, Helffer and Sj\"{o}strand used the space $\widetilde{V}_k$ spanned by eigenfunctions of $\widetilde{\mathcal{H}}$ associated with eigenvalues less than or equal to $k$. This space is not invariant by $h_\theta$, but errors are $O(h^\infty)$ with $h = \theta/2\pi$ and thus do not affect the asymptotic expansion (\ref{introeigen1}). Their methods are based on the semi-classical localizations related to the harmonic oscillators (see also Remark \ref{HSproof}).
     
We shall give another proof based on a comparison between unitary representations of $\mbox{Heis}_3(\mathbb{Z})$ and the real Heisenberg group ${\rm Heis}_3(\mathbb{R})$, which is described in Theorem \ref{introdiscretetoLie}($=$ Theorem \ref{discretetoLie}), a particular case of Theorem \ref{introdecompose}. Note that this relation is somewhat related to the well-known fact that the spectrum of the Harper operator $H_\theta$ for $\theta \in \mathbb{Q}$ has the band structure. 
However, as far as we know, there are no arguments for using this fact to justify mathematically the asymptotic expansion formula (\ref {introeigen1}). 
We believe our method has some advantage over the original proof in \cite{Helffer} since the arguments there need to consider the differential operator $d/ds$ and the multiplication operators conjointly to make the harmonic oscillator. However, our method can handle them separately, which gives more flexible applicability. 

\subsection{A guide for readers who are not familiar with representation theory}
Our results could be applied to various situations appearing in geometry, dynamical systems, and analysis. The central body of arguments is based on a representation theory, especially in orbit philosophy \`{a} la Kirillov in a broad sense.

For the sake of accessibility of some readers who are not familiar with these kinds of subjects, we separate the arguments first in the special case of the Heisenberg group, where we can avoid several terminologies of the orbit method, but describe the essential novelties of this paper. 
For general nilpotent groups, our arguments are primarily based on several results from the literature, supplemented by a few additional arguments from our own work. Thus, there are several quotations from various papers with different notations and terminologies for the same notions. 
We do not try to unify them, but instead, to give a translation to each of them. 
We fear that this may lead to confusion. However, on the other hand, there is a positive aspect to preventing misunderstandings by unifying and making it convenient to refer to each original literature.  

In addition, the arguments are several variants of the fundamental case of nilpotent Lie groups. Thus, we add a summary of this case in Appendices B to compare Chapter \ref{Representationsdiscretenilpotent} and explain examples in Section \ref{Step0-9}. These parts are essentially quoted from \cite{Corwin}, and nothings are included in our originality.

\subsection*{Use of AI and LLM} 
Artificial Intelligence (Microsoft Copilot) was employed primarily for copy-editing purposes.
Additionally, in Subsections 11.12.6, 11.12.7, and 11.12.8, it was utilized to gather supplementary information concerning previously established results. Any inaccuracies generated by the AI during this process were subsequently corrected.
(For reference, see the differences in Subsection 11.12.7 between version 2 and version 3 of arXiv:2509.16848.)

\subsection*{Acknowledgements}
\addcontentsline{toc}{subsection}{Acknowledgements}

First and foremost, we are deeply grateful to Hidenori Fujiwara for his expert guidance on the representation theory of nilpotent Lie groups, which lies at the heart of this project.  In particular, his insightful note \cite{Fujiwara2} and his comprehensive article \cite{Fujiwara1} --- as discussed in Section \ref{nilpotentfujiwara} --- have been indispensable.

We also wish to thank Ali Baklouti, Fumio Hiroshima, Takashi Iwamoto, Satoshi Ishiwata, Tsuyoshi Kajiwara, Hiroshi Kawabi, Motoko Kotani, Hisashi Naito, Ryuya Namba, Hiroyuki Ochiai, Toshikazu Sunada, and Tatsuya Tate for many stimulating discussions and valuable suggestions.  Our thanks extend to Francis Brown, Beno\^it Collins, Kazuhiro Hikami, Katsushi Ito, Masanobu Kaneko, Akishi Kato, Morimichi Kawasaki, Minsung Kim, Nobushige Kurokawa, Toshiki Matsusaka, Masanori Morishita, Satoshi Naito, Peter Sarnak, Yuji Terashima, and Yoshimichi Ueda for generously sharing information on the topics covered in the final chapter.

Last but not least, we offer our profound thanks to Toshikazu Sunada for introducing these themes to us and for his steadfast advice and encouragement over the past forty years.



\section{Outline of the proofs: Floquet-Bloch theory and applications to several asymptotic formulas}\label{outline}

In this chapter, we first outline our version of the Floquet-Bloch theory, along with its applications to the asymptotics of heat kernels. To clarify the distinction between parts of the Floquet-Bloch theory and other parts, mainly perturbation arguments for the twisted Laplacians, we denote (F's) for the former and (P's) for the latter in the starting points of the arguments.

\subsection{Finite extensions}
\begin{description}
\item[(P0)] Our concern is long-time asymptotics of the heat kernel $k_X(t,p,q)$ on $X$ for a $\Gamma$-covering $\pi:X \to M$. We begin with the easier case where $\Gamma$ is a finite group. In this case, $X$ is compact and thus, $k_X(t,p,q)$ can be expressed as follows:
\begin{equation}       k_X(t,p,q) = \sum_{i=0}^\infty e^{-\lambda_i t}\varphi_i(p)\varphi_i(q), \label{heatcompact}  
\end{equation}
where 
\[ 0 = \lambda_0 < \lambda_1 \leq \lambda_2 \leq \cdots  
\]
are eigenvalues of the Laplacian $\Delta_X$ on $X$ and $\{\varphi_i\}_{i=0}^\infty$ is a complete orthonormal system of eigenfunctions $\varphi_i$ associated to eigenvalues $\lambda_i$.  Note that the $0$-th eigenfunction $\varphi_0$ is a constant function $1/\sqrt{{\rm vol}(X)}$ with the volume ${\rm vol}(X)$ of $X$.  If $i \geq 1$, then $\lambda_i > 0$  and thus $e^{-\lambda_i t}$ decays to $0$ if $t \to \infty$. Then, we have 

\[   \lim_{t\to\infty}k_X(t,p,q) = \lim_{t\to\infty}e^{-\lambda_0t}\varphi_0(p)\varphi_0(q) =  \varphi_0(p)\varphi_0(q) = \frac{1}{{\rm vol}(X)}.  
\]
\end{description}

\subsection{Infinite abelian extensions}\label{Infiniteabelian}
Here, we recall the strategy when $\Gamma$ is an infinite abelian group. Let us recall the proof of the leading term of the asymptotic expansion (\ref{Kotaniheat}). Namely, we shall show that if $\Gamma$ is an infinite abelian group of rank $d$, then there exists a constant $C > 0$ such that
\begin{equation} k_X(t,p,q) \sim \frac{C}{t^{d/2}},
\label{leadingKotaniheat}\end{equation}
   
For simplicity, we assume $\Gamma = \mathbb{Z}^d= H_1(M,\mathbb{Z})$. Essential points in the arguments already appear in this case, and it is relatively easy to extend to general finitely generated abelian groups (cf. \cite {Kotani3} for detailed arguments). Here $X$ is noncompact. Although a similar formula to (\ref{heatcompact}) exists in this case, using the spectral decomposition of $\Delta_X$, it appears to be useless for our problem since it contains parts of a continuous spectrum.
\begin{description}
\item[Step 1 (F1)] Instead of the above spectral decomposition formula, we regard the domain $L^2(X)$ of the Laplacian $\Delta_X$ on $X$ as a space ``$L^2(M) \otimes L^2(\Gamma )$'',  rigorously formulated as a space of $L^2$-sections $L^2(E_R)$ of the flat vector bundle associated to the right regular representation of $\Gamma$, a.k.a. local system associated to $R$, where $E_R$ is defined as follows:
\[   E_R \simeq X \times L^2(\Gamma )/\sim, \hspace{12pt}  (p,\varphi) \sim (\gamma p, R(\gamma^{-1})\varphi )\quad \mbox{for}\quad \gamma \in \Gamma.\]

Note that if we regard the covering $X \to M$ as a ``$\Gamma$-principal bundle over $M$'', then $E_R$ is its associated vector bundle with respect to the representation $R$. Moreover, this construction is similar to the representation space of an induced representation (cf. (\ref{inducedrepresentationbundle}) in Section \ref{Inducedrep}).

When $\Gamma$ is abelian,  we can apply the usual Floquet-Bloch theory. To adjust later generalizations, we formulate the Floquet-Bloch theory geometrically as follows: Since $\Gamma = \mathbb{Z}^d$, the right regular representation $R$ is decomposed to a direct integral of one-dimensional irreducible unitary representations, i.e., characters, over the unitary dual $\widehat{\Gamma}$ of $\Gamma$ which is the space of equivalence class of characters $\chi: \Gamma \to U(1)$ with one-dimensional unitary group $U(1)$. Note that $\widehat{\Gamma}$ is isomorphic to a $d$-dimensional torus $U(1)^d$, which can be identified with the Brillouin zone in the usual terminology of the Floquet-Bloch theory in condensed matter physics. Associated with this formula, we have the following direct integral decomposition. 
\[ L^2(E_R) \simeq \int_{\widehat{\Gamma}}^\oplus L^2(E_\chi )d\chi,
\] 
where $L^2(E_\chi )$ is the space of sections of the flat line bundle associated with $\chi$, similarly defined as $L^2(E_R)$. The space $L^2(E_\chi)$ can be identified with the space 
\[ H_\chi = \{ f: X \to \mathbb{C} \mid f(\gamma^{-1}p) =\chi (\gamma ) f(p) \; \mbox{for all}\; \gamma \in \Gamma\}. 
\]
Using this identification, the heat kernel $k_X(t,p,q)$ can be decomposed as 
\begin{equation}  k_X(t,p,q) = \int_{\widehat{\Gamma}}k_\chi(t,p,q)d\chi  \label{Fourierinversionabel}
\end{equation}
where the integrand $k_\chi(t,p,q)$, which is sometimes called the twisted heat kernel, is the Fourier transform of $k_X(t,p,q)$ with respect to $\Gamma = \mathbb{Z}^d$, which is defined as
\begin{equation} 
k_\chi(t,p,q) = \sum_{\gamma \in \Gamma}\chi(\gamma^{-1} )k_X(t,p,\gamma q).  
\end{equation} 
Moreover, (\ref{Fourierinversionabel}) is the Fourier inversion formula.

Moreover, $k_\chi(t,p,q)$ can be expressed by a twisted average as follows:
\begin{equation} 
k_\chi(t,p,q) = \sum_{i=0}^\infty e^{-\lambda _i(\chi)t}\varphi_{i,\chi}(p)\varphi_{i,\chi}(q), \label{twisted eigen}
\end{equation}
where $\lambda_i(\chi )$ and $\varphi_{i,\chi}$ are the $i$-th eigenvalue and eigenfunction of the twisted Laplacian $\Delta_\chi := \Delta_X\left|_{H_\chi}\right.$. 
Note that $\lambda_i(\chi )$ are discrete as eigenvalues $\lambda_i$ of the usual Laplacian $\Delta_M$ on $M$ and $\lambda_i = \lambda_i(\textbf{1})$ for the trivial character $\textbf{1}$.

\item[Step 2 (P1)] The second step analyzes eigenvalues $\lambda_i(\chi )$.  
It is not difficult to show that there is a positive constant $c$ such that $\lambda_i(\chi ) \geq c$ for $i \geq 1$. We also have 
\begin{equation}\label{first}
\lambda_0(\chi ) \geq 0,\qquad \lambda_0(\chi ) = 0 \Longleftrightarrow  \chi = \textbf{1}
\end{equation}
and $\lambda_0(\chi )$ depends
 on $\chi$ smoothly. From this fact, we easily find the first derivative with respect to $\chi$ at $\textbf{1}$ is zero. The essential point is to show that the Hessian of $\lambda_0(\chi)$ at $\chi = \textbf{1}$ is positive definite.

To prove this fact, we use perturbation arguments as follows:
Although the usual setting of the perturbation theory is that operators defined on a fixed domain are varied, our situation here is the converse in that the defining domains $H_\chi$ are varied. However, the operator $\Delta_\chi$ is fixed since it is the restriction of the Laplacian $\Delta_X$ on $X$ to the space $H_\chi$.
 
To reduce our situation to the usual setting of the perturbation theory, we construct a canonical section $s_\chi$ of the line bundle $E_\chi$ (or better, consider it as a section of $U(1)$-principal bundle) and identify $L^2(M)$ with $L^2(E_\chi)$ by the following correspondence:
\[     L^2(M) \ni f \longleftrightarrow fs_\chi \in L^2(E_\chi) \simeq H_\chi.\]
This procedure is a kind of ``Gauge trasformations''.

Associated with this correspondence, we have a unitary equivalence between the twisted operator $L_\chi:= s_\chi^{-1}\circ\Delta_\chi\circ s_\chi$ acting on $L^2(M)$ and the twisted Laplacian $\Delta_\chi$ acting on $L^2(E_\chi ) \simeq H_\chi$.

To construct $s_\chi$, by the de Rham-Hodge Theorem, we can take a harmonic $1$-form $\omega$ satisfying
 \begin{equation} \chi ([\gamma ]) := \chi_\omega([\gamma ]) := \exp \left(2\pi \sqrt{-1}\int_\gamma \omega\right) \notag
\end{equation}
for a closed curve $\gamma$ and its homology class $[\gamma ] \in H_1(M, \mathbb{Z})$. 

Since $\chi$ is a homomorphism from $\Gamma \simeq H_1(M,\mathbb{Z}) \simeq \pi_1/[\pi_1,\pi_1]$ to an abelian group $U(1)$, we can lift $\chi$ to a character $\widetilde{\chi}:\pi_1(M) \to U(1)$ canonically. 
Then, taking a lift $\widetilde{\omega}$ of $\omega$ to the universal covering $\widetilde{M}$ of $M$, we define a function $\widetilde{s}_{\widetilde{\omega}}$ on $\widetilde{M}$ by
\begin{equation} \widetilde{s}_{\widetilde{\omega}} (p) = \exp \left(2\pi\sqrt{-1}\int_{p_0}^p\widetilde{\omega}\right) \label{lineintegral}
\end{equation} 
for some reference point $p_0 \in \widetilde{M}$. 
Note that this is well defined since the line integral $\int_{p_0}^p\widetilde{\omega}$ does not depend on the choice of a curve from $p_0$ to $p$ by the homotopy invariance property because $\widetilde{\omega}$ is a closed form. 
Since $\widetilde{s}_{\widetilde{\omega}}(\gamma p) = \chi_{\omega}([\gamma ])\widetilde{s}_{\widetilde{\omega}}(p)$, this function can be identified with a section $s_\omega = s_{\chi_\omega}$ of $L^2(E_\chi)$. 

Then $L_\chi= s_{\chi_\omega}^{-1}\circ\Delta_{\chi_\omega}\circ s_{\chi_\omega}$ can be computed as
\[  L_\chi f = \Delta_Mf-4\pi\sqrt{-1} \langle \omega ,df \rangle + 4\pi^2|\omega |^2f
\]
where $\langle \cdot, \cdot \rangle$, and $|\cdot |$ are the inner product and the norm of the cotangent space induced from the Riemannian metric of $M$, respectively. 

From this expression, we can compute the hessian ${\rm Hess}_0 \lambda_0$ of $\lambda_0(\chi )$ at $\chi =\chi_\omega = \textbf{1}$ (i.e. $\omega = 0$) which is written as
\[     {\rm Hess}_0\lambda_0 (\omega,\omega) = \frac{8\pi^2}{{\rm vol}(M)}\int_M |\omega|^2 dv_g
\]   
where ${\rm vol}(M)$ is the Riemannian volume and $dv_g$ is
 the Riemannian measure of $M$.

By the Morse lemma, we can take local coordinates $(U, (x^1, x^2, \ldots, x^d))$ around the trivial character $\textbf{1}$ on $\widehat{\Gamma} \simeq (U(1))^d$ such that 
\[   \lambda_0(\chi) = \lambda_0(x^1, x^2, \ldots, x^d) = \sum_{k=1}^d (x^i)^2. 
\] 
Then, we get the conclusion by the following computation:
\begin{align*} k_X(t,p,q) &= \int_{\widehat{\Gamma }}k_\chi(t,p,q) d\chi \\
&= \int_{\widehat{\Gamma }}\sum_{n=0}^\infty e^{-\lambda _n(\chi)t}\varphi_{n,\chi}(p)\varphi_{n,\chi}(q)  d\chi \\
&= \int_Ue^{-\lambda _0(\chi)t}\varphi_{0,\chi}(p)\varphi_{0,\chi}(q)  d\chi + \mbox{error term} \\
&= C\int_U e^{-\sum_{i=1}^d(x^i)^2t}dx^1dx^2\ldots dx^d + \mbox{error term} \\ &\sim C\prod_{i=1}^d \left\{\frac{1}{\sqrt{t}}\int_{-\infty}^{\infty}e^{-{(y^i)}^2}dy^i\right\} \\
&= \frac{C}{t^{d/2}},
\end{align*}
where $U$ is a neighborhood of $\textbf{1}$ in $\widehat{\Gamma}$. Note that this argument is sometimes called the Laplace method.
\end{description}

\subsection{Heisenberg extensions}\label{Heisenbergexplanation}
For the generalization of the above arguments to nilpotent groups, we point out arising difficulties and give a strategy to overcome them in the case when $\Gamma = {\rm Heis}_3(\mathbb{Z})$. The reasons why we explain this case separately are the following:
\begin{description}
\item[{\rm (1)}] Representation theory of ${\rm Heis}_3(\mathbb{R})$ and some part of the theory of ${\rm Heis}_3(\mathbb{Z})$ are explicit and can be explained as a concrete form without knowledge of orbit methods.

\item[{\rm (2)}] Main innovation in this paper has already appeared in this case as the simplest form.
\end{description}

In the following, we use the notations $\rho$ or $\rho_{\rm \tiny{fin}}, \ldots$ for representations of the Heisenberg group ${\rm Heis}_3(\mathbb{Z})$ and $\pi$ or $\pi_{\rm \tiny{fin}}, \ldots$ for those of general discrete torsion-free  nilpotent groups $\Gamma$.
\begin{description}
\item[Step 1 (F2)] In the first step of the previous arguments, we have used the irreducible decomposition formula of the right regular representation $R$ of abelian groups. However, if $\Gamma$ is nonabelian nilpotent, its representation theory is not of type I. In this case, although there exist abstract decomposition formulas of the regular representation,  it is not unique, and the unitary dual $\widehat{\Gamma}$ is a wild space. In conclusion, the (whole) unitary dual has no practical, computable formula.

In the case that $\Gamma = {\rm Heis}_3(\mathbb{Z})$, if we restrict to finite-dimensional unitary representations, we can use the Plancherel formula (Theorem \ref{Pytlik}) due to Pytlik \cite{Pytlik} or its corollary, the Fourier inversion formula (\ref{Fourier1}) (special case of Theorem \ref{introPytlik} $=$ Theorem \ref{newPytlik}). 
The Plancherel measure $\mu$ is finitely additive and supported on finite-dimensional irreducible unitary representations in this formula. This formula is indeed useful. However, the dimensions of the representations appearing in the formula vary, making it difficult to apply perturbation arguments similar to those in Step 2 (P1) of the previous section. To overcome this point, we relate these finite dimensional representations to infinite dimensional irreducible unitary representations, which are called the Schr\"{o}dinger representations of ${\rm Heis}_3(\mathbb{R})$. These representations $\rho_h$ are parametrized by $h \in \mathbb{R} \setminus \{0\}$ and their representation spaces are common $L^2(\mathbb{R})$ for all $h \in \mathbb{R} \setminus \{0\}$.
   
Let us explain briefly the above relation, which is the content of Theorem \ref{introdiscretetoLie} ($=$ Theorem \ref{discretetoLie}).  Finite dimensional irreducible unitary representations $\rho_{{\rm \tiny{fin}},x}$ of $\Gamma$ are parametrized by $x = (x_1,x_2,x_3) \in \widehat{X}:=  (\mathbb{Q} \cap [0,1])\times  [0,1] \times [0,1]$ and the above Plancherel measure $\mu$ is the product measure of the Lebesgues measure $m$ on the interval $[0,1]$ for the second and third factors and the finitely additive measure $\widetilde{m}$ on the first factor $\mathbb{Q} \cap [0,1]$ determined by
\[         \widetilde{m}(\mathbb{Q} \cap [a,b]) = b-a.
\]  
The representations $\rho_{{\rm \tiny{fin}},x}$ are essentially determined from $x_1$, which is, in some sense, a discrete version of the Schr\"{o}dinger representation. If we express $x_1$ as the irreducible fraction $p/q$, then the dimension of $\rho_{{\rm \tiny{fin}},x}$ is $q$ and the role of $x_2, x_3$ for $\rho_{{\rm \tiny{fin}},x}, x = (x_1,x_2,x_3)$ are abelian perturbations of $\rho_{{\rm \tiny{fin}},(x_1,0,0)}$.
The point of our arguments is the following fact: 

\begin{fact}\label{factfluctuation}
The fluctuation by $x_2, x_3$ is like $O(1/q)$ which is independent to the numerator $p$ of $x_1 = p/q$.
\end{fact}

  Moreover, the subset of $\widehat{X}$ consisting of $x=(x_1,x_2,x_3)$ whose first component $x_1 =p/q$ has the denominator smaller than a fixed constant is a null set with respect to the above Plancherel measure $\mu$.

  To relate the Schr\"{o}dinger representation $\rho_h$, we note that if we restrict $\rho_h$ to the discrete subgroup $\Gamma$ when the parameter $h = x_1$ is a rational number $p/q$, the restriction $\rho_h\left|_\Gamma\right.$ to $\Gamma$ of $\rho_h$ can be decomposed as a direct integral 
\begin{equation}\rho_h\left|_\Gamma\right. = {\int_0^1\!\int_0^1}^\oplus\rho_{{\rm \tiny{fin}},(x_1, \{qx_3x_2\}, x_3)}dm(x_2)dm(x_3),\label{decompositiontofinite}
\end{equation} 
where $\{a\}$ denotes the fractional part $a - [a]$. 
We regard this decomposition formula as the approximation formula of the left-hand side $\rho_h|_\Gamma$ by $\rho_{{\rm \tiny{fin}},x}$. 
It means that the left-hand side is the ``average'' of the integrand on the right-hand side, and the fluctuation of the integrand is $O(1/q)$. 
Moreover, for any $ h$ near $0$, it can be approximated by rational numbers with arbitrarily large denominators, which implies that an infinite-dimensional representation $\rho_h$ can be approximated by finite-dimensional representations $\rho_{{\rm \tiny{fin}},x}$ to arbitrary orders.  
This fact means that for a mathematical justification of problems of convergence, e.g., the Fourier transform $\rho_h(f)$ of $L^1$-function $f$ does not in the trace class in general, we use the above Plancherel formula which includes only finite-dimensional representations but for some formal computation of the eigenvalues, we can use the Schr\"{o}dinger representation since it varies smoothly in $h$.
We can freely exchange finite-dimensional representations and infinite-dimensional representations. This procedure is what we call the extension of the Floquet-Bloch theory to the Heisenberg group.
  
\item[Step 2 (P2)] As in Step 2 (P1) in the previous section, we need to investigate the behavior of the eigenvalues of Laplacian $\Delta_{\rho_{{\rm \tiny{fin}},x}}$ twisted by $\rho_{{\rm \tiny{fin}},x}$ near the trivial representation ${\bf 1}$. By Step 1 (F2), it can be reduced to a formal computation of the twisted Laplacian $\Delta_{\rho_h}$ associated to $\rho_h$, replacing $\Delta_{\rho_{{\rm \tiny{fin}},x}}$. These formal computations are, in spirit, the same as the abelian case. However, technically speaking, it is necessary to modify several points in the arguments to proceed with actual computations. One of them is to replace the line integral $\int_{p_0}^p\widetilde{\omega}$ in (\ref{lineintegral}) by the Lie integral of the Lie algebra valued $1$-form (cf. \cite{Benardete}), which is equivalent to Chen's iterated integrals \cite{Chen}. Moreover, our computation is somewhat complicated since we need to carry out infinite rank vector bundles. For example, in this procedure,  the hessian ${\rm Hess}_0 \lambda_0$ should be replaced by the quadratic forms associated with a modified harmonic oscillator $\mathcal{H}$ acting on the fiber $L^2(\mathbb{R})$ of $E_{\rho_h}(L^2(\mathbb{R}))$. However, the arguments are essentially a blending of the so-called Schr\"{o}dinger method in physics literature (cf. \cite{Rammal},\cite{Sakurai}) and a computation carried out in \cite{Kotani3} in the abelian case.   
\end{description}

\subsection{Nilpotent extensions}\label{nilpotentexplanation}

Our practical strategies to deal with these cases mentioned in Section \ref{Introprac} are summarized as follows:
\begin{assertion}\label{guidingassertion}
A practical device for the asymptotics can be obtained from a computation of unitary representations of the nilpotent Lie group $G$, which is the Malcev completion of a discrete group $\Gamma$ near the trivial representation.   
\end{assertion}

We shall justify this assertion in the sequel by applying Theorem \ref{introPytlik} and the generalization to nilpotent groups of Theorem \ref{discretetoLie}. According to the following arguments, it can be proved almost together.
Before proceeding further, we first discuss the detailed statement of the latter.

It is known such a simply connected nilpotent Lie group $G$ admits a lattice $\Gamma$, then it has rational structure ($\mathbb{Q}$-structure, Theorem 5.1.8  in \cite{Corwin}), namely there is a basis $X_1, \ldots, X_n$ of the Lie algebra $\mathfrak{g}$ of $G$ such that the structure constants $c_{ij}^k \in \mathbb{Q}$ which is defined by:
\[  [X_i,X_j] = \sum_{k=1}^nc_{ij}^kX_k,\quad (i,j = 1,2,\ldots,n)  
\]
Let $\mathfrak{g}_\mathbb{Q}$ be the $\mathbb{Q}$-span of the above basis and
$\mathfrak{g}^\ast$ (resp. $\mathfrak{g}^\ast_\mathbb{Q}$) the dual vector space (resp. dual rational vector space) of $\mathfrak{g}$ (resp. $\mathfrak{g}_\mathbb{Q}$ ). 
For an element $l \in \mathfrak{g}^\ast$, its radical $\mathfrak{r}_l$ is defined as the set of elements $X \in \mathfrak{g}$ satisfying $l([X,Y]) = 0$ for any $Y \in \mathfrak{g}$ and a maximal subordinate algebra (a.k.a. polarizing algebra or polarization) $\mathfrak{m}$ is a maximal one among subalgebras $\mathfrak{n}$ of $\mathfrak{g}$ satisfying $l([X,Y]) = 0$ for any $X, Y  \in \mathfrak{n}$. Note that $\mathfrak{m}$ is not uniquely defined in general for $l$. These are said to be rational if $l \in \mathfrak{g}^\ast_\mathbb{Q}$ and $\mathfrak{r}_l, \mathfrak{m} \subset \mathfrak{g}^\ast_\mathbb{Q}$ respectively. Let $\mathbb{Q}^{\rm odd}$ be the set of rational numbers that have odd denominators as their irreducible fraction expressions.

\begin{theorem}{{\rm (} $=$ Theorem \ref{sec4decompose} {\rm ) }}\label{introdecompose}
Assume that $l$ is rational. Then there exists a polarization $\mathfrak{m}$ of dimension $m$, a rational subalgebra of $\mathfrak{g}$ satisfying the following. Let $\{X_1,\ldots, X_m, \dots X_n\}$ be a weak
Malcev basis of $\mathfrak{g}_\mathbb{Q}$, based on $\Gamma$ and passing through $\mathfrak{m}$.

Then, we have the decomposition
\begin{equation} \pi_l|_\Gamma = \int_{[0,1)^{n-m}}^\oplus {\rm Ind}_{M\cap\Gamma}^\Gamma (\chi_{l_t |_{M\cap\Gamma})}dt_{m+1}\cdots dt_n, \label{intropildecomposition}
\end{equation}
where 
\begin{equation*}
l_t := l_{t_{m+1},\ldots,t_n} = {\rm Ad}^\ast (\exp(t_{m+1}X_{m+1}))\cdots{\rm Ad}^\ast(\exp(t_nX_n))l \end{equation*}

for all $t_{m+1}, \ldots t_n \in [0, 1)$.

Moreover if $t = (t_{m+1}, \ldots t_n ) \in ([0,1) \cap \mathbb{Q}^{\rm odd})^{n-m}$, then integrand of the right hand side of (\ref{intropildecomposition}) is further decomposed into finite dimensional unitary representations $\mbox{Ind}_{\mathcal{P}}^\Gamma (\chi_{l_t}\otimes\chi_{\psi_s})$ as:
\begin{equation}
\mbox{Ind}_{M\cap\Gamma}^\Gamma(\chi_{l_t}) = \mbox{Ind}_{\mathcal{P}}^\Gamma( \mbox{Ind}_{M\cap\Gamma}^{\mathcal{P}}(\chi_{l_t})) 
= \int_{[0,1)^{n-m}}^\oplus\mbox{Ind}_{\mathcal{P}}^\Gamma(\chi_{l_t}\otimes \chi_{\psi_s})ds \label{finaldecomposition}
\end{equation}
where $M = \exp \;\mathfrak{m}$ 
and $\mathcal{P}$ is the polarization of $\log \;\Gamma$. (These notions will be explained in detail in Subsections \ref{Step0-5}, \ref{Step0-8}.) 
\end{theorem}

\begin{remark}
\begin{description}
\item[{\rm (1)}]These results are obtained in the procedure of proving the previous theorem. They will provide helpful computational tools and the main conceptual novelty in this paper. We think the most essential novelty of this paper is Theorem \ref{discretetoLie}, although the idea of the proof is not any more difficult. This theorem is a prototype of Theorem \ref{sec4decompose} and almost all arguments here can be reduced to rather simple arguments on abelian groups in virtue of several previous works. 
\item[{\rm (2)}] 
The proofs of the above theorems are heavily dependent on works of several experts, Bekka, Driutti, Fujiwara, Greenleaf, Howe, Johnston, $\ldots$ as in the arguments in Chapter \ref{Representationsdiscretenilpotent}. However, according to the inductive nature of the orbit method, there are some possibilities to be done more naively by repeated use of Theorem \ref{discretetoLie}, which is the simplest version of the above Theorem \ref{introdecompose} and would be proved directly without using the orbit method, in each induction step. 
\end{description}  
\end{remark}

The crucial point for the asymptotics can be obtained from a computation of unitary representations of the nilpotent Lie group $G$, which is the Malcev completion of a discrete group $\Gamma$ near the trivial representation.   

\begin{description}
\item[Step 1 (F3)] 
 Since there seems to be no known generalization of the Pytlik theorem for ${\rm Heis}_3(\mathbb{Z})$ to general nilpotent groups, we recover them as follows: We first decompose the right regular representation $R$ of a discrete nilpotent group $\Gamma$ more coarsely, into factor representations $\tau_\lambda$. Then, we extend $\tau_\lambda$ to a monomial representation $\tau_\lambda^G$ of $G$ and decompose it into irreducible representations $\pi_l$. Then, if its restriction $\pi_l|_\Gamma$ to $\Gamma$ satisfies some ``rationality conditions'', we further decompose as a direct integral of finite-dimensional representations of $\Gamma$ through two steps.  

The following is a little more explanation:

\begin{description}
\item[Substep 0-1] (Section \ref{Step0-1}) First we use a direct integral decomposition formula of the right regular representation $R$ of $\Gamma$ to their ``factor'' representations $\tau_\lambda$ due to \cite{Johnston} or \cite{Bekka2}, see \cite{Thoma}, \cite{Baggett} for more general contexts. The reason why we wrote ``factor'' is that generic representations of integrands are indeed factor representations. However, there are non-factor representations that consist of dense subsets in the defining domain of the Plancherel measure. 

In addition, we note that all of the above representations $\tau_\lambda$ are obtained as induced representations from characters $\lambda$ of the center of $\Gamma$, and their representation spaces are isomorphic subspaces of $L^2(\Gamma)$. 

\item[Substep 0-2] (Section \ref{nilpotentfujiwara}) Then, since the group multiplication law of $\Gamma$ can be expressed as forms of polynomials of certain coordinate functions, similarly as construction of the Malcev completion $G$ from $\Gamma$, we can extend $\tau_\lambda$ to representations $\tau_\lambda^G$ of the nilpotent Lie group $G$ whose representation spaces are isomorphic subspaces of $L^2(G)$.

In this case, we remark that if $\Gamma = {\rm Heis}_3(\mathbb{Z})$, then $\tau_\lambda^G$ is the infinite direct sum of a copy of fixed Sch\"odinger representation $\rho_h$ with $l(Z+Y) = hZ^\ast(Z) = h = \lambda(Z)$ where $Z$ is an element in the center. We take a polarized algebra $\mathfrak{m} = \{Z,Y\}$ of the Lie algebra $\mathfrak{g}$ of $G = {\rm Heis}_3(\mathbb{R})$. Moreover, $Z^\ast$ is the dual element of $Z$ in the vector space dual $\mathfrak{g}^\ast$ of $\mathfrak{g}$. This formula is a particular case of  (\ref{monomialmulti}).       

\item[Substep 0-3] (Section \ref{nilpotentfujiwara}) Then, $\tau_\lambda^G$ is a monomial representation. By Fujiwara's Plancherel theorem (Theor\'{e}m\`{e} 2 in \cite{Fujiwara1}), $\tau_\lambda^G$ can be decomposed as a direct integral of irreducible unitary representations $\pi_l$, which are parametrized by certain parameters $l$, which is the formula (\ref{monomialmulti}). 

\item[Substep 0-4] (Section \ref{stratification}) In the arguments of Section \ref{stratification}, by later arguments for asymptotic expansions, we may adssume $G$ is stratified in later sections. Namely, the Lie algebra $\mathfrak{g}$ of $G$ is generated by $\mathfrak{g}^{(1)}$ in the following decomposition:
\[\mathfrak{g} = \mathfrak{g}^{(1)}\oplus\mathfrak{g}^{(2)}\oplus\cdots \oplus\mathfrak{g}^{(m)}, \]
where $\mathfrak{g}^{(i)} \simeq \mathfrak{g}_i/\mathfrak{g}_{i+1}$, $\mathfrak{g}_1 = \mathfrak{g}$, $\mathfrak{g}_{i+1} = [\mathfrak{g}_1, \mathfrak{g}_i ],\quad i > 1$. and $\mathfrak{g}^{(i)}$ satisfies the relation $[\mathfrak{g}^{(i)},\mathfrak{g}^{(j)}] \subset \mathfrak{g}^{(i+j)}$ and,  in additional, $\mathfrak{g}^{(1)}$ generates $\mathfrak{g}$.

\item[Substep 0-5] (Section \ref{Step0-5}) Next, we choose parameters $l$ suitably to be dense in the domain of the above direct integral and restrict $\pi_l$ again to $\Gamma$.
Then we obtain the first branching formula as a direct integral decomposition of $\pi_l\left|_\Gamma\right.$, which is a modification of Theorem 1.3 in \cite{Bekka1} and the proof of Theorem \ref{nilpLieirredrep} in \cite{Corwin}.

In the case of $\mbox{Heis}_3(\mathbb{Z})$, the arguments here are corresponding to part of the relation (\ref{step2discretetolie}) in Section \ref{relationdiscLie}.

\item[Substep 0-6] (Section \ref{Step0-6}) In the above decomposition, some representations appearing in the integrand are irreducible, but others are not. In the latter case, the extreme one, which satisfies ``rationality conditions'', can be further decomposed into finite-dimensional irreducible representations $\rho_{\rm \tiny{fin}}$ which is equal to $\mbox{Ind}_P^\Gamma (\chi_{l_t}\otimes\chi_{\psi_s})$ in the formula (\ref{finaldecomposition}) in Theorem \ref{introdecompose}. This formula is similar to (\ref{decompositiontofinite}) in spirit. Moreover, we can choose them so that these cases are dense in the domain of the above direct integral. In the case of $\mbox{Heis}_3(\mathbb{Z})$, the arguments here correspond to the relation (\ref{step1discretetolie}) in Section \ref{relationdiscLie}.

\item[Substep 0-7] (Section \ref{Step0-6}) Combining with decompositions in Substeps 0-1, 0-3, 0-5, 0-6, we obtain a substitute of Pytlik's theorem for general nilpotent groups. Here, we need to replace the Plancherel measure that appeared in the previous subsections with the finitely additive measure supported on the above dense subset. This principle is similar in spirit to the original Pytlik Theorem (Theorem \ref{Pytlik}). 
\item[Substep 0-8] (Section \ref{Step0-8})  Finite dimensional irreducible unitary representations $\pi_{{\rm \tiny{fin}}}$ of finitely generated torsion-free discrete nilpotent groups $\Gamma$ are classified in \cite{Howe}. Each representations have similar properties of fluctuation to Fact \ref{factfluctuation} generalizing to $\Gamma$ from ${\rm Heis}_3(\mathbb{Z})$.

\item[Substep 0-9] (Section \ref{Step0-9}) Therefore, as in the Heisenberg case, formal computation can be reduced to spectral analysis associated with irreducible representation $\pi_l$ of $G$. 
The hypo-elliptic operator $H$ in the Conjecture \ref{conj-geod} is written in the following form:
\[  H := -\left(d\pi_l(X_1)^2+\cdots +d\pi_l(X_k)^2\right),
\]
where $X_1, \ldots, X_k$ is a $\mathbb{Q}$-basis of $\mathfrak{g}^{(1)}$. If $\Gamma = \mbox{Heis}_3(\mathbb{Z})$, this operator $H$ is essentially the harmonic oscillator $-\frac{d^2}{du^2}+ u^2$ which is already discussed in the previous section. For a general nilpotent group $\Gamma$, computation reduced to its Malcev completion $G$ as explained above can be performed by the method described in Appendix B.

It should be noted that in a more general framework, the corresponding operator has the form $-(E_1^2 cc+\cdots + E_k^2)+B$ (cf. \cite{Hunt}. However, here we used different sign convention and normalizations.), where $E_i$ is coming from $\mathfrak{g}^{(1)}$ and $B$ from $\mathfrak{g}^{(2)}$. In our case, the property $B=0$ is a consequence of the symmetry of operators. 

\item[Substep 0-10] (Section \ref{Step0-9}) Finally, we notice that analysis of the above irreducible unitary representations of discrete nilpotent groups reduces to those of nilpotent Lie groups. This phenomenon also holds even in corresponding Plancherel measures. Thus, we can utilize the Plancherel theorem and the Fourier inversion formula of the latter as a formal model of the former case.
\end{description}
The above arguments are based on the orbit method, which is a variant of the original case of the nilpotent Lie group. Comparing it to the original version helps the reader understand. Thus, we add a summary of them and three examples in Appendix B.

\begin{remark}\label{directrelationtoG}
One may question why we do not use representations of the Lie group directly without detouring (non-)factor representations of $\Gamma$, extension to the representation of $G$, and finally restriction to $\Gamma$.
 The reason is summarized as follows: For the spectral analysis, similarly to the abelian case, it is crucial to consider the twisted Laplacian associated with representations around the trivial one. If we choose this shortcut, we need to deal with representations that are as trivial as those of $G$ but not its restriction to $\Gamma$, which causes several complexities in the arguments.
\end{remark}
\begin{remark}\label{ExplanationofPractical}
To summarize the foregoing discussion:

In Substep 0-1, the factorization produces a representation of $\Gamma$ which, being induced from the center up to $G$, is monomial. In Substep 0-3, we invoke Fujiwara's theorem to decompose these monomial representations into irreducible representations of $G$, then restrict them back to $\Gamma$ and further decompose the resulting dense subset elements into finite-dimensional representations. This construction yields mutual approximations of $G$-representations by $\Gamma$-representations (and vice versa) while endowing the entire procedure with a rigorous foundation via finite-dimensional representations.

In actual calculations using the nilpotent Floquet-Bloch formalism, it is often more efficient to work directly with the approximate $G$-representations, since they can be handled by the standard orbit method (see Appendix B), rather than juggling the finite-dimensional models employed for justification. Ultimately, it ensures that both approaches agree on the final result.
\end{remark}
\item[Step 2 (P3)]
Concerning the second step of the proof modeled after the abelian case, although essentially the same as ${\rm Heis}_3(\mathbb{Z})$, there are some additional difficulties for general nilpotent groups as explained below.

When $\Gamma = {\rm Heis}_3(\mathbb{Z})$, the hypo-elliptic operator $H$ operator in Theorems \ref{theorem-geod} and \ref{conj-heat}
is a modified harmonic oscillator $\mathcal{H}$ acting on the fiber $L^2(\mathbb{R})$ of $E_{\rho_h}(L^2(\mathbb{R}))$ and thus, we know their eigenvalues and eigenfunctions explicitly.
However, for general nilpotent groups, it seems hopeless to obtain this level of explicitness.

Here, we consider the following two alternatives. One method is to compute the asymptotic expansion in terms of eigenvalues and eigenfunctions of the hypoelliptic operator $H$. In this method, the hypo-elliptic (or sub-elliptic) estimates guarantee the well-definedness of the coefficients of higher-order terms in the asymptotic expansions. Since we do not know the explicit value of eigenvalues and eigenfunctions of $H$ in general, the arguments should stop at this level.

The other method is to express $H$ as an infinite representation matrix $\mathcal{M}_H$ of the basis of eigenfunctions of multi-dimensional harmonic oscillators. Then $\mathcal{M}_H$ is an extended Jacobi matrix, i.e., a near-diagonal matrix. Then, their eigenvalues seem approximately computable by approximating finite-width matrices called finite propagation or finite range operators (see Section \ref{multiharmonic} for details). In addition, higher-order terms are explicitly expressed by using creation and annihilation operators associated with the above harmonic oscillators.   
\end{description}

\subsection{Closed geodesics on hyperbolic Riemann surfaces}

For the asymptotic formula for counting prime geodesics in Theorems \ref{theorem-geod} and \ref{heisenberg-geod}, we use the Selberg trace formula as in the proof of Theorem \ref{abel-geod} for abelian groups. The spectral side of the trace formula can be treated similarly to the above heat kernel asymptotics. On the geometric side, the assumption that $\alpha$ is a conjugacy class of a central element of $\Gamma$ in Theorem \ref{heisenberg-geod} is used in the fact that the following formula holds for such $\alpha$ and finite-dimensional irreducible unitary representations $\rho_{{\rm \tiny{fin}}}$,
\begin{equation}   {\rm tr}(\rho_{{\rm \tiny{fin}}}(f)\rho_{{\rm \tiny{fin}}}(\alpha))
 = {\rm tr}(\rho_{{\rm \tiny{fin}}}(f)){\rm tr}(\rho_{{\rm \tiny{fin}}}(\alpha)) \quad \mbox{for} \quad f \in L^1(\Gamma) \label{central}
\end{equation}
where ${\rm tr}(A) = \frac{1}{{\rm \tiny{dim}}A}{\rm Tr}(A)$ is the normalized trace of the matrix $A$ which is normalization of the usual matrix trace ${\rm Tr}(A)$ satisfying ${\rm tr}(I) = 1$ for the identity matrix $I$ and $\rho_{{\rm \tiny{fin}}}(f)$ is the Fourier transform of $f$ defined by
\[    \rho_{{\rm \tiny{fin}}}(f) = \int_\Gamma f(\sigma)\rho_{{\rm \tiny{fin}}}(\sigma)d\sigma = \sum_{\sigma \in \Gamma}f(\sigma)\rho_{{\rm \tiny{fin}}}(\sigma).
\]
Then, similar arguments to abelian cases in \cite{Phillips}, essentially due to Dirichlet, imply Theorem \ref{heisenberg-geod}. The general case of Theorem \ref{theorem-geod} is similar.

 In the case when $\alpha$ is coming from a non-central element $\sigma$ in Proposition \ref{noncenter}, then we can identify the conjugacy class $\alpha = [\sigma]$ with the element $\pi(\sigma)$ in a finite extension of $\mathbb{Z}^2$ by the canonical projection $\pi:\Gamma \to \Gamma/[\Gamma, \Gamma] \simeq \mathbb{Z}^2$. Therefore, we can apply the same arguments as in abelian cases. 

As already mentioned in the Introduction, when $\alpha$ comes from a non-central element $\sigma$ for general nilpotent groups, the situation is not yet understandable to the author.

\subsection{Contents of later sections}
In the rest of this section, we explain the contents briefly in later sections.

In Chapter 3, we recall the representation theories of ${\rm Heis}_3(\mathbb{Z})$ and ${\rm Heis}_3(\mathbb{R})$ and give a formula which explains a relationship between both theories, which is an essential part of our extension of the Floquet-Bloch theory for ${\rm Heis}_3(\mathbb{Z})$. We explain this case separately since it is explicit and can avoid using the orbit method, at least superficially.

In Chapter 4, we present a representation theory for general discrete nilpotent based on the orbit method, a L \`{a} Killirov especially developed in \cite{Johnston}, \cite{Bekka2}, \cite{Fujiwara2}, \cite{Fujiwara1}, \cite{Bekka1}, \cite{Howe}. Almost all arguments are already prepared in these works; the remaining piece is obtained essentially from a similar idea inherited from the decomposition (\ref{decompositiontofinite}).

Chapter 5 recalls the Lie integrals, Chen's iterated integrals, and their harmonic theory. Using these integrals, we construct a canonical section analogous to the construction of $s_\chi$ in abelian cases. 

In Chapter 6, we give another proof of  ({\ref{introeigen1}) and some related results, which also play a role as the simplest model for later arguments.

In Chapter 7, we compute spectral asymptotics of $\Delta_{\rho_h}$ around $h=0$, corresponding to the trivial representation $\textbf{1}$ in the case of the Heisenberg groups. 

In Chapter 8, we extend computations in Chapter 7 to general nilpotent groups.

In Chapter 9, we provide a proof of the asymptotic expansion of the heat kernel for nilpotent coverings, Theorem \ref{conj-heat}, and several related results, including the local central limit theorem.
  
 In Chapter 10, we give a proof of the geometric analogue of the Chebotarev density theorem \ref{theorem-geod} for nilpotent extensions.  The method is a combination of the preceding arguments and the use of the Selberg trace formula.

In the final Chapter 11, we shall discuss several possibilities of future directions.

In Appendices A and B, we recall the orbit method for simply connected nilpotent Lie groups according to the book authored by Corwin and Greenleaf \cite{Corwin}.
In Appendix C, we recall the harmonic theory of Chen's iterated integral for general torsion-free nilpotent groups mainly for future reference.
The main contents are quoted from Kohno's book \cite{Kohno} with an English translation.

\section{Unitary representations of the discrete Heisenberg group and the real Heisenberg group (Heisenberg Lie group) and their relation}\label{dHeisenberg} 
\subsection{Unitary representations of the discrete Heisenberg group} \label{discreteHeisenberg}

By the reason explained in the first paragraph of Section \ref{Heisenbergexplanation}, we start from the three-dimensional discrete Heisenberg group $\Gamma = {\rm Heis}_3(\mathbb{Z})$. 

To motivate discussions of this and the following sections, let us recall the geometric formulation of the Floquet-Bloch theory for a general discrete group $\Gamma$.
Consider $\pi: X\to M$ a normal covering of a compact Riemannian manifold $M$ with the covering transformation group $\Gamma$. 
The space $L^2(X)$ is naturally identified with the space $L^2(E_R)$ of sections of the flat vector bundle $E_R$ associated with the right regular representation $R$ of $\Gamma$. 
We want to decompose $L^2(E_R)$ as a direct integral of the space of sections of a flat vector bundle associated with irreducible unitary representations of $\Gamma$. However, since $\Gamma$ is not type I, there is no natural decomposition, and it has several difficulties in directly handling the unitary dual $\widehat{\Gamma}$.

To overcome this situation, we can use the following Plancherel theorem due to Pytlik \cite{Pytlik}. 

For $f \in L^1(\Gamma )$ the space of integrable functions on $\Gamma$ and a representation $\pi$ of $\Gamma$, we put the Fourier transform $\pi(f)$ of $f$ by
\[ \pi (f) = \int_\Gamma f(\sigma )\pi (\sigma) d\sigma
\] 
with the left invariant Haar measure $d\sigma$ on $\Gamma$. Let $\widehat{\Gamma}$ denote the unitary dual of $\Gamma$, the set of equivalence classes of irreducible unitary representations of $\Gamma$. Here we have exchanged indexing $x_i, i=1,2,3$ in \cite{Pytlik} to $x_{4-i}$ here in order to adjust the notations used in the next chapter, which are the same as \cite{Corwin}.

\begin{theorem}[\cite{Pytlik}] \label{Pytlik} 
There exists a positive, finitely additive measure $\mu$ on $\widehat{\Gamma}$, supported by the set $\widehat{\Gamma}_{{\rm \tiny{fin}}}$ of finite dimensional representations of $\Gamma$, such that for $f \in L^1(\Gamma )$, the function 
\[ \rho_{{\rm \tiny{fin}},x} \mapsto \frac{1}{{\rm \tiny{dim}}\; \rho_{{\rm \tiny{fin}},x}}{\rm Tr}(\rho_{{\rm \tiny{fin}},x}(f)^\ast\rho_{{\rm \tiny{fin}},x}(f))
\]
is $\mu$-integrable on $\widehat{\Gamma}$ and 
\[      \sum_{n \in \Gamma}|f(n)|^2 = \int_{\widehat{\Gamma}_{\rm fin}}\frac{1}{{\rm \tiny{dim}}\; \rho_{{\rm \tiny{fin}},x}}{\rm Tr}(\rho_{{\rm \tiny{fin}},x}(f)^\ast
\rho_{{\rm  \tiny{fin}},x} (f))d\mu(\rho_{{\rm \tiny{fin}},x}).
\]
\end{theorem}

Here, $\widehat{\Gamma}_{\rm \tiny{fin}}$ can be identified with the set $\widehat{X} :=(\mathbb{Q}\cap [0,1]) \times [0,1] \times [0,1]$ and the Plancherel measure $\mu$ on this set can be expressed as 
\begin{equation}\label{measure}     d\mu(x_1,x_2,x_3) =  d\widetilde{m}(x_1)dm(x_2)dm(x_3)
\end{equation}
for $x = (x_1,x_2,x_3) \in \widehat{X} = \widehat{\Gamma}_{\rm fin}$, where $m$ is the standard Lebesgue measure on $[0,1]$ and $\widetilde{m}$ is finitely additive measure determined by 
\[\widetilde{m}(\mathbb{Q} \cap [a,b]) = b-a. \]

From this theorem, we can easily yield the following Fourier inversion formula for $f \in L^2(\Gamma ), \sigma \in \Gamma$,
\begin{equation}
f(\sigma) = \int_{\widehat{X}}\frac{1}{{\rm \tiny{dim}}\; \rho_{{\rm \tiny{fin}},x}}{\rm Tr}(\rho_{{\rm \tiny{fin}},x} (\sigma^{-1})\rho_{{\rm \tiny{fin}},x} (f))d\mu(\rho_{{\rm \tiny{fin}},x}).\label{Fourier1}
\end{equation}

It is known that finite dimensional irreducible unitary representations of $\Gamma$ (cf. \cite{Pytlik}) are described as follows: 
For every $x = (x_1,x_2,x_3) \in  \widehat{X}$, with $x_1 = p/q$ ($p$ and $q$ are relatively prime), let us define a finite-dimensional representation $\rho_{{\rm \tiny{fin}},x}$ of $\Gamma $, acting on the Hilbert space $\mathcal{H}_x$ of dimension $q$, which will be realized as a space $\mathcal{H}_x$ of all complex, periodic sequence $\{\varphi (k)\}_{k\in\mathbb{Z}}$ with period $q$ as
\begin{equation} (\rho_{{\rm \tiny{fin}},x}(n)\varphi )(k) = \exp \left(\frac{2\pi\sqrt{-1}}{q}(n_3x_3+n_2x_2+n_1qx_1+kn_2qx_1)\right)\varphi (k+n_3) \label{finite rep}
\end{equation}
where
\begin{equation} n = [n_1,n_2,n_3] := \left( \begin{array}{ccc} 1 & n_3 & n_1 \\ 0 & 1 & n_2 \\ 0 & 0 & 1\end{array}\right) \in \Gamma.\label{3heis}
\end{equation}
There is another description (cf. \cite{Davidson}) of $\rho_{{\rm \tiny{fin}},x}$ as follows: For  $\gamma = \exp (2\pi\sqrt{-1}p/q)= \exp (2\pi\sqrt{-1}x_1)$, let  
\begin{equation}
w = [1,0,0],\quad v= [0,1,0],\quad u=[0,0,1] \in \Gamma \label{uvw}
\end{equation}
and \begin{equation}\alpha=\exp\left(\frac{2\pi\sqrt{-1}}{q}x_3\right),\quad \beta = \exp\left(\frac{2\pi\sqrt{-1}}{q}x_2\right) \in U(1). \label{fluctuation}\end{equation}

 Then we have $\rho_{{\rm \tiny{fin}},x}(w) = \gamma$ and
\begin{equation}\label{matrixrep}  \rho_{{\rm \tiny{fin}},x}(u) = \alpha\left(\begin{array}{ccccc}0 & 0 & \cdots & 0 & 1 \\ 1 & 0 & 0 &\cdots & 0 \\ 0 & 1 & 0 & \cdots & 0 \\ \vdots&\ddots &\ddots & \ddots & \vdots \\ 0 &\cdots &0 & 1& 0\end{array}\right), \quad \rho_{{\rm \tiny{fin}},x}(v) = \beta\left(\begin{array}{cccc}1 & 0 & \cdots & 0 \\0 & \gamma & \cdots & 0 \\ \vdots & \ddots & \ddots & \vdots \\ 0 & \cdots & 0 & \gamma^{q-1}\end{array}\right).
\end{equation}
Here, we identify $\mathcal{H}_x$ with $\mathbb{C}^q$ where the above matrices act.
Moreover, the following holds:
\begin{align}
&{}\mbox{pair $(\alpha^q,\beta^q)$ is uniquely determined in $U(1) \times U(1)$ by $\rho_{{\rm \tiny{fin}},x}$} \notag \\  &{}\mbox{and this pair determines $\rho_{{\rm \tiny{fin}},x}$ up to unitary equivalence.} \label{equivalence}
\end{align}

Using the above theorem, we can decompose the flat vector bundle $L^2(E_R)$ associated to the right regular representation $R$ as a direct integral of the space $L^2(E_{\rho_{{\rm \tiny{fin}},x}})$ of sections of the flat vector bundle associated with the representation $\rho_{{\rm \tiny{fin}},x}$ on the fiber $\mathcal{H}_x$. We also have associated decomposition of Laplacian $\Delta_{E_R}$ acting on $L^2(E_R)$ into the Laplacians  $\Delta_{\rho_{{\rm \tiny{fin}},x}}$ acting on $L^2(E_{\rho_{{\rm \tiny{fin}},x}})$. To  perturb eigenvalues of $\Delta_{\rho_{\mbox{\tiny{fin}},x}}$, there is a difficulty that the representation spaces $\mathcal{H}_x$ are varied with respect to $x \in \widehat{X}$.

To overcome this point, we compare $\rho_{{\rm \tiny{fin}},x}$ on $\mathcal{H}_x$ and representations of the real Heisenberg group $G$: 
\begin{equation}  G = {\rm Heis}_3(\mathbb{R}) = \left\{(z,y,x) :=  \left. \left(\begin{array}{ccc} 1 & x & z \\ 0 & 1 & y \\ 0 & 0 & 1\end{array}\right) \right| x,y,z \in \mathbb{R}\right\},\label{Mat}
\end{equation}
as already explained in Chapter $2$.

As a final comment, our version of the extended Floquet-Bloch analysis is applied only to the vicinity of the trivial representation $\text{1}$ in this context. 
However, our asymptotic arguments can also be applied to a vicinity of a particular kind (namely ``rational'') representations which are dense in the unitary dual with respect to the Fell topology. 
This argument is a generalization of \cite{Helffer}, which can be viewed as some considerations in the particular case of the Heisenberg group by different arguments.

\subsection{A relation of unitary representations of the discrete Heisenberg group and the real Heisenberg group}\label{relationdiscLie}
Let us recall irreducible unitary representations of $G = {\rm Heis}_3(\mathbb{R})$ (cf. \cite{Corwin}), which are given by
\begin{description}
\item[(1)] one-dimensional characters which are trivial on the center of $G$,

\item[(2)] infinite dimensional representations $\rho_h$, which are called the Schr\"odinger representations, are parametrized by $h \in \mathbb{R} \setminus \{0\}$: For $f \in L^2(\mathbb{R})$ and $\gamma = (z,y,x) \in G$, 
\begin{equation} (\rho_h(\gamma)f)(s) = (\rho_h((z,y,x))f)(s) = e^{2\pi \sqrt{-1}h(z+sy)}f(s+x) \label{schrep}.
\end{equation}
which is unitary equivalent to 
\begin{equation}
    e^{2\pi \sqrt{-1}(hz+\sqrt{h}sy)}f(s+\sqrt{h}x)=e^{2\pi \sqrt{-1}(hz+\sqrt{h}sy)}e^{\sqrt{h}x\frac{d}{ds}}f(s)   \label{srep}
\end{equation}
\end{description}

Note that the expression (\ref{schrep}) is different from the expression (4) in 2.2.6 Example in  \cite{Corwin}. This difference is caused by the difference of representations of the real Heisenberg algebra ${\rm Lie}({\rm Heis}_3(\mathbb{R}))$. Here, the former expression (\ref{schrep}) is that of the matrix representation (\ref{Mat}) while the latter one is that of exponential coordinates or canonical coordinates of the first kind, which is appeared in the comparison between (\ref{exponetialmultiplication}) and (\ref{matrixmultiplication} in Subsubsection \ref{2n+1dimensionalHeisenberg}
).

Moreover, the following Fourier inversion formula holds. Let $\mathcal{S}(G)$ be the space of rapidly decreasing functions on $G$.

\begin{theorem}\label{Fourierinversionheisenberglie} For  $f \in \mathcal{S}(G)$, we have
\begin{equation} f(\sigma) = \int_{\mathbb{R}\setminus\{0\}}{\rm Tr}(\rho_h(\sigma^{-1})\rho_h(f))|h|dh, \label{Liedecomp}
\end{equation}
\end{theorem}
\noindent where $\rho_h(f)$ denotes the Fourier transform of $f$:
\[   \rho_h(f) = \int_G f(\gamma )\rho_h(\gamma)d\nu(\gamma) \label{FLie}
\]
and a left invariant measure $\nu$ on $G$.

If $f \in L^1(G)$, then $\rho_h(f)$ is a compact operator (cf. Corollary 6.2.16. in \cite{Fujiwara}) but not necessarily in the trace class (see \ref{Step0-9} for more detail explanation),
which is the main reason why the computation of the heat kernel $k_{\rho_h}(t,p,q)$ is a formal one already mentioned in Section \ref{Heisenbergexplanation}.

The latter representations are crucial in the decomposition of the regular representation, specifically in the Fourier inversion formula.

In the case when $h = p/q \in \mathbb{Q}$ ($p,q$ are relatively prime), the restriction of $\rho_h$ to $\Gamma ={\rm Heis}_3(\mathbb{Z})$ can be decomposed as follows:

\begin{theorem}\label{discretetoLie} If $h= x_1 =p/q \in \mathbb{Q} \cap [0,1]$, then 
\begin{equation} \rho_h\left|_{\Gamma}\right. \simeq {\int_0^1}^\oplus\!{\int_0^1}^\oplus\rho_{{\rm \tiny{fin}},(x_1, \{qx_1x_2\}, x_3)}dm(x_2)dm(x_3),\label{discretetoHeisenbergLie}
\end{equation}

where $\{a\}$ is the fractional part $a-[a]$ of $a \in \mathbb{R}$ and $dm$ is the Lebesgue measure on $[0,1]$
\end{theorem}
This theorem is a particular case of Theorem \ref{introdecompose}, which will be proved in Section \ref{Step0-6}. Here, we give proof without explicitly mentioning the orbit methods, which can be viewed as a model case in the arguments of the next chapter.  

The basic strategy is simple and of an abelian nature, based on the following correspondence, which is presented in reverse order of the actual proof:
\begin{description} 
\item[Step 2 of the proof] 
\begin{equation}
L^2(\mathbb{R}) \simeq L^2((\mathbb{R}/\mathbb{Z})\times \mathbb{Z}) \simeq  L^2([0,1))\otimes L^2(\mathbb{Z}) \simeq L^2(E_R(\mathbb{R}/\mathbb{Z})) \label{step2discretetolie}
\end{equation}
where $R$ is the right regular representation of $\mathbb{Z}$ and $L^2(E_R(\mathbb{R}/\mathbb{Z}))$ is the space of $L^2$-sections of the flat vector bundle $E_R(\mathbb{R}/\mathbb{Z})$ associated to $R$. After completing the previous paper \cite{Katsuda0}, we find that the method is essentially the same as in the proof of Theorem 1.3 in \cite{Bekka1}, which will be reproduced in Subsubsection \ref{inductionstep}. However, we leave our version here since it uses fewer terminologies of the representation theory.
\item[Step 1 of the proof] 
\begin{align} L^2(\mathbb{Z}) &\simeq L^2((\mathbb{Z}/q\mathbb{Z})\times q\mathbb{Z}) \simeq  L^2(\mathbb{Z}/q\mathbb{Z})\otimes L^2(q\mathbb{Z}) \notag \\
 &\simeq  L^2(\mathbb{Z}/q\mathbb{Z})\otimes L^2(S^1) \simeq  L^2(\mathbb{Z}/q\mathbb{Z})\otimes L^2([0,1)).\label{step1discretetolie}
\end{align}
\end{description}

\begin{proof} Step 1: First, we extend the action of $\rho_{{\rm \tiny{fin}},x}$ on $H_x$ to that on  $L^2(\mathbb{Z})$ as the same formula as (\ref{finite rep}). We write this action as $\rho_{L^2(\mathbb{Z}),x}$. Then, we have a decomposition 
\begin{equation}   L^2(\mathbb{Z}) \simeq \int_{\widehat{q\mathbb{Z}}}^\oplus\mathcal{H}_{\chi,q}^\mathbb{Z} d\chi   \label{Zdecomp},
\end{equation}
where $\widehat{q\mathbb{Z}}$ is the unitary dual of $q\mathbb{Z}$, which is viewed as the covering transformation group of the canonical covering $\mathbb{Z} \to \mathbb{Z}/q\mathbb{Z}$ and for $\chi:= \exp(2\pi\sqrt{-1}a) \in U(1)$ with $a \in [0,1]$, 
\[ \mathcal{H}_{\chi,q}^\mathbb{Z} = \{\varphi : \mathbb{Z} \to \mathbb{C} \mid \varphi (k+q) = \chi \varphi(k)\}.
\]
Note that $\mathcal{H}_{\chi,q}^\mathbb{Z}$ is an invariant subspace of $\rho_{L^2(\mathbb{Z}),x}$ and its restriction to this space is unitarily equivalent to $\rho_{{\rm \tiny{fin}},(x_1, x_2,x_3+a)}$ on $\mathcal{H}_{(x_1,x_2, x_3+a)}$ which is defined before the formula (\ref{finite rep}). 
In fact, taking $s(k) := \chi^{k/q}= \exp(2\pi\sqrt{-1}ka/q)$, we identify $\mathcal{H}_{\chi,q}^\mathbb{Z}$ with $\mathcal{H}_{(x_1,x_2,x_3+a)}$ by the correspondence
\[  \mathcal{H}_{\chi,q}^\mathbb{Z} \ni s\varphi \longleftrightarrow  \varphi \in \mathcal{H}_{(x_1,x_2, x_3+a)}.
\]
This is well defined by the relations $\rho_{{\rm \tiny{fin}},x}(u) s = (\exp (2\pi\sqrt{-1}a/q)s\rho_{{\rm \tiny{fin}},x}(u)$ and $\rho_{{\rm \tiny{fin}},x}(v)s = s\rho_{{\rm \tiny{fin}},x}(v)$ where $s$ is viewed as a multiplication operator. Thus, we have,
\[ \rho_{L^2(\mathbb{Z}),x} \cong {\int_0^1}^\oplus\rho_{{\rm\tiny{fin}},(x_1,x_2, x_3+a)}dm(a).
\]

Step 2: Next, we extend actions of $\rho_{L^2(\mathbb{Z}),x}$ on $L^2(\mathbb{Z})$ to $L^2(\mathbb{R})$, the same formula as (\ref{finite rep}) again but here the range of the variables $k \in \mathbb{Z}$ changes to $ s \in \mathbb{R}$. We denote this action as $\rho_{L^2(\mathbb{R}),x}$. 

In the following, we relate $L^2(\mathbb{Z})$ and $L^2(\mathbb{R})$ through the  
decompositions (\ref{Zdecomp}) and the following (\ref{Rdecomp}). Namely we compare each integrand $\mathcal{H}_{\chi,q}^\mathbb{Z}$ in the former and $\mathcal{H}_{\chi ,q}^\mathbb{R}$ in the latter.
\begin{equation} L^2(\mathbb{R}) \simeq \int_{\widehat{\mathbb{Z}}}^\oplus \mathcal{H}_{\chi ,q}^\mathbb{R} d\chi,  \label{Rdecomp}
\end{equation}
where
\[\mathcal{H}_{\chi ,q}^\mathbb{R} := \{f:\mathbb{R} \to \mathbb{C} \mid f(s+q) = \chi f(s)\}.
\]
For a positive integer $m$, let us divide the interval $[0,1)$ into $2^m$ subintervals $I_i = [i/2^m, (i+1)/2^m), \; (i = 0, 1,\ldots , 2^m-1)$ of equal length and define
\[ \mathcal{H}_{m,\chi ,q}^\mathbb{R} = \left\{ f \in \mathcal{H}_{\chi ,q}^\mathbb{R} \left| \begin{array}{l}\mbox{$f$ is (locally) constant on the intervals}\\
\mbox{ $[(i/2^m) +r, ((i+1)/2^m) +r)$} \\ \mbox{for $i = 0,1 \ldots ,2^m-1$ , $r = 0,1,\ldots, q-1$} \end{array}\right. \right\}.
\]
Let $f_m \in \mathcal{H}_{m,\chi ,q}^\mathbb{R}$ be a function satisfying 
\[  f_m(s) = \frac{i}{2^m} + r \quad \mbox{if}\; s \in [(i/2^m) +r, ((i+1)/2^m) +r)
\]
for $i = 0,1, \ldots ,2^m-1$ , $r \in \mathbb{Z}$. Define an action $\rho_x^m(n)$ of $\Gamma$ on $\mathcal{H}_{m,\chi ,q}^\mathbb{R}$ by 
\[(\rho_x^m(n)\varphi )(s) = \exp\left(\frac{2\pi\sqrt{-1}}{q}(n_3x_3+n_2x_2+n_1qx_1+f_m(s)n_2qx_1)\right)\varphi (s+n_3).
\]
Note that for $j=0,1,\ldots ,2^m-1$,
\[\mathcal{H}_{m,\chi ,q}^{I_j} = \left\{ f \in \mathcal{H}_{m,\chi ,q}^\mathbb{R} \left| \begin{array}{l}\mbox{$f \equiv 0$ on the complement of the set}\\
\mbox{ $\cup_{r=0}^{q-1}[(j/2^m) +r, ((j+1)/2^m)+r)$}\end{array}\right. \right\}\]
is an invariant subspace of the action of $\rho_x^m(n)$ and its action is unitary equivalent to the actions $\rho_{{\rm \tiny{fin}},x}$ or $\rho_{L^2(\mathbb{Z}),x}$ on the space
\[\mathcal{H}_{\chi,q}^{\mathbb{Z}+j/2^m} = \{f:\mathbb{Z}+j/2^m \to \mathbb{C} \mid f(s+q) = \chi f(s)\}
\]
where $\mathbb{Z}+j/2^m = \{n + j/2^m \mid n \in \mathbb{Z}\}$, which is unitarily equivalent to the action of $\rho_{{\rm \tiny{fin}},(x_1,x_2+(j/2^m)x_3q,x_3)}$ on $\mathcal{H}_{\chi ,q}^\mathbb{Z}$ as in the first step,

Since the decomposition of $[0,1)$ into $2^m$-subintervals corresponds to binary expansions of real number in $[0,1)$ and we have $\rho_{{\rm \tiny{fin}},(x_1,x_2+(j/2^m)x_1q,x_3)}$  converges to $\rho_{{\rm \tiny{fin}},(x_1, x_2+ qx_1x_2,x_3)}$ in the norm topology if $j/2^m$ tends to $x_2$. From an unitary equivalence between $\rho_{{\rm \tiny{fin}},(x_1,qx_1x_2,0)}$ and $\rho_{{\rm \tiny{fin}},(x_1, \{qx_1x_2\},0)}$ by the condition (\ref{equivalence}), we get the conclusion by noting that $\rho_h\left|_\Gamma\right. = \rho_{L^2(\mathbb{R}),(h,0,0)}$. 
\end{proof}

By the definition of finitely additive measure $d\mu$ in Theorem \ref{Pytlik}, for $K > 0$, we put
\[  A_K := \{ x_1 = \frac{p}{q} \in \mathbb{Q} \cap [0,1] \mid q > K\}, \qquad B_K = A_K \times [0,1] \times [0,1].\] 
Then, $B_K$ is a full measure set in $\widehat{X}$ with respect to $\mu$ for any $K > 0$.

 On the other hand the fluctuation of $x_2,x_3$ in $x = (x_1,x_2,x_3) \in \widehat{X})$ with $x_1 = p/q$ for $\rho_{{\rm \tiny{fin}},x}$ is $O(1/q)$ by (\ref{fluctuation}), we have 
\[ \lim_{x_1 \to 0}\rho_{{\rm \tiny{fin}},(x_1,\{qx_1x_2\},x_3)} = \lim_{x_1 \to 0}\rho_{{\rm \tiny{fin}},(x_1,0,0)}
= \lim_{x_1 \to 0}\rho_{{\rm \tiny{fin}},(x_1,x_2,x_3)}.
\]

Therefore, we get the following:
\begin{align*}
&{} \lim_{K \to \infty}\int_{B_K}\frac{1}{{\rm \tiny{dim}}\; \rho_{{\rm \tiny{fin}},x}}{\rm Tr}(\rho_{{\rm \tiny{fin}},x}(\sigma^{-1})\rho_{{\rm \tiny{fin}},(x_1, \{qx_1x_2\}, x_3)}(f))d\mu(x) \\
&=\lim_{K \to \infty}\int_{B_K}\frac{1}{{\rm \tiny{dim}}\; \rho_{{\rm \tiny{fin}},x}}{\rm Tr}(\rho_{{\rm \tiny{fin}},x}(\sigma^{-1})\rho_{{\rm \tiny{fin}},(x_1, x_2, x_3)}(f))d\mu(x) \\
&= \int_{\mathbb{Q}\cap[0,1]}\!\int_0^1\!\int_0^1\frac{1}{{\rm \tiny{dim}}\; \rho_{{\rm \tiny{fin}},x}}{\rm Tr}(\rho_{{\rm \tiny{fin}},x}(\sigma^{-1})\rho_{{\rm \tiny{fin}},(x_1, x_2, x_3)}(f))d\mu(x)  \\
&= \int_{\mathbb{Q}\cap[0,1]}\!\int_0^1\!\int_0^1{\rm tr}(\rho_{{\rm \tiny{fin}},x}(\sigma^{-1})\rho_{{\rm \tiny{fin}},(x_1, x_2, x_3)}(f))d\mu(x). 
\end{align*}

\begin{remark}\label{remark34} By the above formula, we can view the above theorem \ref{discretetoLie} as an approximation formula of the left-hand side by $\rho_{{\rm \tiny{fin}},x}$. Thus, we can freely exchange finite and infinite-dimensional representations. More explicitly, since $e^{\sqrt{p/q}(d/ds)}$ and $e^{2\pi\sqrt{-1}\sqrt{p/q}s}$ are approximated by $\rho_{{\rm \tiny{fin}},(1/q,0,0)}(u)$ and 
$\rho_{{\rm \tiny{fin}},(p/q,0,0)}(v)$ respectively, we replace semigroups $e^{2\pi\sqrt{-1}\sqrt{h}s}$, $e^{\sqrt{h}(d/ds)}$ 
with respect to $\sqrt{h}$ in the expression of $\rho_h$ by semigroups  $e^{\sqrt{h}u_{p/q}}$, $e^{\sqrt{h}v_{p/q}}$ where
\begin{align*}   \tilde{u}_{p/q} := \log(\rho_{{\rm \tiny{fin}},(p/q,0,0)}(u)), &{}\quad \tilde{v}_{p/q} := \log(\rho_{{\rm \tiny{fin}},(p/q,0,0)}(v)) \\
  u_{p/q}:=\sqrt{q/p}\tilde{u}_{p/q} &{} \quad v_{p/q}:=\sqrt{q/p}\tilde{v}_{p/q}
\end{align*}
respectively, whose approximation error can be decreased arbitrarily by taking a sufficiently large denominator $q$.

 Therefore, several formal computations in the later Section \ref{Step0-6} on the Schr\"odinger representations $\rho_h$ near $h=0$ can be justified mathematically. This procedure is what we consider our extension of the Floquet-Bloch theory in the case of the Heisenberg group ${\rm Heis}_3(\mathbb{Z})$. 
\end{remark}

\section{Unitary representations of discrete nilpotent groups, nilpotent Lie groups and their relations}\label{Representationsdiscretenilpotent}

\subsection{Decomposition of regular representations of discrete nilpotent groups into the direct integrals of factor representations: Substep 0-1} \label{nilpotentfactor}
\label{Step0-1}


To prepare several terminologies for the study of representations of $\Gamma$, we quote several paragraphs from \cite{Bekka2} with suitable modifications to adjust to the present context.

A fundamental object is its unitary dual space $\widehat{\Gamma}$, the set of equivalence classes of irreducible unitary representations of $\Gamma$. 
The space $\widehat{\Gamma}$ carries a natural Borel structure, known as the Mackey Borel structure (see \cite{Mackey}). 
A classification of $\widehat{\Gamma}$ is considered as being possible only if $\widehat{\Gamma}$ is a standard Borel space; according to Glimm's celebrated theorem \cite{Glimm}, this is the case if and only if $\Gamma$ is of type I in the following sense.
Recall that a von Neumann algebra is a self-adjoint subalgebra of $\mathcal{L}(H)$ which is closed for the weak operator topology of $\mathcal{L}(H)$, where $H$ is a Hilbert space. A von Neumann algebra $\mathcal{M}$ is a factor if the center of $\mathcal{M}$ consists of the scalar operators.
Let $\pi$ be a unitary representation of $\Gamma$ in a Hilbert space $H$. The von Neumann subalgebra generated by $\pi(\Gamma)$ coincides with the bicommutant $\pi(\Gamma)''$ of $\pi(\Gamma)$ in $\mathcal{L}(H)$ which is closed for the weak operator topology of $\mathcal{L}(H)$, we say that $\pi$ is a factor representation if $\pi(\Gamma)''$ of $\mathcal{L}(H)$ is a factor.

The group $\Gamma$ is of type I if, for every factor representation $\pi$ of $\Gamma$, the factor $\pi(\Gamma)''$ is of type I, that is, $\pi(\Gamma)''$ is isomorphic to the von Neumann algebra $\mathcal{L}(K)$ for some Hilbert space $K$; equivalently, the Hilbert space $H$ of $\pi$ can write as tensor product $K\otimes K'$ of Hilbert spaces in such a way that $\pi$ is equivalent to $\sigma\otimes I_{K'}$ for an irreducible representation $\sigma$ of $\Gamma$ on $K$. Here $I_{K'}$ is the identity operator on $K'$.

By the celebrated theorem of Thoma \cite{Thoma1}, an infinite discrete group $\Gamma$ is of type I if and only if $\Gamma$ contains an abelian subgroup of finite index. 
Thus, a nonabelian discrete infinite nilpotent group $\Gamma$ is not type I, and treating $\widehat{\Gamma}$ directly is impossible.

In the case when $\Gamma$ is the three dimensional discrete Heisenberg group $\mbox{Heis}_3(\mathbb{Z})$, 
We have used finite-dimensional irreducible unitary representations and the Plancherel theorem due to Pytlik \cite{Pytlik} in \cite{Katsuda0} to overcome this situation. 
Even for general discrete nilpotent groups $\Gamma$, we know all their finite-dimensional irreducible unitary representations by \cite{Howe}; however, the Plancherel theorems generalizing the Pytlik theorem in the above to $\Gamma$ are not available at present. Thus, we take another method.

We first recall results in \cite{Bekka2}, which decompose the right regular representation of $\Gamma$ into its factor representations. For related results, please refer \cite{Thoma}, \cite{Johnston}, \cite{Baggett}.
In the following explanations, we quote from several sentences in \cite{Bekka2} with slight modifications. We consider mainly the right regular representation $\rho_\Gamma$ instead of the left regular representation in \cite{Bekka2}.

When $\Gamma$ is not type I, it usually admits several integral decompositions of $\rho_\Gamma$ into irreducible representations, and it is impossible to single
out a canonical one among them. However, when $\Gamma$ is unimodular, $\rho_\Gamma$
does admit a canonical integral decomposition into factor representations; this is the content of a Plancherel theorem, which
we will discuss in the following Theorem \ref{Bekkafactor}.

In order to state the Plancherel theorem for $\Gamma$,
we need to replace the dual space $\widehat{\Gamma}$
by the consideration of Thoma's
dual space ${\rm Ch}(\Gamma )$ which we now introduce.

Recall that a function $t:\Gamma \to \mathbb{C}$ is of positive type if the complex-valued matrix $(t(\gamma^{-1}_j\gamma_i))_{1 \leq i,j \leq n}$
is positive semi-definite for any $\gamma_1, \ldots, \gamma_n$ in $\Gamma$.
A function of positive type $t$ on $\Gamma$ which is constant on conjugacy
classes and normalized (that is, $t(e) = 1$), will be called a trace on $\Gamma$.
The traces on $\Gamma$ will be denoted by ${\rm Tr}(\Gamma)$.

Let $t \in {\rm Tr}(\Gamma)$ and $(\pi_t,H_t,\xi_t)$ be the associated GNS triple (see \cite{Bekka3}, C.4). Then $\tau_t:\pi(\Gamma)''\to \mathbb{C}$, defined by $\tau_t(T) = \langle T\xi_t, \xi_t\rangle$
is a trace on the von Neumann algebra $\pi_t(\Gamma)''$, that is, $\tau_t(T^\ast T) \geq 0$
and $\tau_t(TS) = \tau_t(ST)$ for all $T, S \in \pi_t(\Gamma)''$, moreover, $\tau_t$ is faithful in
the sense that $\tau_t(T^\ast T) > 0$ for every $T \in \pi_t(\Gamma)'', T \neq 0$. 
Observe that $\tau_t(\pi(f)) = t(f)$ for $f \in C[\Gamma]$, where $t$ denotes the linear extension of $t$
to the group algebra $C[\Gamma]$.
The set ${\rm Tr}(\Gamma)$ is a convex subset of the unit ball of $\ell^\infty(\Gamma)$, which is
compact in the topology of pointwise convergence. An extreme point
of ${\rm Tr}(\Gamma)$ is called a character of $\Gamma$; we will refer to ${\rm Ch}(\Gamma)$ as Thoma's dual space.
Since $\Gamma$ is countable, ${\rm Tr}(\Gamma)$ is a compact metrizable space and ${\rm Ch}(\Gamma)$
is easily seen to be a $G_\delta$ subset of ${\rm Tr}(\Gamma)$. So, in contrast to $\widehat{\Gamma}$, Thoma's
dual space ${\rm Ch}(\Gamma)$ is always a standard Borel space.
An important fact is that ${\rm Tr}(\Gamma)$ is a simplex (see \cite{Thoma1}, Satz 1 or \cite{Sakai1}, 3.1.18.); by Choquet theory, this implies that every $\tau \in {\rm Tr}(\Gamma)$ can be represented as integral
\[ \tau = \int_{{\rm Ch}(\Gamma)}td\mu(t) \] 
for a unique probaiblity measure $\mu$ on ${\rm Ch}(\Gamma)$.

As we now explain, the set of characters of $\Gamma$ parametrizes the factor
representations of finite type of $\Gamma$, up to quasi-equivalence; (for more
details, see \cite{Bekka3}, section 11.C).
Recall first that two representations $\pi_1$ and $\pi_2$ of $\Gamma$ are quasi-equivalent
if there exists an isomorphism $\Phi:\pi_1(\Gamma)'' \to \pi_2(\Gamma)''$ of von Neumann
algebras such that $\Phi(\pi_1(\gamma)) = \pi_2(\gamma)$ for every $\gamma \in \Gamma$.
Let $t \in {\rm Ch}(\Gamma)$ and $\pi_t$ the associated GNS representation. Then the map $t \to [\pi_t]$ is a bijection between ${\rm Ch}(\Gamma)$ and the set of quasi-equivalence classes of factor representations of the finite type of $\Gamma$.

The following theorem is a version for countable groups of the Plancherel
theorem due to \cite{Mautner1} and \cite{Segal} which holds more generally, for any unimodular second countable locally compact group, we have that
the right regular representation $\rho_\Gamma$ can be decomposed into a direct integral of factor representation $\pi_t$ over ${\rm Ch}(\Gamma)$ as follows:

\begin{theorem}{\rm (Theorem A in \cite{Bekka2} (Plancherel Theorem for Countable Groups)}. \label{BekkaPlancherel} Let $\Gamma$ be a countable group. There
 exists a probability measure $\mu$ on ${\rm Ch}(\Gamma )$ such that a measurable field of representations $t \mapsto (\pi_t,\mathcal{H}_t)$ of $\Gamma$ on the standard
Borel space ${\rm Ch}(\Gamma )$, and an isomorphism of Hilbert spaces between $\ell^2(\Gamma )$
and $\int_{{\rm Ch}(\Gamma )}^\oplus \mathcal{H}_td\mu(t)$
which transforms the right regular representation $\rho_R$ of $\Gamma$ decomposed as 
\begin{equation} \rho_R := \lambda_\Gamma = \int_{{\rm Ch}(\Gamma )}^\oplus \pi_t d\mu(t). \label{rhoRdecomposition} \end{equation} 
and has the following properties:
\begin{description}
\item[{\rm (i)}] $\pi_t$ is quasi-equivalent to the GNS representation associated to $t$; in particular, the $\pi_t$'s
 are mutually disjoint factor representations of finite type, for $\mu$-almost every $t \in {\rm Ch}(\Gamma)$ ;
\item[{\rm (ii)}]  the von Neumann algebra $L(\Gamma) := \lambda_{\Gamma}(\Gamma)''$ is mapped 

onto the direct integral $\int_{{\rm Ch}(\Gamma )}^\oplus \pi_t(\Gamma)'' d\mu(t)$ of factors;
\item[{\rm (iii)}] for every $f \in C[\Gamma]$, the following Plancherel formula holds:
\begin{equation}
\|f\|_2^2
= \int_{{\rm Ch}(\Gamma )}\tau_t(\pi_t(f)^\ast\pi_t(f))d\mu (t) = \int_{{\rm Ch}(\Gamma )}
t(f^\ast\ast f)d\mu (t).
\label{decompositionfactor}
\end{equation}
\end{description}
 The measure $\mu$ is the unique probability measure on ${\rm Ch}(\Gamma)$  such that the Plancherel formula above holds, which is called the Plancherel measure of $\Gamma$
\end{theorem}

In \cite{Bekka2} Theorem B, Bekka showed that the above Plancherel theorem can be reduced to that on the Choquet space ${\rm Ch}(\Gamma_{\rm fc})$ of FC-centre $\Gamma_{\rm fc}$ or its quotient space ${\rm Ch}(\Gamma_{\rm fc})/K_\Gamma$. Here the FC-centre $\Gamma_{\rm fc}$  of $\Gamma$ is defined as the normal subgroup of
elements in $\Gamma$ with a finite conjugacy class and $K_\Gamma$ is defined as the closure in ${\rm Aut}(\Gamma_{\rm fc})$ of the subgroup ${\rm Ad}(\Gamma)|_{\Gamma_{\rm fc}}$ given
by conjugation with elements from $\Gamma$. Moreover, he showed 
when $\Gamma$ is a discrete subgroup of the connected linear algebraic group $G$ such that the Zariski closure of $\Gamma$ is connected, the
Plancherel formula has a particularly simple form.

\begin{theorem}{{\rm (Corollary D in \cite{Bekka2})}}\label{Bekkafactor}
Let $G$ be a connected linear algebraic group over a field $k$ and let $\Gamma$ be
a countable Zariski dense subgroup of $G$. The Plancherel measure of
$\Gamma$ is the image of the normalized Haar measure $d\chi$ on the character group $\widehat{Z(\Gamma)}$ of the center $Z(\Gamma)$ of $\Gamma$ under the
map
\[\widehat{Z(\Gamma)} \to {\rm Tr}(\Gamma), \qquad \chi \to \tilde{\chi} \]
and the Plancherel formula is given for every element $f$ in the group algebra $\mathbb{C}[\Gamma]$ by
\begin{equation} \|f\|^2_2 = \int_{\widehat{Z(\Gamma)}}\mathcal{F}((f^\ast \ast f)|_{Z(\Gamma)})(\chi)d\chi, \label{Bekkadecopequation}\end{equation} 
where $\mathcal{F}$ is the Fourier transform on the abelian group $Z(\Gamma)$.

The previous conclusion holds in the following two cases:
\begin{description}
\item[{\rm (i)}] $k$ is a countable field of characteristic $0$ and $\Gamma = G(k)$ is the
group of $k$-rational points in $G$;
\item[{\rm (ii)}]
 $k$ is a local field {\rm (} that is, a non discrete locally compact field {\rm )},
$G$ has no proper $k$-subgroup $H$ such that $(G/H)(k)$ is compact,
and $\Gamma$ is a lattice in $G(k)$.
\end{description}
\end{theorem}
This theorem is a generalization of Theorem 3.6 in Johnston \cite{Johnston}. 
Note that if $G$ is the Malcev completion of a discrete, finitely generated, torsion-free, nilpotent group $\Gamma$, then $G$ is a linear algebraic group over $\mathbb{R}$ and has a rational structure (Theorem 5.1.8 in \cite{Corwin}) such that  $\Gamma$ can be expressed as a uniform lattice $G(\mathbb{Z})$ of $G$ (cf. \cite{Corwin}). In this case, a factor representation $\pi_t$ is an induced representation from the character $t$ of the center of $\Gamma$, as explained below.

We recall the more detailed description of $\pi_t$ according to \cite{Johnston}. First of all, we note that the notations are slightly changed below, such as $\tau_\lambda$ (used in \cite{Johnston}) is the same as $\pi_t$ in Theorem \ref{BekkaPlancherel}.

Let $\mathcal{L}$ be a subset of the Lie algebra of the Malcev completion $G$ of $\Gamma$ with $\exp(\mathcal{L}) = \Gamma$ and for $t:\Gamma \to \mathbb{C}, t \in {\rm Ch}(\Gamma)$, we have an element $\lambda $ in the Pontryagin dual $\widehat{\mathcal{L}}$ of $\mathcal{L}$ with $t(\exp a) = \lambda (a)$ for $a \in \mathcal{L}$.
 We define $\tau_\lambda$ to be the representation
of $\Gamma$ induced from the restriction of $\lambda$ to the center $z(\mathcal{L})$ of $\mathcal{L}$ regarded as a character on $z(\Gamma )$ as follows:
The Hilbert space defines it
\[ H_\lambda = \left\{f:\Gamma \to \mathbb{C}\left| \begin{array}{l}\int_{z(\Gamma )\backslash\Gamma}|f(x)|^2dx < \infty, f(z^{-1}\gamma) = \lambda(z)f(\gamma ) \\
\mbox{for}\; z \in z(\Gamma ), \gamma \in \Gamma \end{array}\right. \right\}
\]
with inner product $\langle f,g \rangle = \int_{z(\Gamma )\backslash\Gamma}f(x)\overline{g}(x)dx$. 
Namely, 
\[\tau_\lambda(g):H_\lambda \to H_\lambda; \quad (\tau_\lambda(g)f)(\gamma) = f(\gamma g) \; \mbox{for}\; f \in H_\lambda.
\]
Since elements in the coadjoint orbit closure $\mathcal{O}_\lambda$
take same values on the center of $\mathcal{L}$, $\tau_\lambda$ depends only upon $\mathcal{O}_\lambda$. 
Note that, by Lemma 5.1.5 in \cite{Corwin}, $z(\mathcal{L}) = z(\mathcal{L}_\mathbb{R}) \cap \mathcal{L}$ and $\exp(z(\mathcal{L})) = z(\Gamma )$, it is well-defined and $\tau_\lambda = \pi _t$ appeared in (\ref{decompositionfactor}) for generic $t$. 
It is compatible with the following fact shown in in \cite{Johnston}: If $\mathcal{O}_\lambda$ is a generic
coadjoint orbit closure in $\mathcal{L}$, then $\tau_\lambda$ is a factor representation. 

\subsection{Extensions of (non-)factor representations $\Gamma$ to monomial representations of $G$: Substep 0-2 and their decompositions: Substep 0-3} \label{nilpotentfujiwara}

By the Baker-Campbell-Hausdorff formula, we have 
\[ \exp(p_1X_1+\cdots+p_nX_n) = \exp(Q_1(p)X_1)\cdots \exp(Q_n(p)X_n)
\]
where $Q_i(p) = Q_i(p_1,\ldots, p_n)$ is a polynomial of $p = (p_1,\ldots,p_n)$.
From this formula, we can extend $\tau_\lambda$ to the representation $\widetilde{\tau}_\lambda$ of $G$ defined
on the Hilbert space
\[\widetilde{H}_\lambda = \left\{f:G \to \mathbb{C}\left| \begin{array}{l}\int_{z(G)\backslash G}|f(x)|^2dx < \infty, f(z\gamma) = \lambda(z)f(\gamma ) \\
\mbox{for}\; z \in z(G), \gamma \in G\end{array}\right. \right\}
\]
by 
\[(\widetilde{\tau}_\lambda(g)f)(\gamma) = f(\gamma g) \; \mbox{for} \; f \in \widetilde{H}_\lambda, \; \mbox{for} \; g \in G.
\]

It was pointed out to the author by T. Iwamoto before starting this project that representations of a discrete nilpotent group $\Gamma$ can be extended to its Malcev completion $G$ by the Baker-Campbell-Hausdorff formula. 

Concerning a decomposition of $\widetilde{\tau}_\lambda$, H. Fujiwara \cite{Fujiwara2} kindly informed to the author the following facts;
Since $\widetilde{\tau}_\lambda$ is induced from the character on the center of $G$, it is a monomial representation. Then, we can apply the following Plancherel formula in \cite{Corwin2}. The following few paragraphs are quoted from his note \cite{Fujiwara2}:
\begin{quotation}
For a connected and simply connected nilpotent Lie group $G = \exp \mathfrak{g}$ with Lie algebra $\mathfrak{g}$ and its center $Z = \exp  \mathfrak{z}$, we take $\chi$ a unitary character of $Z$. Take a $f \in \mathfrak{g}^\ast$ such that $\chi(\exp X) = e^{if (X)}$ for any $X \in \mathfrak{z}$ and put $\Lambda = f + \mathfrak{z}^\bot \subset \mathfrak{g}^\ast$. 
The simplest form of the irreducible decomposition of the induced representation $\tau = \mbox{Ind}_Z^G\chi$ of $G$ is
\begin{equation} \tau \simeq \int_\Lambda \pi_\gamma d\gamma,  \label{monomial}
\end{equation}
where $\pi_\gamma$ denotes the irreducible unitary representation of $G$ corresponding to $\gamma \in \Lambda$. In this fashion, there are infinitely many $\pi_\gamma$ which are mutually equivalent.

There is another way to describe this decomposition. Let
\begin{equation} \mathfrak{g}_0 = \mathfrak{z} \subset \mathfrak{g}_1 \subset \cdots \subset \mathfrak{g}_{n-1} \subset \mathfrak{g}_n = \mathfrak{g} \label{fujiwarachain}
\end{equation}
be a sequence of ideals of $\mathfrak{g}$ such that
$\mbox{dim}(\mathfrak{g}_j /\mathfrak{g}_{j-1}) = 1,\; 1 \leq j \leq n$.
Let $\mathfrak{g}_j = \mathbb{R}X_j + \mathfrak{g}_{j-1}$ for $1 \leq j \leq n$. For $l \in \Lambda$, we denote by $l_j \in (\mathfrak{g}_j)^\ast$
the restriction of $l$ to $\mathfrak{g}_j$ . Let $e_j (l)$ be the dimension of the $G$-orbit of
$l_j$ in $(\mathfrak{g}_j )^\ast$. Put $e(l) = (e_1(l), \ldots, e_n(l))$ and $\mathcal{E} = \{e(l)| l \in \Lambda \}$. Then,
there is a unique $e = (e_1, . . . , e_n) \in \mathcal{E}$ such that the layer $U_e = \{l \in \Lambda | e(l) = e\}$ is a Zariski open set of $\Lambda$. Let $T$ be the set of the non-jump indices of $e$, namely $T = \{1 \leq j \leq n| e_j = e_{j-1}\}$ and put
\[V = f + \sum_{j\in T}\mathbb{R}X_j^\ast \subset \Lambda,
\]
here $X_j^\ast$ denotes the element of $\mathfrak{z}^\bot \subset \mathfrak{g}^\ast$ verifying $X_j^\ast(X_k) = \delta_{jk}$.
Then,
\begin{equation}\tau \simeq \int_V^\oplus \infty\pi_v dv. \label{monomialmulti}
\end{equation}
The irreducible unitary representations $\pi_v$ are mutually
inequivalent in this form (\cite{Corwin2}).
\end{quotation}


Moreover, several literatures consider the Plancherel formulae of monomial representations with infinite multiplicities.
As for the present contexts, the following is useful due to Fujiwara \cite{Fujiwara1}. In our case that $\pi _t$ ($\pi _t$ is same as $\tau_\lambda$ in \cite{Johnston}) is
induced from a character of the center $z(\Gamma )$ of $\Gamma$ and thus, the above representation $\widetilde{\tau}_\lambda$ of $G$ which is the extension of $\tau_\lambda$ is the induced from the center of $G$. 
Then, our situation is the more straightforward case that $H$ is a normal subgroup in \cite{Fujiwara1} where $H$ is the notation. The monomial representation is induced by the character of the subgroup $H$.  

The Plancherel formula, in this case, is described as follows, which is our translation of the French text of \cite{Fujiwara1}:

First, we recall an abstract Plancherel formula due to \cite{Bonnet}. For a simply connected, connected, nilpotent Lie group $G$ and a unitary representation $\pi$ of $G$, we denote $\mathcal{H}_\pi^\infty$ the space of $C^\infty$ vectors of $\pi$ and $\mathcal{H}_\pi^{-\infty}$ its antidual. Let $H$ be a closed subgroup of $G$, $dh$ a Haar measure on $H$, and $\chi$ a unitary character of $H$. 
By \cite{Bonnet}, $\chi$ admit, like a distribution of positive type on $G$, the Fourier transform is a couple $(\mu,U)$ where $\mu$ is a positive measure on the unitary dual $\widehat{G}$ of $G$ and $U$ a family of nuclear operators 
$U_\pi :\mathcal{H}_\pi^\infty \to \mathcal{H}_\pi^{-\infty},\; \pi \in \widehat{G}$. Then we have for any function $\phi \in C_0^\infty(G)$, the following abstract Plancherel formula:

\begin{theorem}{{\rm \cite{Bonnet}}}
\begin{equation}
\int_H\phi(h)\chi(h)dh = \int_{\widehat{G}}\mbox{tr}(\pi(\phi)U_\pi)d\mu(\pi) \label{Bonnet1}.
\end{equation}
\end{theorem}
When $H$ is the normal subgroup (our case), a concrete form is explained as a more straightforward case in \cite{Fujiwara1}. 
Let $\mathfrak{g}^{\ast}$ be the dual of the Lie algebra $\mathfrak{g}$ of $G$.  
By the orbit method due to Kirillov \cite{Kirillov1}, (see also \cite{Kirillov2}, \cite{Corwin}), each irreducible unitary representation of $G$ corresponds to a coadjoint orbit of $G$ in $\mathfrak{g}^\ast$. 
Let $H$ be a connected normal subgroup of $G$, and let $\mathfrak{h}$ be its Lie algebra.
A character $\chi$ of $H$ can be written as a form
\[\chi_f(\exp X) = e^{\sqrt{-1}f(X)}\quad \mbox{for}\quad X \in \mathfrak{h}
\]
where $f \in \mathfrak{g}^\ast$ and $\mathfrak{h}$ is its subordinate algebra, namely $f([X,Y]) = 0$ for $X,Y \in \mathfrak{h}$. 
The Plancherel measure $\mu$ in (\ref{Bonnet1}) is the canonical central disintegration of the monomial representation $\tau = \mbox{Ind}_H^G(\chi)$ which is obtained as the image of the Lebesgue measure on the subspace $\Xi = f + \mathfrak{h}^\perp$ which is regarded as a measure $\nu$ on $\mathfrak{g}^\ast$ by the Kirillov map $\kappa: \mathfrak{g}^\ast \to G$, 
where $\mathfrak{h}^\perp$ donotes the orthogonal complement of $\mathfrak{h}$ in $\mathfrak{g}$. 

Since $H$ is a normal subgroup of $G$, $G$ acts on $\mathfrak{h}^\ast$. We denote $G(\xi)$ the stabilizer of $\xi \in \mathfrak{h}^\ast$ of $G$.
For any $\zeta \in \Xi$, we have $G\cdot \zeta \cap \Xi = P\cdot \zeta$, where $P=G(f|\mathfrak{h})$, $f|\mathfrak{h}$ denotes the restriction of $f$ to $\mathfrak{h}$. 
On the other hand, a monomial representation $\tau$ can be written as $\tau = \int_{\widehat{G}}^\oplus m(\pi)\pi d\mu$ with the multiplicites are uniform, equal to the number of $H$-orbits in $\kappa^{-1}(\pi)\cap\Xi$ and one unless $\infty$ (cf. \cite{Grelaud}, \cite{Quint}). 

In this situation, we realize $\pi$ in the support of $\mu$, using a porlarization $\mathfrak{b}$ of $\zeta \in \kappa^{-1}(\pi)\cap\Xi$.
For simplicity, we can choose $\mathfrak{b}$ as $\mathfrak{h} \subset \mathfrak{b}$.
Then we have $H \subset B \subset P$, where $B = \exp(\mathfrak{b})$, and $\pi = \mbox{Ind}_B^G\chi_\zeta$. 
 For $\alpha,\beta \in \mathcal{H}_\pi^\infty$, the function $\alpha\-{\beta}$ is a right $B$-invariant and integrable for the $P$-invariant measure $d\dot{g}$ on $P/B$.   

By the formula
\[ \langle U_\pi\alpha,\beta\rangle = \int_{P/B}\alpha(g)\-{\beta(g)}d\dot{g},
\]
we can define an operator $U_\pi :\mathcal{H}_\pi^\infty \to \mathcal{H}_\pi^{-\infty}$. This operator is positive and self-adjoint.
On the other hand, the integral kernel $K_\phi(x,y)$ for $\phi \in C_0^\infty(G)$ is written as the formula \cite{Rais}
\[    K_\phi(x,y) = \int_B\phi(xby^{-1})\chi_\zeta(b)db,
\]
for $(x,y) \in G \times G$.

It is precisely as Gr\`{e}laud \cite{Grelaud2} that we deduce that
\begin{align}
\mbox{Tr}(\pi(\phi)U_\pi) &= \int_{P/B}K_\phi(g,g)d\dot{g} \notag \\
&= \int_{P/B}d\dot{g}\int_B\phi(gbg^{-1})\chi_\zeta(b)db \notag \\
&= \int_{P/B}d\dot{g}\int_\mathfrak{b}\phi^g(\exp(X))e^{i\zeta(X)}dX
 \label{Fujiwaraformula1} 
\end{align}
with $\phi^g(x) = \phi(gxg^{-1})$. Since $\zeta|\mathfrak{h} \neq 0$, we take  a complement space $\mathfrak{t}$ of $\mathfrak{b}$ in $\mathfrak{g}$ such that $\zeta|\mathfrak{t} = 0$.
In view of $B\cdot \zeta = \zeta + \mathfrak{b}^\perp$, if we denote $\mathcal{F}(\eta)$ the Fourier transformation of $\eta \in C_0^\infty(G)$, we have the Fourier inversion formula as follows;
\begin{align}
\mbox{Tr}(\pi(\phi)U_\pi) &= \int_{P/B}d\dot{g}\int_{\mathfrak{b}^\perp}d\xi\int_\mathfrak{t}dY\int_\mathfrak{b}\phi^g(\exp(X+Y))e^{i\zeta(X)+i\xi(Y)}dX \notag \\
&= \int_{P/B}d\dot{g}\int_{\mathfrak{b}^\perp}d\xi\int_\mathfrak{g}\phi^g(\exp(X))e^{i(\zeta+\xi)(X)}dX \notag \\
&= \int_{P/B}d\dot{g}\int_{B/G(\zeta)}d\dot{b}\int_\mathfrak{g}\phi(\exp(X))e^{ib\cdot \zeta(g^{-1}\cdot X)}dX \notag \\
&= \int_{P/G(\zeta)}\mathcal{F}(\phi\circ \exp)(g\cdot \zeta)d\dot{g} \notag \\
&= \int_{G\cdot \zeta\cap\Xi}\mathcal{F}(\phi\circ \exp)(\lambda)d\lambda \label{Fujiwaraformula2} 
\end{align}
Then we arrive at the following formula,

\begin{theorem}{{\rm (the character formula)}}\label{characterfujiwara}
\begin{equation}
{\rm Tr}(\pi(\phi)U_\pi) = \int_{\Omega(\pi)\cap\Xi}\mathcal{F}(\phi\circ \exp)(\lambda)d\lambda \label{Fujiwaraformula3} 
\end{equation}
where we denote the orbit $\kappa^{-1}(\pi)$ by $\Omega(\pi)$.
\end{theorem}

This theorem is a special case of Theorem 1 in \cite{Fujiwara1}.  
Moreover, this is a form of the Kirillov character formula (cf. \cite{Rais}) when $H$ is trivial.

Then the abstract Plancherel formula (\ref{Bonnet1}) is described as:

\begin{theorem}{{\rm (Plancherel formula)}}\label{plancherelfujiwara}
Let $H$ be a normal subgroup of $G$.
For any function $\phi \in C_0^\infty(G)$ and $f \in \mathfrak{g}^\ast$, 
\begin{align}
\int_H\phi(h)\chi_f(h)dh &= \int_\mathfrak{h}\phi(\exp(X))e^{if(X)}dX \notag \\ &= \int_{\Xi}\mathcal{F}(\phi\circ \exp)(\xi)d\nu(\xi) \notag \\
&= \int_{\widehat{G}}d\mu \int_{\Omega(\pi)\cap\Xi}\mathcal{F}(\phi\circ \exp)(\lambda)d\lambda \\
&= \int_{\widehat{G}}\mbox{tr}(\pi(\phi)U_\pi)d\mu(\pi) \label{FujiwaraPlancherel}.
\end{align}
\end{theorem}
 This theorem is also a special case of Theorem 2 in \cite{Fujiwara1}.

At the end of this section, we note that $\pi$ here is the same as $\pi_l$ in Section \ref{nilpotentexplanation}.

\subsection{Restrictions of irreducible unitary representations of
nilpotent Lie Groups to their lattices: Substep 0-5}\label{Step0-5}
In this section and Section \ref{Step0-6}, we shall give a decomposition formula of the restriction of the lattice $\Gamma$ of the irreducible representation $\pi_l$ of the simply connected nilpotent Lie group $G$. 

Our task is to decompose a restriction of $\pi$ ($= \pi_l$ in Section \ref{nilpotentexplanation}) in the previous section to a lattice $\Gamma$ by a direct integral of smaller (not necessarily irreducible) unitary representations of $\Gamma$. 

This kind of problem has already been studied in  \cite{Bekka1}, and the contents of this section are heavily influenced by \cite{Bekka1}, and also by \cite{Corwin} naturally. We would like to thank Ali Baklouti for informing us in the paper \cite{Bekka1}, answering our questions. 
Our substantial contributions in this chapter (also the main one in the whole paper) are essentially limited to this section and the following Section \ref{Step0-6}. 

Note that, as we already mentioned in Substep 0-5 in Section \ref{nilpotentexplanation},  if the argument here is the case of $\mbox{Heis}_3(\mathbb{Z})$, the arguments here are corresponding to the relation (\ref{step2discretetolie}) in Section \ref{relationdiscLie}.

\subsubsection{Reviewing the case when $\Gamma = \mbox{Heis}_3(\mathbb{Z})$ from the orbit method of general discrete nilpoent groups} 

We start in this section with an example of Theorem \ref{introdecompose} ($=$ Theorem \ref{sec4decompose} )  for the case of the Heisenberg group $\Gamma = \mbox{Heis}_3(\mathbb{Z})$, which is a rewriting of Theorem \ref{introdiscretetoLie} ($=$ Theorem \ref{discretetoLie}).

\begin{example}[Example 1 in \cite{Bekka1}]\label{BekkaHeisenberg1}
Let $\mathfrak{g} = \mathbb{R}X_1+\mathbb{R}X_2+\mathbb{R}X_3$ be the Heisenberg algebra,
defined by the commutator bracket $[X_3, X_2]=X_1$. Let $\{X_1^\ast, X_2^\ast, X_3^\ast\}$ be the associated dual basis of $\mathfrak{g}^\ast$. 
The corresponding Lie group $G = \mathbb{R}^3 = \mbox{Heis}_3(\mathbb{R})$ is
endowed with the product
\[(x_1, x_2, x_3)\cdot (y_1, y_2, y_3)=(x_1+y_1+x_3y_2, x_2+y_2, x_3+y_3),\]
for all $(x_1, x_2, x_3), (y_1, y_2, y_3) \in G$. Let $\Gamma$ be the subgroup
$(\mathbb{Z}, \mathbb{Z}, \mathbb{Z})= \{(p_1, p_2, p_3)| p_1, p_2, p_3 \in \mathbb{Z}\}$ of $G$. Then $\Gamma$ is nothing but $\mbox{Heis}_3(\mathbb{Z})$, obviously a latticevin $G$. Moreover, $\{X_1, X_2, X_3\}$ is clearly a strong Malcev basis of $\mathfrak{g}$, strongly based
on $\Gamma$. 

Let $\pi$ be an irreducible unitary representation of $G$. If $\pi$ is not one-dimensional, 
there exists $\alpha \in \mathbb{R} \setminus \{0\}$ such that 
$\pi \cong \mbox{Ind}_M^G\chi_l$ where
$l =: \alpha X_1^\ast \in \mathfrak{g}^\ast$, and $\mathfrak{m}=\mathbb{R}X_1+\mathbb{R}X_2$ is a polarization in $l$ and $M=\exp(\mathfrak{m})$.
In this case, $\mathfrak{m}$ is also an
ideal of $\mathfrak{g}$ and $\{X_1, X_2, X_3\}$ passes through $\mathfrak{m}$. 
Then, we have
\begin{equation}
\pi|_\Gamma \cong \pi_{l,M}|_\Gamma \cong \int_{[0,1)}^\oplus\mbox{Ind}_{M\cap\Gamma}^\Gamma(\chi_{\alpha X_1^\ast + t\alpha X_2^\ast}|_{M\cap\Gamma})dt. \label{BekkaHeisenberg}
\end{equation}

Let $\chi$ be a character of a normal subgroup
$H$ of $\Gamma$, then the stabilizer $\mbox{St}_H^\Gamma(\chi)$ of $\chi$ in $\Gamma$ is defined by
\[\mbox{St}_H^\Gamma(\chi) = \{\gamma \in \Gamma | \chi(\gamma h\gamma^{-1}) = \chi(h) \; \mbox{for all}\; h \in H\}.
\]
This set is a subgroup of $\Gamma$. 

For $t \in [0, 1)$, let $l_t := \alpha X_1^\ast+t\alpha X_2^\ast$.  Then, 
\[\mbox{St}_{M\cap\Gamma}^\Gamma(\chi_{l_t}) = \left\{\begin{array}{lll}(\mathbb{Z}, \mathbb{Z}, 0) & \mbox{if} & \alpha \not\in \mathbb{Q}\\
(\mathbb{Z}, \mathbb{Z}, q\mathbb{Z}) & \mbox{if} & \alpha = \frac{p}{q} \in \mathbb{Q}\setminus\{0\}\end{array}\right.\]

It is known that Mackey's criterion says $\mbox{Ind}_H^\Gamma(\chi)$
is irreducible if and only if $\mbox{St}_H^\Gamma(\chi)= H$. Hence, if $\alpha \not\in \mathbb{Q}$, then the formula (\ref{BekkaHeisenberg}) is a decomposition into irreducible components.

On the other hand, if $\alpha = \frac{p}{q} \in \mathbb{Q}\setminus\{0\}$, then the following further decomposition formula holds:
\begin{align}   \mbox{Ind}_{M\cap\Gamma}^\Gamma(\chi_{l_t}) &= \mbox{Ind}_{\mathcal{P}}^\Gamma( \mbox{Ind}_{M\cap\Gamma}^{\mathcal{P}}(\chi_{l_t}))  \notag \\
&\cong \mbox{Ind}_{\mathcal{P}}^\Gamma\left( \int_{[0,1)}^\oplus \chi_{l_t}\otimes \chi_{\psi_s}ds  \right)\notag \\ 
&= \int_{[0,1)}^\oplus\mbox{Ind}_{\mathcal{P}}^\Gamma(\chi_{l_t}\otimes \chi_{\psi_s})ds, \label{interHeisenbergdecomposion}
\end{align}
where $\psi_s \in (q\mathbb{Z})^{\ast}$  defined by $\psi_s(qn) = s \in [0,1)$ for $n \in \mathbb{Z}$ and $\mathcal{P} := \exp P$  with the polarizng subalgebra $P$ of $l_t$. Note that $\mathcal{P} =  \mbox{St}_H^\Gamma(\chi_{l_t})$. 

Combining with (\ref{BekkaHeisenberg}), the above decomposition brings back the formula (\ref{introdiscretetoHeisenbergLie}) in Theorem \ref{introdiscretetoLie} ($=$ Theorem \ref{discretetoLie}). It is a particular case of the general formula (\ref{intermidietedecomposion}).
\end{example}

\subsubsection{Branching theorem of nilpotent Lie groups to their lattices I: the structure of proof}\label{Branchingtheorem1}

As already mentioned in the first paragraph of this section, we shall give a decomposition formula of the restriction of the lattice $\Gamma$ of the irreducible representation $\pi_l$ of the simply connected nilpotent Lie group $G$. 

Our aim here is to gerenalize the formula (\ref{BekkaHeisenberg}) for $\mbox{Heis}_3(\mathbb{Z})$ to generalize for any torsion-free nilpotent group $\Gamma$ in the case that $\pi_l$ for a rational $l \in \mathfrak{g}^\ast$, i.e. $l \in \mathfrak{g}_{\mathbb{Q}}^\ast$.

Further decompositions into finite-dimensional representations corresponding to the formula \eqref{interHeisenbergdecomposion} will be performed in Section \ref{Step0-6}.

First, we recall basic terminologies related to the representations $\Gamma$ and $G$ respectively. For a discrete, torsion-free, nilpotent group $\Gamma$, let $G$ be its Malcev completion, i.e., a simply connected nilpotent Lie group that contains $\Gamma$ as a lattice. 
Then $G$ has $\mathbb{Q}$-structure such that $\mathfrak{g}_{\mathbb{Q}}$ is $\mathbb{Q}$-span $\log \Gamma$. Note that the usage of the notion $\log \Gamma$ needs to be careful in general (see Section \ref{Step0-8} in detail).  
It means that there is a linear basis $X_1, \ldots, X_n$ of $\mathfrak{g}$ in $\log \Gamma$ such that the structure constants $c_{ij}^k \in \mathbb{Q}$ which are defined by:
\[  [X_i,X_j] = \sum_{k=1}^nc_{ij}^kX_k,\quad (i,j = 1,2,\ldots,n)  
\]
Conversely, if $G$ admits such a structure, then it contains a uniform lattice $\Gamma$ (cf. 5.1.8 Theorem in \cite{Corwin}). 

 For $l\in\mathfrak{g}_\mathbb{Q}^\ast$, let $\mathfrak{r}_l$ be the radical of $\mathfrak{g}$ and denote $r$ of its dimension. It is known that $\mathfrak{r}_l$ is rational.
Moreover, it is also known that we can choose a polarizing subalgebra $\mathfrak{m}$ as a rational subalgebra of $\mathfrak{g}$. We denote by $m$ its dimension.
Namely, we may assume $\mathfrak{r}_l, \mathfrak{m} \subset \mathfrak{g}_\mathbb{Q}$. Put $R_l = \exp \mathfrak{r}_l$ and $M = \exp \mathfrak{m}$. We also call them the radical and polarization of $l$ by the abuse of naming.

Then, the following theorem holds.

\begin{theorem}\label{firstdecomposition}
Assume that $l \in \mathfrak{g}_{\mathbb{Q}}^\ast$. Then there exists a polarization $\mathfrak{m}$ of dimension $m$, a rational subalgebra of $\mathfrak{g}$ satisfying the following. Let $\{X_1,\ldots, X_m, \dots, X_n\}$ be a weak
Malcev basis of $\mathfrak{g}_\mathbb{Q}$, based on $\Gamma$ and passing through $\mathfrak{m}$.
Then, we have the decomposition
\begin{equation} \pi_l|_\Gamma = \int_{[0,1)^{n-m}}^\oplus {\rm Ind}_{M\cap\Gamma}^\Gamma (\chi_{l_{t_{m+1},\ldots,t_n}}|_{M\cap\Gamma})dt_{m+1}\cdots dt_n \label{pildecomposition}
\end{equation}
where
\[
l_{t_{m+1},\ldots,t_n} = {\rm Ad}^\ast (\exp(t_{m+1}X_{m+1}))\cdots{\rm Ad}^\ast(\exp(t_nX_n))l\]
for all $t_{m+1}, \ldots t_n \in [0, 1)$.
\end{theorem}

\begin{remark} This theorem has some similarity with Theorem 1.3 in \cite{Bekka1} although there exists the following differences:
 In Theorem 1.3 in \cite{Bekka1}, the above polarization $\mathfrak{m}$ is assume to be ideal more strongly and take $\{X_1,\ldots, X_m, \dots, X_n\}$ as a strong Malcev basis of $\mathfrak{g}_\mathbb{Q}$. The choice of a strong Malcev basis does not affect its proof. However, in the above Theorem, we need to choose wisely a weak Malcev basis since we assume that $\mathfrak{m}$ is only a subalgebra. In fact, the proof is based on induction. In each induction step, there are possibilities to change the polarization but keeping the isomorphism class of the unitary representation (see Subsubsection \ref{choicepolarization}).

Note that the above stronger conditions in Theorem 1.3 in \cite{Bekka1} are used in the fact that only the case 1 in the proof of Theorem \ref{firstdecomposition} occurs and the proceeding proof in this case is almost same as the proof of Theorem 1.3 in \cite{Bekka1} and close in spirit to the Step 2 in the proof of Theorem \ref{discretetoLie}. On the other hand, additional arguments necessary in the proof of the above theorem are almost the same as the proof of Theorem \ref{nilpLieirredrep} (2) ($=$ 2.2.2 Theorem in \cite{{Corwin}}).
\end{remark}

The following is a key lemma in the argument below, which is called Kirillov's lemma.

\begin{lemma}{{\rm ($=$ 1.1.12 Lemma (Kirillov's Lemma) in \cite{Corwin})}}\label{Kirillovlemma} Let $\mathfrak{g}$ be a noncommutative nilpotent Lie algebra whose center $\mathfrak{z}(\mathfrak{g})$ is one-dimensional. 
Then $\mathfrak{g}$ can be written as
\begin{equation*}
\mathfrak{g}= \mathbb{R}Z \oplus\mathbb{R}Y
\oplus\mathbb{R}X\oplus\mathfrak{w}=\mathbb{R}X\oplus\mathfrak{g}_0,
\end{equation*}
a vector space direct sum, where
$\mathbb{R}Z= \mathfrak{z}(\mathfrak{g})$, $[X, Y]=Z$, and
$\mathfrak{g}_0 = \mathbb{R}Y\oplus\mathbb{R}Z\oplus\mathfrak{w}$ is the centralizer of $Y$ which is an ideal of codimension $1$.
\end{lemma}

\begin{remark}\label{Kirillovremark}
If $G$ has rational structure, we may assume that $\mathfrak{z}(\mathfrak{g})$, $X, Y, Z, \mathfrak{g}_0, \mathfrak{w}$ are all in $\mathfrak{g}_\mathbb{Q}$. 
\end{remark}

The proof proceeds by induction on the dimension of $\mathfrak{g}$. 
Assume the result is valid for all nilpotent groups of dimension $n$ with rational structure and let $\mathfrak{g}_{\mathbb{Q}}$ be of dimension $n+1$.

If the center $\mathfrak{z}(\mathfrak{g})$ has proper one dimensional rational subspace $\mathfrak{h}_{\mathbb{Q}}$, then we may assume $l(\mathfrak{h}_{\mathbb{Q}})=0$ and consider the quotients $\bar{\mathfrak{g}_{\mathbb{Q}}} = \mathfrak{g}_{\mathbb{Q}}/\mathfrak{h}_{\mathbb{Q}}$ and $\bar{\mathfrak{m}}_{\mathbb{Q}} = \mathfrak{m}_{\mathbb{Q}}/\mathfrak{h}_{\mathbb{Q}}$ and $\bar{l} \in \bar{\mathfrak{g}}^\ast$ induced from $l$. Then, from induction hypothesis $\bar{\mathfrak{g}}_{\mathbb{Q}}$, we easily obtain the result. Thus, we may assume  $\mathfrak{z}(\mathfrak{g}) = \mathbb{R}Z$ for some $Z \in \mathfrak{g}_{\mathbb{Q}}$ satisfying $l(Z) =1$.

Applying Kirillov lemma (Lemma \ref{Kirillovlemma} and remark \ref{Kirillovremark}), we have several items: $\mathfrak{z}(\mathfrak{g})$, $X, Y, Z, \mathfrak{g}_0, \mathfrak{w}$ in $\mathfrak{g}_\mathbb{Q}$. Replacing $Y$ with $Y-l(Y)Z$ if necessary, we may also assume $l(Y) =0$.

There are now two cases to consider.
\begin{description}
\item[Case $1$] $\mathfrak{m}\subset\mathfrak{g}_0$. Let $l_0 = l_{\mathfrak{g}_0}$. Put $G_0 =\exp \mathfrak{g}_0$. 
The inductive hypothesis on $G_0$ implies that it suffice to show $\pi_{l,M} = \mbox{lnd}_{G_0}^G(\pi_{l_0,M})$.  To prove this fact in the Subsubsection \ref{inductionstep}, we can proceed as in the proof of Theorem 1.3 in \cite{Bekka1} with suitable modifications. 
\item[Case $2$]  Assume that $\mathfrak{g}_0$ does not include $\mathfrak{m}$. We will show that there exists
another polarization (maximal subordinate subalgebra) $\mathfrak{m}_1$ with $\mathfrak{m}_1\subset\mathfrak{g}_0$ and $\pi_{l,M} \simeq \pi_{l,M_1}$. Then, we are back to Case $1$ in a compatible way with further decompositions in Section \ref{Step0-6}. We shall give a proof of this fact in Section \ref{choicepolarization}. 
 \end{description}

\subsubsection{Branching theorem of nilpotent Lie groups to their lattices II: Case 1 (the induction step)}\label{inductionstep}

As mentioned above, here we discuss Case 1 of the proof of Theorem \ref{firstdecomposition}. 
Although it is almost the same as the part of the proof of Theorem 1.3 in \cite{Bekka1}, we reproduce the arguments there with suitable modifications for the sake of completeness.

\begin{proof}(Case 1.) Recall the starting setup for Case 1.
$\mathfrak{m}\subset\mathfrak{g}_0$. Let $l_0 = l_{\mathfrak{g}_0}$. Put $G_0 =\exp \mathfrak{g}_0$
The inductive hypothesis on $G_0$ implies that the $\pi_{l,M} = \mbox{lnd}_{G_0}^G(\pi_{l_0,M})$.

By the definition of induction, we can realize the representation $\pi_{l,M}= \mbox{lnd}_{G_0}^G(\pi_{l_0,M})$ in the following way:
Hereafter in the proof, we rewrite $G_0$ to $G_n$ to adjust induction arguments, which is also the same notation as in the proof of Theorem 1.3 in \cite{Bekka1}.

Let $\mathcal{H}$ be the Hilbert space
\begin{equation*}
\mathcal{H} = \left\{ f:\mathbb{R} \to \mathcal{H}(\pi_{l_0,M})\; \mbox{measurable}; \int_\mathbb{R}\|f(t)\|_{\mathcal{H}(\pi_{l_0,M})}^2dt < \infty\right\}.
\end{equation*}
Then, for all $g = g_n\cdot\exp(a X_{n+1})$ (with $g_n \in G_n$ and $a \in \mathbb{R}$), $f \in \mathcal{H}$ and
$t \in \mathbb{R}$,
\begin{equation*}
\pi_{l,M}(g)f(t) = 
\pi_{l_0,M}(\exp(tX_{n+1})\cdot g_n\cdot\exp(-tX_{n+1}))(f(t+a))
\end{equation*}
Now fix $s \in [0, 1)$ and define the following Hilbert space $\mathcal{H}_s$ of functions on
$\mathbb{Z}+s = \{p+s: p \in \mathbb{Z}\}$ by
\begin{equation*}
\mathcal{H}_s = \left\{ f:\mathbb{Z}+s \to \mathcal{H}(\pi_{l_0,M})\; \mbox{measurable}; \sum_{p\in\mathbb{Z}}\|f(p+s)\|_{\mathcal{H}(\pi_{l_0,M})}^2dt < \infty\right\}.
\end{equation*}

Let $\rho_s$ be the representation of $\Gamma$ on $\mathcal{H}_s$ defined by
\begin{equation*}
\rho_s(\gamma)f(p+s) = \pi_{l,M}(g)f(t) = \pi_{l_0,M}(\exp((p+s)X_{n+1})\cdot \gamma_n\cdot\exp(-(p+s)X_{n+1}))( f ( p+q+s)),
\end{equation*}
for all $\gamma = \gamma_n\exp(qX_{n+1}) \in \Gamma$, (with $\gamma_n \in G_n\cap\Gamma$ and $q \in \mathbb{Z}$), $f \in \mathcal{H}_s$ and
$p \in \mathbb{Z}$. 
Now, set
\begin{equation*}
T :\mathcal{H}\to \int_{[0,1)}^\oplus\mathcal{H}_sds,\quad f\mapsto  \int_{[0,1)}^\oplus f_{\mathbb{Z}+ s}ds,
\end{equation*}
where $ds$ is the Lebesgue measure on $[0, 1)$. Then $T$ is a unitary operator,
which clearly intertwines $\pi|_\Gamma$ and $\int_{[0,1)}^\oplus\rho_s ds$.

Define a unitary operator $\Phi_s :\mathcal{H}_s \to \mathcal{H} = l^2(\mathbb{Z}, \mathcal{H}(\pi_{l_0,M}))$ by $\Phi_s(f)( p)=
f(p+s)$, for all $f \in \mathcal{H}_s$ and $p \in \mathbb{Z}$, and set $\pi_s(\gamma) = \Phi_s\circ \rho_s(\gamma)\circ\Phi_s^{-1}$, for $\gamma \in \Gamma$.
Then, for all $f \in \mathcal{H}_0, \gamma = \gamma_n\cdot\exp(qX_{n+1}) \in \Gamma$ (with $\gamma_n \in G_n\cap\Gamma$ and $q \in \mathbb{Z})$,
we have
\begin{equation*}
(\pi_s(\gamma)f)(p) = \pi_n^s(\exp(pX_{n+1})\cdot \gamma_n\cdot\exp(-pX_{n+1}))( f ( p+q)),
\end{equation*}
where $\pi_n^s$ is the representation of $G_n$ defined by
\begin{equation*}
(\pi_n^s(g'_n) = \pi_{l_0,M}(\exp(sX_{n+1})\cdot g'_n\cdot\exp(-sX_{n+1})),
\end{equation*}
for all $g_n \in G_n$. It follows that the representation $\pi_s$ of $\Gamma$ on $\mathcal{H}_0$ is induced by the restriction to $G_n\cap\Gamma$ of the representation $\pi_n^s$. Hence
\begin{align*}
\pi_l|_\Gamma &\cong \int_{[0,1)}^\oplus \rho_s ds \cong \int_{[0,1)}^\oplus \pi_s ds \cong \int_{[0,1)}^\oplus \mbox{Ind}_{G_n\cap\Gamma}^\Gamma(\pi_n^s|_{G_n\cap\Gamma}) ds \\
& \cong \mbox{Ind}_{G_n\cap\Gamma}^\Gamma\left(\int_{[0,1)}^\oplus \pi_n^s|_{G_n\cap\Gamma} ds\right).
\end{align*}

Now, the representation $\pi_n^s \cong \mbox{Ind}_{G_n}^M(\mbox{Ad}^\ast(\chi_{\exp(-sX_{n+1})}l_n)$ is irreducible and
$G_n\cap\Gamma$ is a lattice in $G_n$. By induction, $\pi_n^s|_{G_n\cap\Gamma}$ is then unitarily equivalent
to
\begin{align*}
\int_{[0,1)^{n-m}}^\oplus &\mbox{Ind}_{M\cap\Gamma}^{G_n\cap\Gamma}(\mbox{Ad}^\ast(\chi_{\exp(-t_{m+1}X_{m+1})\cdot\exp(-t_nX_n)})\mbox{Ad}^\ast(\chi_{\exp(-sX_{n+1})}l_n)|_{M\cap\Gamma})\\
 &dt_{m+1}\cdots dt_n
\end{align*}
where $dt_{m+1}\cdots dt_n$ is the Lebesgue measure on $[0,1)^{n-m}$. So,
\begin{equation}
\pi_l|_\Gamma \cong \int_{[0,1)^{n+1-m}}^\oplus \mbox{Ind}_{M\cap\Gamma}^\Gamma(\chi_{l_{t_{m+1},\cdots, t_n,s}|_{M\cap\Gamma}})dt_{m+1}\cdots dt_nds, \label{case1last}
\end{equation}
where
\begin{equation*}
l_{t_{m+1},\cdots, t_n,s} =\mbox{Ad}^\ast(\chi_{\exp(-t_{m+1}X_{m+1})\cdot\exp(-t_nX_n)\cdot\exp(-sX_{n+1})})l
\end{equation*}
Hence, the proof of Case 1. 
\end{proof}

\subsubsection{Branching theorem of nilpotent Lie groups to their lattices III: Case 2 (the choice of the polarization)}\label{choicepolarization}
As mentioned above, here we discuss Case 2 of the proof of Theorem \ref{firstdecomposition}. 
Although this time it is almost same as the part of the proof of to the proof of Theorem \ref{nilpLieirredrep}(2)(2.2.2 Theorem in \cite{Corwin}), we reproduce here the arguments there with suitable modification for the sake of completeness.

\begin{proof}(Case 2.) In this case, $\mathfrak{m} \not\subset \mathfrak{g}_0$. Here we shall show that there is another polarization $\mathfrak{m}_1$ satisfying $\mathfrak{m}_1 \subset \mathfrak{g}_0$ with a unitary equivalence
\[\mbox{Ind}_{M}^G(\chi_l|_{\mathfrak{m}}) \cong \mbox{Ind}_{M_1}^G(\chi_l|_{{\mathfrak{m}_1}})\]
for $M = \exp\mathfrak{m}$ and $M_1 = \exp\mathfrak{m}_1$.

Let $\mathfrak{m}_0 = \mathfrak{m}\cap\mathfrak{g}_0$.
Since $\mathfrak{m} \not\subset \mathfrak{g}_0$, $\mathfrak{m}_0$ has codimension one in  $\mathfrak{m}$ 
and we may assume that $\mathbb{R}X$ is a complement to $\mathfrak{m}_0$ in $\mathfrak{m}$, namely $\mathfrak{m}_0\oplus\mathbb{R}X = \mathfrak{m}$. 
In fact, there is some $X_0 \in \mathfrak{g}_0$ such that $X-X_0 \in \mathfrak{m}$. Then, we can replace $X$ by $X -X_0$ without affecting anything in Kirillov's lemma.
Replacing $X$ with $X-l(X)Z$ if necessary, we may also assume that $l(X)=0$. 
Now let $\mathfrak{m}_1 = \mathfrak{m}_0\oplus\mathbb{R}Y$.
Here, note that $Y \not\in \mathfrak{m}_0$, for otherwise we would have $Y \in \mathfrak{m}$, 
$X \in \mathfrak{m}$ and $l([X, Y])=l(Z)= 1$, in conflict with the subordination property of $\mathfrak{m}$. 
Then $\mathfrak{m}_1 \subset \mathfrak{g}_0$, $\mbox{dim}\; \mathfrak{m}_1 = \mbox{dim}\; \mathfrak{m}$, and
$\mathfrak{m}_1$ is a subalgebra since $Y$ is central in $\mathfrak{g}_0$. 
Since $l([\mathfrak{m}_0,\mathfrak{m}_0])=0$ and
$l([\mathfrak{m}_0,Y])=0$, 
we have $l([\mathfrak{m}_1, \mathfrak{m}_1])=0$, and $\mathfrak{m}_1$ is a maximal subordinate
subalgebra i.e, polarization, inside $\mathfrak{g}_0$.

Let $\mathfrak{k}_0 = \mathfrak{m}_0\cap\mbox{ker}\;l$, $\mathfrak{k}=\mathfrak{m}_1\oplus\mathfrak{m} = \mathfrak{m}_0\oplus\mathbb{R}X\oplus\mathbb{R}Y$, $K_0 = \exp\mathfrak{k}_0$, $K = \exp\mathfrak{k}$. 
We shall show that $\mathfrak{k}$ is a subalgebra of $\mathfrak{g}$ and that
\begin{equation} \mbox{Ind}_M^K(\chi_{l,M}) \cong \mbox{Ind}_{M_1}^K(\chi_{l,{M_1}}). \label{MM1}
\end{equation}
By induction in stages, inducing from $K$ to $G$ gives $\pi_{l,M} = \pi_{l,{M_1}}$, completing the proof of the theorem.

$K$ is a group i.e. $\mathfrak{k}$ is a Lie algebra. Moreover, $\mathfrak{k}_0$ is an ideal. 
To prove this fact, we need to prove that, if $W_1 \in \mathfrak{k}_0$ and $W_2 \in \mathfrak{k}$, 
then $l([W_1,W_2])=0$ and $[W_1,W_2] \in \mathfrak{m}_0$. 
If $W_2 \in \mathfrak{m}_0$, these follow because $\mathfrak{m}_0$ is subordinate. 
If $W_2 = Y$, they are also easy, since $[\mathfrak{m}_0,Y] = 0$. Finally, if $W_2 = X$, then $W_1, X \in \mathfrak{m}$, hence $l([W_1,X]) = 0$; 
furthermore, $[W_1,X] = W'_1+\alpha X$ for some $W'_1 \in \mathfrak{m}_0, [W_1,[W_1, X]] = W''_1+\alpha^2X$, etc. 
We must have $\alpha = 0$ because $\mathfrak{m}$ is nilpotent.

Next, notice that $\mathfrak{m}, \mathfrak{m}_1$ are both maximal subordinate subalgebras for $\mathfrak{k}$. Let $B_l(X,Y) :=l([X, Y])$ the skew-symmetric bilinear form on $\mathfrak{g}$. Then, the above $\mathfrak{m}, \mathfrak{m}_1$ also both $B_l$-isotropic subalgebras, both have codimension $1$ in $\mathfrak{k}$, and $B_l$ is nontrivial on $\mathfrak{k}$ since $l([X, Y]) \neq 0$.

Let $\tau_1 = \mbox{lnd}_{M_1}^K(\chi_{l,{M_1}})$ and $\tau = \mbox{lnd}_M^K(\chi_{l,M})$.
Then both $\tau_1$ and $\tau$ contain $K_0$ in their kernels. 
We will show this for $\tau$, the proof for $\tau_1$  
being nearly the same. 
We may realize $\tau$ on $L^2(\mathbb{R})$, using $\exp(\mathbb{R}Y)$ as a transversal for $M\backslash K$. 
If $k_0 \in K_0$, we have
\begin{align*}
\tau(k_0)f(t) &= f(\exp tY\cdot k_0\cdot\exp (-t)Y\cdot\exp tY) \\
&= \chi_{l,M}(\exp tY\cdot k_0\cdot\exp (-t)Y)f(t)= f(t),
\end{align*}
since $\exp tY\cdot k_0\cdot\exp (-t)Y \in K_0$ and $l(\mathfrak{k}_0)=0$.

Define $\bar{\tau}_1, \bar{\tau}$ on $\bar{K} = K_0\backslash K$ in the obvious way. To prove $\tau_1\cong\tau$, it suffices to prove that $\bar{\tau}_1\cong\bar{\tau}$. But $\bar{K}\cong \mbox{Heis}_3(\mathbb{R})$, and $\bar{\tau}_1, \bar{\tau}$ are two explicit representations on it.

Here are their descriptions: since $\mathbb{R}Z\subset\mathfrak{m}_0$ and $Z \not\subset \mbox{ker} l$, we have $\mathfrak{m}_0 = \mathfrak{k}_0\oplus\mathbb{R}Z$, so that
\begin{align*}\mathfrak{k} &= \mathfrak{k}_0\oplus\mathbb{R}Z \oplus\mathbb{R}Y\oplus\mathbb{R}X, \\
\mathfrak{m} &= \mathfrak{k}_0\oplus\mathbb{R}Z\oplus\mathbb{R}X, \\
\mathfrak{m}_1 &= \mathfrak{k}_0\oplus\mathbb{R}Z\oplus\mathbb{R}Y, \\
\mathfrak{k}\cap\mbox{ker} l &= \mathfrak{k}_0\oplus\mathbb{R}Y\oplus\mathbb{R}X.
\end{align*}

When we identify $\mathfrak{k}_0\backslash\mathfrak{k} = \mathfrak{h}_1 := \mbox{Lie}(\mbox{Heis}_3(\mathbb{R}))$, $\{Z ,Y, X\}$ becomes the standard basis for
 $\mathfrak{h}_1$, $\bar{\mathfrak{m}} = \mathbb{R}Z\oplus\mathbb{R}X, \bar{\mathfrak{m}}_1 = \mathbb{R}Z\oplus\mathbb{R}Y$, and the representations are

\begin{align*}
\mathfrak{\tau}_1&: \chi_1(\exp(zZ+ yY)) = e^{2\pi iz}, \quad \mbox{induced from $\bar{M}_1$ to $\mbox{Heis}_3(\mathbb{R})$},  \\
\mathfrak{\tau}_1&: \chi_1(\exp(zZ+ xX)) = e^{2\pi iz}, \quad \mbox{induced from $\bar{M}$ to $\mbox{Heis}_3(\mathbb{R})$}. 
\end{align*}

These representations can both be modeled in $L^2(\mathbb{R})$ if we identify
$M_1\backslash\mbox{Heis}_3(\mathbb{R}) \cong \exp \mathbb{R}X \cong \mathbb{R}$, $M\backslash\mbox{Heis}_3(\mathbb{R}) \cong \exp \mathbb{R}Y \cong \mathbb{R}$ respectively.

A simple calculation based on the standard models of these representations
shows that the basic one-parameter subgroups $\exp \mathbb{R}Z, \exp \mathbb{R}Y, \exp \mathbb{R}X$ in $\mbox{Heis}_3(\mathbb{R})$ have the following explicit actions on $L^2(\mathbb{R})$:
\begin{align*}
\bar{\mathfrak{\tau}}_1(\exp zZ)f(t) &:= e^{2\pi iz}f(t), \quad \bar{\mathfrak{\tau}}(\exp zZ)f(t) := e^{2\pi iz}f(t), \\
\bar{\mathfrak{\tau}}_1(\exp yY)f(t) &:= e^{2\pi ity}f(t), \quad \bar{\mathfrak{\tau}}(\exp yY)f(t) := f(t+y), \\
\bar{\mathfrak{\tau}}_1(\exp xX)f(t) &:= f(t+x), \quad \bar{\mathfrak{\tau}}(\exp xX)f(t) := e^{-2\pi itx}f(t), 
\end{align*}

The similarity of these actions is evident and strongly suggests that the
Euclidean Fourier transform $\mathcal{F}:L^2(\mathbb{R})\to L^2(\mathbb{R})$ defined by
\begin{equation}
\mathcal{F}f(t) = \int_\mathbb{R} e^{-2\pi itx}f(x) dx \quad (dx\; \mbox{normalized on}\; [0, 1]) \label{EuclidFourier}
\end{equation}
is an $L^2$-isometry intertwining $\bar{\mathfrak{\tau}}_1$ and $\bar{\mathfrak{\tau}}$. To verify this, it suffices to show
that it intertwines the one-parameter actions for $X$ and $Y$:
\begin{align*}
\bar{\mathfrak{\tau}}_1(\exp yY)\mathcal{F}f &= \mathcal{F}\bar{\mathfrak{\tau}}(\exp yY)f, \\
\bar{\mathfrak{\tau}}_1(\exp xX)\mathcal{F}f &= \mathcal{F}\bar{\mathfrak{\tau}}(\exp xX)f,
\end{align*}
for all $f \in L^2(\mathbb{R})$
 but these are well-known properties of $\mathcal{F}$, and easily verified. Thus $\bar{\mathfrak{\tau}}_1 \cong \bar{\mathfrak{\tau}}$ and the proof is complete.
\end{proof}

\begin{remark} The above intertwining between $\bar{\mathfrak{\tau}}_1$ and $\bar{\mathfrak{\tau}}$ is compatible with the decomposition into the direct integral of finite-dimensional representations (\ref{discretetoHeisenbergLie}) in Theorem \ref{discretetoLie}. In fact, comparing with two matrices $\rho_{{\rm \tiny{fin}},x}(u)$, $\rho_{{\rm \tiny{fin}},x}(v)$ in (\ref{matrixrep} ) with $x_2=x_3$ for $x =(x_1,x_2,x_3)$ in the finite dimensional representation $\rho_{{\rm \tiny{fin}},x}$ of $\mbox{Heis}_3(\mathbb{Z})$, we see that the finite-dimensional analogue of the above Fourier transform $\mathcal{F}$  is nothing but the diagonalizing matrix $P$ from $\rho_{{\rm \tiny{fin}},x}(u)$ to $\rho_{{\rm \tiny{fin}},x}(v)$  which is obiously satisfied the above mentioned compatibilities.
\end{remark}

\subsection{Finite dimensional irreducible unitary representations: Substep 0-8}
\label{Step0-8}
Before proceeding to the next step for further decompositions, we recall the result of the complete classification of finite-dimensional irreducible unitary representations of discrete, finitely generated, torsion-free, nilpotent groups, following \cite{Howe}.

Let $\Gamma$ will always denote a discrete, finitely generated,
torsion-free, nilpotent group, $e$ will be its identity element, $\Gamma^{(1)} =\Gamma, \Gamma^{(2)} = [\Gamma,\Gamma], \ldots, \Gamma^{(n)} = [\Gamma,\Gamma^{(n-1)}], \ldots,$  its descending central series. 
If $\Gamma^{(k+1)}= \{e\}$ with $\Gamma^{(k)}\neq \{e\}$, we say $\Gamma$ is $k$-step nilpotent.
Let $G$ be a Malcev completion of $\Gamma$, namely $\Gamma$ is a uniform lattice of $G$. The exponential map $\exp: \mathfrak{g} \to G$ is a diffeomorphism, and $\log: G \to \mathfrak{g}$ is its inverse. We would like to say
$\log \Gamma \subset \mathfrak{g}$ is the lie algebra of $\Gamma$, which is alredy appeared in Subsubsection \ref{Branchingtheorem1}. However, this generally makes no sense because $\log\Gamma$ must not be closed under addition or undertaking commutators (Lie bracket). 
But suppose $G$ is $k$-step nilpotent, and
$L \subset G$ is a uniform lattice, such that $[L, L] \subset k!L$. 
Then, a straightforward calculation using the Baker-Campbell-Hausdorff formula shows that $\Gamma = \exp L$ is a subgroup of $G$. We will call such a lattice $L$, and $\exp L$, elementarily exponentiable, or e.e., for short. It is noticed that the general $\Gamma$ can be reduced in this case in Proposition 0 in \cite{Howe}.
If $H \subset G$ is a subgroup, we say $H$ is saturated if $g^n \in H$
implies $g \in H$, for any $g \in G$. Given $H$, $H_s$ will denote the smallest
saturated subgroup containing $H$. If $H$ is normal, $H$ is
saturated if and only if $G/H$ is torsion-free.

There is a parallel to Kirillov's theory, i.e., based on the orbit philosophy in these discrete settings. 
We also need a notion like ``polarizing subalgebra (a.k.a. maximal subordinate subalgebra)''.
We now provide this. Let $L =\log \Gamma$ for $\Gamma$ e.e.. 
If $\psi \in \hat{L}$ is rational on $\Gamma^{(2)}$, then the radical $R_\psi := \{X \in L |B_\psi(X,Y) := \psi([X,Y]) = 0\; \mbox{for all}\; Y \in L \}$ is finite index in $L$. Moreover, by changing $\psi$ by a character of $\Gamma/\Gamma^{(2)}$, we may assume this $\psi$ is rational.
Then, $B_\psi$ factors to a non-degenerate antisymmetric form $\widetilde{B}_\psi$ on $(\log(\Gamma)/\log(\Gamma^{(2)}))/R_\psi$.  Then we define a polarizing subalgebra $P$ of $\widetilde{B}_\psi$ such that $P \supset R_\psi$ and $P/R_\psi$ forms a half-polarization of $\widetilde{B}_\psi$, i.e., its maximal isotropic subalgebra. Note that ${}^\#(P/R_\psi) = {}^\#(L/P)$.  

In the case when $\Gamma = {\rm Heis}_3(\mathbb{Z})$, if $l = \alpha_1X_1^\ast +\alpha_2X_2^\ast$ with $\alpha_1 = p/q$, then $R_\psi = (\mathbb{Z}, q\mathbb{Z}, q\mathbb{Z})$ and $P =  (\mathbb{Z}, \mathbb{Z}, q\mathbb{Z})$ for the notation of Example \ref{BekkaHeisenberg1}

Then, the finite-dimensional irreducible unitary representation of $\Gamma$ is classified as follows;  

\begin{theorem}[{\rm Theorem 1 in \cite{Howe}}]\label{Howefinite}
Let $\Gamma$ be an e.e., discrete, finitely generated, torsion
free-nilpotent group. Let $L = \log \Gamma, L' = 2L, \Gamma' = \exp L', L_s^{(2)} = \log\Gamma^{(2)}$, and $n\Gamma = \exp (nL)$ for $n > 2$.
\begin{description}
\item[{\rm (a)}] Let $O$ be a finite ${\rm Ad}^\ast\Gamma$ orbit in $\hat{L}$, and $\psi \in O$. 
Let $n$ be the period of $\psi|_{L_s^{(2)}}$. Call $n$ the period of $O$. 
If $n$ is odd, then a finite-dimensional irreducible unitary representation $U_0$ of $\Gamma$ may be associated with $O$ in the following manner. 
If $P$ is an e.e. polarizing subalgebra for $\psi$, and $\Pi = \exp P$, put $\psi|_\Pi = \tilde{\psi}$.
 Then $U_0 = U^{\tilde{\psi}} = {\rm Ind}_P^\Gamma(\chi_{\tilde{\psi}})$ where the chacter $\chi_{\tilde{\psi}} = e^{i\tilde{\psi}}$.
The dimension of $U_0$ is ${}^\#(O)^{1/2}$, and the character $\zeta_0$ of $U_0$ is given by
$\zeta_0 = {}^\#(O)^{-1/2}\sum_{\varphi\in O}\varphi$.
 All representations of the form $\chi\otimes V$, with $\chi \in \widehat{\Gamma/\Gamma_s^{(2)}}$, 
and $V$ defined modulo $\Gamma, n \mbox{odd}$, are realized in this manner.

\item[{\rm (b)}] In general, there is a surjective map from finite-dimensional
irreducible unitary representations of $\Gamma'$ to finite ${\rm Ad}^\ast\Gamma$ orbits in $\widehat{L}'$. 
This map is at most ${}^\#(\Gamma/\Gamma')$-to-one. 
If $\{U_i\}_{i=1}^l$ map onto the orbit $O$, then the $U_i$ are permuted transitively among themselves by the
action ${\rm Ad}^\ast \Gamma/\Gamma'$, so that they define a point in $M(\Gamma', \Gamma )$. $l$ is a
power of $2$, and if $m$ is the common dimension of the $U_i$, $lm^2 ={}^\#(O)$.
The sum of the characters $\xi_i$ of the $U_i$ is given by the formula $\sum \xi_i =m^{-1}\sum_{\varphi \in O}\varphi$.
\end{description}
\end{theorem} 

\subsection{Further decomposition to finite-dimensional irreducible unitary representations: Substep 0-6}\label{Step0-6}
\subsubsection{Decomposition of $\mbox{Ind}_{M\cap\Gamma}^\Gamma (\chi_{l_{t_{m+1},\ldots,t_n}}|_{M\cap\Gamma})$ }
Here, we shall give a further decomposition of $\mbox{Ind}_{M\cap\Gamma}^\Gamma (\chi_{l_{t_{m+1},\ldots,t_n}}|_{M\cap\Gamma})$ appeared as an integrand of the righthand side of the formula (\ref{pildecomposition}) under some technical condition on $l \in \mathfrak{g}_{\mathbb{Q}}^\ast$ as a direct integral of finite-dimensional unitary irreducible representations. 

In succession to Section \ref{Step0-5} and also, as we already mentioned in Substep 0-5 in Section \ref{nilpotentexplanation}, if we apply the argument here to the case of $\mbox{Heis}_3(\mathbb{Z})$ in Section \ref{relationdiscLie}, it is appeared as the relation (\ref{step1discretetolie}) there.

We recall the setting again. Write $\Gamma = \mathbb{Z}$-span$\{X_1,\ldots, X_n\} \cong \mathbb{Z}^n$ for the weak Malcev basis $\{X_1,\ldots, X_n\} \cong \mathbb{Z}^n$ in Theorem \ref{firstdecomposition}. Moreover $l\in\mathfrak{g}_\mathbb{Q}^\ast$, the radical $\mathfrak{r}_l$ and a polarizing subalgebra $\mathfrak{m}$ contained in $\mathfrak{g}_\mathbb{Q}$ appeared in be in Theorem \ref{firstdecomposition}, whose dimensions are $r$ and $m$ respectively.  Recall $R_l = \exp \mathfrak{r}_l$ and $M = \exp \mathfrak{m}$. 

Furthermore, recall several notions related to representations of $\Gamma$.
Take $\psi \in L^\ast$ with $L = \log \Gamma$ satisfying $\psi = l|_L$ for some $l \in \mathfrak{g}_\mathbb{Q}^\ast$. Let $R_\psi$ and $P$ be the radical and a polarization of $\psi$ in $\Gamma$.
Then, by definition, 
$P = \log(\mbox{St}_H^\Gamma(\chi_\psi))$ is the polaring subalgebra of the restriction $\psi$ in $\log \Gamma$. 

Take $l \in \mathfrak{g}_{\mathbb{Q}}^\ast$ satifying $l(X_i) = t_i = p_i/q_i, \; i=m+1,\ldots,n$ such that 
\begin{itemize}
\item $q_i$ is positive odd integer and $p_i$ is an integer satisfying $ 0\leq p_i \leq q_i$,
\item for the least common multiple of $q_i'$s $q$, it satisfies
\begin{equation} R_\psi \supset \mathbb{Z}^p \times (q\mathbb{Z})^{m-p} \times (q\mathbb{Z})^{n-m} = \mathbb{Z}^p \times (q\mathbb{Z})^{2(n-m)}.  \label{rpsi}
\end{equation} 
\item furthermore, 
\begin{align} 
 &{}  \mathfrak{m}\cap \log(\Gamma) \cong \mathbb{Z}^m \notag \\
 &{}  P \cong \mathbb{Z}^m \times q_{m+1}\mathbb{Z} \times \cdots \times q_n\mathbb{Z}  \label{mloggamma}
\end{align} 
\end{itemize}
Here, the reason why we consider the condition that $q_i$ is an odd integer is technical, whose irreducible representation is simpler than the other cases as in Theorem \ref{Howefinite} (a). 

Then, we get a decomposition
\begin{equation*}
\pi_{l,M}|_\Gamma \cong \int_{[0, 1)^{n-m}}^\oplus
\mbox{Ind}_{M\cap\Gamma}^\Gamma(\chi_{l_{t_{m+1},\ldots., t_n}}|_{M\cap\Gamma})dt_{m+1}\cdots dt_n,
\end{equation*}
where
\begin{equation*}
l_{t_{m+1},\ldots,t_n} := \mbox{Ad}^\ast(\exp(-t_{m+1}X_{m+1})\cdots\mbox{Ad}^\ast(\exp (-t_nX_n)) l
\end{equation*}
for all $t_{m+1}, \ldots, t_n \in [0, 1)$.

Let $\psi = l|_{\log \Gamma}$ for $l$ satisfying the above conditions. Then, we see the induced representation Ind$_{M\cap\Gamma}^{\mathcal{P}}(\textbf{1})$ of the trivial representation $\textbf{1}$ is isomorphic to the right regular representation $R_{\mathcal{P}/\mathcal{M}}$ of $\mathcal{P}/\mathcal{M}$ with $\mathcal{P} = \exp(P)$, 
$M = \exp(\mathfrak{m})$ and $\mathcal{M} = M \cap \Gamma$ 

Since $P/(\mathfrak{m} \cap \log(\Gamma)) \cong  q_{m+1}\mathbb{Z} \times \cdots \times q_n\mathbb{Z}$, we have 
\[  L^2(\mathcal{P}/\mathcal{M}) \cong L^2(q_{m+1}\mathbb{Z} \times \cdots \times q_n\mathbb{Z}) \cong L^2(\mathbb{Z}^{(n-m)}) \cong L^2(U(1)^{(n-m)}) \cong L^2(([0,1)^{(n-m)})).
\] 
and
\[  R_{\mathcal{P}/\mathcal{M}} \cong \int_{[0,1)^{n-m}}^\oplus \chi_{\psi_s}ds
\]
where 
\[ \chi_{\psi_s}(k) = \chi_{\psi_{(s_{m+1},\ldots ,s_n)}}(k_{m+1},\ldots, k_n) = e^{2\pi\sqrt{-1}((s_{m+1}/q_{m+1})k_{m+1}+\cdots +(s_n/q_n)k_n)}\]
for \[k =(k_{m+1},\ldots,k_n) \in \mathbb{Z}^{n-m} \cong \mathcal{P}/\mathcal{M} \]
and $ds := ds_{m+1}\ldots ds_n$ is the Lebesgue measure on $[0,1)^{n-m}$.
  
Then we have
\begin{align}   \mbox{Ind}_{M\cap\Gamma}^\Gamma(\chi_{l_t}) &= \mbox{Ind}_{\mathcal{P}}^\Gamma( \mbox{Ind}_{M\cap\Gamma}^{\mathcal{P}}(\chi_{l_t}))  \notag \\
&\cong \mbox{Ind}_{\mathcal{P}}^\Gamma\left( \int_{[0,1)^{n-m}}^\oplus \chi_{l_t}\otimes \chi_{\psi_s}ds  \right)\notag \\ 
&= \int_{[0,1)^{n-m}}^\oplus\mbox{Ind}_{\mathcal{P}}^\Gamma(\chi_{l_t}\otimes \chi_{\psi_s})ds \label{intermidietedecomposion}
\end{align}
Note that each integrand of the right-hand side is a finite-dimensional representation; these decompositions are what we wanted. Thus, combining with Theorems \ref{characterfujiwara} and \ref{plancherelfujiwara}, we have the following theorem: Let $\mathbb{Q}^{\rm odd}$ be the set of rational numbers which have odd denominators as their irreducible fraction expressions.  

\begin{theorem}{\rm ($=$ Theorem \ref{introdecompose} )}\label{sec4decompose}

Assume that $l \in \mathfrak{g}_{\mathbb{Q}}^\ast$. Then there exists a polarization $\mathfrak{m}$ of dimension $m$, a rational subalgebra of $\mathfrak{g}$ satisfying the following. Let $\{X_1,\ldots, X_m, \ldots, X_n\}$ be a weak
Malcev basis of $\mathfrak{g}_\mathbb{Q}$, based on $\Gamma$ and passing through $\mathfrak{m}$.
Then, we have the decomposition
\begin{equation} \pi_l|_\Gamma = \int_{[0,1)^{n-m}}^\oplus {\rm Ind}_{M\cap\Gamma}^\Gamma (\chi_{l_t |_{M\cap\Gamma})}dt_{m+1}\cdots dt_n
 \label{intropildecomposition}
\end{equation}
where 
\begin{equation*}
l_t := l_{t_{m+1},\ldots,t_n} = {\rm Ad}^\ast (\exp(t_{m+1}X_{m+1}))\cdots{\rm Ad}^\ast(\exp(t_nX_n))l \end{equation*}
for all $t_{m+1}, \ldots, t_n \in [0, 1)$.

Moreover if $t = (t_{m+1}, \ldots, t_n ) \in ([0,1) \cap \mathbb{Q}^{\rm odd})^{n-m}$, then integrand of the right hand side of (\ref{intropildecomposition}) is further decomposed into finite dimensional unitary representations ${\rm Ind}_{\mathcal{P}}^\Gamma (\chi_{l_t}\otimes\chi_{\psi_s})$ as:
\begin{equation}
{\rm Ind}_{M\cap\Gamma}^\Gamma(\chi_{l_t}) = {\rm Ind}_{\mathcal{P}}^\Gamma({\rm Ind}_{M\cap\Gamma}^{\mathcal{P}}(\chi_{l_t})) 
= \int_{[0,1)^{n-m}}^\oplus{\rm Ind}_{\mathcal{P}}^\Gamma(\chi_{l_t}\otimes \chi_{\psi_s})ds \label{finaldecomposition}
\end{equation}
where $M = \exp \;\mathfrak{m}$, $\Gamma^{(2)} = [\Gamma,\Gamma]$, $\Gamma_s^{(2)}$ is the smallest saturated subgroup of $\Gamma^{(2)}$ and $\mathcal{P}$ is the polarization of $\log \;\Gamma$.
\end{theorem}
Note that $\chi_{\psi_s}$ is the same as $\chi$ in \ref{Howefinite}.


\subsubsection{The Fourier inversion formula for general discrete nilpotent groups: a generalization of Pytlik's theorem}

Finally if we combine Theorems \ref{Bekkafactor} with (\ref{FujiwaraPlancherel}), (\ref{pildecomposition}) but modifying as follows:
\begin{description}
\item Replace integral domain $\widehat{Z(\Gamma)}$ in (\ref{Bekkadecopequation}) and of $G\cdot \zeta\cap\Xi$ in (\ref{Fujiwaraformula2})  
to the dense subsets whose element satisfies rationality conditions and the assumption of Theorem \ref{Howefinite} (a) if we restrict the domain of definition to $\Gamma$
\item Replace the Plancherel measures $d\chi$ in (\ref{Bekkadecopequation}) and $d\lambda$ in (\ref{Fujiwaraformula2}) to finitely additive measures $d\widetilde{\chi}$ and $d\widetilde{\lambda}$ which are defined on rational subsets of $\widehat{Z(\Gamma)}$ and $G\cdot \zeta\cap\Xi$ respectively and whose values are same as those of $d\chi$ and $d\lambda$ on their closures.
\end{description}

The above modifications follow the same spirit as Pytlik's theorem (Theorem \ref{Pytlik}) and the Fourier inversion formula (\ref{Fourier1}) for the three-dimensional discrete Heisenberg group $\mbox{Heis}_3(\mathbb{Z})$. Then, we have the following generalization of them to discrete, torsion-free, nilpotent groups.

\begin{theorem}{\rm (}$=$ Theorem \ref{introPytlik}{\rm )}\label{newPytlik} 
 For a finitely generated discrete torsion-free nilpotent group $\Gamma$, 
there exists a positive, finitely additive measure $\mu$ on the unitary dual $\widehat{\Gamma}$ of $\Gamma$, supported by the set $\widehat{\Gamma}_{\rm fin}$ of (equivalence classes of) finite-dimensional representation of $\Gamma$, such that for $f \in L^1(\Gamma )$, the function 
\[ \pi \mapsto {\rm tr}(\pi f) := \frac{1}{{\rm dim}\; \pi}{\rm Tr}(\pi (f))
\]
is $\mu$-integrable on $\hat{\Gamma}$ and 
\begin{equation}     f(0) = \int_{\widehat{\Gamma}_{\rm fin}}\frac{1}{{\rm dim}\; \pi}{\rm Tr}(\pi (f))d\mu(\pi ). \label{nilpotentPytlik2}
\end{equation}
\end{theorem}

\begin{remark}\label{PlancherelLie} 
\begin{description}
\item[{\rm (1)}] As in the case of the Heisenberg group in Remark \ref{remark34}, it should be noted that once establishing the above formula, 
it also gives an approximation of the left-hand side by each integrand of the right-hand side by regarding their fluctuations as an error bound. Moreover, the fluctuation can be easily estimated using Theorem \ref{Howefinite}, which is essentially the same as (\ref{fluctuation}) in the case of the Heisenberg group.

\item[{\rm (2)}] Once we establish these approximation results, concerning a vicinity of the trivial representation at least, we can replace the Plancherel measure $d\mu$ in the formula (\ref{nilpotentPytlik1}) ($=$ (\ref{nilpotentPytlik2})) to the one $\mbox{Tr}\;\pi_l(\phi) dl$ of (\ref{B2}) in Theorem \ref{FourierinversionLie} in the actual computation. In fact, this is based on the fact that the approximation errors are improved as the denominators of the irreducible fractions of the rational numbers in approximations become larger and larger.

\end{description}

\end{remark}

\subsection{Filtrations of Lie algebras and stratified Lie groups: Substep 0-4}\label{stratification}

In this section, we recall from \cite{Alexopoulos3} notions of a stratified Lie algebra and stratification of a general nilpotent Lie algebra. 

First, we recall the filtration of the Lie algebra. 
Let $\mathfrak{g}$ be the Lie algebra of nilpotent Lie group $G$, which we
identify with the left-invariant vector fields on $G$.
We set $\mathfrak{g}_1 = \mathfrak{g}$ and $\mathfrak{g}_{i+1} = [\mathfrak{g}_1, \mathfrak{g}_i ],\quad i > 1$. Since $\mathfrak{g}$ is nilpotent, we have a filtration
\[\mathfrak{g} = \mathfrak{g}_1 \supset \mathfrak{g}_2 \supset \cdots \supset \mathfrak{g}_n \supset \mathfrak{g}_{n+1} = \{0\},\quad \mathfrak{g}_n \neq \{0\}.
\]
We consider linear subspaces $\mathfrak{g}^{(1)}, \ldots, \mathfrak{g}^{(n)}$ of $\mathfrak{g}$ such that
\begin{equation}\mathfrak{g}_i = \mathfrak{g}^{(i)} \oplus \cdots \oplus \mathfrak{g}^{(n)}, \quad 1 \leq i \leq n.  \label{liealgdecomposition}
\end{equation}

We set
\begin{align}
&{} g_0 = 0,\; g_i = \mbox{dim}(\mathfrak{g}/\mathfrak{g}_{i+1}) = \mbox{dim}(\mathfrak{g}^{(1)}\oplus \cdots \oplus\mathfrak{g}^{(i)}),\quad 1 \leq i \leq n, \\
&{} \sigma(j) = i,\quad \mbox{for}\quad g_{i-1} < j \leq g_i,\\
&{} q = g_n = \mbox{dim}(\mathfrak{g}).\end{align}\

Notice that the homogeneous dimension $D$ of $G$ is given by
\[D = \sigma(1)+ \cdots + \cdots +\sigma(q).\]

We say that $\mathfrak{g}$ is stratified if the following conditions hold:
\begin{align}
&{} \mathfrak{g} = \mathfrak{g}_1 = \mathfrak{g}^{(1)} \oplus \cdots \oplus \mathfrak{g}^{(n)},  \label{liestratified1} \\
&{} [\mathfrak{g}^{(i)},\mathfrak{g}^{(j)}] \subset \mathfrak{g}^{(i+j)}, \label{liestratified2} \\
&{}  \mathfrak{g}^{(1)}\; \mbox{generates}\; \mathfrak{g}\; \mbox{as a Lie algebra},  \label{liestratified3}
\end{align}

We consider a basis $\{X_1, \ldots,X_q \}$ of $\mathfrak{g}$ such that $\{X_{g_{i-1}+1}, \ldots ,X_{g_i}\}$
is a basis of $\mathfrak{g}^{(i)}, \quad 1 \leq i \leq n$.

On the linear space $\mathfrak{g}$, we define the Lie bracket $[\cdot,\cdot]_0$ by setting
\begin{equation} [X_i,X_j ]_0 = \mbox{pr}_{\mathfrak{g}^{\sigma (i)+\sigma (j)}}[X_i,X_j] \label{zerobracket}
\end{equation}
where $\mbox{pr}_{\mathfrak{g}^{\sigma (i)+\sigma (j)}}$ is the projection onto $\mathfrak{g}^{\sigma (i)+\sigma (j)}$.
We denote by $\mathfrak{g}_0$ the Lie algebra $\mathfrak{g}_0 = (\mathfrak{g}, [\cdot,\cdot ]_0)$. Note that $\mathfrak{g}_0$ is nilpotent.

Take the exponential coordinates of the second kind (or Malcev coordinates)
\[\phi : \mathbb{R}^q \to G, \qquad \phi : x = (x_q,\ldots, x_1) \mapsto  \exp(x_qX_q)\cdots \exp(x_1X_1).
\]

Let $\delta_\varepsilon, \; \varepsilon > 0$
, be the family of dilations of $G$ defined by
\[ \delta_\varepsilon : (x_q,\ldots, x_1) \mapsto
(\varepsilon^{\sigma (q)}x_q,\ldots, \varepsilon^{\sigma (1)}x_1)
.\]
Also, let $\ast_\varepsilon, \; \varepsilon > 0$, be the family of group products defined by
\[x\ast_\varepsilon y = \delta_\varepsilon[(\delta_{\varepsilon^{-1}}x)(\delta_{\varepsilon^{-1}}y)] \]
and let
\[ x \ast_0 y = \lim_{\varepsilon \to 0}x \ast_\varepsilon y. \]
Then $G_0 = (G, \ast_0)$ is a stratified nilpotent Lie group whose Lie algebra is isomorphic to $\mathfrak{g}_0$. 

Next, we take canonical $\ast_0$-coordinates on $G_0$. Here we denote $\ast = \ast_0$. 
\begin{equation} \phi^\ast :\mathbb{R}^q \to G_0, \qquad \phi^\ast : x = (x_{q,\ast},\ldots, x_{1,\ast}) \mapsto  \exp(x_{q,\ast}X_q)\ast \cdots \ast\exp(x_{1,\ast}X_1). \label{stratificationlie}
\end{equation}
Note that the above $\delta_\varepsilon$ naturally defines dilatation on  $\mathfrak{g}_0$ and, thus, on $G_0$.

We identify $G$ and $G_0$ as a differentiable manifold. Then, the relations of the above two coordinates are described as follows:
\begin{align}
x_{i,\ast} &= x_i &(1 \leq i \leq \sigma(1)+\sigma(2)) \\
x_{i,\ast} &= x_i + \sum_{0<|K|\leq k}C^KP_K(x)  \quad &(\sum_{j=1}^k\sigma(j)+1 \leq i \leq \sum_{j=1}^{k+1}\sigma(j), 2 \leq k \leq n-1) \label{stratifiedcanonical}
\end{align}
where $K$ stands for a multi-index $(i_1,\ldots,i_\ell)$, $|K| = \sigma(i_1)+\cdots +\sigma(i_\ell)$ and $P_K(x) := x_{i_1}\cdots x_{i_\ell}$. This formula can be derived from (\ref{zerobracket}) inductively (cf \cite{Ishiwata},\cite{Ishiwata4}).

\subsection{Computations of hypo-elliptic operators: Substep 0-9, with several examples}\label{Step0-9}

In the above arguments, especially the decomposition formula (\ref{finaldecomposition}) and the structure theorem of finite dimensional irreducible representations of $\Gamma$ in Theorem \ref{Howefinite} (i), for any element $\pi$ in a dense subset of irreducible unitary dual of the Malcev completion $G$ of $\Gamma$, the restriction $\pi|_\Gamma$ can be 
approximated by finite dimensional representations of $\Gamma$ in arbitrality order.
Thus, the Plancherel measure $\mu$ in Theorem  \ref{newPytlik} also approximates $dl$ in Theorem \ref{FourierinversionLie}. 

There is a problem of trace properties of test functions mentioned in Section \ref{Heisenbergexplanation} and after Theorem \ref{Fourierinversionheisenberglie} for the Heisenberg group $\mbox{Heis}_3(\mathbb{Z})$. We explain here in more detail. When we regard a function $f \in L^1(\Gamma)$ as a function $\tilde{f} \in L^1(G)$, it seems natural to define the latter as a locally constant function which takes a value $f(\sigma)$ on $\sigma \mathcal{D}$ where $\mathcal{D}$ ia a fundamental domain of the canonical covering $G \to G/\Gamma$. 
In this case $\tilde{f}$ in $L^1(G)$ but not in $\mathcal{S}(G)$. Then, the operator $\pi_l(\tilde{f})$ is a compact operator but not a trace class operator; thus, we could not apply the Fourier inversion formula (\ref{FourierinversionLie}), and thus, the computation above is a formal one.
The formal manipulations of irreducible representations of $G$ can be a substitute for actual computation for those of $\Gamma$ in virtue of Theorem \ref{introdecompose} ($=$ Theorem \ref{sec4decompose}).
 
In addition, we may reduce to the case where $G$ is stratified by the previous section, our hypoelliptic operator $\mathcal{H} = \mathcal{H}^G$ is expressed as 
\begin{equation} \mathcal{H} = \mathcal{H}^G = c_1d\pi_l(X_1)^2+\cdots c_kd\pi_l(X_k)^2, \label{generalhypo1}\end{equation}     
where $c_j, \; 1 \leq j \leq k$ is a positive constant determined in later $d\pi_l$ is the differential of the irreducible representation $\pi_l$ with $l$ as an element of the dense subset of the principal orbits appeared in Theorem \ref{FourierinversionLie} and $X_1,\ldots, X_k$ is an orthonormal basis of $\mathfrak{g}^{(1)}$ of the stratified Lie algebra $\mathfrak{g}$ of $G_0 = (G, \ast_0)$ defined by (\ref{stratificationlie}). For a more detailed explanation, see (\ref{detailedgeneralhypo1}).
 
We notice that $\mathcal{H}$ can be expressed as a polynomial of the multiplication operator $x_i$ and the differential operators $\partial /\partial x_i$ for some local coordinates $\{x_i\}$ of $G$.

Here, we give several computations for three examples in Appendix B.
 
\begin{example}[$(2n+1)$ dimensional Heisenberg group $H_n = {\rm Heis}_{2n+1}(\mathbb{Z})$]

From the Fourier inversion formula (\ref{2n+1heisenbergplancherel}), the principal orbit can be expressed as $\{\alpha l_1|\alpha \neq 0\}$ and the Plancherel measure is $|\alpha|^nd\alpha$ and thus, the polynomial growth order $d$ satisfies $d/2 = n+1$.  
By (\ref{2n+1heisenbergirred}), the corresponding unitary irreducible representations are 
\begin{equation}
\pi_{\alpha l_1}(z, y, x)\tilde{f}(t)= e^{2\pi\sqrt{-1}\alpha (z+t\cdot y+(l/2)x\cdot y)}\tilde{f}(t+x) \quad (\lambda \neq 0).
\end{equation}
which are unitary equivalent to 
\begin{equation}
\widetilde{\pi}_{\alpha l_1}(z, y, x)\tilde{f}(t)= e^{2\pi\sqrt{-1}\sqrt{\alpha} (z+t\cdot y+(l/2)x\cdot y)}\tilde{f}(t+\sqrt{\alpha}x) \quad (\lambda \neq 0).
\end{equation}
If $n=1$, this is the same representation as (\ref{srep}) but using different coordinates as explained there and replacing $\alpha$ here with $h$ there.

Therefore, the hypo-elliptic operator $\mathcal{H} = \mathcal{H}_{H_n}$ is expressed as
\begin{align*}
\mathcal{H} &= \sum_{i=1}^n \left( d\pi_{l_1}(X_i)^2+d\pi_{l_1}(Y_i)^2\right)\\
&= \sum_{i=1}^n \left( \left(\frac{\partial}{\partial t_i}\right)^2-4\pi^2t_i^2\right)
\end{align*}
acting  on $L^2(\mathbb{R}^n)$, which is the $n$-dimensional Harmonic oscillator.
Its eigenvalues are $4\pi^2\sum_{i=1}^n(2k_i+1)$ with $k_i = 0,1,2,\ldots,$ and the special value $\zeta_{\mathcal{H}}(d/2)$ of the spectral zeta function of $\mathcal{H}$ at $d/2 = n+1$ is
\[ \frac{1}{16\pi^4}\zeta_{\mathcal{H}}(d/2) =\sum_{i=1}^n\sum_{k_i =0}^\infty \frac{1}{k_1^{n+1}+\cdots +k_n^{n+1}}.\] 
In particular, when $n=1$, 
\[\zeta_{\mathcal{H}}(d/2) = \zeta_{\mathcal{H}}(2) = \sum_{i=1}^\infty \frac{1}{(2k+1)^2} = \sum_{k=1}^\infty \frac{1}{k^2}-\frac{1}{4}\sum_{k=1}^\infty \frac{1}{k^2} = \frac34\zeta(2) = \frac{\pi^2}{8}
\]
where $\zeta(s)$ is the Riemann zeta function.

\end{example}

\begin{example}[Engel group $E_4$]\label{Engelquartic}
From the Fourier inversion formula (\ref{engelplancherel}), the principal orbit can be expressed as $\{\delta l_1 + \beta l_3|\delta, \beta \neq 0\}$ and the Plancherel measure is $|\delta|d\delta d\beta$. Since the Lie algebra $\mathfrak{g}= \mathfrak{e}_4$ of $E_4$ has the following stratified structure of $\mathfrak{g}= \mathfrak{g}^{(1)}\oplus \mathfrak{g}^{(2)}\oplus \mathfrak{g}^{(3)}$ with
\[ \mathfrak{g}^{(1)}=\{W,X\}, \mathfrak{g}^{(2)}=\{Y\}, \mathfrak{g}^{(2)}=\{Z\},
\] 
the polynomial growth order $d=7$ is coming from the fact that the weight of elements in $\mathfrak{g}^{(i)}$ is equal to $i$ and $7 = 1+1+2+3$.
 This coincides with the computation for the Planchrel measure $|\delta|d\delta d\beta$ where the weights of $\delta$ and $\beta$ are $3$ and $1$ respectively and $7 = 3 \times 2 +1$.
 
By (\ref{engelirred}), the corresponding unitary irreducible representations are
\begin{align*}
&{}[\pi_{\delta,\beta}(z,y,x,w)\tilde{f}](t) \\
&{}= e^{2\pi\sqrt{-1}\beta x}e^{2\pi\sqrt{-1}\delta(z+ty+(l/2)t^2x +(1/2)twx +(l/2)wy + (l/6)w^2x)}\tilde{f}(t+ w). 
\end{align*} 

Therefore, the hypo-elliptic operator $\mathcal{H} = \mathcal{H}_{E_4}$ is expressed as
\begin{align}
\mathcal{H} &= \left( d\pi_{\delta,\beta}(W)^2+d\pi_{\delta,\beta}(Y)^2\right) \notag \\
&= \left(\frac{\partial}{\partial t}\right)^2+t^4 \label{engelhypo}
\end{align}
for $\tilde{f} \in L^2(\mathbb{R})$. This operator is called a (modified) quartic operator.

Its eigenvalues seem not to be known explicitly; however, the special value $\zeta_{\mathcal{H}}(d/2)$ of the spectral zeta function of $\mathcal{H}$ at $d/2 = 7/2$ can be expressed as a certain integral including the Bessel functions. This fact is obtained by the method explained in Appendix C in \cite{Voros} or \cite{Voros2} as follows: We shall express the special values of the spectral zeta function
\[   \zeta_{\mathcal{H}}(s) = \sum_{k=1}^\infty \frac{1}{\lambda_k^s}
\]
where $\lambda_k$ is the $k$-th eigenvalue of $\mathcal{H}$. Note that the spectrum of $\mathcal{H}$ consists of discrete eigenvalues. 

Although it is enough to consider the case $M=2$ for the quartic oscillator here, we quote the general integer $M$ cases as in Appendix C in \cite{Voros}.

Let the homogeneous Schr\"odinger operator $\widehat{H} = \frac{d^2}{dq^2}  + q^{2M}$ be the Hamiltonian here. 

Our computations rely on several hypergeometric classical functions involving the Bessel functions $K_\mu$ and $I_\mu$ in \cite{Erdelyi} and the generalized hypergeometric function ${}_pF_q\left(\begin{array}{ccc} a_1 &  \ldots &  a_p \\  b_1 &  \ldots &  b_q\end{array};z\right)$ in \cite{Slater} are used.

Bessel's differential equation is defined as follows: for $\mu \not\in \mathbb{Z}$, 
\[ z^2\frac{dw^2}{dz^2}+z\frac{dw}{dz}+(z^2-\mu^2) = 0. \]
The functions $J_\mu(z)$ and $J_{-\mu}(z)$ are independent solutions, where 
\[ J_\mu(z) = \sum_{m=1}^\infty \frac{(-1)^m(z/2)^{2m+\mu}}{m!\Gamma(m+\mu+1)}\]is called Bessel function of the first kind and $z, \mu$ are called its variable and order respectively.

If $z$ is replaced by $iz$, then Bessel's differential equation becomes
\[ z^2\frac{dw^2}{dz^2}+z\frac{dw}{dz}-(z^2+\mu^2) = 0. \]
The functions $J_\mu(iz)$ and $J_{-\mu}(iz)$ are independent solutions, but the following functions are more often used.
\begin{align*}
I_\mu(z)&= e^{-(1/2)i\mu\pi}J_\mu(e^{(1/2)i\pi}z) \; (\mbox{Modified Bessel function of the first kind})\\
K_\mu(z)&= \frac{\pi}{2\sin \mu\pi}\left(I_{-\mu}(z) - I_\mu(z)\right) \; (\mbox{Modified Bessel function of the third kind})
\end{align*}
Next, we define the generalized hypergeometric function ${}_pF_q\left(\begin{array}{ccc} a_1 &  \ldots &  a_p \\  b_1 &  \ldots &  b_q\end{array};z\right)$ in (C.9), (C.10) in \cite{Voros} by
\[ {}_pF_q\left(\begin{array}{ccc} a_1 &  \ldots &  a_p \\  b_1 &  \ldots &  b_q\end{array};z\right) = 
\Gamma\left(\begin{array}{ccc} b_1 &  \ldots &  b_q \\ a_1 &  \ldots &  a_p \end{array}\right)\prod_{n=0}^\infty\Gamma\left(\begin{array}{ccc} a_1+n &  \ldots &  a_p+n \\  b_1+n &  \ldots &  b_q+n\end{array}\right)\frac{z^n}{n!}
\]
with the notation
\[ \Gamma\left(\begin{array}{ccc} b_1 & \ldots & b_q \\ a_1 & \ldots & a_p \end{array}\right) = \frac{\prod_{k=1}^q\Gamma(b_k)}{\prod_{j=1}^p\Gamma(a_j)},
\]
where $\Gamma(s)$ is usual gamma function.
Connection to Bessel functions at (C.11), (C.12) in \cite{Voros} are 
\begin{align*} 
I_\mu(z)^2 &= \frac{(z/2)^{2\mu}}{\Gamma(\mu+1)^2}{}_1F_2\left(\begin{array}{ccc} {} & \mu+1/2  & {} \\  1+\mu &  {}  & 1+2\mu\end{array};z^2\right) \\
I_\mu(z)I_{-\mu}(z) &= \frac{\sin \mu\pi}{ \mu\pi}{}_1F_2\left(\begin{array}{ccc} {} & 1/2  & {} \\  1+\mu &  {}  & 1-\mu\end{array};z^2\right)
\end{align*}

Our basic formula is the resolvent kernel $R(E;q,q')$ of the operator $\hat{H}$ in (C.14) in \cite{Voros} 
\[  (\hat{H}-E)^{-1}(q,q') = R(E;q,q') =W(E)^{-1}\psi_+(q_<)\psi_-(q_>),\]
where $W(E)=\psi_-\psi'_+-\psi_+\psi'_-$ is the Wronskian, $\psi_\pm$ are two solutions recessive at $q = \mp \infty$, namely $\psi_\mp$  grows exponentially fast (is dominant) relative to $\psi_\pm$ (recessive) at $q = \mp \infty$ (p.230 \cite{Voros}) of the equation $(\hat{H}-E)\psi_\pm = 0$ and $q_< = \min(q, q')$, and $ q_> = \max(q, q')$. 

If $\psi_\pm$ for $E=0$, these are expressible in terms of Bessel functions using Lommel's transformation (\S 7.2.8 \cite{Erdelyi}). If we set
\[   \mu = (2M+1)^{-1}, z=q^{M+1}/(M+1) \quad (q>0).\] 
Then, the solutions were normalized according to 
\begin{align*}
\psi_{+}(q_1) &:= \psi_{+,q_0}(q_1,E=0) \sim p(q_1)^{-1/2}\exp\left(i\int_{q_0}^{q_1}p(q)dq\right),\quad q_1 \to -\infty \\
\psi_{-}(q_1)&:= \psi_{-,q_0}(q_1,E=0) \sim p(q_1)^{-1/2}\exp\left(-i\int_{q_0}^{q_1}p(q)dq\right),\quad q_1 \to \infty 
\end{align*}
where $p(q) = -iq^M$ (cf. (4.13) \cite{Voros}) with $q_0 = 0$ read:
\begin{align*}
\psi_-(q)&=  \psi_+(-q) = (2i/\pi)^{1/2}(2\mu)^{(1/2)-\mu}z^\mu K_\mu(z)\quad (q > 0) \\
\psi_+(-q) &= \psi_-(-q) = (2i/\pi)^{1/2}(2\mu)^{(1/2)-\mu}z^\mu(K_\mu(z)+(\sin \mu\pi)^{-1}I_\mu(z)\quad (q > 0)    
\end{align*}
and $W = W(E=0) = 2i(\sin \mu\pi)^{-1}$ as (C.16),(C.17) in \cite{Voros}.
Then, by using the kernel $R(q,q'):=R(0;q,q')$ of the integral expression of $(\hat{H})^{-n}$, we have, as (C.19) in \cite{Voros}, the special value $\zeta_{\hat{H}}(n)$ of the spectral zeta function of $\hat{H}$ at integers $n = 1,2,\ldots$ can be expressed as:
\[ \zeta_{\hat{H}}(n) = \int_{\mathbb{R}^n}R(q_1,q_2)R(q_2,q_3)\ldots R(q_n,q_1)dq_1\ldots dq_n. \] 
 
Moreover, our required value $\zeta_{\hat{H}}(7/2)$ of the quartic oscillator $\hat{H} = -d^2/dq^2 + q^4$ is expressed as
\[\zeta_{\hat{H}}(7/2) = \int_{\mathbb{R}^4}R(q_1,q_2)R(q_2,q_3)R(q_3,q_4)B(q_4,q_1)dq_1\ldots dq_4,
\] 
where $B(q,q')$ is the integral kernel of the square root $\hat{H}^{-1/2}$ of $R :=  \hat{H}^{-1}$ which can be defined as 
\[ R^{1/2} = 2\|R\|^{1/2}(R/4\|R\|)^{1/2} = 2\|R\|^{1/2}\sum_{n=1}^\infty(-1)^n\left(\begin{array}{c}1/2 \\ n\end{array}\right)\left(I-\frac{R}{4\|R\|}\right)^n.
\]
It is well defined since the operator norm $\|R\|$ of $R$ is bounded.

\end{example}

\begin{example}[Nilpotent group of $4 \times 4$ matrices $N_4$]

From the Fourier inversion formula (\ref{4upperplancherel}), the principal orbit can be expressed as $\{\alpha l_1 +\gamma_2 l_5 |\alpha_1,\gamma_2 \neq 0\}$ and the Plancherel measure is $|\alpha|^2 d\alpha d\gamma_2$.
 Since the Lie algebra $\mathfrak{g}= \mathfrak{n}_4$ of $N_4$ has the following stratified structure of $\mathfrak{g}= \mathfrak{g}^{(1)}\oplus \mathfrak{g}^{(2)}\oplus \mathfrak{g}^{(3)}$ with
\[ \mathfrak{g}^{(1)}=\{X_1,X_2,X_3\}, \mathfrak{g}^{(2)}=\{Y_1,Y_2\}, \mathfrak{g}^{(2)}=\{Z\}.
\] 
The polynomial growth order $d=10$ is coming from the fact that the weight of elements in $\mathfrak{g}^{(i)}$ is equal to $i$ and $10 = 1+1+1+2+2+3$. This coincides with the computation for the Planchrel measure $|\alpha|^2 d\alpha d\gamma_2$ where the weights of $\alpha$ and $\gamma_2$ are $3$ and $1$ respectively and $10 = 3 \times 3 +1$. 

By (\ref{4upperirred}), the corresponding unitary irreducible representations are 
\begin{align}
&{}[\pi_{\alpha',\gamma'_2}\tilde{f}](t_l,t_3) \notag \\
&= e^{2\pi\sqrt{-1}\gamma'_2x_2}e^{2\pi\sqrt{-1}\alpha'[z + (1/2)x_1y_2 -x_3y_1 -t_1x_2t_3-(1/2)x_1x_2t_3 -(1/2)t_1x_2x_3+(1/3)x_1x_2x_3]}\notag \\
&{} \cdot\tilde{f}(t_1+x_1,t_3 +x_3), \notag 
\end{align}
for $\tilde{f} \in L^2(\mathbb{R}^2)$.

Therefore, the hypo-elliptic operator $\mathcal{H}= \mathcal{H}_{N_4}$ is expressed as
\begin{align}
\mathcal{H} &= \left( d\pi_{l}(X_1)^2+d\pi_{l_1}(X_3)^2+d\pi_{l}(X_2)^2\right) \notag \\
&= \left( \left(\frac{\partial}{\partial t_1}\right)^2+\left(\frac{\partial}{\partial t_3}\right)^2-4\pi^2(t_1t_3)^2\right) \label{4upperhypo}
\end{align}
acting  on $L^2(\mathbb{R}^2)$.
The explicit values of its eigenvalues and eigenfunctions are unknown. 
Although there seem to be some possibilities to express the special value $\zeta_{\mathcal{H}}$ of the spectral zeta function of $\mathcal{H}$ at $d/2 = 5$ as some integrals, we have not yet obtained these formulas at present. 
\end{example}

\section{Another proof of the semiclassical expansion for the Harper operator}\label{Anotherproofwilkinson}   
Before going into further steps for proofs of results of Subsections \ref{IntroChebotarev} and \ref{IntroHeat}, we explain our mathematical justification of the semiclassical expansion ({\ref{introeigen1}) and (\ref{introformal}) since it can be considered as simple models containing some essential idea in the arguments in Sections \ref{sectionasym} and \ref{asymnilpotent}

We first recall several reasons why mathematical justifications are necessary for the above expansion formula.
\begin{description}
\item[\rm{(1)}] As we already explained in the introduction, formal manipulation like (\ref{introformal}) is the form that the bounded operator $h_\theta$ would be approximated by the harmonic oscillator $\widetilde{\mathcal{H}}$, which is an unbounded operator, and thus, this causes the ``error'' to be something complicated.

\item[\rm{(2)}] As one may be seen from the shape of the Hofstadter butterfly, the behavior of the edges of the spectrum $\sigma(H_\theta)$ is not smooth with respect to $\theta$ in general: it is not $C^1$ at rational $\theta$'s (cf. \cite{Bellissard}).  

\item[\rm{(3)}]When $\theta$ is a rational number, the spectrum $\sigma(H_\theta)$ consists of the unions of finite closed intervals (i.e., finite bands). Then we have the question: Which point in each band should be chosen for each rational $\theta$ in the expansion (\ref{introeigen1})?
\end{description}

In some literature of physics, the above points (2) and (3) are explained as follows: In the point (3), the width of each band in the spectrum $\sigma(H_\theta) =\sigma(h_\theta)$ has an estimate $O(e^{-c/\theta})$ and thus the choice of a point in each band does not affect the (polynomial) expansion (\ref{introeigen1}), which is sometimes called ``non-perturbative effects''. Moreover, this freedom of choice naturally eliminates the problem of (2) since we can choose a point in each band varying smoothly with respect to $\theta$.

In the arguments in \cite{Helffer}, the above non-perturbative estimates
are obtained simultaneously with (\ref{introeigen1}). On the other hand, we can prove this estimate after establishing (\ref{introeigen1}).

The expression $h_\theta$ of $H_\theta$ in (\ref{introhtheta}) is coming from the Schr\"odinger representation $\rho_h$ with $\theta =2\pi h = 2\pi x_3 = 2\pi p/q$. Then, the semiclassical expansion (\ref{introeigen1}) and (\ref{introformal}) ($=$ (\ref{formal})) is essentially the same as the above investigation of the behavior of the eigenvalue of the twisted Laplacian $\Delta_{\rho_{{\rm \tiny{fin}},x}}$ and $\Delta_{\rho_h}$ based on Theorem \ref{discretetoLie} and Remark \ref{remark34}. Namely, our mathematical justification comes from a simple application of the above freeness of the interchange between $\rho_h$ and $\rho_{\mbox{\tiny{fin}}, x}$.
 
We think our method has some advantage to the original proof in \cite{Helffer}  since we can justify approximations of the operators $e^{\sqrt{h}\frac{1}{\sqrt{-1}}\frac{d}{ds}}$ and $e^{2\pi\sqrt{-1}\sqrt{h}s}$ individually while the original case in \cite{Helffer} is an approximation conjointly, since the harmonic oscillator $\widetilde{\mathcal{H}} =-\frac{d^2}{ds^2}+s^2$ is coming from the sum of the above operators and their hermitian conjugates.

In the rest of this chapter, we prove the above non-perturbative estimate that the width of each band is $O(e^{-c/\theta})$ when $\theta \in 2\pi\mathbb{Q}$. The method for computing the spectrum $\sigma(H_\theta)$ of $H_\theta$ for rational $\theta$ appears to be well known (cf. \cite{Duan}). Let us recall that the following expression: For $\theta = 2\pi x_1 = 2\pi/q$, 
\begin{equation}
\sigma(H_\theta) = \bigcup_{x_2, x_3 \in [0,1]} \sigma(H_x),
\end{equation}
where $x = (x_1,x_2,x_3) \in \widehat{X}$ as before and 
\begin{align*}
H_x &= \rho_{\mbox{\tiny{fin}},x}(u) + (\rho_{\mbox{\tiny{fin}},x}(u))^\ast +\rho_{\mbox{\tiny{fin}},x}(v) + (\rho_{\mbox{\tiny{fin}},x}(v))^\ast  \\
&= \left(\begin{array}{ccccccc}c_0& e^{-2\pi\sqrt{-1}x_3/q} & 0 &\cdots & 0 & 0 & e^{2\pi\sqrt{-1}x_3/q} \\ 
e^{2\pi\sqrt{-1}x_3/q} & c_1 & e^{-2\pi\sqrt{-1}x_3/q} & \cdots & 0 & 0 & 0 \\ \vdots & \vdots & \ddots & \ddots & \ddots & \vdots & \vdots \\  \vdots & \vdots & \ddots & \ddots & \ddots & \vdots & \vdots \\   \vdots & \vdots & \ddots & \ddots & \ddots & \vdots & \vdots \\0 & 0 & 0 & \cdots & e^{2\pi\sqrt{-1}x_3/q} & c_{q-2} & e^{-2\pi\sqrt{-1}x_3/q} \\
e^{-2\pi\sqrt{-1}x_3/q} & 0 & 0 & \cdots & 0 & e^{2\pi\sqrt{-1}x_3/q} & c_{q-1} \end{array}\right).
\end{align*}
where $c_k = 2\cos (2\pi(x_2+k)/q), \; k= 0,1,\ldots, q-1$.

Next, we compute the characteristic equation 
\begin{equation}
F_x(E) = \mbox{det} (H_x - E\textbf{1}_q) = 0. \label{characteristic}
\end{equation}
where $\textbf{1}_q$ is the identity  matrix of degree $q$.
As in \cite{Hasegawa}, it is straightforward to check that
\[  F_x(E) = F_{(x_1,x_2,0)}(E) -2\cos(2\pi x_3) +2.
\]
Using the symmetry under the mapping $(x_2, x_3) \to (-x_3, x_2)$, one finds that the equation (\ref{characteristic}) can be simplified to
\begin{equation} F_{(x_1,0,0)}(E) + 4 = 2(\cos(2\pi x_3) + \cos(2\pi x_2)) \label{level}
 \end{equation}

To estimate fluctuations of the solutions $E$ of the characteristic equation (\ref{characteristic}) with respect to $x_2,x_3$, which implies estimates of the width of bands at $\theta = 2\pi x_1 = 2\pi/q$, it suffices to estimate of the first derivative $\frac{d}{dE}\left|_{E = E_i}F_{(x_1,0,0)}(E)\right. = F'_{(x_1,0,0)}(E_i)$ at a solution $E_i$ of the equation $F_{(x_1,0,0)}(E) = 0$.

Let us denote
\[  F_{(x_1,0,0)}(E) = (E_1-E)\cdots(E_q-E).
\]
Combining with (\ref{eigen}), we  have
\begin{align*}
  \left|F'_{(x_1,0,0)}(E_i)\right| &= \Bigl|\prod_{\substack{j =1 \\ j \neq i}}^q(E_j-E_i)\Bigr| \\ &=  \Bigl|\prod_{\substack{j =1 \\ j \neq i}}^q 2(j-i)\theta\Bigr|
\end{align*}
By Stirling's formula, the right-hand side is estimated from below by
\[   2^{q-1}\left((q/2)!\right)^2\left(\frac{2\pi}{q}\right)^q \sim 
2^{q-1}\left(\sqrt{\pi q}\left(\frac{q/2}{e}\right)^{q/2}\right)^2\left(\frac{2\pi}{q}\right)^q =\frac{\pi q}{2}\left(\frac{2\pi}{e}\right)^q > 2^{q-1}.
\]
which implies that the width of bands are $O(e^{-cq}) = O(e^{-c/\theta})$.

\begin{remark}[Laplacian under the constant magnetic field on the plane $\mathbb{R}^2$ and the Harper operator defined on the square lattice $\mathbb{Z}^2$]\label{meshapproximation}
In classical mechanics, a charged particle $A$ moves in a circle called a cyclotron orbit under a constant magnetic field on a plane.
Correspondingly, in quantum mechanics,  the size of the possible energy of charged particles appears as the spectrum of the Laplacian $\Delta_B$ under the constant magnetic field with flux $B$.

If $B=0$, then $\Delta_B$ is the usual Laplacian $\Delta$ and its spectrum are non-negative half line $[0,\infty )$. However, if $B \neq 0$, the spectrum degenerates and only takes discrete values. This value is known as the Landau level, which coincides with the eigenvalue of the harmonic oscillator up to a universal constant.

Here, we explain that the Wilkinson formula (\ref{introeigen1}) approximates the Landau level. Although this fact is known by some experts (cf. \cite{Analytis}), we put it since it is not common knowledge despite its exciting aspects. Moreover, the exponential decay estimate of the width of the bands also implies the degeneracies of the spectrum, i.e., the Landau levels, which consist of a point spectrum with infinite multiplicities.    

Let us consider a square lattice $\delta\mathbb{Z}^2$ with the side-length $\delta$, which is the same as the distance between the nearest neighboring two points, as a discrete approximation of the plane $\mathbb{R}^2$. Accordingly, the Laplacian $\Delta_B$ is also approximated by the magnetic discrete Laplacian $\Delta_{B,\delta}$ on this lattice $\delta\mathbb{Z}^2$.
Next if we rescall the lattice $\delta\mathbb{Z}^2$ by multiplying $1/\delta$, then the operator $\Delta_{B,\delta}$ acting on $L^2(\delta\mathbb{Z}^2)$ is transform to the operator $\Delta_{\delta B,1}$ acting on $L^2(\mathbb{Z}^2)$. Put $\theta = \delta B$, then $\Delta_{\delta B,1}$ is written as $4I-H_\theta$, where $I$ is the idenity operator and $H_\theta$ is the Harper operator defined in (\ref{introharperfirst}).
Then, by the Wilkinson formula (\ref{introeigen1}), the spectrum $\widetilde{E}_n$ of $\Delta_{\delta B,1}$ is written as 
\[\widetilde{E}_n = E_n + 4 = (2n+1)\theta +o(\theta), \quad \theta = \delta B
\]
Moreover, if we translate back to the spectrum $\lambda_n$ of $\Delta_{B,\delta}$ as
\[ \lambda_n = \widetilde{E}_n/\delta = (2n+1)B + o(1) \]
which gives our aimed approximation of the Laplacian $\Delta_B$ under the magnetic field of the flux $B$.
\end{remark}

\begin{remark}[Mathematical justification of the Wilkinson formula by Helffer and Sj\"ostrand]\label{HSproof}
For comparison with our method based on Theorem \ref{discretetoLie}, we provide a brief outline of the arguments in \cite{Helffer01} and \cite{Helffer} very roughly.

By the (formal) Taylor expansion formula, this operator can be expressed as 
\begin{equation}
      h_\theta = -4 + \left(-\frac{d^2}{ds^2}+s^2\right)\theta + O(\theta^2) \label{formal}.
\end{equation}
The coefficient of the linear part in $\theta$ is the harmonic oscillator $\widetilde{\mathcal{H}} := -\frac{d^2}{ds^2}+s^2$ whose eigenvalues are $n + \frac12, n=0, 1, 2 ,\ldots$, which implies the Wilkinson formula (\ref{introeigen1}). However, (\ref{formal}) is a form in which a bounded operator $h_\theta$ is approximated by an unbounded operator $\widetilde{\mathcal{H}}$. 


We explain an idea in their method in \cite{Helffer}, \cite{Helffer01} briefly as follows: If the domain $L^2(\mathbb{R})$ of both operators $h_\theta$ and $\mathcal{H}$ are exhausted by a sequence of common invariant finite-dimensional spaces $V_k$, $k = 1, 2, \ldots$, then the expansion (\ref{formal}) can be considered as a limit of the Taylor expansion of the restriction $h_\theta\left|_{V_k}\right.$ of $h_\theta$. However, there is no such sequence. Instead of $V_k$, Helffer and Sj\"{o}strand used the space $\widetilde{V}_k$ spanned by eigenfunctions of $\widetilde{\mathcal{H}}$ associated with eigenvalues less than or equal to $k$. This space is not invariant by $h_\theta$, but errors are $O(h^\infty)$ with $h = \theta/2\pi$ and thus do not affect the asymptotic expansion (\ref{introeigen1}). 

Their methods are based on the semiclassical localizations related to the harmonic oscillators. Here, we will provide a brief explanation.

To begin with a simpler situation, we start with the case of the Schr\"odinger operator $H_{W,h}:= -h^2\partial{d^2}{dx^2}+W(x)$ with a double-well type potential $W(x)$. Namely,the graph of $W(x)$ is the following figure: 
\begin{figure}[H]
\centering
\includegraphics[width=140mm]{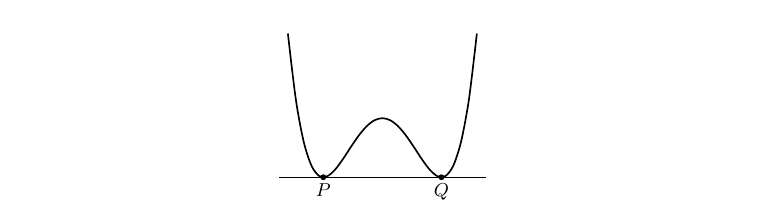}
\caption{Double well potential}
\end{figure}

Consider the eigenfunctions for the lowest eigenvalue $\lambda_h$ of $H_{W,h}$ in the semiclassical limit $h\to 0$. Since $\lambda_h \to 0$ in the limit, the corresponding eigenfunctions  $\varphi_{P,h}(x)$ and $\varphi_{Q,h}(x)$ are localized at the bottom points $P, Q$ of $W(x)$ respectively. 
More strongly, if the supports of $\varphi_{P,h}(x)$ and $\varphi_{Q,h}(x)$ were disjoint, then they would not interfere with each other and thus, a similar argument to the above using virtual common invariant finite-dimensional spaces $V_k$ could work. 
However, this is different. The support of each eigenfunction is the whole real line $\mathbb{R}$. In fact, if it vanishes on an open interval, then it is identically $0$ on $\mathbb{R}$ by the unique continuation property of elliptic differential operators.

Nonetheless, the following facts are available to compensate for the above. The eigenfunctions $\varphi_{P,h}(x)$ and $\varphi_{Q,h}(x)$ decays as $O(e^{-1/h})$ or $O(h^\infty)$ outside the vicinity of $P$ (resp. $Q$). This error is non-perturbative and, thus, does not affect the power series expansion. This error is sometimes called the tunnel effect.
,
In the case of the Harper operator, in particular, for their expression of the form (\ref{introhtheta}) as
\begin{equation*}h_\theta = -2\cos \left(\sqrt{\theta}\frac{d}{\sqrt{-1}ds}\right)-2\cos (\sqrt{\theta}s),  \end{equation*}
their symbol is written as $-2\cos (\sqrt{\theta}\xi)-2\cos (\sqrt{\theta}s)$
and their bottom points are located on the set 
\[A:= \{((2n+1)\sqrt{\theta}\pi\xi, (2n+1)\sqrt{\theta}\pi x)\mid n \in \mathbb{Z}\}.\] 
Similarly to the above double well case, the arguments are based on the localizations on each point on $A$, which are approximated by the harmonic oscillators locally at each point.
In other words, the Harper operator is the operator with infinite wells, and Helffer and Sj\"{o}strand \cite{Helffer}, \cite{Helffer01} studied these arguments very carefully and made them rigorous.

These arguments are helpful in other situations; for example, in a Japanese book by Mikio Furuta \cite{Furuta} on the index theorem, this idea was explained as an argument based on the supersymmetric harmonic oscillator. 
Furthermore, this type of argument was discussed in \cite{Helffer02} for a mathematical explanation of  Witten's Morse theory \cite{Witten} and another proof and extension \cite{BismutZhang} of the Ray-Singer conjecture \cite{Ray}, which is originally proved by Cheeger \cite{Cheeger} and M\"uller \cite{Muller} independently, and several works by B. Simon such as \cite{Simon1},  \cite{Simon2},  \cite{Simon3}.
\end{remark}

\section{A relation of the twisted Laplacian on flat vector bundles associated to unitary representations of $\Gamma$ over $M$ and twisted operators acting on $L^2(\mathbb{R}^m)$-valued functions on $M$}\label{reductionpathintegral}

\subsection{The Lie integrals and Chen's iterated integrals}\label{Lieintegral}
We need to analyze the eigenvalues of the twisted Laplacian $\Delta_{\rho_{\rm \tiny{fin}}}$ acting  on the space $L^2(E_{\rho_{\rm \tiny{fin}}})$ of sections of a flat vector bundle $E_{\rho_{\rm \tiny{fin}}}$ associated to finite dimensional unitary representation $\rho_{\rm \tiny{fin}}$ near the trivial representation $\textbf{1}$. 
We note that $\rho_{\rm \tiny{fin}} = \mbox{Ind}_P^\Gamma (\chi_{l_t}\otimes\chi_{\psi_s})$ where the notation of the left-hand side is appeared in Substep 0-6 in Section \ref{nilpotentexplanation} while that of right-hand side in the formula (\ref{finaldecomposition}) in Theorem \ref{introdecompose} ($=$ Theorem \ref{sec4decompose}).
Moreover, from this theorem, we can replace them with a formal computation of the eigenvalues of $\Delta_{\pi_l}$ acting on the space of sections $L^2(E_{\pi_l})$ of a flat vector bundle $E_{\pi_l}$ associated to infinite dimensional representation appeared in the formula (\ref{intropildecomposition}) in the above theorem near $\textbf{1}$.

We follow the same strategy in case $\Gamma$ is abelian. As explained in Section \ref{Infiniteabelian}, we have to adjust the situation to the usual perturbation arguments. For this purpose, we need to construct a canonical section $s_{\pi_l}$ of the flat $U(L^2(\mathbb{R}^m))$-principal bundle $U(E_{\pi_l})$ associated to $E_{\rho_{\pi_l}}$, whose fibers are the group $U(L^2(\mathbb{R}^m))$ of unitary transformations of the fiber $L^2(\mathbb{R}^m)$ of $E_{\pi_l}$, where $L^2(\mathbb{R}^m)$ is the model of representation space of $\pi_l$. To construct $s_{\pi_l}$, we replace the line integrals in (\ref{lineintegral}) with the Lie integrals or Chen's iterated integrals.

Let us briefly recall the definition of the Lie integral, as given in (\cite{Benardete}), but we use the right Lie integral instead of the left one. 
For a Lie group $G$, its Lie algebra $\mathfrak{g}$ and a continuous path $f:[a,b] \to \mathfrak{g}$ and the partition 
\[\Delta: a = u_0 < u_1 < \ldots < u_n =b \]
of the interval $[a,b]$, we define its definite right Lie integral $\displaystyle{{}_R\!\!\int_a^b\!f}$ to be the element of $G$ equal to 
\begin{align*} &{}\lim_{n \to \infty} \prod_{i=n}^1\exp(f(u^*_i)(u_i-u_{i-1}))  \\ &= \lim_{n \to \infty}\exp(f(u^*_n)(u_n-u_{n-1}))\cdots\exp(f(u^*_1)(u_1-u_0)) 
\end{align*}
where $\exp \colon \mathfrak{g} \to G$ is the exponential map, and the limit is taken over partitions and choices of $u^*_i \in (u_{i-1},u_i]$. 
The definite left Lie integral $\displaystyle{{}_L\!\!\int_a^b\!f}$ is similarly defined as
\[ \lim_{n \to \infty} \prod_{i=1}^n\exp(f(u^*_i)(u_i-u_{i-1}))
\]  These are well-defined similarly to the standard Riemann integral of continuous functions.

Next, for a compact Riemannian manifold $M$, let $\alpha : [a,b] \to M$ be a smooth path in $M$ and $\omega \colon TM \to \mathfrak{g}$ be a $\mathfrak{g}$ valued one form on $M$. 
Then, $\omega(\alpha'(u)):[a,b] \to \mathfrak{g}$ is a continuous path as above, and the right Lie integral ${}_R\!\int_\alpha \omega$ of $\omega$ is defined as 
${}_R\!\int_a^b \omega(\alpha'(u))du. 
$ 
  The left Lie integral ${}_L\!\int_\alpha \omega$ of $\omega$ is similarly defined as ${}_L\!\int_a^b \omega(\alpha'(u))du. 
$ 

Observe that ${}_R\!\int_\alpha \omega$ and ${}_L\!\int_\alpha \omega$ are independent of the parametrization of $\alpha$.
Also note that if $\alpha$ and $\beta$ are smooth paths in $M$ and their concatenation $\alpha\beta$ is defined, then we have \[{}_R\!\!\int_{\alpha\beta}\omega = {}_R\!\!\int_\beta \omega\; {}_R\!\!\int_\alpha \omega.\] 
Similarly we have ${}_L\!\int_{\alpha\beta}\omega = {}_L\!\int_\alpha \omega\;{}_L\!\int_\beta \omega$.

Furthermore, we notice that a $\mathfrak{g}$-valued one form $\omega$ can be identified with a connection one form $\widetilde{\omega}$ on the trivial $G$-principal bundle $M \times G$ by $\widetilde{\omega} = \pi^\ast \omega$,  where
$\pi: M \times G \to M$ is the canonical projection. Then the parallel transport $P\left|_{\alpha(a)}^{\alpha(t)}\right.$ with respect to $\widetilde{\omega}$ along a curve $\alpha : [a,b] \to M$ from $\alpha(a)$ to $\alpha(t)$ on $M \times G$ can be written as
\[  P\left|_{\alpha(a)}^{\alpha(t)}\right.(\alpha(a),v) = \left(\alpha(t), \left({}_R\!\!\int_a^t \omega(\alpha'(u))du\right)v \right) \in M \times G. 
\]
We say $\omega$ is flat if it satisfies 
\begin{equation}
d\omega + [\omega,\omega] = 0 \label{flatform}
\end{equation}
(\cite{Spivak} p. 376) which corresponds that $\widetilde{\omega}$ defines flat connection. (Some authors use different conventions in defining wedge products and the exterior derivative; for them, a flat one-form satisfies $d\omega +\frac12[\omega, \omega]$ (\cite{KobayashiN})).
If $\omega$ is flat and paths $\alpha$ and $\beta$ are homotopic relative to their endpoints, then we have ${}_R\!\int_\alpha \omega = {}_R\!\int_\beta \omega$. Note that if $G$ is abelian, then the flat condition is $d\omega = 0$, which means $\omega$ is closed; thus, the above result is derived from the Stokes formula.

If $G$ is a Lie subgroup of $\mbox{GL}(n, \mathbb{R})$, then the right Lie integral of $\mathfrak{g}$-valued one forms can be computed using Chen's iterated integrals. 
Let $\{\eta_i\}_{i=1}^n$ be a collection of real-valued one forms on a manifold $M$. If $\alpha \colon [a,b] \to M$ is a smooth path, then the iterated integral\begin{align}\int_S\eta_1\eta_2\cdots\eta_n &:= \int_Sf_n(u_n)f_{n-1}(u_{n-1})\cdots f_1(u_1)du_1du_2\ldots du_n \notag \\
&= \int_a^b\left(\int_a^{u_n}\left(\cdots\left(\int_a^{u_2}f_1(u_1)du_1\right)\cdots\right)du_{n-1}\right)du_n \label{section6lteratedintegral}
\end{align}
where $S$ is the simplex $\{(u_1,\ldots ,u_n) \mid a \leq u_1 \leq \ldots \leq u_n \leq b\}$ and $f_i(u) = \eta_i(\alpha '(u))$.

\begin{theorem} Let $G$ be a Lie subgroup of $\mbox{GL}(n, \mathbb{R})$, $\omega \colon TM \to \mathfrak{g}$, a $\mathfrak{g}$ valued one form and $\alpha \colon [a,b] \to M$ a smooth path in $M$. Then the right Lie integral $\displaystyle{{}_R\!\!\int_\alpha \omega}$ is equal to the sum of an infinite series of iterated integrals
\begin{equation} I + \int_\alpha\omega + \int_\alpha\omega\omega + \int_\alpha\omega\omega\omega + \cdots. \label{power} \end{equation}
where $I$ is the identity matrix. 

Furthermore, if $G$ is a nilpotent and a Lie subgroup of the group of upper triangular matrices (with one down the diagonal), then the terms of the series become zero after finitely many steps.
\end{theorem}

In this theorem, the above sum (\ref{power}) is nothing but a solution of the equation of the parallel transport with respect to the connection one form $\widetilde{\omega}$ 
\[ df = \tilde{\omega}f
\]
by the Picard's iteration method.

In the case of $G$ is ${\rm Heis}_3(\mathbb{R})$ and $\Gamma$ is ${\rm Heis}_3(\mathbb{Z})$, this $\omega$ can be taken as a form of upper triangular matrices
\begin{equation}  \omega = \left( \begin{array}{ccc} 
0 & \omega_1 & \omega_{12} \\ 0 & 0 & \omega_2 \\ 0 & 0 & 0 \end{array}\right) \label{oneform}
\end{equation}
where $\omega_1, \omega_2, \omega_{12}$ are real valued one forms.
Then, the right  Lie integral ${}_R\!\int_{\alpha}\omega$ can be expressed as follows:
\begin{align*}
 {}_R\!\int_{\alpha}\omega &= I + \int _{\alpha}\omega + \int _\alpha \omega \omega \\
&=  \left(\begin{array}{ccc}
1 & 0 & 0 \\
0 & 1 & 0 \\
0 & 0 & 1
\end{array}
\right) +  \left(\begin{array}{ccc}
0 & \int _{\alpha} \omega _1 & \int _{\alpha} \omega_{12} \\
0 & 0 &  \int _{\alpha} \omega _2 \\
0 & 0 & 0
\end{array}
\right) +  \left(\begin{array}{ccc}
0 & 0 & \int _{\alpha}\omega _2\omega _1 \\
0 & 0 & 0 \\
0 & 0 & 0
\end{array}
\right) \\
&= \left(\begin{array}{ccc}
1 & \int _{\alpha} \omega _1 & \int _{\alpha} \omega_{12} + \int _{\alpha}\omega _2\omega _1 \\
0 & 1 &  \int _{\alpha} \omega _2 \\
0 & 0 & 1
\end{array}
\right)
\end{align*}
where $\int_\alpha \omega_2\omega_1$ is an iterated integral and their flat condition (\ref{flatform}) is equivalent to 
\begin{equation}d\omega_{12} = -\omega_1 \wedge \omega_2.\label{flatomega}\end{equation} 

\subsection{$\pi_1$-de Rham theorem}

In the construction of a particular section $s_\chi$ in Step 2 (P1) of the proof of (\ref{leadingKotaniheat}), we have utilized the de Rham-Hodge Theorem to establish the correspondence between characters $\chi = \chi_\omega$ and harmonic one-forms $\omega$. In the present situation, a substitute for the de Rham Theorem is the following theorem:
\begin{theorem}[$\pi_1$-de Rham Theorem {\rm \cite{Benardete}}]\label{pi1derham}
Let $M$ be a closed manifold, and $\Gamma$ be a lattice of a simply connected nilpotent Lie group $G$. Take a surjective homomorphism $\Phi \colon \pi_1(M,p) \to \Gamma$. Then, there exists a Lie algebra $\mathfrak{g}$ valued flat one form $\omega$ such that for $\gamma \in \pi_1(M,p)$,
\begin{equation}    \Phi (\gamma ) = {}_R\!\!\int_\gamma \omega.
\label{RH} \end{equation}
\end{theorem}  

This theorem means that $\omega$ (or $\widetilde{\omega}$) is a solution of the Riemann-Hilbert problem for the monodromy representation $\Phi$ of $\pi_1(M,p)$. Note that the choice of $\omega$ is not unique in general. Namely, the Riemann-Hilbert problem here is in the smooth category but not holomorphic in the usual setting.

Thus, we need to choose a suitable $\omega$, which makes later computation simple and, in some sense, a canonical expression of the coefficients of the asymptotic expansion, which is a harmonic representation.
In \cite{Chen} or \cite{Sullivan}, certain canonical expressions are explained when $M$ has the Riemannian metric. We follow the former method.

To explain harmonic theory in the following section, we recall an outline of the proof presented in \cite{Benardete}. For another proof, see the comments in the last few paragraphs in subsection 3.3 in \cite{Benardete}. 

For a fixed $p \in M$, let $\Pi = \pi_1(M, p)$ and $G$ be a Malcev completion of $\Gamma$.
Since $G$ is a connected, simply connected nilpotent Lie group, $\mbox{exp}:\mathfrak{g}\to G$ is a diffeomorphism. Therefore, $G$ is contractible, which implies that the higher homotopy groups of $N:= G/\Gamma$ vanish. That is, $N$ is a $K(\Gamma, 1)$ space. Then the homomorphisms $\Phi:\Pi\to\Gamma$ determine continuous maps $F:(M, p) \to (N, \Gamma e)$ unique up to homotopy, which induces $\Phi$ as in \cite{Benardete}. Since any continuous map between compact manifolds is homotopic to a smooth map, the $F$ can be chosen to be smooth.
 
Let $\widehat{\omega}$ be the canonical $\mathfrak{g}$ valued (right invariant) one-form on $N$ obtained by left translating tangent vectors to the identity. 
For example, in the case when $\Gamma = \mbox{Heis}_3(\mathbb{Z})$, it can be written as 
\begin{equation}  \widehat{\omega} = (dg)g^{-1} = \left(\begin{array}{ccc}
0 & dx & -xdy+dz \\
0 & 0 & dy \\
0 & 0 & 0
\end{array}
\right)\quad \mbox{for} \quad g = \left(\begin{array}{ccc}
0 & x & z \\
0 & 0 & y \\
0 & 0 & 0
\end{array}
\right) \label{formonHeisenberg}
\end{equation}
 Since, by definition, a curve $\gamma$ in $N$ is a ``right'' antiderivative of the curve $\omega (\gamma')$ in $\mathfrak{g}$, it follows that the definite integral 
\[ {}_R\int_\gamma \widehat{\omega} = \gamma (1)\gamma^{-1}.
\]
Since $\widehat{\omega}$ is right-invariant, it induces the desired one-form $\omega := F^\ast\widehat{\omega}$ on $M$. Although it is not difficult to see that wi has the desired properties, we omit it (see \cite{Benardete}).

\subsection{Relations to harmonic theory}\label{harmoniclietheory}
In the case when $\Gamma = {\rm Heis}_3(\mathbb{Z})$,
we can choose $\omega_1, \omega_2, \omega_{12}$ so that
\begin{conditions} \label{harmoniccoexact}
\begin{description}
\item[\rm{(1)}] $\omega_1, \omega_2$ are harmonic forms satisfying 
\begin{equation}( \omega_1, \omega_2 )_{L^2(M)} :=\int_M\langle \omega_1, \omega_2 \rangle = 0\label{ortho1} \end{equation}

\item[\rm{(2)}] $\omega_{12}$ is a coexact form, i.e. there exists two form $\omega^{(2)}$ such that $\delta\omega^{(2)} = \omega_{12}$  with the formal adjoint $\delta$ of the exterior differential $d$.
\end{description}
\end{conditions}
Concerning the orthogonality condition (\ref{ortho1}), it is achieved by changing expression of the matrix $\omega$ in (\ref{oneform}) with suitable choice of a basis of $\mathbb{R}^3$ such that $\omega_1, \omega_2$ they forms a base of $\mathfrak{g}^{(1)}$ where $\mathfrak{g}^{(1)}$ is a component of the decomposition.
Moreover, it should be remarked that by the flat condition (\ref{flatomega}), the co-exact components of $\omega_{12}$ are determined. Then the energy density $|d\Phi|$ of the map 
\[ \Phi(p) = {}_R\!\!\int_{p_0}^p \omega \quad \mbox{for some}\quad p_0 \in M \]satisfies
\[\|d\Phi\|_{L^2(M)}^2 := \int_M|d\Phi|^2 = \|\omega\|_{L^2(M)}^2 = \|\omega_1\|_{L^2(M)}^2 + \|\omega_2\|_{L^2(M)}^2 + \|\omega_{12}\|_{L^2(M)}^2.\]    

In Theorem \ref{pi1derham}, it is the energy minimizing in its homotopy class if and only if 
$\omega_1, \omega_2$ are harmonic, and $\omega_{12}$ is co-exact as in the above conditions for the following reason. The ambiguity of the choice of the solution $\omega_{12}$ of the equation (\ref{flatomega}) arises from $d$-exact part $\omega_{12}^d$ and harmonic part $\omega_{12}^h$ in the Hodge decomposion $\omega_{12} = \omega_{12}^d+\omega_{12}^h+ \omega_{12}^\delta$ of $\omega_{12}$. Since this decomposition is orthogonal, we have 
\[\|\omega_{12}^\delta\|_{L^2(M)} \leq \|\omega_{12}^d\|_{L^2(M)}+\|\omega_{12}^h\|_{L^2(M)}+ \|\omega_{12}^\delta\|_{L^2(M)} = \|\omega_{12}\|_{L^2(M)}
\]
which implies the above choice that  $\omega_{12}$ is co-exact gives energy minimizing. In this case, we call $\Phi$ a generalized Abel-Jacobi map since it is a generalization of the usual Abel-Jacobi maps for the abelian case.

\begin{remark}\label{harmonicrepresentative}
 \begin{equation}{\rm Lie}({\rm Heis}_3(\mathbb{R})) = \mathfrak{g} = \mathfrak{g}^{(1)}\oplus\mathfrak{g}^{(2)}\label{gradedLie} \end{equation}
as a natural graded Lie algebra structure on $\mathfrak{g}$ with the center $\mathfrak{g}^{(2)}$. 
It should be remarked that if we choose $\omega_1, \omega_2, \omega_{12}$ satisfying the above conditions (1), (2), then the map determined by the right Lie integral of $\omega$ is energy minimizing and thus, a harmonic map from $M$ to the Heisenberg manifold ${\rm Heis}_3(\mathbb{R})/\Gamma$. 

It would be interesting to compare this map and minimal surfaces in ${\rm Heis}_3(\mathbb{R})$ studied in \cite{Dorfmeister1}, \cite{Dorfmeister2}. Concerning this topic, we will provide further comments in Section \ref{Beyondharmonicity}.  
\end{remark}

In Appendix \ref{Homologyconnection}, we shall discuss harmonic theory for general discrete nilpotent groups.

\section{Asymptotic expansion of the lowest eigenvalue of the twisted Laplacian: the Heisenberg group $\mbox{Heis}_3(\mathbb{Z})$}\label{sectionasym}
\subsection{Reduction of the twisted Laplacians}\label{Reductiontwisted}

Here we shall construct a canonical section $s_{\pi_l}$ in  of $E_{\pi_l}$ which is a section of $U(L^2(\mathbb{R}))$-principal bundle $U(E_{\pi_l})$ over $M$ where $U(L^2(\mathbb{R}^m))$ is the group of unitary operators on $L^2(\mathbb{R})$.  We also compute an explicit form of an associated twisted operator $L_h$ acting on the space $L^2(M, L^2(\mathbb{R}))$ of $L^2(\mathbb{R})$-valued functions on $M$.

Here, we describe detailed computation in the case of $\mbox{Heis}_3(\mathbb{Z})$. In this case we used notations, $\rho_h$, $s_h$, $U(E_{\rho_h})$ instead of $\pi_l$, $s_{\pi_l}$, $U(E_{\pi_l})$ used for general nilpotent groups $\Gamma$ and $m=1$ in $L^2(M, L^2(\mathbb{R})^m)$ 
For a surjective homomorphism $\Phi:\pi_1(M) \to \Gamma$ in the assumption of Theorem \ref{heisenberg-geod}, by $\pi_1$-de Rham Theorem \ref{pi1derham}, we take an associate flat $\mathfrak{g}$-valued one form \[\omega = \left( \begin{array}{ccc} 
0 & \omega_1 & \omega_{12} \\ 0 & 0 & \omega_2 \\ 0 & 0 & 0 \end{array}\right) \]
such that 
\[\Phi (\gamma ) = {}_R\!\!\int_\gamma \omega\]
where $\omega_1, \omega_2, \omega_{12}$ satisfy Conditions \ref{harmoniccoexact}.
Then we define a canonical section $s_h \in U(E_{\rho_h})$ by 
\begin{equation} s_h(p) = \rho_h\left( {}_R\!\!\int_{p_0}^p \omega \right).
 \label{sectionsh}
 \end{equation}

Note that if $f$ is in $L^2(M,L^2(\mathbb{R}))$, then $\varphi := s_hf$ is in $L^2(E_{\rho_h})$.

We shall compute  $\Delta_{\rho_h}\varphi$ as follows:
\begin{align}\label{Laplacian}
\Delta_{\rho_h}\varphi &= \delta d \left(\rho_h\left({}_R\!\!\int_{p_0}^p \omega \right)f\right) \notag \\
&= \delta d \left(e^{2\pi\sqrt{-1}\left( h\int_{p_0}^p(\omega_{12} + \omega_2\omega_1) + \sqrt{h}\left(\int_{p_0}^p\omega_2\right)s\right)}e^{\sqrt{h}\left(\int_{p_0}^p\omega_1\right)\frac{d}{ds}}f\right) \notag \\
&= \delta \left\{ e^{2\pi\sqrt{-1}\left( h\int_{p_0}^p(\omega_{12} + \omega_2\omega_1) + \sqrt{h}\left(\int_{p_0}^p\omega_2\right)s\right)}\right.\\
&{} \qquad \times \left.2\pi\sqrt{-1}\left( h(\omega_{12} + \omega_2\int_{p_0}^p\omega_1) +\sqrt{h}\omega_2 s \right)e^{\sqrt{h}\left(\int_{p_0}^p\omega_1\right)\frac{d}{ds}}f\right\}\notag \\ 
&{} +\, \delta\left\{\left(e^{2\pi\sqrt{-1}\left( h\int_{p_0}^p(\omega_{12} + \omega_2\omega_1)+ \sqrt{h}\left(\int_{p_0}^p\omega_2\right)s\right)}\right)e^{\sqrt{h}\left(\int_{p_0}^p\omega_1\right)\frac{d}{ds}}\left(\sqrt{h}\omega_1\frac{d}{ds}\right)f\right\} \notag  \\
&{} +\, \delta \left\{e^{2\pi\sqrt{-1}\left( h\int_{p_0}^p(\omega_{12} + \omega_2\omega_1) + \sqrt{h}\left(\int_{p_0}^p\omega_2\right)s\right)}e^{\sqrt{h}\left(\int_{p_0}^p\omega_1\right)\frac{d}{ds}}df\right\} \notag
 \end{align}

Note that 
\begin{equation}
[s, e^{\frac{d}{ds}}] = -e^{\frac{d}{ds}} \label{sdscommutation}
\end{equation}
by
\[([s, e^{\frac{d}{ds}}]f)(s) = (se^{\frac{d}{ds}})f(s)-e^{\frac{d}{ds}}(sf(s)) = sf(s+1)-(s+1)f(s+1) = -f(s+1) = (-e^{\frac{d}{ds}}f)(s),\]
the first term of the right-most side of the above equality (\ref{Laplacian}) is equal to 
\begin{align*}
&{} \delta\left\{e^{2\pi\sqrt{-1}\left( h\int_{p_0}^p(\omega_{12} + \omega_2\omega_1) )+ \sqrt{h}\left(\int_{p_0}^p\omega_2\right)s\right)}e^{\sqrt{h}\left(\int_{p_0}^p\omega_1\right)\frac{d}{ds}} \right.\\ 
&{} \qquad \left. \times 2\pi\sqrt{-1}\left( h\left(\omega_{12} + \omega_2\int_{p_0}^p\omega_1\right) +\sqrt{h}\omega_2 s -h\omega_2\int_{p_0}^p\omega_1 \right)f\right\} \\
&= 2\pi\sqrt{-1}\delta\left\{e^{2\pi\sqrt{-1}\left( h\int_{p_0}^p(\omega_{12} + \omega_2\omega_1) )+ \sqrt{h}\left(\int_{p_0}^p\omega_2\right)s\right)}
e^{\sqrt{h}\left(\int_{p_0}^p\omega_1\right)\frac{d}{ds}}( h\omega_{12} + \sqrt{h}\omega_2 s)f\right\} 
\end{align*}

Since $\delta(f\omega) = f\delta\omega - \langle df,\omega \rangle$, and
 $\omega_1, \omega_2$ and $\omega_{12}$ are co-closed forms, combining with (\ref{sdscommutation}), the right-most side of (\ref{Laplacian}) is equal to
\begin{align}
&{} e^{2\pi\sqrt{-1}\left( h\int_{p_0}^p(\omega_{12} + \omega_2\omega_1) 
+ \sqrt{h}\left(\int_{p_0}^p\omega_2\right)s\right)}e^{\sqrt{h}\left(\int_{p_0}^p\omega_1\right)\frac{d}{ds}} \notag \\
&\times \left\{\left( 4\pi^2\left| h\omega_{12} +\sqrt{h}\omega_2 s \right|^2\right. \right. \notag \\
&{} \quad \qquad -4\pi\sqrt{-1}\left\langle h\omega_{12} +\sqrt{h}\omega_2 s, \sqrt{h}\omega_1\frac{d}{ds}\right\rangle \label{right611}\\
&{} \quad \qquad  - h|\omega_1|^2\left(\frac{d}{ds}\right)^2 \Big)f  \notag \\
&{} -\, \left.2\left\langle 2\pi\sqrt{-1}h(\omega_{12} +\sqrt{h}\omega_2 s) + \sqrt{h}\omega_1\frac{d}{ds}, df\right\rangle  +\Delta_M f \right\}\notag
\end{align}
where  $\Delta_M$ is the Laplacian on $M$.

Thus the eigen-equation $\Delta_{\rho_h}\varphi = \lambda(h)\varphi$ is equivalent to $L_hf = \lambda(h)f$, where
\begin{equation}           L_h = L^{(0)} +L^{(1)}h^{1/2} + L^{(2)}h + L^{(3)}h^{3/2}+L^{(4)}h^2, \label{Lexpansion} 
\end{equation}
and

\begin{align}
L^{(0)}f &= \Delta_Mf,  \label{H0}\\
L^{(1)}f &= -\left\langle 4\pi\sqrt{-1}\omega_2 s +2\omega_1\frac{d}{ds}, df\right\rangle, \label{H1}\\
L^{(2)}f &= \Big(4\pi^2|\omega_2|^2s^2-4\pi\sqrt{-1}\langle \omega_2,\omega_1 \rangle s\frac{d}{ds}-|\omega_1|^2\left(\frac{d}{ds}\right)^2\Big)f \notag \\
&{}-4\pi\sqrt{-1}
\left\langle \omega_{12}, df\right\rangle, \label{H2}\\
L^{(3)}f &= \Big(\left\langle 8\pi^2\omega_{12},\omega_2 s\right\rangle -4\pi\sqrt{-1}\left\langle\omega_{12}, \omega_1\frac{d}{ds}\right\rangle\Big) f, \label{H3}\\
L^{(4)}f &= 4\pi^2|\omega_{12}|^2f. \label{H4}
\end{align}

\subsection{The spectrum of twisted Laplacians and the Kazhdan distance: Sunada's theorem}\label{sunadatheorem} 
In the case where $\Gamma$ is abelian, by (\ref{first}), we need only consider the case where characters $\chi$ are near $\textbf{1}$. The following Theorem \ref{sunada} indicates that the same holds for general discrete groups.


For a $\Gamma$-covering $X \to M$  and a unitary representation $\rho \colon \Gamma \to U(V)$ of $\Gamma$ on a separable Hilbert space $V$, we associate a flat vector bundle $E_\rho$. Let $\lambda_0(\rho)$ be the infimum of the spectrum of the associated twisted Laplacian $\Delta_\rho$ acting on the sections of $E_\rho$. Sunada \cite{Sunada} compares $\lambda_0(\rho)$ with the Kazhdan distance $\delta_A(\rho, \mathbf{1})$ between $\rho$ and the trivial representation $\mathbf{1}$. Here $\delta_A(\rho, \mathbf{1})$ is defined by 
\[ \delta_A(\rho,\mathbf{1}) = \inf_{\substack{v \in V \\ \|v\| = 1}}\sup_{\sigma \in A}\|\rho(\sigma)v-v\|,
\]  
where 
$A$ is a finite set of generators of $\Gamma$. Then, the following holds:
\begin{theorem}[\rm{\cite{Sunada}}]\label{sunada} There exist positive constants $c_1$ and $c_2$ depending only on $M$ and $A$ such that for all $\rho$,
\begin{equation}   c_1\delta_A(\rho,\mathbf{1})^2 \leq \lambda_0(\rho) \leq c_2\delta_A(\rho,\mathbf{1})^2. \label{sunadainequality}
\end{equation}
\end{theorem}

\subsection{Computations for the second derivative of the lowest eigenvalues $\lambda_{0,i}(h)$}\label{subsection72}

In later arguments in this chapter, we continue to use the notations $\rho_h$, $s_h$, $U(E_{\rho_h})$, etc., as in Section \ref{Reductiontwisted} since $\Gamma = \mbox{Heis}_3(\mathbb{Z})$.

Here we consider eigenvalue problem for function $f(x,s) \in L^2(M,L^2(\mathbb{R}))$ where $x \in M$ and $s \in \mathbb{R}$,
\[      L_hf = \lambda(h)f.
\]

Our computation here is formal; thus, we assume the spectrum of our operator consists of discrete eigenvalues only. A mathematical justification of these formal computations can be done through Theorem  \ref{discretetoLie}. Namely, we replace this procedure using finite-dimensional representations $\rho_{\mbox{\tiny{fin}},x}$ in the place of $\rho_h$ here.

We place its eigenvalues $\lambda_{k,i}(h)$ in order satisfying 
\begin{equation}       \lambda_{k,i}(h) \leq \lambda_{k+1,i}(h),  \quad \lambda_{k,i}(h) \leq \lambda_{k,i+1}(h) \label{eigen3}
\end{equation}
where $\{\lambda_{k,i}(h)\}_{i=1}^\infty$ is the cluster of eigenvalues of $\Delta_{\rho_h}$ such that each eigenvalue $\lambda_{k,i}(h)$ converges to the $k$-th eigenvalue $\lambda_k$ of the Laplacian $\Delta_M$ of $M$ when $h \to 0$.
We consider a complete orthonormal system of eigenfunctions $f_{k,i,h}(x,s)$ associated with $\lambda_{k,i}(h)$.

By Theorem \ref{sunada}, $\lambda_{0,i}(h)$ are nonnegative and equal to $0$ if and only if $h  = 0$ in a vicinity of $0$. 
Thus, we only need to care about the behavior of $\lambda_{0,i}(h)$ near $h=0$.

We expand $\lambda_{0,i}(h)$ and $f_{0,i,h}$ in terms of the power of $\sqrt{h}$. In the following, we use abbreviated notations $\lambda_i(h) := \lambda_{0,i}(h)$ and $f_{i,h} := f_{0,i,h}$. Then we have
\begin{align*}
\lambda_i(h) &= \sum_{j=0}^\infty \lambda_i^{(j)}h^{j/2} \\
 f_{i,h} &= \sum_{j=0}^\infty f_i^{(j)}h^{j/2}
\end{align*}

By the minimality of $\lambda_i(0)$, we know that $\lambda_i^{(0)} = \lambda _i^{(1)}= 0$.  

We also know that 
$f_i^{(0)}(x,s) =f_i^{(0)}(s)$, namely  $f_i^{(0)}$ is independent to the variable $x \in M$. In what follows, we shall choose $f_i^{(0)}$ as the $i$-th eigenfunction of the modified harmonic oscillator $\mathcal{H}$ acting on $L^2(\mathbb{R})$ which will be defined in later (\ref{HC}).

To compute $\lambda_i^{(2)}$, we take derivatives up to second order of the eigen-equation  $L_hf_{i,h} = \lambda_i(h)f_{i,h}$ with respect to $\sqrt{h}$ at $\sqrt{h} =0$, then we have
\begin{equation}\label{second}  L^{(2)}f_i^{(0)} + 2L^{(1)}f_i^{(1)} + L^{(0)}f_i^{(2)} = \lambda_i^{(2)}f_i^{(0)} +2\lambda_i^{(1)}f_i^{(1)}+\lambda_i^{(0)}f_i^{(2)} \notag
\end{equation}
where $L^{(k)}$ is the $k$-th derivative of $L_h$ with respect to $\sqrt{h}$ at $\sqrt{h}= 0$ which is consistent with our previous notation (\ref{Lexpansion}).

Integrating this equation over $M$, since $f_i^{(0)}$ is constant in the variable $x$, we obtain from (\ref{H2}),
\begin{align}\int_ML^{(2)}f_i^{(0)} &= \Big(4\pi^2\|\omega_2\|_{L^2(M)}^2s^2-4\pi\sqrt{-1}( \omega_2,\omega_1 )_{L^2(M)} s\frac{d}{ds}-\|\omega_1\|_{L^2(M)}^2\left(\frac{d}{ds}\right)^2\Big)f_i^{(0)} \notag\\ &{} \quad -4\pi\sqrt{-1}(\omega_{12}, df_i^{(0)})_{L^2(M)} 
\notag\\
&= \left(-\|\omega_1\|_{L^2(M)}^2\left(\frac{d}{ds}\right)^2 + 4\pi^2\|\omega_2\|_{L^2(M)}^2s^2\right)f_i^{(0)} =:\mathcal{H}f_i^{(0)}, \label{Heisenharmonic}
\end{align}
where $\mathcal{H}$ is the modified harmonic oscillator defined in later (\ref{HC}).
Here we have used the following formulas: By $\delta\omega_1=\delta\omega_2 =\delta\omega_{12} = 0$ and Condition \ref{harmoniccoexact}, 
\begin{align*}
(\omega,df)_{L^2(M)}&=(\delta\omega,f)_{L^2(M)}, \\  ( \omega_1,\omega_2)_{L^2(M)} &= 0. 
\end{align*}

Similarly, from (\ref{H1}), (\ref{H0}), we have, by $\int_M\Delta f = 0$,
\begin{align*}
\int_ML^{(1)}f_i^{(1)} &=-\int_M \left( 4\pi\sqrt{-1}\delta\omega_2 s +2\delta\omega_1\frac{d}{ds}\right)f_i^{(1)} = 0 \\
\int_M L^{(0)}f_i^{(2)} &= \int_M \Delta_Mf_i^{(2)} = 0 
\end{align*}
Thus, if we choose as $f_i^{(0)}(s) = \psi_i(s)$ the normalized eigenfunctions of $\mathcal{H}$, then we obtain 
\begin{equation} {\rm vol}(M)\lambda_i^{(2)} = \mu_i = 2\pi\|\omega_1\|_{L^2(M)}\|\omega_2\|_{L^2(M)}\left(i+\frac{1}{2}\right) = 2\pi{\rm vol}(\widehat{H})\left(i+\frac{1}{2}\right)  \; (i= 0,1,2,\ldots) \label{harmeigen}
\end{equation}
where $\mu_i$ is the $i$-th eigenvalue and $\psi_i(s)$ is its associated normalized eigenfunction of a modified harmonic oscillator, 
\begin{equation}\mathcal{H} :=-\|\omega_1\|_{L^2(M)}^2\left(\frac{d}{ds}\right)^2 +4\pi^2\|\omega_2\|_{L^2(M)}^2s^2, \label{HC}
\end{equation}
and ${\rm vol}(\widehat{H})$ is the volume of the character group $\widehat{H}$  of $\Gamma/[\Gamma,\Gamma]$ with $\Gamma = \mbox{Heis}_3(\mathbb{Z})$.

\subsection{Higher order terms: An algorithm for computation} \label{Higher}
In this section, although we shall not give concrete forms of higher-order terms, we shall introduce an algorithm to compute each term explicitly.

To proceed to the higher order expansion, as explained in Chapter 2, we use arguments in \cite{Kotani3} combined with a variant of Rayleigh-Schr\"{o}dinger perturbation theory in quantum mechanics (see \cite{Sakurai}, \cite{Rammal}, \cite{Helffer}). We emphasize again that our computation of the asymptotic series is formal in this context. However, it can be justified by replacing $\rho_h$ to finite dimensinal representations $\rho_{{\rm \tiny{fin}},(x_1, \{qx_1x_2\}, x_3)}$ through (\ref{discretetoHeisenbergLie}) in Theorem \ref{discretetoLie}.

 Consider the eigenvalue problem on $L^2(M,L^2(\mathbb{R})) \cong L^2(M \times \mathbb{R}, \mathbb{C})$,
\begin{align}
L_hf_{i,h}(x,s) &= \lambda_i(h)f_{i,h}(x,s) \label{eigen}\\
(f_{i,h},f_{j,h})_{L^2(M \times \mathbb{R})} &:= \int_M\int_{\mathbb{R}}f_{i,h}(x,s)\overline{f_{j,h}(x,s)} dxds = \delta_{ij} \label{normal}
\end{align}  

Take a complete orthonormal basis $\{\psi_k(s)\}_{k=0}^\infty$ of eigenfunctions of $\mathcal{H}$ in $L^2(\mathbb{R})$.

From the computations of terms of orders $h^0, h^{1/2}, h$, we have $\lambda_i^{(0)} = \lambda_i^{(1)} = 0$, thus put
\begin{align}
f_{i,h}(x,s) &= \sum_{j=0}^\infty f_i^{(j)}(x,s)h^{j/2}  \notag \\ &= \sum_{j=0}^\infty \sum_{k=0}^\infty a_{ik}^{(j)}(x)\psi_k(s)h^{j/2} \notag \\ &= \frac{1}{\sqrt{{\rm vol}(M)}}\psi_i(s) + \sum_{j=1}^\infty \sum_{k=0}^\infty a_{ik}^{(j)}(x)\psi_k(s)h^{j/2}  \label{asym2}\\
\lambda_i(h) &= \lambda_i^{(0)} +\lambda_i^{(1)}h^{1/2} +\lambda_i^{(2)}h + \lambda_i^{(3)}h^{3/2} + \ldots \notag \\ &= \frac{\mu_i}{{\rm vol}(M)}h + \lambda_i^{(3)}h^{3/2} + \ldots. \label{asym3}
\end{align}

Next we shall give an algorithm to compute $\lambda_i^{(j)}\; (j = 3,4,\ldots)$ and $a_{ik}^{(j)} \; (j=1,2,\ldots)$ inductively with respect to $j$. As an initial step of the induction, we note that  $a_{i,k}^{(0)}(x) = \frac{\delta_{ik}}{\sqrt{{\rm vol}(M)}}$.

\subsubsection{Step 1: Differential equations for coefficients  $a_{ik}^{(n)}(x)$ of asymptotic expansions}
Assume that we have information of $\lambda_i^{(j)}$ for $j \leq n$ and  $a_{i\ell}^{(j)}(x)$ for $j \leq n-1$ and $k-2 \leq \ell \leq k+2$.

Differentiating both sides of (\ref{eigen}) $n$ times with respect to $\sqrt{h}$, we have  
\[ L^{(0)}f_i^{(n)} - \lambda_i^{(n)}f_i^{(0)} = \sum_{j=1}^n L^{(j)}f_i^{(n-j)} + \sum_{j=0}^{n-1}\lambda_i^{(j)}f_i^{(n-j)}. 
\]
Note that $L^{(j)} = 0 $ provided $j \geq 5$.
Taking an inner product with $\psi_k(s)$ in $L^2(\mathbb{R})$,  by $L^{(0)} = \Delta_M$ and $f_i^{(0)} = \frac{1}{\sqrt{{\rm vol}(M)}}\psi_i$, we have
\begin{align} 
\Delta_M a_{ik}^{(n)}(x) -\lambda_i^{(n)}\delta_{ik} &= -\sum_{j=1}^n \left( L^{(j)}f_i^{(n-j)}, \psi_k(s)\right)_{L^2(\mathbb{R})} + \sum_{j=0}^{n-1}\left(\lambda_i^{(j)}f_i^{(n-j)}, \psi_k(s)\right)_{L^2(\mathbb{R})} \notag  \\
&= -\sum_{j=1}^n \sum_{\ell=1}^\infty \left( L^{(j)}a_{i\ell}^{(n-j)}(x)\psi_\ell(s), \psi_k(s)\right)_{L^2(\mathbb{R})} \notag \\
&\quad + \sum_{j=0}^{n-1}\sum_{\ell=1}^\infty\left(\lambda_i^{(j)}a_{i\ell}^{(n-j)}(x)\psi_\ell(s), \psi_k(s)\right)_{L^2(\mathbb{R})} \notag \\
&= -\sum_{j=1}^n \sum_{\ell=1}^\infty \left( L^{(j)}a_{i\ell}^{(n-j)}(x)\psi_\ell(s), \psi_k(s)\right)_{L^2(\mathbb{R})} + \sum_{j=0}^{n-1}\lambda_i^{(j)}a_{ik}^{(n-j)}(x).\label{aij}
\end{align} 
This is our desired differential equation for $a_{ik}^{(n)}(x)$.

To compute the right-hand side, we use the creation and annihilation operators
$a^\dagger$ and $a$, respectively, associated with the modified harmonic oscillator $\mathcal{H}$ defined as follows: Put $p = \frac{1}{\sqrt{-1}}\frac{d}{ds}$ which is called the momentum operator. Then we define
\begin{align} a &= \sqrt{\frac{\pi\|\omega_2\|_{L^2(M)}}{\|\omega_1\|_{L^2(M)}}}s + \sqrt{-1}\sqrt{\frac{\|\omega_1\|_{L^2(M)}}{4\pi\|\omega_2\|_{L^2(M)}}}p, \label{creationop} \\
a^\dagger &= \sqrt{\frac{\pi\|\omega_2\|_{L^2(M)}}{\|\omega_1\|_{L^2(M)}}}s - \sqrt{-1}\sqrt{\frac{\|\omega_1\|_{L^2(M)}}{4\pi\|\omega_2\|_{L^2(M)}}}p. \label{annihilationop}
\end{align}

Then, we have 
\[ \mathcal{H} = 4\pi\|\omega_1\|_{L^2(M)}\|\omega_2\|_{L^2(M)}\left(a^\dagger a + \frac{1}{2}\right).
\] 
We can compute eigenvalues of $\mathcal{H}$ from this formula.

The normalized eigenfunction $\psi_m(s)$ can be explicitly written as 
\[  \psi_m(s) = \sqrt{\frac{4}{\pi2^mm!}}H_m(\alpha s)\exp\left(- \frac{\alpha^2}{2}s^2\right),
\]
where $\alpha = 2\pi\|\omega_2\|_{L^2(M)}/\|\omega_1\|_{L^2(M)}$ and $H_m(z)$ is the $m$-th Hermite polynomial normalized as $H_0(z) = 1$. The eigenfunction $\psi_m(s)$ satisfies
\begin{equation}\label{aadagger} a\psi_m(s) = \sqrt{m}\psi_{m-1}(s), \quad a^\dagger \psi_m(s) = \sqrt{m+1}\psi_{m+1}(s) \quad (m = 0,1,2,\ldots ).
\end{equation}
Note that
\begin{align} s\psi_m(s)&= \left(\frac{1}{2}\sqrt{\frac{\|\omega_1\|_{L^2(M)}}{\pi\|\omega_2\|_{L^2(M)}}}(a+a^\dag)\right)\psi_m(s) \notag \\
&= \left(\sqrt{\frac{\|\omega_1\|_{L^2(M)}}{4\pi\|\omega_2\|_{L^2(M)}}}\right)(\sqrt{m}\psi_{m-1}(s)+\sqrt{m+1}\psi_{m+1}(s)), \label{saction} \\ 
p\psi_m(s)&= \frac{1}{\sqrt{-1}}\frac{d}{ds}\psi_m(s) = \left(\frac{1}{2\sqrt{-1}}\sqrt{\frac{4\pi\|\omega_2\|_{L^2(M)}}{\|\omega_1\|_{L^2(M)}}}(a-a^\dag)\right)\psi_m(s) \notag \\
&=\left(\frac{1}{\sqrt{-1}}\sqrt{\frac{\pi\|\omega_2\|_{L^2(M)}}{\|\omega_1\|_{L^2(M)}}}\right)(\sqrt{m}\psi_{m-1}(s)-\sqrt{m+1}\psi_{m+1}(s)). \label{daction}
\end{align}
and 
\begin{align} 
s^2\psi_m(s) &= \frac{\|\omega_1\|_{L^2(M)}}{4\pi\|\omega_2\|_{L^2(M)}}\left(\sqrt{m(m-1)}\psi_{m-2}(s)+(2m+1)\psi_m(s)+ \sqrt{(m+2)(m+1)}\psi_{m+2}(s)\right) \label{ssquare} \\ 
s\frac{d}{ds}\psi_m(s) &= \frac12(a+a^\dag)(a-a^\dag)\psi_m(s) = \frac12\left(a^2-[a,a^\dag]+(a^\dag)^2\right)\psi_m(s) \notag \\
&= \frac12\left(\sqrt{m(m-1)}\psi_{m-2}(s)-\psi_m(s)+ \sqrt{(m+2)(m+1)}\psi_{m+2}(s)\right).  \label{sds} \\
\left(\frac{d}{ds}\right)^2\psi_m(s) &= \frac{\pi\|\omega_2\|_{L^2(M)}}{\|\omega_1\|_{L^2(M)}}\left(\sqrt{m(m-1)}\psi_{m-2}(s)-(2m+1)\psi_m(s)+ \sqrt{(m+2)(m+1)}\psi_{m+2}(s)\right) \label{ddssquare}
\end{align}

Then, we consider each terms $\left( L^{(j)}f_i^{(n-j)}, \psi_k\right)_{L^2(\mathbb{R})}$.
Since the action of $L^{(j)}$, $j = 0,1,2,3,4$ on $L^2(\mathbb{R}$) is expressed as a quadratic polynomials of $s, p$, and hence of $a, a^\dag$, thus, by the above formula (\ref{aadagger}), then only coefficints of $\psi_{k-2}$, $\psi_{k-1}$, $\psi_{k}$, $\psi_{k+1}$,  $\psi_{k+2}$ of the orthogonal decompositon of $f_i^{(n-j)}$ with respect to $\{\psi_\ell\}_{\ell = 0}^\infty$ are relevant to the above terms which are computable explicitly. 
For the other terms $\left(\lambda_i^{(j)}f_i^{(n-j)}, \psi_k\right)_{L^2(\mathbb{R})}$ can be computed more directly.
Thus we can compute the differential equation (\ref{aij}) of $a_{ik}^{(n)}(x)$ explicitly. To determine $a_{ik}^{(n)}(x)$, we only need to know $\int_Ma_{ik}^{(n)}(x)$ which will be done in Step 3 below.

\subsubsection{Step 2: Computation of $\lambda_i^{(n)}$}
Assume that we have information of $\lambda_i^{(j)}$ for $1 \leq j \leq n-1$ and  $a_{ik}^{(j)}(x)$ for $1 \leq j \leq n-1$ and $k-2 \leq \ell \leq k+2$.

Integrating (\ref{aij}) over $M$, we have
\begin{equation}   -\mbox{vol}(M)\lambda_i^{(n)} = -\int_M\sum_{j=1}^n \left( L^{(j)}f_i^{(n-j)}, \psi_i\right)_{L^2(\mathbb{R})} + \sum_{j=0}^{n-1}\left(\lambda_i^{(j)}f_i^{(n-j)}, \psi_i\right)_{L^2(\mathbb{R})} \label{lambda}
\end{equation}
As explained in Step 1, we can compute the right-hand side of (\ref{lambda}) and obtain $\lambda_i^{(n)}$.

\subsubsection{Step 3: Computation of $a_{ik}^{(n)}(x)$}
Assume that we have information of $\lambda_i^{(j)}$ for $1 \leq j \leq n$ and  $a_{ik}^{(j)}(x)$ for $1 \leq j \leq n-1$ and $k-2 \leq \ell \leq k+2$.

To compute $a_{ik}^{(n)}(x)$, by Step 1, we only need to know the following matrix: 
\[  B^{(n)} := \left([a]_{ik}^{(n)}\right), \quad   [a]_{ik}^{(n)} = \int_Ma_{ik}^{(n)}(x)dx
\]
If we know this information, then $a_{ik}^{(n)}(x)$ can be computed using the Green operator of $\Delta_M$ similarly to \cite{Kotani3}.

Decompose this matrix $B^{(n)}$ as
\[     B^{(n)} = B^{(n)}_{\mbox{\tiny{herm}}} +B^{(n)}_{\mbox{\tiny{skew}}}
\]
where 
\[B^{(n)}_{\mbox{\tiny{herm}}} := \frac12\left(B^{(n)} +{}^t\overline{B^{(n)}}\right) = \left([a]_{ik}^{(n),h}\right), \quad [a]_{ik}^{(n),h} := \frac12\left([a]_{ik}^{(n)}+\overline{[a]_{ik}^{(n)}}\right) \]
 is the Hermitian part and 
\[B^{(n)}_{\mbox{\tiny{skew}}}:=\frac12\left( B^{(n)} - {}^t\overline{B^{(n)}}\right)= \left([a]_{ik}^{(n),s}\right), \quad [a]_{ik}^{(n),s}:= \frac12\left([a]_{ik}^{(n)}-\overline{[a]_{ik}^{(n)}}\right)
\]
is the skew-Hermitian part of $B^{(n)}$ respectively.

\textbf{Substep 1} (We can rearrange a basis $\{f_{i,h}(x,s)\}$ consisted by eigenfunctions of $L_h$ so that $B^{(n)}_{{\rm \tiny{skew}}} = O$ with zero matrix $O$):

We shall show this fact inductively. Since  $B^{(0)}_{{\rm \tiny{skew}}} = (\delta_{ik}/\sqrt{{\rm vol}(M)})_{{\rm \tiny{skew}}} = O$, we may assume that $B^{(j)}_{{\rm \tiny{skew}}} = O$ for $j < n$ as the inductive hypothesis. Put $U^{(n)} = e^{-B^{(n)}_{{\rm \tiny{skew}}}h^{n/2}}$.  
If $B^{(n)}_{{\rm \tiny{skew}}} \neq O$, we replace $f_{i,h}(p,s)$ with
\[ \widetilde{f}_{i,h}(x,s) := U^{(n)}f_{i,h}(x,s) =: \left(\sum_{j=0}^\infty \sum_{k=0}^\infty \widetilde{a}_{i,k}^{(j)}(x)h^{j/2}\psi_k(s) \right).
\]
Since $U^{(n)}$ is a unitary matrix and \[
U^{(n)} = I - B^{(n)}_{{\rm \tiny{skew}}}h^{n/2} + o(h^{n/2}) 
\] with the identity matrix $I$, we have 
\[ (\widetilde{f}_{i,h}, \widetilde{f}_{j,h})_{L^2(M\times \mathbb{R})} = \delta_{ij},\quad \widetilde{a}_{ik}^{(j)}(x) = a_{ik}^{(j)}(x) \qquad \mbox{for}\quad 1 \leq j \leq n-1.  
\]

We can show that the skew-hermitian part $\widetilde{B}^{(n)}_{{\rm \tiny{skew}}}$ of new matrix $\widetilde{B}^{(n)} = \left(\int_M \widetilde{a}_{ij}^{(n)}\right)$ is $O$ as follows:

Since $U^{(n)}$ acts on the fiber $L^2(\mathbb{R})$. At the same time, $L_h$ is an operator acting on a base manifold $M$, i.e., the action on the fiber $L^2(\mathbb{R})$ is a scalar differential operator with constant coefficients. Therefore the actions of $U^{(n)}$  and $L_h$ are commutative. Then if $\varphi_h$ is an eigenfunction of $L_h$, then $U^{(n)}\varphi_h$ is so and thus several computations for $f_{i,h}$ in  Steps $1, 2$ also hold for $U^{(n)}f_{i,h}$. 

On the other hand, we see
\[  \int_M\widetilde{a}_{ik}^{(n)} = \int_M a_{ik}^{(n)} - [a]_{ik}^{(n),s}  = [a]_{ik}^{(n),h} =\frac12\int_M\left(a_{ik}^{(n)} + \overline{a_{ki}^{(n)}}\right).  
\]
Namely, we have that the skew-hermitian part of the new matrix $\widetilde{B}^{(n)}:= U^{(n)}B^{(n)}$ is the zero matrix $O$.

\textbf{Substep 2} (Computation of $B^{(n)}$):

From Substep 1, we may assume $B^{(n)}_{{\rm \tiny{skew}}} = O$ and thus $B^{(n)} = B^{(n)}_{{\rm \tiny{herm}}}$. We proceed with the computation by differentiating the orthogonal relation (\ref{normal}) $n$ times with respect to $h$. Then,from $f_k^{(0)} = \psi_k/\sqrt{{\rm vol}(M)}$,  
\begin{align}   \int_M a_{ik}^{(n)}(x)dx &= \frac12\int_M\left(a_{ik}^{(n)}(x)dx + \overline{a_{ki}^{(n)}(x)}\right) \notag \\
&= \frac12\int_M\int_{\mathbb{R}}\left(\sum_{\ell=0}^\infty a_{i\ell}^{(n)}(x)\psi_\ell(s)\overline{\psi_k(s)} +  \psi_k(s)\overline{\sum_{\ell=0}^\infty a_{\ell i}^{(n)}(x)\psi_\ell(s)}\right) \notag \\
&= \frac{1}{2}\sqrt{{\rm vol}(M)}\int_M\int_{\mathbb{R}}\left(f_i^{(n)}(x)\overline{f_k^{(0)}(x,s)} + f_i^{(0)}(x) \overline{f_k^{(n)}(x,s)}\right)dx \notag \\
&= -\frac{1}{2}\sqrt{{\rm vol}(M)}\int_M\int_{\mathbb{R}}\sum_{j=1}^{n-1}f_i^{(j)}(x,s)\overline{f_k^{(n-j)}(x,s)} \notag \\
&= -\frac{1}{2}\sqrt{{\rm vol}(M)}\int_M\int_{\mathbb{R}}\sum_{j=1}^{n-1}\sum_{p,q = 0}^\infty a_{ip}^{(j)}(x)\psi_p(s)\overline{a_{kq}^{(n-j)}(x)\psi_q(s)} \notag \\
&= -\frac{1}{2}\sqrt{{\rm vol}(M)}\int_M\sum_{j=1}^{n-1}\sum_{p= 0}^\infty  a_{ip}^{(j)}(x)\overline{a_{kp}^{(n-j)}(x)}. \label{computationofaikn}
\end{align}
Namely, each component $\int_M a_{ik}^{(n)}(x)$ of $B^{(n)}$ can be written as integrals of functions of $a_{i\ell}^{(j)}(x)$ for $j \leq n-1$ over $M$ and thus, we know that $B^{(n)}$.
 
Combining the equation (\ref{aij}), we can get a solution  $a_{ik}^{(n)}(x)$,which is written in terms of  the Green operator of $\Delta_M$ and the information of $\lambda_i^{(j)}$ for $j \leq n$ and  $a_{ik}^{(j)}(x)$ for $1 \leq j \leq n-1$.

\subsection{Concrete computation of the subleading term $\lambda_i^{(4)}$}\label{concrete}
We have already computed from (\ref{sunadainequality}) and (\ref{harmeigen}) as
\begin{align} \lambda_i^{(0)} &= \lambda_i^{(1)} = 0 \notag \\  \lambda_i^{(2)} &= \frac{\mu_i}{{\rm vol}(M)} = \frac{2\pi}{{\rm vol}(M)}\|\omega_1\|_{L^2(M)}\|\omega_2\|_{L^2(M)}\left(i+\frac{1}{2}\right) = \frac{2\pi{\rm vol}(\widehat{H})}{{\rm vol}(M)}\left(i+\frac{1}{2}\right)  \label{concreteprepare} \\ 
f_i^{(0)} &\equiv \psi_i/\sqrt{{\rm vol}(M)} \quad \mbox{i.e.} \quad a_{ij}^{(0)}= \delta_{ij}/\sqrt{{\rm vol}(M)}.
\end{align}
Moreover, by the symmetry of the eigenvalues of the twisted operator $L_h$ with respect to the parameter $h$ at $h=0$,

we have 
\begin{equation} \lambda_i^{(2k+1)} = 0  \label{vanishodd}
\end{equation}
 
To obtain $\lambda_i^{(4)}$, we compute $a_{ij}^{(k)}$ and $\lambda_i^{(k)}$ for $k \leq 3$ in order. The following computations are based on (a) (\ref{aij}), (b) (\ref{computationofaikn}), 
\begin{description}
\item[{\rm (1)}]$a_{ij}^{(1)}$: 
\begin{description}
\item[{\rm (a)}] From $\lambda_i^{(0)} = \lambda_i^{(1)} =0$ and $f_i^{(0)} \equiv \psi_i$, we have 
\begin{equation} \Delta_M a_{ik}^{(1)} = - \left( L^{(1)}\psi_i, \psi_k\right)_{L^2(\mathbb{R})} = 0.\label{aij1}
\end{equation} 
Thus, we obtain $a_{ik}^{(1)} \equiv \mbox{constant}$.
\item[{\rm (b)}] The $(i,k)$-component of $B^{(1)}_{{\rm \tiny{herm}}}$,
\[ \int_M a_{ik}^{(1)}(x) = -\frac{1}{2}\int_M\sum_{j=0}^{0}\sum_{p= 0}^\infty\left(\begin{array}{c}1 \\ j\end{array}\right)a_{ip}^{(j)}(x)\overline{a_{kp}^{(1-j)}(x)} = 0.
\] 
Thus, combining with $a_{ik}^{(1)} \equiv \mbox{constant}$, we have 
\begin{equation}
a_{ik}^{(1)} \equiv 0\quad \mbox{and}\quad f_{i}^{(1)}=\sum_{k=0}^\infty a_{ik}^{(1)}\psi_k = 0. \label{aik1}
\end{equation}
\end{description}

\item[{\rm (2)}]$a_{ij}^{(2)}$: 
\begin{description}
\item[{\rm (a)}] By $d\psi_\ell =0$, we compute 
\begin{align}
L^{(2)}\psi_\ell &= -\left(4\pi^2|\omega_2|^2s^2-4\pi\sqrt{-1}\langle \omega_2,\omega_1 \rangle s\frac{d}{ds} -|\omega_1|^2\left(\frac{d}{ds}\right)^2\right)\psi_\ell  \notag \\
&{} -4\pi\sqrt{-1}\langle \omega_{12}, d\psi_\ell \rangle \notag \\
&= -2\pi|\omega_1||\omega_2|\left(\ell+\frac{1}{2}\right)\psi_\ell \notag\\
&{} -4\pi\sqrt{-1}\langle \omega_2,\omega_1 \rangle\left((\sqrt{\ell(\ell-1)}\psi_{\ell-2} -\psi_\ell + \sqrt{(\ell+2)(\ell+1)}\psi_{\ell+2}\right) \label{L2psi}
\end{align}

Here we have used the relations (\ref{saction}) and (\ref{daction}) and
\begin{equation*}
\left(4\pi^2|\omega_2|^2s^2 -|\omega_1|^2\left(\frac{d}{ds}\right)^2\right)\psi_\ell = 2\pi|\omega_1||\omega_2|\left(a^\dagger a + \frac{1}{2}\right)\psi_\ell = 2\pi|\omega_1||\omega_2|\left(\ell+\frac{1}{2}\right)\psi_\ell.
\end{equation*}
By $(a^\dagger a - a a^\dagger)\psi_\ell = \ell\psi_\ell - (\ell+1)\psi_\ell = -\psi_\ell$,
\begin{align*}
&{} -4\pi\sqrt{-1}h\langle \omega_2,\omega_1 \rangle s\frac{d}{ds}\psi_\ell
= -4\pi\sqrt{-1}\langle \omega_2,\omega_1 \rangle(a+a^\dag)(a-a^\dag)\psi_\ell  \\
&{} = -4\pi\sqrt{-1}\langle \omega_2,\omega_1 \rangle(\sqrt{\ell(\ell-1)}\psi_{\ell-2} - \psi_\ell + \sqrt{(\ell+2)(\ell+1)}\psi_{\ell+2})
\end{align*}

Then, from (\ref{concreteprepare}), we have 
\begin{align}
 \Delta_M a_{ik}^{(2)} &= \lambda_i^{(2)}\delta_{ik} 
 -\sum_{j=1}^{2}\sum_{\ell=1}^\infty \left( L^{(j)}a_{i\ell}^{(2-j)}\psi_\ell, \psi_k\right)_{L^2(\mathbb{R})} - \sum_{j=0}^{1}\lambda_i^{(j)}a_{ik}^{(2-j)} \notag \\ &=  \lambda_i^{(2)}\delta_{ik} - \sum_{\ell=1}^\infty \left( L^{(2)}a_{i\ell}^{(0)}\psi_\ell, \psi_k\right)_{L^2(\mathbb{R})}\notag \\ 
&= \frac{2\pi{\rm vol}(\widehat{H})}{{\rm vol}(M)}\left(i+\frac{1}{2}\right)\delta_{ik}-4\pi|\omega_1||\omega_2|\left(i+\frac{1}{2}\right)\left(\psi_i,\psi_k\right)_{L^2(\mathbb{R})} \notag \\ 
&{} -\frac{4\pi\sqrt{-1}}{{\rm vol}(M)}\langle \omega_2,\omega_1 \rangle\left(\sqrt{i(i-1)}\psi_{i-2}-\psi_i +\sqrt{(i+2)(i+1)}\psi_{i+2}, \psi_k\right)_{L^2(\mathbb{R})} \notag \\ 
&= \Big\{\Big(\frac{2\pi{\rm vol}(\widehat{H})}{{\rm  vol}(M)}-2\pi|\omega_1||\omega_2|\Big)\left(i+\frac{1}{2}\right)\Big\}\delta_{ik} \notag \\ 
&{} -\frac{4\pi\sqrt{-1}}{{\rm vol}(M)}\langle \omega_2,\omega_1 \rangle(\sqrt{(i+2)(i+1)}\delta_{{i-2},k}-\delta_{ik}+ \sqrt{i(i-1)}\delta_{{i+2},k}). \label{aij2} 
\end{align}

\item[{\rm (b)}] The $(i,k)$-component of $B^{(2)}$
\[
 \int_M a_{ik}^{(2)}(x) = -\frac{1}{2}\int_M\sum_{j=1}^{1}\sum_{p= 0}^\infty a_{ip}^{(j)}(x)\overline{a_{kp}^{(2-j)}(x)} = -\frac{1}{2}\sum_{p= 0}^\infty a_{ip}^{(1)}(x)\overline{a_{kp}^{(1)}(x)}= 0 
\]
 
We denote by $G$ the Green operator of $\Delta_M$ on $M$, that is, the operator acting on $C^\infty(M)$ satisfying
\[ I = H +G\Delta_M, \quad G\Delta_M = \Delta_MG, \quad GH = HG = O,
\] 
where $I$ is the identity operator and $H$ is the harmonic projection, namely $Hf = \frac{1}{{\rm vol}(M)}\int_Mf$. Then by (\ref{aij2}) and $G(\delta_{ik}) =0$, we have
\begin{align}
a_{ik}^{(2)}(x) &= (H +G\Delta_M)a_{ik}^{(2)}(x) \notag \\ 
&= \frac{2\pi{\rm vol}(\widehat{H})}{{\rm  vol}(M)}\left(i+\frac{1}{2}\right)G(\delta_{ik})-4\pi G(|\omega_1||\omega_2|)(x)\left(i+\frac{1}{2}\right)\delta_{ik}\notag \\
&{} -\frac{4\pi\sqrt{-1}}{{\rm vol}(M)}G(\langle \omega_2,\omega_1 \rangle)(x)(\sqrt{(i+2)(i+1)}\delta_{{i-2},k}-\delta_{ik}+ \sqrt{i(i-1)}\delta_{{i+2},k}) \notag \\
&=-4\pi G(|\omega_1||\omega_2|)(x)\left(i+\frac{1}{2}\right)\psi_i \notag \\ 
&{} -\frac{4\pi\sqrt{-1}}{{\rm vol}(M)}G(\langle \omega_2,\omega_1 \rangle)(x)(\sqrt{(i+2)(i+1)}\delta_{{i-2},k}-\delta_{ik}+ \sqrt{i(i-1)}\delta_{{i+2},k})  
\end{align}
where we denote 
\begin{align*}
G(|\omega_1||\omega_2|)(x) &:= G(|(\omega_1)_x|,|(\omega_2)_x|) \\
G(\langle \omega_2,\omega_1 \rangle)(x) &:= G(\langle (\omega_2)_x,(\omega_1)_x \rangle)
\end{align*}
\end{description}
\end{description}
From the above computations, we have 
\begin{align}
f_i^{(0)}(x,s) &=  \frac{1}{\sqrt{{\rm vol}(M)}}\psi_i(s)  \label{fi0}\\
f_i^{(1)}(x,s) &=  \sum_{k=0}^\infty a_{ik}^{(1)}(x)\psi_k(s) = 0 \label{fi1}\\
f_i^{(2)}(x,s) 
&=  \sum_{k=0}^\infty a_{ik}^{(2)}(x)\psi_k(s) \notag \\
&= -\frac{4\pi}{\sqrt{{\rm vol}(M)}}G(|\omega_1||\omega_2|)(x)\left(i+\frac{1}{2}\right)\psi_i\notag \\
&{}-\frac{4\pi\sqrt{-1}}{{\rm vol}(M)}G(\langle \omega_2,\omega_1 \rangle)(x)(\sqrt{(i+2)(i+1)}\psi_{i-2}-\psi_i+ \sqrt{i(i-1)}\psi_{i+2})  \label{fi2}
\end{align}

Combinig with (\ref{lambda}), (\ref{vanishodd}) and (\ref{L2psi}), we have 
\begin{align*} \lambda_i^{(4)} &= -\int_M\sum_{j=1}^4 \left( L^{(j)}f_i^{(4-j)}, \psi_i\right)_{L^2(\mathbb{R})} + \sum_{j=0}^{3}\left(\lambda_i^{(j)}f_i^{(4-j)}, \psi_i\right)_{L^2(\mathbb{R})} dx \\
&=  -\int_M\left( (L^{(2)}-\lambda_i^{(2)})f_i^{(2)}+L^{(4)}f_i^{(0)}, \psi_i\right)_{L^2(\mathbb{R})} dx. 
\end{align*}
By (\ref{L2psi}), we have
\begin{align} 
\int_M\left( L^{(2)}f_i^{(2)}, \psi_i\right)_{L^2(\mathbb{R})} dx
&= \int_M\Big(-\frac{4\pi}{\sqrt{{\rm vol}(M)}}G(|\omega_1||\omega_2|)(x)\left(i+\frac{1}{2}\right)L^{(2)}\psi_i\notag \\
&{}-\frac{4\pi\sqrt{-1}}{{\rm vol}(M)}G(\langle \omega_2,\omega_1 \rangle)(x)(\sqrt{(i+2)(i+1)}L^{(2)}\psi_{i-2}-L^{(2)}\psi_i+ \sqrt{i(i-1)}L^{(2)}\psi_{i+2}, \Big), \psi_i\Big)_{L^2(\mathbb{R})} dx \notag \\
&= \int_M\frac{8\pi^2}{\sqrt{{\rm vol}(M)}}|\omega_1||\omega_2|G(|\omega_1||\omega_2|)(x)\left(i+\frac{1}{2}\right)^2dx\notag \\
&{} -\int_M\Big(\frac{16\pi^2\sqrt{-1}}{\sqrt{{\rm vol}(M)}}\langle \omega_2,\omega_1 \rangle G(|\omega_1||\omega_2|)(x)\left(i+\frac{1}{2}\right)dx \notag \\
&{} -8\pi\sqrt{-1}\sqrt{(i+2)(i+1)i(i-1)}\int_M\langle \omega_2,\omega_1 \rangle dx
\end{align}
Since $\lambda_i^{(4)} \in \mathbb{R}$, the right-hand side is, in fact,  equal to
\[  \int_M\frac{8\pi^2}{\sqrt{{\rm vol}(M)}}|\omega_1||\omega_2|G(|\omega_1||\omega_2|)(x)\left(i+\frac{1}{2}\right)^2dx.
\]
Combinig with computaions on the other parts based on (\ref{concreteprepare}), (\ref{H4}), we have 
\begin{align}
\lambda_i^{(4)} &=  -\int_M\left( (L^{(2)}-\lambda_i^{(2)})f_i^{(2)}+L^{(4)}f_i^{(0)}, \psi_i\right)_{L^2(\mathbb{R})} dx \notag \\
&= \int_M\frac{8\pi^2}{\sqrt{{\rm vol}(M)}}|\omega_1||\omega_2|G(|\omega_1||\omega_2|)(x)\left(i+\frac{1}{2}\right)^2 
\notag \\
&{} + \frac{8\pi^2{\rm vol}(\widehat{H})}{({\rm vol}(M))^2}\left(i+\frac{1}{2}\right)^2 G(|\omega_1||\omega_2|)(x)+ \frac{4\pi^2}{{\rm vol}(M)}|\omega_{12}|^2dx. \label{4thderivativelambda}
\end{align}

\section{Asymptotic expansion of the lowest eigenvalue of the twisted Laplacian: general nilpotent groups}\label{asymnilpotent}

The formal computation of the previous chapter heavily depends on the explicit nature of the Sch\"odinger representation and the one-dimensional Harmonic oscillator. Although irreducible unitary representations of general nilpotent Lie groups can be written rather explicitly (cf. Appendices \ref{B.1} -\ref{B.3}), we could
compute the hypo-elliptic operator $\mathcal{H}$. However, exact values of eigenvalues and exact forms of eigenfunctions of this $\mathcal{H}$ are unknown in general. Thus, we could not expect to obtain the same level of explicitness as the one-dimensional Harmonic oscillator for general hypo-elliptic operators $\mathcal{H}$ associated with these representations.

However, regardless of that, there are two ways to proceed with generalizations to general nilpotent groups.

The first one is directly based on the spectral analysis of the hypo-elliptic operator $\mathcal{H}$ mentioned above. Namely, we consider an expansion using eigenvalues and eigenfunctions of this operator, combined with estimates of higher-order terms, which are achieved through subelliptic estimates (e.g., \cite{HelfferNou}, \cite{Nourrigat}).

The second one is based on an expression using the basis of multi-dimensional harmonic oscillators. In this case, we can not diagonalize the matrix representation of general $\mathcal{H}$, but it can be expressed as a finite-width band matrix near the diagonal. 
This matrix is also known as a finite propagation or finite range operator. The latter naming is similar to short or long-range potential in quantum mechanics. 
On the other hand, for higher-order terms, more explicit expressions than the first one are available. The coefficients for each term can be written in terms of the basis of multi-dimensional harmonic oscillators. 

These two methods have advantages and disadvantages, as explained in the following subsections. 

\subsection{Spectral analysis of a hypo-elliptic operator}\label{spectralhypoelliptic}
By the Birkhoff embedding theorem (1.1.11 Theorem in \cite{Corwin}), the nilpotent Lie algebra $\mathfrak{g}$ over $\mathbb{R}$ is a finite-dimensional Lie subalgebra of the linear Lie algebra $\mathfrak{gl}_n(V)$ of a vector space $V$. Thus, we may assume that its Lie bracket can be expressed as 
\[   [X,Y] = XY-YX \quad \mbox{in} \quad \mathfrak{gl}_n(V).
\]
In addition, by Section \ref{stratification}, we may assume $\mathfrak{g}$ is stratified, namely it satifies conditions, (\ref{liestratified1}), (\ref{liestratified2}), (\ref{liestratified3}).

Since we have the surjective homomorphism $\Phi : \pi_1(M) \to \Gamma$ and $G$ is the Malcev completion of $\Gamma$, we may identify $\mathfrak{g}^{(1)}$ with the image of $H_1(M,\mathbb{R}) \simeq (\pi_1(M)/[\pi_1(M),\pi_1(M)])\otimes \mathbb{R}$ under the induced map from the natullary extended map $\widetilde{\Phi}: \pi_1(M)\otimes\mathbb{R} \to G$ of $\Phi$. Thus, the basis $X_1,\ldots, X_k$ of $\mathfrak{g}^{(1)}$ can be viewed as a subset of the basis $X_1,\ldots, X_m$ of $H_1(M,\mathbb{R})$ in Section \ref{harmoniclietheory} around the definition (\ref{powerseriesconnection}).  

Next, we view the flat connection $\omega$ in (\ref{RH}) in Theorem \ref{pi1derham} as a formal power series connection. Then, we can use several formulas after (\ref{powerseriesconnection}) by replacing index $m$ there with $k$ here.

For each irreducible unitary representations $\pi_l$ of nilpoten Lie group $G$ in Theorem \ref{FourierinversionLie}, we can construct the canonical section 
$S_{\pi_l}$ in $U(E_{\pi_l})$ by generalizing $\rho_h$, $U(E_{\rho_h})$ and $S_{\rho_h}$ in \ref{sectionsh} for $\mbox{Heis}_3(\mathbb{Z})$ to general discrete nilpotent group $\Gamma$ whose Malcev completion is $G$, namely, 
\begin{equation} S_{\pi_l} = \pi_l\left({}_R\!\!\int_{p_0}^p \omega \right). \label{sectionpil}
\end{equation}

We need to choose the representation $\pi_l$ corresponding to the generic orbits, which are orbits of maximal dimension. In the case of $\mbox{Heis}_3(\mathbb{R})$, this $\pi_l$ is nothing but the Schr\"odinger representaion $\rho_h$.
The Concrete calculation of $\pi_l$ is presented in Appendix B. For the descriptions of the generic orbits, we refer to Theorems \ref{corwin316}, \ref{corwin318},  \ref{corwin319}. Then, we choose an element $l$ in the generic orbits arbitrarily and construct $\pi_l$ as the induced representation from the character on the maximal subordinate algebra (also known as the polarizing subalgebra) $\mathfrak{m}$ of $l$. There may be many choices of $\mathfrak{m}$ for each $l$, but these induced representations are unitary equivalent to each other. Although it is not generally possible to choose canonically, there are several hints for choosing $\mathfrak{m}$ conveniently in section 3.2 in \cite{Corwin}.

Next we put a weight $\hbar$ compatible with strafication, namely $\hbar^i$ for elemnts in $\mathfrak{g}^{(i)}$ and replacing $l$ to $\hbar l$, expand the twited Laplacian $\Delta_{\pi_{\hbar l}}$ as a formal power series $\sum_{i=0}^\infty L^{(i)}\hbar^i$. Then, similar computations as (\ref{Laplacian})- (\ref{H4}) and Section \ref{subsection72} imply that $L^{(0)}=\Delta_M$ and $L^{(1)}$ is a skew-hermitian operator. Moreover, the next coefficient $L^{(2)}$ can be expresed as a following form:
\begin{equation}
L^{(2)} = \mathcal{H}^G + \mbox{skew-hermitian operator} \label{generalnilpotentsecond}
\end{equation}
where $\mathcal{H}^G$ is the hypo-elliptic operator appearing in the following form  
\begin{equation} \mathcal{H}^G = \|\omega_1\|_{L^2(M)}^2d\pi_l(X_1)^2+\cdots +\|\omega_k\|_{L^2(M)}^2d\pi_l(X_k)^2, \label{generalhypo2}
\end{equation}
and 
$\omega_1, \ldots ,\omega_k$ are harmonic forms on $M$ which consist $\mathfrak{g}^{(1)}$-part of the $\mathfrak{g}$-valued form determined by $\pi_1$-de Rham Theorem \ref{pi1derham}. We may assume $L^2$-inner product $(\omega_i,\omega_j)_{L^2(M)}=0$ for $i\neq j$ by the same reason explained after Condition \ref{harmoniccoexact}.
In the case when $\Gamma = \mbox{Heis}_3(\mathbb{Z})$, $\mathcal{H}^G$ is the modified harmonic oscillator $\mathcal{H}$ in (\ref{Heisenharmonic}) in (\ref{HC}).

We notice that the twisted Laplacian $\Delta_{\pi_{\hbar l}}$ acting on the space of sections on the flat vector bundle $E_{\pi_{\hbar l}}$ associated to $\pi_{\hbar l}$ is self-adjoint. Here $\hbar = \sqrt{h}$ when $\Gamma = \mbox{Heis}_3(\mathbb{Z})$. Then, we can ignore the skew-hermitian parts of each term in the computation of the expansions of its eigenvalues $\lambda_{k,i}(\pi_{\hbar l})$, where the indexings are the same as those given in Section \ref{subsection72}. Therefore, we have the following expansion of the eigenvalue $\lambda_{0,i}(\pi_{\hbar l})$ of $\Delta_{\pi_{\hbar l}}$:
\begin{equation}
\lambda_{0,i}(\pi_{\hbar l}) = \sum_{j=0}^\infty \lambda_{0,i}^{(j)}\hbar ^j = \mu_i \hbar ^2 + \cdots   \label{nilpotentlambdaexpansion}  
\end{equation}
where $\mu_i$ is the $i$-th eigenvalue of $\mathcal{H}^G$. In the second equlity, we have used the fact that $\lambda_{0,i}^{(0)} = \lambda_{0,i}^{(1)} = 0$ which are derived from the fact $\lambda_{0,i}(\pi_{\hbar l}) \geq 0$ and takes minimum value $0$ at $\hbar = 0$ as same reason as in the previous chapter.

The operator $\mathcal{H}^G$ is hypo-elliptic and maximally polynomial in the sense of \cite{HelfferNou} and, thus, has a discrete spectrum and the special value $\zeta_{\mathcal{H}}(d/2)$ is convergent. 

For the convenience of our readers, we explain the reason for the positive definiteness of the Dirichlet form associated with $\mathcal{H}^G$, which is equivalent to all eigenvalues $\mu_i$ of $\mathcal{H}^G$ being positive. However, this fact is widely known.

 First, we recall the Kirillov-Kostant symplectic form $\omega$ on the tangent space $T_l\mathcal{O}_l$ at $l$ of the coadjoint orbit $\mathcal{O}_l$ through $l \in \mathfrak{g}^\ast$. 
This is defined as $\omega(X,Y):= l([X,Y])$  by identifying $T_l\mathcal{O}_l$ with $g/\mathfrak{r}_l$, where $\mathfrak{r}_l := \{X | l([X,Y]) = 0 \; \mbox{for any}\; Y \in \mathfrak{g}\}$ is the radical of $l$. 
Then, for any maximal subordinate subalgebra $\mathfrak{m}$ which is a maximal one among subalgebras $\mathfrak{n}$ of $\mathfrak{g}$ satisfying $l([X,Y]) = 0$ for any $X, Y \in \mathfrak{n}$, the quotient space $\mathfrak{m}/\mathfrak{r}_l$ is a Lagrangian subspace of $\mathfrak{g}/\mathfrak{r}_l$ by definition of $\mathfrak{m}$. 
Moreover, the representation space of the induced representation $\pi_l$ has a model $L^2(\mathbb{R}^k)$ with $k = \mbox{dim}(\mathfrak{m}/\mathfrak{r}_l) = \frac12\mbox{dim}(\mathfrak{g}/\mathfrak{r}_l)$. 
In the concrete expression of $\pi_l$, actions coming from the chacter $\chi_l$ of elements in $\mathfrak{m}$ are multiplication of elements of $U(1)$ on a function $f \in L^2(\mathbb{R}^k)$ and actions coming from the complement $\mathfrak{g}/\mathfrak{m}$  of $\mathfrak{m}/\mathfrak{r}_l$ in $\mathfrak{g}/\mathfrak{r}_l$ are traslation of varibles of a function $f$. Namely, if we consider the derivative $d\pi_l(X)$, then it acts as a multiplication of a polynomial in the former and a derivation for the latter cases. Furthermore, since we may assume $\mathfrak{g}$ is stratified by Section \ref{stratification}, $\mathfrak{g}^{(1)}$ generates $\mathfrak{g}$ and thus, there is a system of polynomial $p_i(d\pi_l(X_1), \ldots, d\pi_l(X_m))$  for $i=1, \ldots, k$

Moreover, concerning higher-order terms of the asymptotic expansion, by sub-elliptic estimates \cite{HelfferNou}, \cite{Nourrigat}, we can estimate the norms of higher derivatives of eigenfunctions in conjunction with the estimates of eigenvalues. Therefore, even in general nilpotent groups, the coefficients of higher-order terms of expansions can be estimated, although their detailed information is less explicit than that of the harmonic oscillator. 

\subsection{Concrete construction of the hypo-elliptic operator $\mathcal{H}^G$}\label{constracthypoelliptic}

We shall explain how to construct the hypo-elliptic operator $\mathcal{H}^G$
\begin{description}
\item[{\rm (1)}] Given a torsion-free discrete nilpotent group $\Gamma$, consider its Malcev completion $G$, which can be done naturally since the multiplication rule of $\Gamma$ is expressed as polynomials with respect to their suitable coordinates  (cf. Section \ref{nilpotentfujiwara})

\item[{\rm (2)}] By the procedure explained in Section \ref{stratification}, regard $G$ as stratified Lie group $G = (G,\ast)$. Hereafter, we assume $G$ is a stratified nilpotent Lie group.

\item[{\rm (3)}] Using the stratification, choose the rational strong Malcev basis $\{X_i| 1 \leq i \leq n:= \mbox{dim}\;\mathfrak{g}\}$ of $\mathfrak{g}$. It can be done by such as taking iterated Lie brackets of elements in the basis $\{X_i| 1\le i \le n:= \mbox{dim}\;\mathfrak{g}^{(1)}\}$ of $\mathfrak{g}^{(1)}$ where we used the same notations in Subsubsection \ref{stratification}. 

Moreover, we can choose this basis as a rational one, namely we may assume $X_i \in \mathfrak{g}_\mathbb{Q}$ for all $i$, since the existence of the lattice $\Gamma$ in $G$ is equivalent to $G$ equips $\mathbb{Q}$-structure (cf. 5.1.8 Theorem in \cite{Corwin}).

\item[{\rm (4)}] For each element $X_i$, compute its coadjoint action on $\mathfrak{g}^\ast_\mathbb{Q}$. 
From the collection of their data, choose a generic element $l \in \mathfrak{g}^\ast_\mathbb{Q}$ such that the isotropy algebra $\mathfrak{g}_l$ of $\mbox{ad}^\ast(\mathfrak{g})$ at $l$ has minimal dimension. 
Since the dimension of the coadjoint orbit $\mathcal{O}_l$ through $l$ is equal to the codimension of $\mathfrak{g}_l$, this orbit $\mathcal{O}_l$ is a generic orbit.

\item[{\rm (5)}] Take any $l$ in the generic orbit $\mathcal{O}$ and compute the radical $\mathfrak{r}_l$ which is uniquely determined by $l$. We can also assume $\mathfrak{r}_l$ is a rational ideal.

\item[{\rm (6)}] We want to choose a polarization $\mathfrak{m}$ in $\mathfrak{g}_\mathbb{Q}$, but there can be many polarizing subalgebras $\mathfrak{m}$. 
Although they are subject to certain restrictions, for example, $\mathfrak{m} \supset \mathfrak{r}_l$, there is no systematic way to construct them all, given $l$, which is one of the complications of the theory. 

The following construction, essentially due to Vergne \cite{Vergne}, selects a polarizing subalgebra in a canonical way for any  $l \in \mathfrak{g}^\ast_\mathbb{Q}$, given a strong Malcev basis in $\mathfrak{g}_\mathbb{Q}$.

\begin{theorem}[{\rm 1.3.6 Theorem in \cite{Corwin}}]\label{methodVergne}
Let $(O) \subset \mathfrak{g}_1  \subset \mathfrak{g}_2 \subset \cdots \subset \mathfrak{g}_n$ be a chain of ideals in $\mathfrak{g}_\mathbb{Q}$ such that $\mbox{dim}\;\mathfrak{g}_j = j$. 
Given $l \in \mathfrak{g}^\ast_\mathbb{Q}$, let $l_j = l|_{\mathfrak{g}_j}$. 
Then $\mathfrak{m}_l := \sum_{j=1}^n \mathfrak{r}_{l_j} \subset \mathfrak{g}_\mathbb{Q}$ is a polarizing subalgebra for $l$.
\end{theorem}

The above chain of ideals in $\mathfrak{g}_\mathbb{Q}$ can be taken by $\mathfrak{g}_j = \mathbb{R}\mbox{-span}\{X_1,\ldots,X_j\}$ for the strong Malcev basis $\{X_i| 1\le i \le n := \mbox{dim}\;\mathfrak{g}_\mathbb{Q}\}$ chosen in the item ($3$).

\item[{\rm (7)}] To describe the structure of generic orbit $\mathcal{O} = \mathcal{O}_l$, we apply Theorem \ref{corwin316} for $V = \mathfrak{g}^\ast_\mathbb{Q}$ there (see quotation \cite{Fujiwara2} in  Section \ref{nilpotentfujiwara} for similar description). 
Then,let $\mathfrak{g}_j^\ast$ be the rational dual Lie algebra of $\mathfrak{g}_j$, $l_j = l|_{\mathfrak{g}_j^\ast}$ and $d_j (l)$ be the dimension of the $G$-orbit of $l_j$ in $(\mathfrak{g}_j )^\ast$. 
Put $d(l) = (d_1(l), \ldots, d_n(l))$ and $\mathcal{D} = \{d(l)| l \in \mathcal{O} \}$. Then, there is a unique $d = (d_1, . . . , d_n) \in \mathcal{D}$ such that the layer $U = \{l \in V | d(l) = d\}$ is a Zariski open set of $V$.
 Let $T$ be the set of the non-jump indices of $d$, namely $T = \{j| d_j = d_{j-1}\}$ and $S = \{1,\ldots,n\} \setminus T$. Then $\mbox{dim}\;\mathcal{O} = {}^\sharp S$ 
Let $V_S =\mathbb{R}\mbox{-span}\{X_i: i \in S\}$, $V_T = \mathbb{R}\mbox{-span}\{X_i: i \in T\}$. Then, for each element $t(X) = \sum_{j\in T}t_jX_j \in V_T$, the component $t = \{t_j|j\in T\}$ parametrize generic orbits $\mathcal{O}_t := \mbox{$G$-orbit of}\;t(X)$.

\item[{\rm (8)}] For each $l \in \mathcal{O}_t\cap\mathfrak{g}^\ast_\mathbb{Q}$, we choose a polarization subalgebra $\mathfrak{m}$, for example, as in Theorem \ref{methodVergne} and construct irreducible unitary representation $\pi_l$ as an induced representation $\mbox{Ind}_M^G(\chi_l)$ with $M = \exp\;\mathfrak{m}$ from the character $\chi_l = e^{2\pi\sqrt{-1}l(X)}$ of $\mathfrak{m}$ to $G$.
Then, combining with the explanation in the previous section, the  hypoelliptic operator $\mathcal{H} = \mathcal{H}^G$ is constructed as 
\begin{equation} 
\mathcal{H}^G = \|\omega_1\|_{L^2(M)}^2d\pi_l(X_1)^2+\cdots +\|\omega_k\|_{L^2(M)}^2d\pi_l(X_k)^2, \label{detailedgeneralhypo1}
\end{equation}
\end{description}

\subsection{Some computations based on the multi-dimensional harmonic oscillator expressions}\label{multiharmonic}

Since the representation space $\mathcal{H}_{\pi_l}$ of the above $\pi_l$ is modeled on $L^2(\mathbb{R}^k)$ for some $k$, namely the former is isomorphic to 
the latter space. For a suitable coordinates $s^1, \ldots, s^k$ of $\mathbb{R}^k$, all the coefficients in asymptotic expansions are polynomials of $s^i$'s and $\frac{\partial}{\partial s^i}$'s for $i=1,\ldots,k$. Then, similarly to (\ref{creationop}) and (\ref{annihilationop}), we take the creation and annihilation operators $a_i^\dagger$ and $a_i$ associated to $k$-dimentional harmonic oscillator $\mathcal{H}_k$ defined, up to some constant $c_{M,\Gamma}$ depending on the geometry of the base space $M$ and the covering transformation group $\Gamma$, as  
\[\mathcal{H}_k = -\sum_{i=1}^k\left(\frac{\partial}{\partial s^i}\right)^2 + \sum_{i=1}^k(s^i)^2 = \sum_{i=1}^k (a_i^\dagger a_i + \frac12),
\]
where
\[ a_i = \frac{1}{\sqrt{2}}\left(s^i + \frac{\partial}{\partial s^i}\right),\quad  a_i^\dagger = \frac{1}{\sqrt{2}}\left(s^i - \frac{\partial}{\partial s^i}\right). \] 
Hereafter in this section, we ignore the above constant $c_{M,\Gamma}$ for the sake of brevity.

Since 
\[ \frac{\partial}{\partial s^i} = \frac{1}{\sqrt{2}}(a_i - a_i^\dagger ),\quad s^i = \frac{1}{\sqrt{2}}(a_i + a_i^\dagger ),
\]
we know that all the coefficients in asymptotic expansions are polynomials of $a_i^\dagger$'s and $a_i$'s which map $\psi_{i,\ell}$ to $\sqrt{\ell}\psi_{i,\ell-1}$ and $\sqrt{\ell+1}\psi_{i,\ell+1}$ respectively for a complete orthonormal system $\{\psi_{i,\ell}\}_{(i,\ell) = (1,0)}^{(k,\infty )}$ of eigenfunctions of $\mathcal{H}_k$ which can be written as 
\[\psi_{(i_1,\ldots,i_k),(\ell_1,\dots, \ell_k)}(s_1,\ldots,s_k) =\psi_{i_1,\ell_1}(s_1)\cdot\psi_{i_2,\ell_2}(s_2)\cdots \psi_{i_,\ell_k}(s_k),\] 
where $\psi_{i,\ell}(s_i)$ are $\ell$-th eigenfunctions of the one dimensional harmonic oscillator $\mathcal{H}_i = -\left(\frac{\partial}{\partial s^i}\right)^2 + (s^i)^2$.

Then, the representation matrix $\mathcal{H}^G = \left(h_{ij}\right)$ of the hypo-elliptic operator $\mathcal{H}^G$ is a finite-width band matrix near the diagonal, namely satisfying the following;
\[   h_{ij} = 0 \quad \mbox{if $|i-j| > K$}
\] 
for some $K >0$ depending only on $k$. 

The explicit expression of the eigenvalues of the above $\mathcal{H}^G$ is still obscure. The other method for computing the special value $\zeta_{\mathcal{H}}(d/2)$, similar to the computations done in the case of the quartic oscillator in Example \ref{Engelquartic}, is not known. 
However, its integral expression or several properties of the number-theoretic nature seem to have some interest, see Section \ref{leadingterm} for further discussion.
On the other hand, the higher-order terms can be explicitly computed. Thus, there are some merits in using the above basis.

\section{Long time asymptotics of heat kernels for nilpotent extensions}\label{Longtimeheat}

In this chapter, we first prove the special case Theorem  \ref{heisenberg-heat} of $\mbox{Heis}_3(\mathbb{Z})$ and also the local central limit theorem in this case, and next, prove the general case Theorem \ref{conj-heat}.

\subsection{Expression of the leading and second terms of long time asymptotic expansion for the Heisenberg group $\mbox{Heis}_3(\mathbb{Z})$}\label{leadingHeisenberg}
In this chapter, we shall prove Theorem \ref{heisenberg-heat}.  By the Fourier inversion formula (\ref{Fourier1}), we compute the heat kernel $k_X(t,p,q)$ on the covering manifold $X$ as follows: 
\begin{equation} 
k_X(t,p,q) = \int_{\widehat{X}}{\rm tr}(k_{\rho_{{\rm \tiny{fin}},x}}(t,p,q))d\mu \label{Heatasymptoticsmain1}
\end{equation}
where $\widehat{X} = \widehat{\Gamma}_{\rm fin} = (\mathbb{Q}\cap [0,1]) \times [0,1] \times [0,1]$  with $\Gamma = \mbox{Heis}_3(\mathbb{Z})$ and 
\begin{equation}k_{\rho_{{\rm \tiny{fin}},x}}(t,p,q) =\sum_{\gamma \in \Gamma}\rho_{{\rm \tiny{fin}},x}(\gamma^{-1})k_X(t, p, \gamma q).\label{Heatasymptoticsmain2}
\end{equation}

By Theorem \ref{sunada}, the right-hand side of the above equation can be expressed as
\begin{equation}
\int_{[0,1]\times[0,1]\times(\mathbb{Q}\cap([0,\varepsilon)\cup(1-\varepsilon,1]))}{\rm tr}(k_{\rho_{{\rm \tiny{fin}},x}}(t,p,q))d\mu + O(e^{-c_\varepsilon t}) \label{reductionbysunada} 
\end{equation}

Then, by Theorem \ref{discretetoLie}, we can replace finite-dimensional representations $\rho_{{\rm \tiny{fin}},x}$ by (infinite-dimensional) Schr\"odinger representations $\rho_h$, and then, the first term of the righthand side of (\ref{reductionbysunada}) can be transformed into the following formal expression: for $p,q \in X$ and $s,u \in \mathbb{R}$,
\begin{align}
\int_{-\varepsilon}^{\varepsilon}k_{\rho_h}((p,s),(q,u)) &= \int_{-\varepsilon}^{\varepsilon}\sum_{k=0}^\infty\sum_{i=0}^\infty {\rm Tr}\Big(e^{-\lambda_{k,i}(h)t}\varphi_{k,i,h}(p)\varphi_{k,i,h}^\ast(q))\Big) \notag \\
&=  \int_{-\varepsilon}^{\varepsilon}\sum_{i=0}^\infty {\rm Tr}\Big(e^{-\lambda_{0,i}(h)t}\varphi_{0,i,h}(p)\varphi_{0,i,h}^\ast(q))\Big) + O(e^{-ct}).\label{HeatSch1}
\end{align}
Since 
\[ \varphi_{0,i,h}(p) = s_h(p)f_{0,i,h}(p) = \rho_h\left({}_R\!\!\int_{p_0}^p\omega \right)f_{0,i,h}(p)
\]
and the trace property  $\mbox{Tr}(AB) ={\rm Tr}(BA)$, we have the formal expression 
\begin{equation} \left(\sum_{i=0}^\infty e^{-\lambda_{0,i}(h)t}\varphi_{0,i,h}(p)\varphi_{0,i,h}^\ast(q)\right)  = E_h(t)S_h(p,q) = S_h(p,q)E_h(t)   \label{Sh}  
\end{equation}
where
\[ E_h(t) = \sum_{i=0}^\infty e^{-\lambda_{0,i}(h)t}f_{0,i,h}(p)f_{0,i,h}^\ast(q), \quad S_h(p,q) = \rho_h\left({}_R\!\!\int_q^p\omega\right) 
\]

Combining with the computations of each term of the asymptotic expansion (\ref{asym2}), (\ref{asym3}) in the previous chapter, we apply Lemma 4.4 in \cite{Kotani3}. Note that the variable $u$ in \cite{Kotani3} there corresponds to $\sqrt{h}$ here. Then, we integrate the leading term of the right-hand side of (\ref{Sh}) with respect to the Plancherel measure in (\ref{Liedecomp}). 
From (\ref{asym3}) and the Morse lemma, there exist a coordinate $u_i$ around $h=0$ such that 
\begin{align}
\lambda_{0,i}(h) &= \lambda_i(h) = \sum_{j=0}^\infty\lambda_i^{(j)}h^{j/2} = \mu_ih + \lambda_i^{(4)}h^2 + \cdots = \mu_iu_i^2, \notag \\
h^{1/2} :&= h_i(u_i) = u_i + k_i(u_i) \quad \mbox{for}\quad  i=0,1,\cdots  \label{eigenmorse}
\end{align} 
Here we note that $|k_i(u_i)| = O(u_i^3)$ since 
\begin{align*}
\lambda_i^{(2)}h + \lambda_i^{(4)}h^2 + \cdots &= \lambda_i^{(2)}(u_i + k_i(u_i))^2 + \lambda_i^{(4)}(u_i + k_i(u_i))^4 + \cdots \\
&= \lambda_i^{(2)}(u_i^2 + 2u_ik_i(u_i) + k_i^2(u_i)) + \lambda_i^{(4)}u_i^4 + O(u_i^5), 
\end{align*}
which implies the fourth order of terms satisfy the relation $2\lambda_i^{(2)}u_ik_i(u_i) + \lambda_i^{(4)}u_i^4 =0$ and then, $k_i(u_i) = \frac{\lambda_i^{(4)}}{2\lambda_i^{(2)}}u_i^3 + o(u_i^3)$ and 
\begin{equation}
h^{1/2} = u_i + k_i(u_i) =u_i\left(1+\frac{\lambda_i^{(4)}}{2\lambda_i^{(2)}}u_i^2 + o(u_i^2)\right) \label{fromhtoui}.
\end{equation}

Then, the first term of the right-hand side of (\ref{HeatSch1}) is equal to 
\begin{align}&{}\int_{-\varepsilon}^{\varepsilon}\int_{\mathbb{R}}\sum_{i=0}^\infty e^{-\lambda_{0,i}(h)t}\varphi_{0,i,h}(p)\varphi_{0,i,h}^\ast(q)ds|h|dh \notag \\ 
&= \int_{-\varepsilon}^{\varepsilon}\sum_{i=0}^\infty e^{-\lambda_{0,i}(h)t}f_{0,i,h}(p)f_{0,i,h}^\ast(q)\rho_h\left({}_R\!\!\int_q^p\omega\right)|h|dh 
\notag \\ 
&= \int_{-\varepsilon}^{\varepsilon}\int_{\mathbb{R}}\rho_h\left({}_R\!\!\int_q^p\omega\right)\sum_{i=0}^\infty e^{-\lambda_i^{(2)} u_i^2t}\left(\sum_{j=0}^\infty f_i^{(j)}(p,s)h^{j/2}\right)\left(\overline{\sum_{j=0}^\infty f_i^{(j)}(q,s)h^{j/2}}\right)ds|h|dh \label{asymheatheisenberg1} \\
 &= \sum_{i=0}^\infty\int_{-\varepsilon}^{\varepsilon}\int_{\mathbb{R}}e^{-\lambda_i^{(2)} u_i^2t}\overbrace{\rho_h\left({}_R\!\!\int_q^p\omega\right)}^{\textcircled{1}}\overbrace{\left(\psi_i(s) + f_i^{(2)}(p,s)h+\cdots\right)\left(\overline{\psi_i(s) + f_i^{(2)}(q,s)h+\cdots}\right)}^{\textcircled{2}}ds\overbrace{|h|dh}^{\textcircled{3}}. \notag
\end{align}

We further compute the rightmost of (\ref{asymheatheisenberg1}).
\begin{description}
\item[{\rm (a)}](Computation of part \textcircled{1})

Part \textcircled{1} is equal to
\begin{align}
\rho_h\left({}_R\!\!\int_q^p\omega\right) 
&= e^{2\pi\sqrt{-1}\left( h\int_q^p(\omega_{12} + \omega_2\omega_1) + \sqrt{h}\left(\int_q^p\omega_2\right)s\right)}e^{\sqrt{h}\left(\int_q^p\omega_1\right)\frac{d}{ds}}\notag \\
&= e^{2\pi\sqrt{-1}\left( (u_i+o(u_i^2))^2\int_q^p(\omega_{12} + \omega_2\omega_1) + (u_i+o(u_i^2))\left(\int_q^p\omega_2\right)s\right)}e^{(u_i+o(u_i^2))\left(\int_q^p\omega_1\right)\frac{d}{ds}}\notag \\
&= \bigg[1+2\pi\sqrt{-1}\left(\int_q^p\omega_2\right)su_i + \bigg\{-2\pi^2\left(\int_q^p\omega_2\right)^2s^2 \notag \\
&{}+2\pi\sqrt{-1}\Big(\int_q^p(\omega_{12} + \omega_2\omega_1)\big\}u_i^2+ o(u_i^2)\bigg]\notag \\
&{} \times \bigg\{1+\left(\int_q^p\omega_1\frac{d}{ds}\right)u_i + \frac12\Big(\int_q^p\omega_1\Big)^2\Big(\frac{d}{ds}\Big)^2 + o(u_i^2)\bigg\}\notag \\
&= 1 + \bigg\{2\pi\sqrt{-1}\left(\int_q^p\omega_2\right)s+\left(\int_q^p\omega_1\frac{d}{ds}\right)\bigg\}u_i \notag \\
&+ \bigg\{2\pi\sqrt{-1}\left(\int_q^p\omega_2\right)\left(\int_q^p\omega_1\right)s\frac{d}{ds}-2\pi^2\left(\int_q^p\omega_2\right)^2s^2 + \frac12\Big(\int_q^p\omega_1\Big)^2\Big(\frac{d}{ds}\Big)^2\notag \\
& +2\pi\sqrt{-1}\int_q^p(\omega_{12} + \omega_2\omega_1)\bigg\}u_i^2 + o(u_i^2).
\end{align}
This computation implies the following:
\begin{align}
\rho_h\left({}_R\!\!\int_q^p\omega\right)\psi_i 
&= \bigg[1 + \bigg\{2\pi\sqrt{-1}\left(\int_q^p\omega_2\right)s+\left(\int_q^p\omega_1\right)\frac{d}{ds}\bigg\}u_i \notag \\
&+ \bigg\{2\pi\sqrt{-1}\left(\int_q^p\omega_2\right)\left(\int_q^p\omega_1\right)s\frac{d}{ds}-\widetilde{\mathcal{H}}+2\pi\sqrt{-1}\int_q^p(\omega_{12} + \omega_2\omega_1)\bigg\}u_i^2 + o(u_i^2)\bigg]\psi_i. \label{rhoh}
\end{align}
where $\widetilde{\mathcal{H}}$ is another modified harmonic oscillator 
\begin{equation}\widetilde{\mathcal{H}} :=-\frac12\Big(\int_q^p\omega_1\Big)^2\Big(\frac{d}{ds}\Big)^2 +2\pi^2\left(\int_q^p\omega_2\right)^2s^2, \label{1HC}
\end{equation}
Similarly to previously defined modified harmonic oscillator $\mathcal{H}$ in (\ref{HC}), we can compute eigenvalues $\nu_i$ of $\widetilde{\mathcal{H}}$ as
\[ \nu_i = \pi\Big(\int_q^p\omega_1\Big)\Big(\int_q^p\omega_2\Big)\Big(i + \frac12\Big).
\]
\item[{\rm (b)}](Action of part \textcircled{1} to the main term of \textcircled{2})

We proceed further computation of $\displaystyle{\int_{\mathbb{R}}\rho_h\left({}_R\!\!\int_q^p\omega\right)\psi_i(s)\overline{\psi_j(s)}ds}$ up to the second order term of $u_i$.
Combinig with (\ref{saction}),(\ref{daction}) and (\ref{sds}), the right-hand side of (\ref{rhoh}) is equal to
\begin{align}
\rho_h\left({}_R\!\!\int_q^p\omega\right)\psi_i 
&= \psi_i + \bigg\{2\pi\sqrt{-1}\Big(\int_q^p\omega_2\sqrt{\frac{\|\omega_1\|_{L^2(M)}}{4\pi\|\omega_2\|_{L^2(M)}}}\Big)(\sqrt{i}\psi_{i-1}(s)+\sqrt{i+1}\psi_{i+1}(s))\notag \\ &{} +\int_q^p\omega_1\sqrt{\frac{\pi\|\omega_2\|_{L^2(M)}}{\|\omega_1\|_{L^2(M)}}}(\sqrt{i}\psi_{i-1}(s)-\sqrt{i+1}\psi_{i+1}(s))\bigg\}u_i \notag \\
&{} + \bigg\{2\pi\sqrt{-1}\left(\int_q^p\omega_2\right)\left(\int_q^p\omega_1\right)\frac12\left(\sqrt{i(i-1)}\psi_{i-2}(s)-\psi_i(s)+ \sqrt{(i+2)(i+1)}\psi_{i+2}(s)\right)\notag \\ 
&{} - \pi\Big(\int_q^p\omega_1\Big)\Big(\int_q^p\omega_2\Big)\Big(i + \frac12\Big)\psi_i +2\pi\sqrt{-1}\Big(\int_q^p(\omega_{12} + \omega_2\omega_1)\Big)\psi_i\bigg\}u_i^2  \label{rhopsi}
\end{align}

From $\int_{\mathbb{R}}\psi_i(s)\overline{\psi_j(s)}ds = \delta_{ij}$, we have
\begin{align}
&{}\int_{\mathbb{R}}\rho_h\left({}_R\!\!\int_q^p\omega\right)\psi_i(s)\overline{\psi_j(s)}ds \notag \\ &=  \delta_{ij}+\bigg\{\sqrt{i}\bigg(2\pi\sqrt{-1}\int_q^p\omega_2\sqrt{\frac{\|\omega_1\|_{L^2(M)}}{4\pi\|\omega_2\|_{L^2(M)}}}+\int_q^p\omega_1\sqrt{\frac{\pi\|\omega_2\|_{L^2(M)}}{\|\omega_1\|_{L^2(M)}}}\bigg)\delta_{{i-1},j}\notag \\ 
&{}+\sqrt{i+1}\bigg(2\pi\sqrt{-1}\int_q^p\omega_2\sqrt{\frac{\|\omega_1\|_{L^2(M)}}{4\pi\|\omega_2\|_{L^2(M)}}}-\int_q^p\omega_1\sqrt{\frac{\pi\|\omega_2\|_{L^2(M)}}{\|\omega_1\|_{L^2(M)}}}\bigg)\delta_{{i+1},j}\bigg\}u_i \notag \\ 
&{}+\bigg[-\pi\bigg(\int_q^p\omega_2\bigg)\bigg(\int_q^p\omega_1\bigg)\bigg(i + \frac12-\sqrt{-1}\bigg)+2\pi\sqrt{-1}\int_q^p(\omega_{12} + \omega_2\omega_1)\delta_{ij}\notag \\ 
&{}+  \pi\sqrt{-1}\bigg(\int_q^p\omega_2\bigg)\bigg(\int_q^p\omega_1\bigg)\bigg(\sqrt{i(i-1)}\delta_{{i-2},j}+ \sqrt{(i+2)(i+1)}\delta_{{i+2},j}\bigg)\bigg]u_i^2 \label{Sh1}
\end{align}

Thus, we have
\begin{align}
&{}\int_{\mathbb{R}}\rho_h\left({}_R\!\!\int_q^p\omega\right)\psi_i(s)\overline{\psi_i(s)}ds \notag \\ 
&= 1+\bigg[-\pi\Big(\int_q^p\omega_2\Big)\Big(\int_q^p\omega_1\Big)\Big(i + \frac12-\sqrt{-1}\Big)+2\pi\sqrt{-1}\int_q^p(\omega_{12} + \omega_2\omega_1) \Big)\bigg]u_i^2 \label{Sh11}
\end{align}
and 
\begin{align}
&{}\int_{\mathbb{R}}\rho_h\left({}_R\!\!\int_q^p\omega\right)\psi_i(s)\overline{\psi_{i-2}(s)}ds \notag \\ &= \pi\sqrt{-1}\bigg[\left(\int_q^p\omega_2\right)\left(\int_q^p\omega_1\right)\left(\sqrt{i(i-1)}\right)\bigg]u_i^2 
&{}\label{Sh12}
\end{align}
and 
\begin{align}
&{}\int_{\mathbb{R}}\rho_h\left({}_R\!\!\int_q^p\omega\right)\psi_i(s)\overline{\psi_{i+2}(s)}ds \notag \\ &= \pi\sqrt{-1}\bigg[\left(\int_q^p\omega_2\right)\left(\int_q^p\omega_1\right)\left(\sqrt{(i+2)(i+1)}\right)\bigg]u_i^2 
&{}\label{Sh13}
\end{align}

\item[{\rm (c)}](Action of part \textcircled{1} to the second term of \textcircled{2})

Combinig (\ref{fi2}) with (\ref{Sh11}), (\ref{Sh12})and (\ref{Sh13}), we have
\begin{align}
&{} \int_{\mathbb{R}}\rho_h\left({}_R\!\!\int_q^p\omega\right)f_i^{(2)}(p,s)\overline{\psi_i(s)}ds  \notag \\
&=\int_{\mathbb{R}}\bigg[ -\frac{4\pi}{\sqrt{{\rm vol}(M)}}G(|\omega_1||\omega_2|)(p)\left(i+\frac{1}{2}\right)\rho_h\left({}_R\!\!\int_q^p\omega\right)\psi_i\notag \\
&{}-\frac{4\pi\sqrt{-1}}{{\rm vol}(M)}G(\langle \omega_2,\omega_1 \rangle)(p)(\sqrt{(i+2)(i+1)}\rho_h\left({}_R\!\!\int_q^p\omega\right)\psi_{i-2}-\rho_h\left({}_R\!\!\int_q^p\omega\right)\psi_i+ \sqrt{i(i-1)}\rho_h\left({}_R\!\!\int_q^p\omega\right)\psi_{i+2})\bigg] \overline{\psi_i(s)}ds   \notag \\
& = \int_{\mathbb{R}}\bigg(-\frac{4\pi}{\sqrt{{\rm vol}(M)}}G(|\omega_1||\omega_2|)(p)\left(i+\frac{1}{2}\right)-\frac{4\pi\sqrt{-1}}{{\rm vol}(M)}G(\langle \omega_2,\omega_1 \rangle)(p)\bigg)\rho_h\left({}_R\!\!\int_q^p\omega\right)\psi_i\overline{\psi_i(s)}ds   \notag \\
&{}-\frac{4\pi\sqrt{-1}}{{\rm vol}(M)}G(\langle \omega_2,\omega_1 \rangle)(p)\int_{\mathbb{R}}\Big(\sqrt{(i+2)(i+1)}\rho_h\left({}_R\!\!\int_q^p\omega\right)\psi_{i-2}\overline{\psi_i(s)}ds + \int_{\mathbb{R}}\sqrt{i(i-1)}\rho_h\left({}_R\!\!\int_q^p\omega\right)\psi_{i+2}\Big)\overline{\psi_i(s)}ds    \notag \\ 
&= \bigg(-\frac{4\pi}{\sqrt{{\rm vol}(M)}}G(|\omega_1||\omega_2|)(p)\left(i+\frac{1}{2}\right)-\frac{4\pi\sqrt{-1}}{{\rm vol}(M)}G(\langle \omega_2,\omega_1 \rangle)(p)\bigg)\Big(1+\bigg[-\pi\Big(\int_q^p\omega_2\Big)\notag \\
&{} \quad \quad \Big(\int_q^p\omega_1\Big)\Big(i + \frac12-2\sqrt{-1}\Big)+\sqrt{-1}\int_q^p(\omega_{12} + \omega_2\omega_1) \Big)\bigg]u_i^2\bigg)   \notag \\
&{}-\frac{8\pi^2}{{\rm vol}(M)}G(\langle \omega_2,\omega_1 \rangle)(p)\Big(\sqrt{(i+2)(i+1)i(i-1)}\left(\int_q^p\omega_2\right)\left(\int_q^p\omega_1\right)u_i^2. \label{Sh2}  
\end{align}

Next, from (\ref{fi0}), (\ref{fi1}) and (\ref{fi2}), the part \textcircled{2} is equal to
\begin{align}
&{}{\rm Re}\bigg[\int_{\mathbb{R}}\bigg(\rho_h\left({}_R\!\!\int_q^p\omega\right)\left(\psi_i(s) + \frac12f_i^{(2)}(p,s)h_i(u_i) + \cdots \right)\bigg)\overline{\left(\psi_i(s)  + \frac12f_i^{(2)}(q,s)h_i(u_i)+ \cdots 
  \right)} ds\bigg] \notag \\
&={\rm Re}\bigg[ \int_{\mathbb{R}}\rho_h\left({}_R\!\!\int_q^p\omega\right)\psi_i(s)\overline{\psi_i(s)}  + \rho_h\left({}_R\!\!\int_q^p\omega\right)\Big(f_i^{(2)}(p,s)\overline{\psi_i(s)} + \psi_i(s)\overline{f_i^{(2)}(q,s)} \Big) u_i^2ds\bigg] \notag \\
&={\rm Re}\bigg[ \int_{\mathbb{R}}\rho_h\left({}_R\!\!\int_q^p\omega\right)\psi_i(s)\overline{\psi_i(s)}  + \Big(\rho_h\left({}_R\!\!\int_q^p\omega\right)f_i^{(2)}(p,s)\overline{\psi_i(s)} + \psi_i(s)\overline{\rho_h\left({}_R\!\!\int_p^q\omega\right)f_i^{(2)}(q,s)} \Big) u_i^2ds\bigg] \notag \\
&= 1-\pi\Big(i + \frac12\Big)\bigg[\Big(\int_q^p\omega_2\Big)\Big(\int_q^p\omega_1\Big) + \frac{4\pi}{\sqrt{{\rm vol}(M)}}\big(G(|\omega_1||\omega_2|)(p)-G(|\omega_1||\omega_2|)(q)\big)\bigg]u_i^2
  \label{ficompute}
\end{align}
Here we have used the fact that 
\[
\]

\item[{\rm (d)}](Computation of part \textcircled{3})

From (\ref{fromhtoui}), recall that
\[k_i(u_i) = \frac{\lambda_i^{(4)}}{2\lambda_i^{(2)}}u_i^3\] up to the third order of $u_i$.
Then, we compute up to the third order of $u_i$ as follows: 
\begin{align}
|h|dh &=  (u_i + k_i(u_i))^2d(u_i + k_i(u_i))^2 = (u_i^2 + 2u_ik_i(u_i)+o(u_i^5))(2u_i +2(k_i(u_i)+u_ik_i'(u_i) + o(u_i^3))du_i \notag \\
&= \big(2u_i^3 + (4u_i^2)k_i(u_i) +4u_i^2(k_i(u_i)+u_ik_i'(u_i))+o(u_i^5)\big)du_i\notag \\
&= \bigg(2+ \frac{10\lambda_i^{(4)}}{\lambda_i^{(2)}}u_i^2+o(u_i^2)\bigg)u_i^3du_i 
\label{Sfinal}
\end{align}

\item[{\rm (e)}](Final computation)

Combining (\ref{Sh1}), (\ref{ficompute}) and (\ref{Sfinal}), we compute the left-hand side of (\ref{asymheatheisenberg1}) as follows:

\begin{align}
&{}\int_{-\varepsilon}^{\varepsilon}\Big(\int_{\mathbb{R}}\sum_{i=0}^\infty e^{-\lambda_{0,i}(h)t}\varphi_{0,i,h}(p)\varphi_{0,i,h}^\ast(q)ds\Big)|h|dh = \sum_{i=0}^\infty \int_{-\varepsilon}^{\varepsilon}\Big(\int_{\mathbb{R}}e^{-\lambda_i^{(2)} u^2t}(1+K_iu_i^2+o(u_i^2))ds\Big)u_i^3du_i \notag \\
&{} = \sum_{i=0}^\infty \int_{-\varepsilon\sqrt{\lambda_i^{(2)}t}}^{\varepsilon\sqrt{\lambda_i^{(2)}t}}\bigg[\int_{\mathbb{R}}e^{-w^2}
\Big(1+K_i\Big(\frac{w}{\sqrt{\lambda_i^{(2)}t}}\Big)^2+o\Big(\Big(\frac{w}{\sqrt{\lambda_i^{(2)}t}}\Big)^2\Big)\Big)\Big(\frac{w}{\sqrt{\lambda_i^{(2)}t}}\Big)^3ds\bigg]d\Big(\frac{w}{\sqrt{\lambda_i^{(2)}t}}\Big) \label{asymptoheisen1} 
\end{align}
where we take change of variable as $w =\sqrt{\mu_it}u_i$ and $K$ in the above,  is expressed as follows: 

\begin{equation}
K_i = -\pi\Big(i + \frac12\Big)\bigg\{\Big(\int_q^p\omega_2\Big)\Big(\int_q^p\omega_1\Big) + \frac{4\pi}{\sqrt{{\rm vol}(M)}}\big(G(|\omega_1||\omega_2|)(p)-G(|\omega_1||\omega_2|)(q)\big)\bigg\}
+ \frac{10\lambda_i^{(4)}}{\lambda_i^{(2)}}
 \label{heatasymp2nd}
\end{equation}

Next, recalling (\ref{harmeigen}
\[  \lambda_i^{(2)} = \frac{2\pi{\rm vol}(\widehat{H})}{{\rm vol}(M)}\Big(i+\frac{1}{2}\Big) = \frac{\pi{\rm vol}(\widehat{H})}{{\rm vol}(M)}(2i+1),
\]
we compute the following:
\begin{align}
\sum_{i=0}^\infty \frac{1}{(\lambda_i^{(2)})^s} &= \Big(\frac{{\rm vol}(M)}{\pi{\rm vol}(\widehat{H})}\Big)^s \sum_{i=0}^\infty \frac{1}{(2i+1)^s} = \Big(\frac{{\rm vol}(M)}{\pi{\rm vol}(\widehat{H})}\Big)^s \Big(\sum_{i=1}^\infty \frac{1}{i^s}-\sum_{i=1}^\infty \frac{1}{(2i)^s}\Big) \notag \\ &= \Big(\frac{{\rm vol}(M)}{\pi{\rm vol}(\widehat{H})}\Big)^s\Big(\frac{2^s-1}{2^s}\zeta(s)\Big)\label{heisenspectralzeta} 
\end{align}
where $\zeta(s)$ is the Riemann zeta function. From the solution $\zeta(2) = \frac{\pi^2}{6}$ of the Basel problem by Euler, we have
\begin{equation}
\sum_{i=0}^\infty \frac{1}{(\lambda_i^{(2)})^2} =  \Big(\frac{{\rm vol}(M)}{\pi{\rm vol}(\widehat{H})}\Big)^2\Big(\frac{3}{4}\zeta(2)\Big) = \frac{\pi^2}{8}\Big(\frac{{\rm vol}(M)}{\pi{\rm vol}(\widehat{H})}\Big)^2 = \frac{({\rm vol}(M))^2}{8({\rm vol}(\widehat{H}))^2}
\label{leadingheatcoefficient}
\end{equation}

Finally, combinig altogether such as (\ref{Heatasymptoticsmain1}), (\ref{Heatasymptoticsmain2}), (\ref{reductionbysunada}) with the equality $\int_{0}^\infty x^3e^{-x^2}dx = \frac12\int_{0}^\infty te^{-t}dx = \Gamma(1)/2 = 1$, the above computations (\ref{heatasymp2nd}) and (\ref{heisenspectralzeta}) imply the following asymptotic expansion: 
\begin{theorem}\label{leadingsecond}
\begin{align}
k_X(t,p,q) &\sim \sum_{i=0}^\infty \int_0^{\varepsilon\sqrt{\lambda_i^{(2)}t}}e^{-w^2}
\Big(1+K_i\Big(\frac{w}{\sqrt{\lambda_i^{(2)}t}}\Big)^2+o\Big(\Big(\frac{w}{\sqrt{\lambda_i^{(2)}t}}\Big)^2\Big)\Big)\Big(\frac{w}{\sqrt{\lambda_i^{(2)}t}}\Big)^3d\Big(\frac{w}{\sqrt{\lambda_i^{(2)}t}}\Big) \notag \\
&{} = \frac{({\rm vol}(M))^2}{8({\rm vol}(\widehat{H}))^2}\Big(1+\frac{c_1}{t}+\cdots \Big)  \label{heatheisenberg}
\end{align}
where   
\begin{align}
c_1 &= \frac{8({\rm vol}(\widehat{H}))^2}{({\rm vol}(M))^2}\sum_{i=0}^\infty\frac{K_i}{(\lambda_i^{(2)})^3} \notag \\ 
&= \frac{8({\rm vol}(\widehat{H}))^2}{({\rm vol}(M))^2}\sum_{i=0}^\infty\frac{1}{(\lambda_i^{(2)})^3}\bigg[(-\pi\Big(i + \frac12\Big)\notag \\
& \times \bigg\{\Big(\int_q^p\omega_2\Big)\Big(\int_q^p\omega_1\Big) + \frac{4\pi}{\sqrt{{\rm vol}(M)}}\big(G(|\omega_1||\omega_2|)(p)-G(|\omega_1||\omega_2|)(q)\big)\bigg\}
+ \frac{10\lambda_i^{(4)}}{\lambda_i^{(2)}}\bigg] \notag \\ 
&= \frac{8({\rm vol}(\widehat{H}))^2}{({\rm vol}(M))^2}\sum_{i=0}^\infty\bigg[(-\frac{\pi}{2}(2i + 1)^{-2}\frac{({\rm vol}(M))^3}{\pi^3({\rm vol}(\widehat{H}))^3} \notag \\ 
& \times \bigg\{\Big(\int_q^p\omega_2\Big)\Big(\int_q^p\omega_1\Big)+ \frac{4\pi}{\sqrt{{\rm vol}(M)}}\big(G(|\omega_1||\omega_2|)(p)-G(|\omega_1||\omega_2|)(q)\big)\bigg\}\notag \\
&+ \int_M\frac{({\rm vol}(M))^{7/2}}{\pi^2({\rm vol}(\widehat{H}))^4}|\omega_1||\omega_2|G(|\omega_1||\omega_2|)(x)(2i+1)^{-2} 
\notag \\
&{} + \frac{(2{\rm vol}(M))^2}{\pi^2({\rm vol}(\widehat{H}))^3}G(|\omega_1||\omega_2|)(x)(2i+1)^{-2}+ \frac{4({\rm vol}(M))^3}{\pi^2({\rm vol}(\widehat{H}))^4}|\omega_{12}(x)|^2(2i+1)^{-4}dx \bigg]\notag \\ 
&= -\frac{1}{2}\frac{{\rm vol}(M)}{{\rm vol}(\widehat{H})}\bigg\{\Big(\int_q^p\omega_2\Big)\Big(\int_q^p\omega_1\Big) + \frac{4\pi}{\sqrt{{\rm vol}(M)}}\big(G(|\omega_1||\omega_2|)(p)-G(|\omega_1||\omega_2|)(q)\big)\bigg\}\notag \\
&+ \int_M\frac{({\rm vol}(M))^{3/2}}{({\rm vol}(\widehat{H}))^2}|\omega_1||\omega_2|G(|\omega_1||\omega_2|)(x) + \frac{2}{{\rm vol}(\widehat{H})}G(|\omega_1||\omega_2|)(x)
+ \frac{\pi^2{\rm vol}(M)}{3({\rm vol}(\widehat{H}))^2}|\omega_{12}(x)|^2dx \label{heatasymp2nd}
\end{align}
\end{theorem}
\end{description}

\subsection{An estimate of error term} To show the asymptotic expansion (\ref{heisenberg-heat}), we need to estimate the error term. To do that we give here rough estimates of $\lambda_i^{(n)}$ and $a_{ik}^{(n)}(x)$ when $n,k \to \infty$. 
The dependence of the variable $x \in M$ is coming from the $x$-derivatives of iterated integrals of $\omega_1, \omega_2, \omega_{12}$ which are uniform on $n, k$,  thus we can ignore this dependence concerning $n,k \to \infty$. 
As explained in the last paragraph of Step 1 in Section \ref{Higher}, the action of $L^{(j)}$, $j = 0,1,2,3,4$ on $L^2(\mathbb{R}$) is expressed as a quadratic polynomial of $s, p$ and hence of $a a^\dag$, thus, by the formula (\ref{aadagger}), the ratio of $a_{ik}^{(n)}(x)$ and $a_{ik}^{(n+1)}(x)$ can be estimated roughly $O(k^2)$ and thus  $a_{ik}^{(n)}(x)$ is comparable to $O(k^{2n})$. A similar estimate also holds for $\lambda_i^{(n)}$.

To compute the error terms of (\ref{heisenberg-heat}), we need to care about the effect of $S_h(p,q)$ in (\ref{Sh}), which is coming from the Taylor expansion with respect to $\sqrt{h}$ of the following: 
\begin{equation} \pi_l\left({}_R\!\!\int_{p_0}^p \omega \right)f = e^{2\pi\sqrt{-1}\left( h\int_{p_0}^p(\omega_{12} + \omega_2\omega_1) + \sqrt{h}\left(\int_{p_0}^p\omega_2\right)s\right)}e^{\sqrt{h}\left(\int_{p_0}^p\omega_1\right)\frac{d}{ds}}f 
\end{equation}
which is also comparable to $O(k^{2n})$ for the coefficients of $\psi_k$.
Therefore the contribution to the coefficient of $\psi_k$ in the $n$-order terms with respect to $\sqrt{h}$ of the asymptotic expansion of $k_{\pi_l}(t,x,y)$ is estimated as follows: By $\mu_k =  O(k)$, we have
\[ e^{-\mu_kht}O((\sqrt{h}k^2t)^n) =  O(e^{-kht}(\sqrt{h}k^2t)^n),
\]  
which is sufficiently tiny uniformly in $k$ for fixed $n$, provided $t$ is sufficiently large, which gives an estimate of the asymptotic expansion error terms (\ref{heisenberg-heat}). Combining with the formal computation of the previous section \ref{Higher}, Theorem \ref{discretetoLie} and the Fourier inversion formula (\ref{Fourier1}), we have proved Theorem \ref{heisenberg-heat}.

\subsection{Comparison with explicit formula of the heat kernel on $\mbox{Heis}_3(\mathbb{R})$}\label{compariontoexplicitformula}

The heat kernel $k_H(t, p, q)$ of tha Laplacian of $\mbox{Heis}_3(\mathbb{R})$
can be expressed as the following integral form: For $p = (x,y,z)$ and $q = (\tilde{x},\tilde{y},\tilde{z})$ (cf. (10.2.40) in \cite{Calin}),
\[k_H(t, p, q) = \frac{1}{(2\pi t)^2}\int_{-\infty}^\infty e^{\frac{\sqrt{-1}\tau\left(z-\tilde{z}+\frac{\tilde{y}x-\tilde{x}y}{2}\right)+\frac{\tau \cosh \tau}{2\sinh \tau}\left((x-\tilde{x})^2+(y-\tilde{y})^2\right)-\frac{\tau^2}{2}}{t}}\frac{\tau}{\sinh \tau}d\tau
\]

In this formula, we put $p=q$ and let $t\to\infty$, then the integral of the right-hand side becomes
\[\int_{-\infty}^\infty \frac{\tau}{\sinh \tau}d\tau = \frac{\pi^2}{2}
\]
Thus, $k_H(t, p, q)$ satisfies
\begin{equation} k_H(t, p, q) \sim \frac{1}{(2\pi t)^2}\frac{\pi^2}{2} = \frac{1}{8t^2} \label{explicitcomputation} 
\end{equation}

On the other hand, in the case when $X = \mbox{Heis}_3(\mathbb{R})$ and $M = \mbox{Heis}_3(\mathbb{R})/\mbox{Heis}_3(\mathbb{Z})$, by (\ref{formonHeisenberg}),
we have 
\begin{equation}
\|\omega_1\|_{L^2(M)} = \Big\|\int_0^1 dx\Big\| = 1, \quad \|\omega_2\|_{L^2(M)} = \Big\|\int_0^1 dy\Big\| = 1. \label{omegaequalone}
\end{equation}

Then, the leading coefficient of the asymptotic expansion is
\[  \zeta_\mathcal{H}(2)\int_{-\infty}^\infty e^{-\alpha^2}|\alpha|d\alpha = \sum_{n=0}^\infty \frac{1}{\left(2\pi(n+1/2)\right)^2}\int_0^\infty e^{-r}dr =\frac{1}{\pi^2}\sum_{n=0}^\infty\frac{1}{(2n+1)^2} = \frac18 
\] 
where \[
\mathcal{H} =-\|\omega_1\|_{L^2(M)}^2\left(\frac{d}{ds}\right)^2 +4\pi^2\|\omega_2\|_{L^2(M)}^2s^2 \]
in (\ref{HC}) with $\|\omega_1\|_{L^2(M)} = \|\omega_2\|_{L^2(M)} =1$ as (\ref{omegaequalone}) and ${\rm vol}(M)=1$, ${\rm vol}(\widehat{H})=1$: This computation coincides with the above (\ref{explicitcomputation}).

\subsection{A local central limit theorem for heat kernels}\label{heatclt}
The results in this section are obtained by combining the previous arguments with the method used in \cite{Kotani3}, \cite{Kotani4}. 

Here, we return to the right-hand side of the equality (\ref{asymheatheisenberg1}).
We continue to compute the part \textcircled{1} there, simplifying to the second-order approximations as follows: 
\begin{align}
\rho_h\left({}_R\!\!\int_q^p\omega\right) 
&= e^{2\pi\sqrt{-1}\left( h\int_q^p(\omega_{12} + \omega_2\omega_1) + \sqrt{h}\left(\int_q^p\omega_2\right)s\right)}e^{\sqrt{h}\left(\int_q^p\omega_1\right)\frac{d}{ds}}\notag \\
&=e^{2\pi\sqrt{-1}\left(\int_q^p\omega_2 s\right)(u_i+o(u_i^2))+\int_q^p(\omega_{12} + \omega_2\omega_1)(u_i+o(u_i^2))^2}e^{\left(\int_q^p\omega_1\frac{d}{ds}\right)(u_i+o(u_i^2))}\notag \\
&= e^{2\pi\sqrt{-1}\left(\int_q^p\omega_2 s + \int_q^p\omega_1\frac{d}{ds}\right)u_i+\Big(\int_q^p(\omega_{12} + \omega_2\omega_1 + \frac12\int_q^p\omega_2 \int_q^p\omega_1\Big)u_i^2+o(u_i^2)}\label{simplerhoh}
\end{align}
where we have used the Campbell-Hausdorff formula in the third equality.

Since we only compute the leading term of the asymptotics expansion related to the local central limit theorem, the higher order terms (i.e., greater than 0-th order terms) of the parts \textcircled{2} and \textcircled{3} are irrelevant in computation, the equality (\ref{asymheatheisenberg1}) simplifies to the following one:

\begin{align}
&{}\int_{-\varepsilon}^{\varepsilon}\Big(\int_{\mathbb{R}}\sum_{i=0}^\infty e^{-\lambda_{0,i}(h)t}\varphi_{0,i,h}(p)\varphi_{0,i,h}^\ast(q)ds\Big)|h|dh 
\\ 
&= \int_{-\varepsilon}^{\varepsilon}\Big(\sum_{i=0}^\infty e^{-\lambda_{0,i}(h)t}\rho_h\Big({}_R\!\!\int_q^p\omega\Big)f_{0,i,h}(p)f_{0,i,h}^\ast(q)ds\Big)|h|dh \label{lctrefgeodesic} \\ 
&= \sum_{i=0}^\infty\int_{-\varepsilon}^{\varepsilon}\Big[\int_{\mathbb{R}}e^{-\lambda_i^{(2)} u_i^2t}e^{2\pi\sqrt{-1}\left(\int_q^p\omega_2 s + \int_q^p\omega_1\frac{d}{ds}\right)u_i+\Big(\int_q^p(\omega_{12} + \omega_2\omega_1 + \frac12\int_q^p\omega_2 \int_q^p\omega_1\Big)u_i^2+o(u_i^2)}\psi_i^2(s)ds\Big]u_i^3du_i \notag
\end{align}

By changing variable $u_i$ to $w = \sqrt{t}u_i$, the leading term of the right-hand side is equal to 
\begin{align}
&{}{\scriptsize \sum_{i=0}^\infty\int_{-\varepsilon\sqrt{t}}^{\varepsilon\sqrt{t}}\Big[\int_{\mathbb{R}}e^{-\lambda_i^{(2)}w^2}e^{2\pi\sqrt{-1}(\int_q^p\omega_2 s + \int_q^p\omega_1\frac{d}{ds})\frac{w}{\sqrt{t}}+(\int_q^p(\omega_{12} + \omega_2\omega_1) + \frac12\int_q^p\omega_2 \int_q^p\omega_1)(\frac{w}{\sqrt{t}})^2}\psi_i^2(s)ds\Big] \left(\frac{w}{\sqrt{t}}\right)^3 d\left(\frac{w}{\sqrt{t}}\right)} \notag \\
&= \sum_{i=0}^\infty\int_{-\varepsilon\sqrt{t}}^{\varepsilon\sqrt{t}}e^{-\lambda_i^{(2)}w^2}e^{2\pi\sqrt{-1}\Big(\int_q^p\omega_2 s + \int_q^p\omega_1\frac{d}{ds}\Big)\frac{w}{\sqrt{t}}+\Big(\int_q^p(\omega_{12} + \omega_2\omega_1) + \frac12\int_q^p\omega_2 \int_q^p\omega_1\Big)\Big(\frac{w}{\sqrt{t}}\Big)^2}\Big(\frac{w}{\sqrt{t}}\Big)^3 d\Big(\frac{w}{\sqrt{t}}\Big) \label{locallimit1}
\end{align}
Recall that $\lambda_i^{(2)} = \mu_i/{\rm vol}(M)$ in (\ref{harmeigen}) and $\mu_i$ is $i$-th eigenvalue of the following operator 
\begin{equation}
\mathcal{H} :=-\|\omega_1\|_{L^2(M)}^2\left(\frac{d}{ds}\right)^2 +4\pi^2\|\omega_2\|_{L^2(M)}^2s^2. \label{2HC}
\end{equation}

Then, we have
\begin{align}
&{}-\frac{\mathcal{H}}{{\rm vol}(M)}\left(\frac{w}{\sqrt{t}}\right)^2 + 2\pi\sqrt{-1}\left(\left(\int_q^p\omega_2\right) s + \left(\int_q^p\omega_1\right)\frac{d}{ds}\right)\frac{w}{\sqrt{t}} \notag \\
& ={\scriptsize \frac{\|\omega_1\|_{L^2(M)}^2}{{\rm vol}(M)}\left(\frac{d}{ds}\right)^2\left(\frac{w}{\sqrt{t}}\right)^2  + 2\pi\sqrt{-1}\left(\int_q^p\omega_1\right)\left(\frac{d}{ds}\right)\frac{w}{\sqrt{t}}-\frac{4\pi^2\|\omega_2\|_{L^2(M)}^2}{{\rm vol}(M)}s^2w^2 + 2\pi\sqrt{-1}\left(\int_q^p\omega_2\right) s\frac{w}{\sqrt{t}}}\notag \\
& = \left(\left(\|\omega_1\|_{L^2(M)}/V \right)\frac{d}{ds}w + \frac{\pi\sqrt{-1}V}{\|\omega_1\|_{L^2(M)}}\left(\int_q^p\omega_1\right)\frac{1}{\sqrt{t}}\right)^2 + \frac{\pi^2V^2\left(\int_q^p\omega_1\right)^2}{t} \notag \\
& -\left(\left(2\pi\|\omega_2\|_{L^2(M)}/V\right)sw + \frac{V\sqrt{-1}}{2\|\omega_2\|_{L^2(M)}}\left(\int_q^p\omega_2\right)\frac{1}{\sqrt{t}}\right)^2- \frac{\left(\int_q^p\omega_2\right)^2V^2}{t}
\end{align}
where we denote $V = \sqrt{{\rm vol}(M)}$.
This implies that the right-hand side of (\ref{locallimit1}) is equal to

\begin{align}
&{} \sum_{i=0}^\infty\int_{-\varepsilon\sqrt{t}}^{\varepsilon\sqrt{t}}e^{\bigg[\left(\left(\|\omega_1\|_{L^2(M)}/V \right)\frac{d}{ds}w + \frac{\pi\sqrt{-1}V}{\|\omega_1\|_{L^2(M)}}\left(\int_q^p\omega_1\right)\frac{1}{\sqrt{t}}\right)^2-\left(\left(2\pi\|\omega_2\|_{L^2(M)}/V\right)sw + \frac{\sqrt{-1}V}{2\|\omega_2\|_{L^2(M)}}\left(\int_q^p\omega_2\right)\frac{1}{\sqrt{t}}\right)^2\bigg]} \notag \\ 
&\times e^{2\pi\sqrt{-1}\left(\int_q^p(\omega_{12} + \omega_2\omega_1) + \frac12\int_q^p\omega_2 \int_q^p\omega_1\right)\left(\frac{w}{\sqrt{t}}\right)^2}e^{- \frac{\left(V\int_q^p\omega_2\right)^2+\left(\pi V\int_q^p\omega_1\right)^2}{t}}\left(\frac{w}{\sqrt{t}}\right)^3 d\left(\frac{w}{\sqrt{t}}\right)ds \notag \\
&\sim \sum_{i=0}^\infty\frac{1}{\left(\lambda_i^{(2)}+\frac{2\pi\sqrt{-1}}{t}\Big(\int_q^p(\omega_{12} + \omega_2\omega_1) + \frac12\int_q^p\omega_2 \int_q^p\omega_1\Big)\right)^2}e^{- \frac{\left(V\int_q^p\omega_2\right)^2+\left(\pi V\int_q^p\omega_1\right)^2}{t}}
\end{align}

In conclusion, we have the following local central limit theorem:
\begin{theorem}\label{HeatCLT}
\begin{equation}
\lim_{t\to\infty}\left|e^{\frac{\left(\int_q^p\omega_2\right)^2+\left(\pi^2\int_q^p\omega_1\right)^2}{t}}k_X(t,p,q) -\sum_{i=0}^\infty\frac{{\rm vol}(M)}{\left(2\pi{\rm vol}(\widehat{H})\left(2i+1\right)\right)^2}\right| = 0 \label{heatclt}
\end{equation}
\end{theorem} 
Note that 
\[   \sum_{i=0}^\infty\frac{{\rm vol}(M)}{\left(2\pi{\rm vol}(\widehat{H})\left(2i+1\right)\right)^2} = \frac{{\rm vol}(M)}{8({\rm vol}(\widehat{H}))^2} = \frac{{\rm vol}(M)}{8}
\] by the computation in the previous subsections. 

Finally, one can obtain more detailed asymptotic results by similar arguments in the previous subsections.

\subsection{Expression of the leading term and an algorithm for computing higher order terms for general torsion-free  nilpotent group $\Gamma$}\label{leadingnilpotent}
In this chapter, we shall prove Theorem \ref{conj-heat}. Recall the Fourier inversion formula (\ref{nilpotentPytlik1}) ($=$ (\ref{nilpotentPytlik2})) and change the notation from $\pi$ there, to $\pi_{\rm \tiny{fin}}$ here. We shall compute the heat kernel $k_X(t,p,q)$ on the covering manifold $X$. Several parts are essentially the same as the previous subsections for the discrete Heisenberg group $\mbox{Heis}_3(\mathbb{Z})$.

 For $\sigma \in \Gamma$ and 
\begin{equation*} 
k_X(t,p,\sigma q) = \int_{\widehat{X}}{\rm tr}(\pi_{\rm \tiny{fin}}(\sigma^{-1})k_{\pi_{\rm \tiny{fin}}}(t,p,q))d\mu
\end{equation*}
where 
\[k_{\pi_{\rm \tiny{fin}}}(t,p,q) =\sum_{\gamma \in \Gamma}\pi_{\rm \tiny{fin}}(\gamma)k_X(t, p, \gamma q).
\]

By Sunada's theorem (\ref{sunada}), the right-hand side of the above equation can be expressed as
\begin{equation*}
\int_ W{\rm tr}(\pi_{\rm \tiny{fin}}(\sigma^{-1} )k_{\pi_{\rm \tiny{fin}}}(t,p,q))d\mu + O(e^{-c t}), 
\end{equation*}
where $W$ is a vicinity of trivial representation $\textbf{1}$.


Then, by Theorem  \ref{introdecompose} ($=$ Theorem \ref{sec4decompose}), we can replace finite dimentional representation $\pi_{\rm \tiny{fin}}$ of $\Gamma$ by infinite dimensional representation $\pi_l$ of $G$ associated to $l \in \mathfrak{g}^\ast$ and replace $k_{\pi_{\rm \tiny{fin}}}(t,p,q))$ the following formal expression:
\begin{align}
k_{\pi_l}(p,q)(s,u)) &= \sum_{k=0}^\infty\sum_{i=0}^\infty e^{-\lambda_{k,i}(l)t}\varphi_{k,i,l}(p)\varphi_{k,i,l}^\ast(q)) \notag \\
&=  \sum_{i=0}^\infty e^{-\lambda_{0,i}(l)t}\varphi_{0,i,l}(p)\varphi_{0,i,l}^\ast(q)) + O(e^{-ct}).\label{HeatSch}
\end{align}
where $\lambda_{k,i}(l)$ and $\varphi_{k,i,l}(p)$ are eigenvalues and normalized eigenfunctions of the twisted Laplacian $\Delta_{\pi_l}$ acting on the space of sections on the flat vector bundle $E_{\pi_l}$ associated to $\pi_l$ given in Section \ref{spectralhypoelliptic}.

Since 
\[ \varphi_{0,i,l}(p) = s_{\pi_l}(p)f_{0,i,l}(p) = \pi_l\left({}_R\!\!\int_{p_0}^p\omega \right)f_{0,i,l}(p)
\]
and the trace property  $\mbox{Tr}(AB) ={\rm Tr}(BA)$, we obtain, similarly as Section \ref{leadingHeisenberg}, the following expression:  
\begin{equation} {\rm Tr}\left(\pi_l(\sigma^{-1}  )\sum_{i=0}^\infty e^{-\lambda_{0,i}(l)t}\varphi_{0,i,l}(p)\varphi_{0,i,l}^\ast(q)\right)  = {\rm Tr}(\pi_l(\sigma^{-1}) E_{\pi_l}(t)S_l(p,q))   \label{Sh}  
\end{equation}
where
\[ E_{\pi_l}(t) = \sum_{i=0}^\infty e^{-\lambda_{0,i}(l)t}f_{0,i,l}(p)f_{0,i,l}^\ast(q), \quad S_l(p,q) = \pi_l\left({}_R\!\!\int_q^p\omega\right) 
\]

Combining with the methods of computations of each term of the asymptotic expansion in the previous chapter, we apply Lemma 4.4 in \cite{Kotani3}.

Note that the variable $u$ in \cite{Kotani3} there corresponds to $\sqrt{|l|}$ here. ($\sqrt{h}$ in the case of $\mbox{Heis}_3(\mathbb{Z})$). Then, in considering Remark \ref{PlancherelLie} (2), we integrate the main term of the right-hand side of (\ref{HeatSch}) with respect to the Plancherel measure $|\mbox{Pf}(l)|dl$ in (\ref{FourierinversionLie}), which we recall the definition:
Let $\{X_1,\ldots,X_n\}$ be a strong Malcev basis for a nilpotent Lie algebra $\mathfrak{g}$ and 
let $\{l_1,\ldots, l_n\}$ be the dual basis for $\mathfrak{g}^\ast$. 
Take a set $U$ of generic coadjoint orbits, index sets 
\[ S = \{i_1 < \cdots < i_{2k}\} \] 
and $T$ which will appear in Theorem \ref{corwin319}, and (absolute value of) Pfaffian $|\mbox{Pf}(l)|$ satisfies
\[ |\mbox{Pf}(l)|^2 = \det \;B, \quad B=(B_l(X_i,X_j)) := (l([X_i,X_j])). \] 

If we replace $l$ to $\hbar l$ with the parameter $\hbar >0$, then
the Plancherel measure $|\mbox{Pf}(\hbar l)|dl$ satisfies the following property:
\[   \delta_\hbar^\ast(|\mbox{Pf}(l)|dl) = |\mbox{Pf}(\hbar l)|dl= \hbar^d|\mbox{Pf}(l)|dl
\]
where $\delta_\hbar:\mathfrak{g}\to\mathfrak{g}$ is the dilatation of $G_0$ which is a stratification of $G$ and $d$ is the polynomial growth order of $G$ which is equal to that of $G_0$.

Next, we have
\begin{align*}&{}\int_{U\cap V_T}{\rm Tr}\left(\sum_{i=0}^\infty e^{-\lambda_{0,i}(\hbar l)t}\varphi_{0,i,\hbar l}(p)\varphi_{0,i,\hbar l}^\ast(q)\right)|\mbox{Pf}(\hbar l)|d(\hbar l) \\ 
&\sim \sum_{i=0}^\infty\sum_{j=2}^\infty\int_{W_\eta}e^{-\lambda_i^{(2)} \hbar^2t}p_{ij}(\sqrt{|\hbar l|})|\mbox{Pf}(\hbar l)|d(\hbar l) \\ 
&= \sum_{j=2}^\infty\sum_{i=0}^\infty \frac{1}{(\mu_it)^{d/2}}\int_{W_{\mu_it\eta}}e^{-k^2}p_{ij}(k/\sqrt{\mu_it})|k|^{d-1}dk \quad (k= \sqrt{\mu_it}\hbar)  \\
&\sim \frac{C\zeta_{\mathcal{H}}(d/2)}{t^{d/2}}\left(1 + \frac{c_1}{t} + \ldots \right). 
\end{align*}
Here we have used that the asymptotic expansion of the eigenvalue $\lambda_{0,i}(\pi_{\hbar l})$ of $\Delta_{\pi_{\hbar l}}$ is in the following form:
\begin{equation*}
\lambda_{0,i}(\pi_{\hbar l}) := \lambda_{0,i}(\hbar l) = \sum_{j=0}^\infty \lambda_{0,i}^{(j)}\hbar^j = \lambda_i^{(2)} \hbar^2 + \cdots   \label{nilpotentlambdaexpansion} 
\end{equation*}
where $C= \int_{-\infty}^\infty  e^{-k^2}|k|^{d-1}dk = 2\Gamma(d/2)$ and $p_{ij}(t)$ is a homogeneous polynomial of degree $j-2$ with $p_{i2} =1$ which is determined by the asymptotic expansions obtained in the previous subsections combining with those of $\pi_{\hbar l}(\sigma^{-1})$ and $S_{\pi_{\hbar l}}(p,q)$.

\section{Asymptotics of prime closed geodesics of compact Riemann surface for nilpotent extensions}\label{Asymptoticsclosedgeodesics}

In this chapter, we shall prove Theorem \ref{theorem-geod}. The proof combines some modified arguments given in \cite{Phillips} in the case of abelian extensions and spectral estimates in previous sections.
\subsection{Asymptotics formulas for prime closed geodesics on compact Riemann surface for nilpotent extensions}\label{Asymptoticsclosedgeodesics}
Let $M$ be a compact Riemann surface of genus $g$ with constant negative curvature $-1$ and $\Gamma$ be a torsion-free discrete nilpotent group. Take a surjective homomorphism $\Phi \colon \pi_1(M) \to \Gamma$. 

Here we use Theorem \ref{introPytlik} and change the notation from unitary finite-dimensional representations $\pi$ of $\Gamma$ to $\pi_{\rm \tiny{fin}}$ here.
We consider the twisted Laplacian $\Delta_{\pi_{\rm \tiny{fin}}}$ acting on the space $L^2(E_{\pi_{\rm \tiny{fin}}})$ of sections of a  flat vector bundle $E_{\pi_{\rm \tiny{fin}}}$ associated to $\pi_{\rm \tiny{fin}}$. 
We recall the Selberg trace formula for $\Delta_{\pi_{\rm \tiny{fin}}}$.  

\begin{align}
\sum_{j=0}^\infty \hat{h}(r_j(\pi_{\rm \tiny{fin}})) &= 2(g-1)\int_{-\infty}^\infty r \tanh (\pi r)\hat {h}(r)dr \\ &{} + \sum_{\gamma:\mbox{\footnotesize{primitive}}}\sum_{k=1}^\infty \frac{{\rm tr}(\pi_{\rm \tiny{fin}}(\Phi(\gamma^k))\label{Selbergtrace})\ell(\gamma )}{\sinh (k\ell(\gamma)/2)}h(k\ell(\gamma )),
\end{align}
where 
${\rm tr}$ is the normalized trace $\frac{1}{{\rm \tiny{dim}}(\pi_{\rm \tiny{fin}})}{\rm Tr}$ with usual matrix trace ${\rm Tr}$ as before, 
$h$ is an even $C^\infty$ function on $\mathbb{R}$ with compact support and $r_j(\pi_{\rm \tiny{fin}})$ is a  number satisfying $\lambda_j(\pi_{\rm \tiny{fin}}) = 1/4 + r_j(\pi_{\rm \tiny{fin}})^2$, $\lambda_j(\pi_{\rm \tiny{fin}})$ is the $j$-th eigenvalue of $\Delta_{\pi_{\rm \tiny{fin}}}$, $\hat{h}$ is the usual Fourier transform of $h$:   
\[     \hat{h}(r) = \int_{-\infty}^\infty e^{\sqrt{-1}rs}h(s)ds,
\]
 $\gamma$ is a primitive conjugacy class of $\pi_1(M)$, which is identified with a prime closed geodesic in on $M$ and $\ell(\gamma)$ is the length of a close geodesic $\gamma$.

For a central element $\alpha$ in $\Gamma$, we have,  for $\beta \in \Gamma$, 
\[       {\rm tr}(\pi_{\rm \tiny{fin}}(\alpha)){\rm tr}(\pi_{\rm \tiny{fin}}(\beta))= {\rm tr}(\pi_{\rm \tiny{fin}}(\alpha\beta)),
\]
which implies that the Fourier inversion formula (\ref{Fourier1}) can be modified to
\begin{equation}
 f(\alpha) = \int_{\widehat{X}}{\rm tr}(\pi_{\rm \tiny{fin}} (\alpha^{-1}f))d\mu(x) 
= \int_{\widehat{X}}{\rm tr}(\pi_{\rm \tiny{fin}} (\alpha^{-1})){\rm tr}(\pi_{\mbox{\tiny{fin}},x} (f))d\mu(x). \label{Fourier2}
\end{equation}

Multiplying ${\rm tr}(\pi_{\rm \tiny{fin}}(\alpha^{-1}))$ and integrating over $\hat{X}$, then by the above, the trace formula (\ref{Selbergtrace}) becomes 
\begin{align}\label{trace}
\begin{split}
\int_{\widehat{\Gamma}}\sum_{j=0}^\infty \widehat{h}(r_j(\pi_{\rm \tiny{fin}})){\rm tr}(\pi_{\rm \tiny{fin}}(\alpha^{-1})) d\mu(x) &=2(g-1)\delta_{0,\alpha}\int_{-\infty}^\infty r \tanh (\pi r)\widehat {h}(r)dr \\ &+\hspace{-12pt} \sum_{\substack{\gamma:\mbox{\footnotesize{primitive}}\\ \Phi(\gamma^k) = \alpha }}\sum_{k=1}^\infty \frac{\ell(\gamma )}{\sinh (k\ell(\gamma)/2)}h(k\ell(\gamma ))
\end{split}
\end{align}

We now choose the function $h$ as the mollified characteristic function of the interval $[-T,T]$. 
For this purpose, we  take a smooth even function $k$ on $\mathbb{R}$ of compact support such that 
\[ \int_{-\infty}^\infty k(s)ds = 1 \quad \mbox{and} \quad k(s) > 0.
\]
Setting $k_\varepsilon(s) = (1/\varepsilon)k(s/\varepsilon)$, $\varepsilon$ to be chosen in terms of $T$, we define $h$ as the convolution
\[ h(s) = (\chi_{[-T,T]}\ast k_\varepsilon)(s).
\]
Then, we have
\[  \hat{h}(r) = \frac{2\sin (Tr)}{r}\hat{k}(\varepsilon r),
\]
where $\hat{k}(r)$ is rapidly decreasing.

Note that the set $\{r_j(\pi_{\rm \tiny{fin}})\}$ can be divided into the clusters $R_k = \{r_{k,i}(x)\}$, $k = 0,1,2,\ldots$ such that 
\[\lim_{x \to (0,0,0)} r_{k,i} = r_k\]  
where $\lambda_k = \frac14+r_k^2$ and $\lambda_k$ is the $k$-th eigenvalue of the Laplacian $\Delta_M$ on $M$. Thus, we have
\begin{align} 
\lambda_{0,i}(x) &= \frac{\lambda_i^{(2)}}{2}x^2+\cdots  \label{asymplambda} \\
r_{0,i}(x)&= \frac{\sqrt{-1}}{2}(1 - 2\lambda_{0,i}(x) + \cdots) \notag \\
&= \frac{\sqrt{-1}}{2}\left(1 -\lambda_i^{(2)}x^2 + \cdots \right) \label{asymprzero}  
\end{align} 
By Theorem \ref{sunada}, we know that there exists a constant $\nu$ such that
\begin{align} 
{}& \mbox{Im}\;r_{0,i}(x) \leq \nu <\frac12  \quad (\mbox{i.e.}\; \lambda_{0,i}(\pi_{\mbox{\tiny{fin}},x}) \geq \frac14-\nu^2 > 0) \quad \mbox{if} \quad x \not\in U  \notag \\
{}& \mbox{Im}\;r_{k,i}(x) \leq \nu < \frac12 \quad (\mbox{i.e.}\; \lambda_{k,i}(\pi_{\mbox{\tiny{fin}},x}) \geq \frac14-\nu^2 > 0) \quad \mbox{if}\quad k \geq 1
\end{align}
where $U$ is a neighborhood of the trivial representation $\textbf{1}$ in $\widehat{X}$.
Then, by the same argument as \cite{Phillips}, we know that the left hand side $L$ of (\ref{trace}) is
\begin{equation*}
L = 2\int_U \sum_{i=1}^{{\rm \scriptsize{dim}}(\pi_{\rm \tiny{fin}})}\frac{\sinh (-\sqrt{-1}r_{0,i}(x)T)}{-\sqrt{-1}r_{0,i}(x)}\hat{k}(\varepsilon r_{0,i}(x)){\rm tr}(\pi_{\rm \tiny{fin}}(\alpha^{-1}))d\mu(x) + O\left(\frac{T}{\varepsilon^2}+e^{\nu T}\right). 
\end{equation*} 
Replacing $\hat{k}(\varepsilon r_{0,i}(x))$ by $1+O(\varepsilon)$ and $\frac{\sinh(-\sqrt{-1}r_{0,i}(x)T)}{-\sqrt{-1}r_{0,i}(x)}$ by $|\frac{e^{-\sqrt{-1}r_{0,i}(x)T}}{-\sqrt{-1}r_{0,i}(x)}|$, this becomes
\begin{equation}
L = e^{T/2}\int_U \sum_{i=1}^{{\rm \scriptsize{dim}}(\tiny{\pi_{\rm \tiny{fin}}})}\left|\frac{e^{(-\sqrt{-1}r_{0,i}(x)-1/2)T}}{-\sqrt{-1}r_{0,i}(x)}\right|{\rm tr}(\pi_{\rm \tiny{fin}}(\alpha^{-1}))d\mu(x) + O\left(\varepsilon e^{T/2}+\frac{T}{\varepsilon^2}+e^{\nu T}\right).\label{lefthandside}
\end{equation} 

For the right hand side $R$ of (\ref{trace}) we use the simple estimate
\[\sum_{\substack{\gamma:{\rm \footnotesize{primitive}}\\ T < \ell(\gamma) < T+\varepsilon }}1 = O(\varepsilon e^T)
\]
to show that the contribution of the sum $k >1$ is $(\sqrt{\varepsilon}e^{T/2})$ and the contribution from the $\varepsilon$-smoothing of $k_\varepsilon$ is $(T^2/\varepsilon^4)$. Thus
\begin{equation}  R = \sum_{\substack{\gamma:{\rm \footnotesize{primitive}}\\ \Phi(\gamma) = \alpha \\ \ell(\gamma) \leq T}}\frac{\ell(\gamma)}{\sinh(\ell(\gamma)/2)} +  O\left(\frac{T^2}{\varepsilon^4}+\sqrt{\varepsilon}e^{T/2} \right) \label{righthandside} 
\end{equation}

Choosing $\varepsilon = e^{-\delta T}$ for some $\delta > 0$, we conclude from (\ref{lefthandside}) and (\ref{righthandside} ) that
\begin{align}\label{LR}   \sum_{\substack{\gamma:{\rm \footnotesize{primitive}}\\ \Phi(\gamma) = \alpha \\ \ell(\gamma) \leq T}}\frac{\ell(\gamma)}{\sinh(\ell(\gamma)/2)} &= e^{T/2}\int_U \sum_{i=1}^{{\rm \scriptsize{dim}}(\pi_{\rm \tiny{fin}})}e^{(-\sqrt{-1}r_{0,i}(x)-1/2)T}\frac{{\rm tr}(\pi_{\rm \tiny{fin}}(\alpha^{-1}))}{-\sqrt{-1}r_{0,i}(x)}d\mu(x) \notag \\ &{}+  O(e^{\nu_1T/2} ) \quad (\nu_1 < 1/2) \notag \\
&\sim \frac{Ce^{T/2}}{T^{d/2}}(1+\frac{c_1}{T^{1/2}} +\frac{c_2}{T} +\cdots).
\end{align}

Then the same argument as \cite{Phillips} implies Theorem \ref{theorem-geod}.

\subsection{More detail on asymptotics for prime closed geodesics of $\mbox{Heis}_3(\mathbb{Z})$-extensions}\label{AsymptoticsHeisenberggeodesics}
\subsubsection{Computations of the subleading term}
In the case of an abelian extension, this is done by Kotani \cite{Kotani1}. Nilpotent case can be essentially reduced to the computation of the left hand side (L)of (\ref{trace}) in more detail, namely, the asymptotics (\ref{asymplambda}) and (\ref{asymprzero}) refine to 
\begin{align} 
\lambda_{0,i}(x) &= \frac{\lambda_i^{(2)}}{2}x^2+\frac{\lambda_i^{(4)}}{24}x^4+\cdots  \label{asymplambda} \\
r_{0,i}(x)&= \frac{\sqrt{-1}}{2}(1 - 2\lambda_{0,i}(x) + \cdots) \notag \\
&= \frac{\sqrt{-1}}{2}\left(1 -\lambda_i^{(2)}x^2 - \frac{\lambda_i^{(4)}}{12}x^4+\cdots \right) \label{higherasymprzero}  
\end{align} 

Combining this formula with the arguments in the previous chapter, we readily obtain 

\begin{theorem}\label{sectiontenheisenberg-geod} Let $M$ be a compact Riemann surface with the constant negative curvature $-1$ of genus $g$ and $\Gamma$ be ${\rm Heis}_3(\mathbb{Z})$ and $\Phi:\pi_1(M) \to \Gamma$ be the standard surjective homomorphism defined in Theorem \ref{heisenberg-geod}.
For a conjugacy class $\alpha$ of a central element of $\Gamma$, we have the following asymptotic expansion:
\[    \pi (x,\Phi,\alpha) \sim \frac{Ce^x}{x^3}\left(1 +\frac{c_1}{x} +\frac{c_2}{x^2} + \cdots \right),
\]
where 
\begin{align}
C &= \left(\frac{{\rm vol}\; M}{2\pi\|\omega_1\|_{L^2(M)}\|\omega_2\|_{L^2(M)}}\right)^{2}\zeta_H(2) = \frac{\pi^2(g-1)^2}{8}.  \label{leadingcoefficient} \\
c_1 &= \frac{\lambda_i^{(4)}}{12\lambda_i^{(2)}} =  \frac{1}{12\lambda_i^{(2)}}\bigg[\int_M\frac{8\pi^2}{\sqrt{{\rm vol}(M)}}|\omega_1||\omega_2|G(|\omega_1||\omega_2|)(x)\left(i+\frac{1}{2}\right)^2 
\notag \\
&{} + \frac{8\pi^2{\rm vol}(\widehat{H})}{({\rm vol}(M))^2}\left(i+\frac{1}{2}\right)^2 G(|\omega_1||\omega_2|)(x)+ \frac{4\pi^2}{{\rm vol}(M)}|\omega_{12}|^2dx\bigg] \label{subleadingcoefficient} 
\end{align} 
\end{theorem}

\subsubsection{Asymptotics for noncentral conjugacy classes}
For the proof of Proposition \ref{noncenter}, we write a conjugacy class $\alpha = [\gamma]$ of the element $\gamma = u^av^bw^c$, where $u, v, w$ are defined in (\ref{uvw}). Since
\[    u\gamma u^{-1} = \gamma w^b,\quad v\gamma v^{-1} = \gamma w^a, \quad w\gamma w^{-1} = \gamma,  \]
we have
\[[\gamma] = \{\gamma w^{nc} \mid c\;  \mbox{is the GCD(a,b)(the greatest common divisor of $a$ and $b$}), n \in \mathbb{Z}\}.  
\]
This set can be identified with an element of a finite extension of the abelian group ${\rm Heis}_3(\mathbb{Z})/[{\rm Heis}_3(\mathbb{Z}),{\rm Heis}_3(\mathbb{Z})] \simeq \mathbb{Z}^2$. Therefore, we can apply the same arguments to abelian cases.
The local central limit theorem is also reduced to analyze the case of an abelian extension, which is already studied in \cite{Lalley}, \cite{Babillot}, and \cite{Anantharaman2}.

\subsection{A local central limit theorem on the asymptotics of closed geodesics for ${\rm Heis}_3(\mathbb{Z})$-extension}
In the case of an abelian extension, Lalley \cite{Lalley} considered the local central limit theorem and the large deviation principle for the asymptotics of the length of prime closed geodesics. Later, Babillot and Ledrappier \cite{Babillot} and Anantharaman \cite{Anantharaman2} further developed the concept. Here, we only study the local central limit theorem. The large deviation principle will be discussed in the next chapter.

Although the above results apply to more general hyperbolic flows in the case of a nilpotent extension, we consider only the case of the geodesic flow on a compact Riemannian surface with constant negative curvature $-1$ for the ${\rm Heis}_3(\mathbb{Z})$-extension. It seems not so very difficult to generalize the same situation to the above situation for a general nilpotent extension. It would be naturally done in the effort explained in Section \ref{variable}. 

Our method here for the central limit theorem for the geodesic flow of the Riemann surface with the constant negative curvature $-1$ about ${\rm Heis}_3(\mathbb{Z})$-extension starts from replacing $\alpha \in \Gamma$ to $\alpha_T (\sim \beta\sqrt{T}) \in \Gamma$ with $\beta \in [\Gamma,\Gamma] \otimes \mathbb{R}$ in the formula (\ref{trace}).  The actual difference resulting from this replacement is evident on the left-hand side of the trace formula (\ref{lefthandside}). The analysis is performed by a slight modification to that in Section \ref{heatclt} as follows: 
\begin{align}
L &= 2\int_U \sum_{i=1}^{{\rm \scriptsize{dim}}(\rho_{\rm \tiny{fin}})}\frac{\sinh (-\sqrt{-1}r_{0,i}(x)T)}{-\sqrt{-1}r_{0,i}(x)}\hat{k}(\varepsilon r_{0,i}(x)){\rm tr}(\rho_{\rm \tiny{fin}}(\alpha_T^{-1}))d\mu(x) \notag 
\\
&\sim e^{T/2}\int_U \sum_{i=1}^\infty e^{-\sqrt{-1}r_{0,i}(x)-1/2)T}\frac{{\rm Tr}(\rho_h(\alpha_T^{-1}))}{(-\sqrt{-1}r_{0,i}(x))T}d\mu(x).\label{refgeodesic}
\end{align} 
We replace the term $e^{-\lambda_{0,i}(h)t}\Big(\rho_h\left({}_R\!\!\int_q^p\omega\right)\Big)$ in the right-hand side of (\ref{lctrefgeodesic}) to $e^{(-\sqrt{-1}r_{0,i}(x)-1/2)T)}(\rho_h(\frac{\beta}{\sqrt{T}})$ in the right-hand side of (\ref{refgeodesic}).

Then, comparing  asymptotics (\ref{asymplambda}) and (\ref{asymprzero}), we obtain, similarly to Theorem \ref{HeatCLT}, the following theorem:  

\begin{theorem}\label{GeodesicCLT}
\begin{equation}
\lim_{x\to\infty}\Bigg|e^{\beta^2T}\pi(x,\alpha_T) -\sum_{i=0}^\infty\frac{1}{\pi\Big(i+\frac12\Big)^2}\Bigg| = 0
\end{equation}
\end{theorem} 


\section{Future directions}

In this chapter, we survey the pertinent literature and outline several promising directions for future work in our noncommutative Floquet-Bloch framework. We then develop asymptotic estimates that emerge from the interplay between finite-dimensional representations of discrete groups and infinite-dimensional representations of Lie groups.

A natural extension is to move from abelian to nilpotent structures---a shift experts will recognize at once. This generalization uncovers new relationships in the nilpotent setting, presenting fresh perspectives that merit thorough investigation. Many of these ideas originate in our own speculations, expectations, and conjectures within specific mathematical contexts.

We also introduce a handful of loosely related topics that, while not central to this paper's core arguments, may inspire future research. Given the wide-ranging disciplines represented in our references, the reader should note that some works have not been exhaustively reviewed, so misinterpretations or gaps in understanding remain possible.

From a more speculative vantage point, one could further extend the notion of nilpotent structures to pro-nilpotent ones. Such a broader viewpoint promises to expand the applicability of our results well beyond the scenarios considered here.

\subsection{Variable negative curvature and mixing Anosov flows}\label{variable}

The leading term of the asymptotics (Conjecture \ref{conj-flow}) appears to be provable by combining our version of Floquet-Bloch theory with techniques from hyperbolic dynamical systems as discussed in \cite{Lalley}, \cite{Pollicott1}, \cite{Katsuda2}, and \cite{Babillot}.

For the asymptotic expansions (Conjecture \ref{conj-geod}), it also seems provable provided that the zero-free regions of the dynamical L-functions are extended beyond the critical line, which is contingent on the applicability of the methods developed in \cite{Dolgopyat}. These cases include geodesic flows of compact Riemann surfaces \cite{Dolgopyat}, \cite{Anantharaman1}, compact manifolds with $1/9$-pinched negative curvature \cite{Giulietti}, and smooth mixing Anosov flows in dimension three \cite{Tsujii}. In a previous preprint \cite{Katsuda0}, we provided a rough sketch of the first case, following arguments in the abelian case given in \cite{Anantharaman1}.

One of the remaining tasks is to confirm that the methods developed in \cite{Dolgopyat} are also applicable to L-functions, not just the aforementioned cases of zeta functions.

We need to consider L-functions twisted by finite-dimensional representations $\rho$ whose dimensions approach infinity as $\rho \to \textbf{1}$. Finite-dimensional convergence must be compatible with uniform approximations in Theorem \ref{introdiscretetoLie} ($=$ Theorem \ref{discretetoLie}) and Theorem \ref{introdecompose} ($=$ Theorem \ref{sec4decompose}), which is not a trivial problem; it could be addressed using our method without significant difficulties.

\subsection{Explicit description of coefficients in higher order terms}  

In the context of hyperbolic dynamical systems, the coefficients in the higher-order terms of the asymptotic expansion for the number of prime closed orbits in a conjugacy class of a discrete group $\Gamma$ and its extensions bear similarities to the asymptotic expansions of certain path integrals in physics. When $\Gamma$ is abelian, Kotani and Sunada \cite{Kotani5} and Kotani \cite{Kotani6} compute these coefficients using the derivatives of the pressure in terms of higher covariances. Their results involve summing several terms in formal expansions as convergent forms. For generalization to nilpotent cases, renormalization procedures, similar to those proposed by Connes and Kreimer \cite{Connes1}, \cite{Connes2}, \cite{Connes3}, or regularization schemes related to Hairer et al. (see \cite{Hairer}, \cite{Bailleul}), seem beneficial.

More ambitiously, we aim to investigate the geometric meaning of each term in the asymptotic expansions, akin to other counting problems.

\subsection{The leading term of the asymptotic expansion, its integral expression, transcendence, periods}\label{leadingterm}
In Theorems \ref{theorem-geod}, \ref{conj-heat}, the coefficients of the leading terms for asymptotic expansions are expressed using the special values of the spectral zeta functions $\zeta_H(d/2)$ of the hypo-elliptic operator $H$ associated to the irreducible unitary representations of $G$ corresponding to the open coadjoint orbits. 
Their explicit expressions in terms of the known constants may not be possible in general.
However, as in Theorems \ref{engel-geod}, there seems to be a possibility to express the special value of the spectral zeta function of the quartic oscillator $\zeta_H(7/2)$ in Example \ref{Engelquartic} by a specific period integral, such as the Bessel functions that appeared there. 
Moreover, the Bessel functions themselves have integral expressions involving the exponential function (cf. \cite{BlochEsnault1}, \cite{BlochEsnault2}, \cite{Fresan1}). In this regard, it would be interesting whether they are periods in the sense of Kontsevich and Zagier \cite{KontsevichZagier} or exponential periods as in  \cite{Fresan1}. 
There are somewhat related facts as follows: First, in \cite{Fresan1}, it is proved that $\Gamma(1/2)=\sqrt{\pi}$ is not a period but an exponential period under the Grothendieck period conjecture. 

Concerning this, we have asked Francis Brown about a related problem. When he gave a talk about the motivic nature of integer values of the Riemann zeta function $\zeta(2), \zeta(3), \ldots$, his comment was that $\zeta(1)$ is related to the exponential period or exponential motive. This relates to the above comment on $\sqrt{\pi}$ through the functional equality, $\zeta(0)$, and Stirling's formula.  

However, he also notes that the half-integer values, in particular $\zeta(1/2)$, would be beyond their reach.

 In the respect of Manin comment that the motive of supersingular elliptic curve over $\mathbb{F}_q$ can be seen as a ``spinorial square root'' of the Tate motive quoted in \cite{Kapranov} and \cite{Ramachandran}, combining with the integral expression of the Bessel function, the above $\zeta_H(7/2)$ has ``exponential nature'' in a double sense. Thus, it seems to be an interesting test case for higher motives.

Second, Costello commented that several constants that appeared in his book \cite{Costello} of perturbative quantum field theory are periods. 
It appears that the path integral and the integral expression of the heat kernel in this note share a similar vein, as explained in \cite{Costello}.   

Third, as we have seen, the harmonic oscillator $H = -\frac{d^2}{ds^2}+s^2$ naturally associated with the Heisenberg group produces a spectral zeta function $\zeta_H(s)$ that, at its core, coincides with the Riemann zeta function $\zeta(s)$. Namely, the eigenvalues of $H$ are $2n+1$ and thus 
\[\zeta(s) = \sum_{n=1}^\infty\frac{1}{n^s} = \sum_{n=1}^\infty\frac{1}{(2n)^s} + \sum_{n=1}^\infty\frac{1}{(2n+1)^s} = \frac{1}{2^s}\zeta(s) + \zeta_H(s)\]
Namely \[\zeta(s) = \frac{2^s}{2^s-1}\zeta_H(s)\]
By contrast, in number theory one also encounters versions of the Riemann zeta function arising as the spectral zeta function $\zeta_{S^1}(s)$ of the Laplacian $\Delta_{S^1}$ on $S^1$. For example, 
the constant term of an Eisenstein series, endowed with a distinguished origin on the circle at infinity $S^1$. 
In fact, since the eigenvalues of $\Delta_{S^1}$ are $n^2$, we have the equality $\zeta(s) = \zeta_{S^1}(s/2)$.t

Natural questions arising the above are whether intrinsic explanation of the equality $\zeta_H(s) =  \frac{2^s-1}{2^s}\zeta_{S^1}(s/2)$ without numerical computaion of each eigenvalues or why the diffence of their arguments $s$ and $s/2$ appeared despite both $H$ and $\Delta_{S^1}$ are the second order differential operators. 
Moreover, concerning the quartic oscillator above, we have a question about whether there exists a counterpart analogous to $\Delta_{S^1}$, as in the case of $H$. 

Recently, Otsubo and Yamazaki \cite{OtsuboYamazaki} studied Gauss and Jacobi sums as invertible Chow motives over a finite field $\kappa$, and remarked on page $2$ of \cite{OtsuboYamazaki}; 
\begin{quotation}
assuming Beilinson's conjecture (cf. 1.0 in \cite{Beilinson01}) on rational $=$ numerical equivalence plus the Tate conjecture (cf. Proposition 2.21 in \cite{Milne}), every simple object in $\mbox{Chow}(\kappa, \mathbb{Q})$ becomes invertible. They further assert that isomorphism classes of these simples form a group isomorphic to the group of Weil $q$-numbers $W_q(\mathbb{Q})$. 
\end{quotation}

Here, one more question arising is whether this could be considered as a ``finite'' analogue of the above questions of spectral zeta functions.

\subsection{From abelian to nilpotent}

A vast amount of literature uses the (usual) Floquet-Bloch theory. (See, e.g., Survey article of Kuchment \cite{Kuchment} and his future books announced there.)

Especially, some of them have a structure in which one of their main arguments contains the perturbation arguments of eigenvalues of twisted operators in the vicinity of the trivial character $\textbf{1}$. Similarly, even in the nilpotent case, we think essentially the same strategies are available since the leading terms are described by $\mathfrak{g}^{(1)}$ part as explained in the previous sections, which essentially come from the abelian quotients $G/[G,G]$. We describe the following two examples as concrete candidates.

\subsubsection{Structure of Green function and the Martin kernels} These kinds of results are considered in \cite{Murata} in the case of abelian extensions. They give the asymptotics at infinity of a Green function for an elliptic equation with periodic coefficients on $\mathbb{R}^d$ and also completely determine the Martin compactification of $\mathbb{R}^d$ with respect to an elliptic equation with periodic coefficients by using the exact asymptotics at infinity of the Green function, which gives an affirmative answer to one of Agmon's conjecture. It seems natural to consider the extended problem for nilpotent groups using our version of noncommutative Floquet-Bloch theory.

\subsubsection{Horocycle flows, frame flows}

Locally mixing properties of horocyclic or frame flows for abelian coverings are investigated by several authors \cite{Babillot}, \cite{Ledrappier}, \cite{Oh}, \cite{Pan}. \cite{Sarig }.  For example, Oh and Pan \cite{Oh} proved the local mixing theorem for the frame flow of $\mathbb{Z} ^d$-covers $X$ of closed hyperbolic $3$-manifolds $X_0$. 
We shall recall their result very briefly. The hyperbolic $3$-manifolds $X_0$ and $x$ can be expressed as $X_0 =\Gamma_0\backslash H^3$ with the hyperbolic $3$-space $H^3$ for some co-compact lattice $\Gamma_0$ of $G = \mbox{PSL}_2(\mathbb{C})$ and $X =\Gamma\backslash H^3$ that $\mathbb{Z}^d$-cover $\Gamma\backslash G \to \Gamma_0\backslash G$ with $\Gamma\backslash\Gamma_0 \simeq \mathbb{Z}^d$. Then, the frame bundle $F(M)$ (resp. $F(\hat{M})$) of $X_0$ (resp. $X$) is isomorphic to the homogeneous space $\Gamma_0\backslash G$ (resp. $\Gamma\backslash G$) and  the frame flow on $F(M)$ corresponds to the right multiplication of \[a_t = \left(\begin{array}{cc} e^{t/2} & 0 \\ 0 & e^{-t/2}\end{array}\right)\]
on $\Gamma_0\backslash G$. 

A probability measure $\mu$ on $\Gamma\backslash G$ has the local mixing property for
$\{a_t\}$ if there exist a positive function $\alpha$ on $\mathbb{R} > 0$ and a non-trivial Radon measure
$m$ on $\Gamma\backslash G$ such that for any function $\psi$ in the set $C_c(\Gamma\backslash G)$ of continuous function on $\Gamma\backslash G$ with compact support,
\[
\lim_{t\to\infty}\alpha(t)\int\psi(x)d\mu_t(x) = \int\psi dm(x),
\]
where $\mu_t$ is the probability
measures on $\Gamma\backslash G$ of $\mu$ translated by the flow $a_t$, which is defined by
\[   \mu_t(\psi) = \int\psi(xa_t)d\mu.
\]
\begin{theorem}[special case of Theorem 1.6 in \cite{Oh}]  For $\psi_1, \psi_2 \in C_c(\Gamma\backslash G)$,
\[ \lim_{t\to\infty}t^{d/2}e^{(D-\delta)t}\int_{\Gamma\backslash G}
\psi_1(xa_t)\psi_2(x) dx =
\frac{m^{\mbox{BR}_+}(\psi_1)m^{\mbox{BR}_-}(\psi_2)}{
(2\pi\sigma)^{d/2}m^{\mbox{BMS}}(\Gamma_0\backslash G)},
\]
where where $D =2$ is the volume entropy of $H^3$, $\delta$ is the critical exponent of $\Gamma_0$, $m^{\mbox{BR}_+}$ is
the Burger-Roblin measure on $\Gamma\backslash G$ for the expanding horospherical subgroup, $m^{\mbox{BMS}}$ is the Bowen-Margulis-Sullivan measure on $\Gamma_0\backslash G$, $m^{\mbox{BR}_-}$ is
the Burger-Roblin measure on $\Gamma\backslash G$ for the contracting horospherical subgroup, and
\[ \sigma = \det(\nabla^2 P(0))^{1/d} > 0 \]
is a constant given in (3.6) in \cite{Oh} with the function $P(u)$ defined from the pressure function as in Section 3.3 in \cite{Oh}.
\end{theorem}
This theorem is the precise form of the local mixing theorem (Theorem 1.5 in \cite{Oh}).

Their proof is based on extending
the symbolic dynamics approach of studying the geodesic flow on the unit tangent bundle $T^1(X_0)$ as the
suspension flow on $\Sigma \times \mathbb{R}/ \sim$ for a subshift of finite type $(\Sigma, \sigma)$. The at flow on a $\mathbb{Z}^d$-cover $X$ can be studied via the suspension flow on
\[ \tilde{\Sigma}:= \Sigma \times \mathbb{Z}^d \times M \times \mathbb{R}/ \sim\]
where the equivalence is defined via the shift map $\sigma: \Sigma \to\Sigma$, the first return
time map $\tau:\Sigma\to\mathbb{R}$, the $\mathbb{Z}^d$-coordinate map $f :\Sigma\to\mathbb{Z}^d$, the holonomy map $\theta:\Sigma\to M$, and $M$ is the compact subgroup which is the centralizer of $A =\{a_t\}$.  

The asymptotic behavior of the correlation function of the suspension
flow on $\Sigma$ with respect to the BMS measure can then be investigated using analytic properties of the associated Ruelle transfer operators $L_{s,v,\mu}$ of three parameters $s \in \mathbb{C}$, $v \in \hat{\mathbb{Z}^d}$, $\mu \in \hat{M}$where $\hat{\mathbb{Z}^d}$ and $\hat{M}$ denote the unitary dual of $\mathbb{Z}^d$ and of $M$ respectively.

Their proof of the desired analytic properties of $(1-L_{s,v,\mu})^{
-1}$ is based on the study of
the generalized length spectrum of $\Gamma_0$ relative to $\Gamma$:
\[\mathcal{G}L(\Gamma_0, \Gamma) := \{(\ell(\gamma), f(\gamma), \theta(\gamma)) \in \mathbb{R} \times \mathbb{Z}^d \times M : \gamma \in \Gamma_0\backslash \Gamma\}.\]
The correlation function for the BMS measure can then be expressed in terms of
the operator $(1-L_{s,v,\mu})^{-1}$ via an appropriate Laplace/Fourier transform. We then
perform the necessary Fourier analysis to extract the main term coming from the
residue. Finally, we can deduce the precise asymptotic of the correlation function for the Haar measure from that for the BMS measure using ideas originated in Roblin's work (see Theorem 4.10 in \cite{Oh}).

These arguments seem to be generalized to nilpotent extensions by replacing the above ingredients concerning around $\hat{\mathbb{Z}^d}$, namely the usual Floquet-Bloch analysis with our nilpotent version in this paper, which answers some conjectural statements at Remark 1.8. (4) in \cite{Oh}. 
More recently, several works related to these topics have been published:\cite {Dolgopyat2}, \cite{Flaminio}, and \cite{Forni}. It seems to be within the scope of the problem to some extent to generalize to nilpotent extensions.

\subsubsection{The Liouville Riemann-Roch theorem for covering} 
Our arguments in the main body of the paper are based on the perturbation around the trivial representation. However, these are also applied to finite-dimensional representations, similar to the arguments surrounding the rational flux for the Harper operator in \cite{Helffer}. Combining with Pytlik's Theorem or its generalizations, Theorem \ref{introPytlik}, we could proceed with some similar analysis done on the whole set of the Brillouin zone in the usual Floquet-Bloch theory to generalize to nilpotent cases approximately. Namely, we can establish asymptotic analysis around a point in the dense set in the unitary dual. These arguments may be helpful in the nilpotent generalization of the (abelian) Liouville-Riemann-Roch theorem in \cite{Kha}. Note that this kind of extension would be possible, as it has already appeared in a comment of Kuchment in \cite{Kuchment}.

\subsubsection{The Livshits theorems}

The Livshits theorems have played an essential role in the cohomological study of dynamical systems theory for several decades. Following the introduction of \cite{Gogolev}, we recall several notions.

If the dynamical system is a smooth flow $X^t:M\to M$ generated by the
vector field $X$ the equation
\[ \varphi =L_Xu \]
It is called the cohomological equation. Here $u:M \to \mathbb{R}$ is an unknown function, and $L_X$ is the Lie derivative in the direction of $X$. If this equation has a solution, the function $\varphi:M \to \mathbb{R}$ is called an $X^t$-coboundary. Accordingly, two continuous
functions $\varphi$ and $\psi$ are said to be $X^t$-cohomologous if $\varphi-\psi$ is a coboundary.

The fundamental question is to decide whether the space of $X^t$-coboundaries $\mathcal{B}_X$ is closed in the appropriate topology (H\"older, smooth, etc.). 
If this is the case, then $\mathcal{B}_X$ is given by the intersection of kernels of all $X^t$-invariant bounded linear functionals. 
In the case of a transitive Anosov flow on a compact manifold $M$, this problem was solved by A. Livshits in the classical paper \cite{Livsic}. He proved that in
this case, the space $\mathcal{B}_X$, as a subspace of the space of H\"older continuous functions,
is characterized by countably many functionals given by integration over periodic orbits of $X^t$. Much later, A. Katok and A. Kononenko introduced a new type of functionals called periodic cycle functionals to handle the case of partially hyperbolic dynamical systems. They proved that the space of coboundaries is closed for partially hyperbolic diffeomorphisms and flows that satisfy certain local accessibility conditions on their stable and unstable foliations \cite{KatokKononenko}.

By revisiting the idea of Katok-Kononenko and defining a version of periodic cycle functionals, A. Gogolev and F. R. Hertz \cite{Gogolev} prove several versions of an abelian Livshits Theorem for transitive Anosov flows, where abelian cohomology is a specific weaker equivalence relation than cohomology, of which the standard cohomology on the universal abelian cover gives a relation. They give applications to improving the Croke-Otal marked length spectrum rigidity theorem \cite{Croke}, \cite{Otal}, and a rigidity problem of contact Anosov flows. Both of these applications also strongly rely on
on R. Sharp's work on homologically full Anosov flows \cite{Sharp1}.

Obviously, the problem of the possibility of generalization to nilpotent from Abelian naturally arises.

\subsection{Prime orbit theorem vs. Chebotarev density theorem}
The strategy of proof of Theorem \ref{primegeodesictheorem} (the prime orbit theorem) seems to be  classified roughly into the following three methods:
\begin{description} 
\item[{\rm a}](Harmonic analysis) This is originated from the celebrated work of Selberg \cite{Selberg}. There are vast generalizations related to number theory. Currently, this method appears to apply to homogeneous spaces, as it is based on the relation between the eigenvalues of the Laplacian and the conjugacy class of primitive elements of the fundamental group. The connection to the norm of the latter and the geometric length of closed orbits is obscure for non-homogeneous spaces.  
  
Nevertheless, there is a conjectural formula like Gutzwiller's semiclassical trace formula \cite{Gutzwiller}. As for developments of the mathematical side, Faure and Tsujii's works \cite{FaureTsujii1}, \cite{FaureTsujii2}, \cite{FaureTsujii3} are remarkable.

\item[{\rm b}](Dynamical method I) This method may be started by Margulis's thesis \cite{Margulis}, \cite{Margulis2}. Recently, the case of nonpositively curved rank one manifolds by {Ricks} \cite{Ricks} and closed geodesics on surfaces without conjugate points by Climenhaga, Knieper, and War \cite{Climenhaga}, and Eskin and Mirzakhani \cite{EskinMirzakhani1}, Eskin, Mirzakhani, and Rafi \cite{EskinMirzakhaniRafi} to Teichm\"uller geodesic flows on the moduli space of hyperbolic Riemann surfaces, the prime orbit theorem is established. 
\item[{\rm c}](Dynamical method II) This method is essentially used to study spectral properties of the transfer operator, sometimes called the Ruelle or Koopman operators by  Parry and Pollicott \cite{Parry1}, \cite{Parry3}.
\end{description} 

Methods (a) and (c) are related to spectral analysis. Thus, several forms of Chebotarev density theorem are also established using the idea of number theory or the Floquet-Bloch theory. On the other hand, it has not yet been established a method to extend the Chebotarev-type theorems to case (b), and it seems interesting to remain with these open problems.

\subsection{Large deviations, heat kernels with bias and non-symmetric random walks}\label{Largedeviationstrategy}

The Large deviation is one of the further topics for asymptotics of Sections \ref{Longtimeheat} and \ref{Asymptoticsclosedgeodesics}. 
There is a certain similarity between significant deviation results for symmetric Brownian motions or symmetric random walks and the asymptotic results for biased or non-symmetric ones. 
An example of the above phenomena is the analysis at 1.11 in \cite{Alexopoulos1} for the asymptotics of $n$-th convolution power $\mu^n$ of a non-centered walk in a neighborhood of the origin and the result of the significant deviation treated in \cite{Tanaka}. 

The latter research has a larger body of literature than the former. 
Thus, we discuss the results of biased random walks, especially about remarkable results of B\'{e}nard and E. Breuillard \cite{Benard}, \cite{Benard2}, which also include a survey of previous results on this topic.

In \cite{Benard}, the authors establish the central limit theorem of random walks on a simply connected nilpotent Lie group $G$, allowing a possible bias and several related results. One of the remarkable ones is a criterion for whether the limiting hypoelliptic diffusion process has full support. They obtain these results by giving a new stratification of the Lie algebra $\mathfrak{g}$ of $G$. 
Moreover, in \cite{Benard2}, the results are strengthened to the local central limit theorem under the non-lattice condition. Their methods are based on the usage of the characteristic function with the above new stratification, which is a kind of Fourier analysis.

Since the nilpotent Floquet-Bloch analysis developed in this paper is a refinement of the usual Fourier analysis, combining it with their new stratification, our method will give further understanding. 
For example, applications to the lattice case, or a simplification of the central part of the proof of their local limit theorem, which is called domain reduction (Proposition 3.8), by replacing their path swap arguments with similar arguments given in the Section \ref{sunadatheorem}.

\subsection{Long range random walks}
Asymptotics of long-range random walks are considered in \cite{ZQChen1}, \cite{ZQChen2}. Although we do not know whether the Floquet-Bloch theory is directly applicable to these problems, it is applicable as a problem of two limits: one is the approximation by exhaustion of usual random walks, and the other is the long-time asymptotics of them. Obviously, it is not easy to obtain meaningful results solely through this heuristic in general. However, in our case, rather exact information available for the asymptotics of finite range random walks in approximation processes, such as effective estimates of error terms, it would be reasonable to expect substantial results.

\subsection{Discrete approximations of several differential operators}
It is a natural and traditional problem in analysis that the approximation of the differential operator by its discrete versions, namely difference operators, arises in various contexts such as homogenization theory, the finite element method (FEM), and the boundary element method (BEM).
Recently, there have been several results on this topic; for example, Nakamura et al. \cite{SNakamura1}, \cite{Mikami} considered the norm resolvent convergence and the convergence of the spectrum as its corollary and Ishiwata and Kawabi \cite{Ishiwata6} investigate $C^0$-semigroup generated by a Schr\"{o}dinger operator with drift on a complete Riemannian manifold is approximated by the discrete semigroups associated with a family of discrete time random walks with killing in a flow on a sequence of proximity graphs. 

Moreover, Burago et al. \cite{DBurago} investigate an approximation of the spectrum of the connection Laplacian on vector bundles over compact metric measure spaces by convolution-type operators and the graph connection Laplacians.

One of the natural problems following \cite{SNakamura1}, \cite{Mikami}, or \cite{Ishiwata6} is an extension operator under a magnetic field on a complete Riemannian manifold. This problem is also considered as an extension of \cite{DBurago} to the noncompact case.

As explained in Remark \ref{meshapproximation}, the Wilkinson formula (\ref{introeigen1}) approximates the Landau level. However, this case is highly symmetric, and it is desirable to be able to handle a more robust problem.

Our proposal here is to apply the idea of JOKU-IKOU (oper), which will be explained in Section \ref{mRHoper}, to construct suitable approximation operators. Moreover, we think some possibilities exist to use something somewhat similar to Theorem \ref{discretetoLie}. 
\subsection{ Sub-leading terms of spectral convergence: Collapsing, tropicalization}

In the field of spectral geometry, the convergence of a sequence of compact Riemannian manifolds $M_i$ and its relationship to the spectral behavior of differential operators $D_{M_i}$ defined on these manifolds is a frequently studied problem.  
In particular, the situation where the volume $\mathrm{vol}(M_i)$ tends to zero is referred to as a collapsing situation. A typical scenario in this context is that, for sufficiently large $i$, the manifold $M_i$ becomes a fiber bundle over a limiting space $M_\infty$, with fibers that are nilmanifolds, which are compact quotients of simply-connected nilpotent Lie groups.  
For instance, if the sectional curvatures of $M_i$ are uniformly bounded, the Cheeger-Fukaya-Gromov theory \cite{CheegerFukayaGromov} ensures that such a fiber bundle structure arises.

Moreover, under the additional assumption that the diameters of $M_i$ are uniformly bounded, Fukaya \cite{Fukaya} proved the continuity of the eigenvalues of the Laplacian $\Delta_{M_i}$ with respect to the measured Gromov-Hausdorff topology.  
While the leading term in the convergence typically does not reflect the spectral information of the fibers, the subleading or error terms may encode such fiber-specific spectral data.

Suppose we normalize the diameters of the collapsing fibers to be scaled to 1. In that case, the situation becomes analogous to the case of the Harper operator $H_\theta$ discussed in Remark \ref{meshapproximation}, where the mesh size is similarly normalized to 1.  
In that setting, the second-order term in the expansion around the zero flux limit of the magnetic field corresponds to the harmonic oscillator.

A similar phenomenon can be considered in the tropical limit at the boundary of the compactification of the moduli space of hyperbolic manifolds, as follows:  
When a sequence of hyperbolic manifolds $M_i$ approaches the boundary while maintaining constant curvature $-1$, the diameter $\mathrm{diam}(M_i)$ of $M_i = (M_i, g_i)$ tends to infinity as $i \to \infty$. 

By rescaling each metric so that its diameter becomes one, we obtain the modified sequence  
\[
\overline{M}_i \;:=\;\bigl(M_i,\;g_i/\mathrm{diam}(M_i)\bigr),
\]  
which converges to a graph $\overline{M}_\infty$, called the tropical limit.  

In this setting, Burger \cite{Burger} studied how the eigenvalues of the Laplacian on $\overline{M}_i$ converge to those of the graph Laplacian on $\overline{M}_\infty$ (see also Section \ref{Chebyshevbias}). Although these arguments in \cite{Burger} are presented for hyperbolic surfaces, we note that the collar lemma--a key ingredient--follows purely from hyperbolic trigonometry and hence extends to higher dimensions.
The behavior of subleading and error terms in this convergence is also of significant interest.

Although not directly related, there may be conceptual connections to the topic discussed in Section \ref{operGai}, where Abouzaid, Ganatra, Iritani, and Sheridan~\cite{AbouzaidGIS} investigated subleading terms in tropicalization.  
For example, they considered expressions such as $\zeta(2) = \pi^2 / 6$ arising as subleading terms in Maslov dequantization. They used such approximations to compute asymptotic expansions of period integrals, relating them to the Gamma class of the mirror manifold.  
Similar phenomena may also appear in the context of spectral invariants.

\subsection{ More general group extensions}

\subsubsection{ First guideline: Rough asymptotics}
Suppose we consider the asymptotics results for extensions beyond nilpotent groups. In that case, the following rough estimate of return probabilities on discrete groups $\Gamma$ may serve as the first guideline for future research. The results would have been known since the late $80$s (cf. \cite{Saloff1}). 

Let us introduce notions to describe the results. Two functions $f, g$ on $\mathbb{R}$ are called comparable, which is denoted by $f\asymp g$, if there exist constants $C_1, C_2 >0$ such that
\[       C_1f(C_1x) \leq g(x) \leq C_2f(C_2x) \quad \mbox{for} \; x \in \mathbb{R}. 
\]
Then, the transition probability $p(t):= k_\Gamma(t,p,p)$ (a.k.a. the return probability) of a simple random walk on the discrete group $\Gamma$, which is a uniform lattice of the simply connected Lie group $G$, is described in the following Theorem:
\begin{theorem} \label{roughasymtotic}
\begin{description}
\item[{\rm (a)}] If $\Gamma$ has polynomial growth of order $d$, then $p(t) \asymp \frac{1}{t^{d/2}}$.
\item[{\rm (b)}]If $\Gamma$ is amenable group of exponential growth, $p(t) \asymp e^{-t^{1/3}}$
\item[{\rm (c)}]If $\Gamma$ is non amenable group, then $p(t) \asymp e^{-t}$
\end{description}
\end{theorem}

The content of this paper is a detailed study of case (a). We make a few supplementary comments in order.

\subsubsection{ Solvable groups}
\begin{description}
\item{\rm (i)} One example of case (b) is a uniform lattice $\Gamma$ of the solvable Lie group $G$. This $\Gamma$ satisfies a property of polycyclicity. This class of groups seems to be the next target of research in the context of this paper.

It should be noted that many discrete solvable groups do not realize uniform lattices of Lie groups. There are several examples in the class of solvable groups whose $p(t)$ have different behaviors from case (b). 
In general, a complete understanding of the asymptotics of $p(t)$ for general solvable groups has not been reached, even roughly, as above. 

Thus, it is not easy to perform similar investigations as in the case of nilpotent groups. However, we would like to point out some possibilities for proceeding further. 
In the case of nilpotent groups, our methods ultimately reduce to finite-dimensional representations, dense in the irreducible representations in uniform topology. 
Concerning more general discrete groups, Yoshimichi Ueda has informed us that the article \cite{Lubotzky} contains an assertion that residually finite or amenable discrete groups have such a density property in Fell topology, which is weaker than uniform topology. 
Although it may not be sufficient to develop further using only this fact, as explained in a comment after Theorem \ref{introdiscretetoLie}, it would provide good support. 
To obtain more concrete and detailed information, such as Theorems \ref{theorem-geod} and \ref{conj-heat} for general discrete groups, we anticipate that a variant of the orbit method can be developed for further investigations, in line with the philosophy presented in \cite{Kirillov2}.

\end{description}

\subsubsection{ Non-amenable groups}
\begin{description}
\item[{\rm (ii)}] If one is interested in the asymptotics of hyperbolic flows, we remind that a similar kind of rough results as Theorem \ref{roughasymtotic} for amenable extensions are already obtained in \cite{Dougall}.
 
\item[{\rm (iii)}] In the case when $\Gamma$ is not amenable, like case (c), the situations are more complicated, and several solutions seem to happen. As for hyperbolic flow, Sarnak asked the author for the existence of examples such that 
\begin{itemize}
\item $\Gamma$ is non-amenable.
\item $\pi(x,\Phi,0)$ with $0 = \mbox{ker}(\Phi:\pi_1(M) \to \Gamma)$ has exponential growth,  
\end{itemize}

Define a surjective homomorphism $\Phi$ from the fundamental group $\Gamma_g := \pi_1(M)$ of surfaces $M$ with genus $g \geq 2$ or more to the free group $F_2 =\langle u,v \rangle$ modifying the definition 
\begin{equation*}\Phi(a_i) = u_i, \quad  \Phi(b_i) = 1,\; (\mbox{identity element})\quad i = 1,\ldots, g   \end{equation*} 
in (\ref{pitofreetonil}) to
\[\Phi(a_1) = u,\; \Phi(a_2) = v \; \Phi(a_i) = \Phi(b_j) = e \quad 3 \leq i \leq g, 1 \leq j \leq g.
\]

Let us examine the inverse image of the conjugacy classes of elements of $F_2$ under this mapping. It seems relatively easy to understand through a simple combinatorial argument that there will be some degree of exponential increase with respect to the word length. 
Furthermore, the relationship between word length and the hyperbolic length of the closed geodesics corresponding is discussed in \cite{GekhtmanTaylorTiozzo}, which implies an easy answer. 
Concerning this, his reaction to this response was that he was not aware of the literature \cite{GekhtmanTaylorTiozzo}, but other discussions were essentially the same as those in his mind.

\item[{\rm (iv)}] (Integrated density of state) A more systematic answer can be obtained by developing several recent results in the following subsection. Before going into detail, we notice that it looks difficult to approximate by finite objects to $\Gamma$ itself or its regular representation in spectral analysis at first glance. The reasons come from the fact that there is a certain difference between the spectrum of Schr\"odinger operator and its integrated density of state. 
For the definition and more details, see \cite{Shubin}, \cite{Adachi1}.  
\end{description}

\subsubsection{ Surface groups}
\begin{description}

\item[{\rm (a)}] (Strong convergence)
Nevertheless, to (iv) in the previous subsection, in the special case that $\Gamma$ is the fundamental group of a compact Riemann surface with genus $g \geq 2$ (surface group) or more general hyperbolic groups, there is a breakthrough in the field related to the notion of strong convergence.
Here we quote recent results due to M. Magee, D. Puder, and R. van Handel \cite{MageePudervanHandel} (see also \cite{HideMaceraThomas}).
Let $S_n$ be the $n$-th symmetric group consisting of permutations of $n$ elements. Usual $n$-dimensional matrix representation is not irreducible but has an $(n-1)$-dimensional irreducible unitary subrepresentation 
\[ \mbox{std}:S_n \to U(n-1)
\]
obtained by removing the nonzero invariant vectors. This std is usually called the standard representation of $S_n$. 
For a positive integer $n$, the set of homomorphisms $\mbox{Hom}(\Gamma_g,S_n)$ is finite, and we equip it with the uniform probability measure. 
Let $C(\Gamma_g)$ denote the group algebra of $\Gamma_g$.
Every linear representation of $\Gamma_g$ extends linearly to one of $C(\Gamma_g)$.
\begin{theorem}[{\rm Theorem 1.1 in \cite{MageePudervanHandel}}]\label{MageePvH}
For all $x \in C(\Gamma_g)$, for $\phi_n \in \mbox{Hom}(\Gamma_g,S_n)$ uniformly random, as $n \to \infty$, and for the right regular representation $R$ of $\Gamma_g$,
\begin{equation}
\|\mbox{std}\circ\phi_n (x)\| \to \|R(x)\| \label{strongnormconv}
\end{equation}
in probability. The norms on each side are operator norms
\end{theorem}
In the language of strong convergence of unitary representations, the above theorem says
precisely that the random representations ${\rm std}\circ\phi_n$ converge strongly in probability to $R$.  
This is an improvement of \cite{Louder}. This latter reference is informed by Beno\^it Collins answering our inquiry. For a survey of related results, see \cite{Magee1}, Section 1.4 of \cite{MageePudervanHandel}, and a recent survey article \cite{vanHandel}.

Moreover, this theorem provides a possibility to investigate the spectrum of the Laplacian on the $\Gamma_g$-covering $X_g$ of a compact manifold $M$ by approximating it with the above finite-dimensional unitary representations.

We notice that there are also other remarkable results which seem to give a hint to these problems recently by Anantharaman and Monk \cite{AnantharamanMonk1}, \cite{AnantharamanMonk2}.

\item[{\rm (b)}] (Moduli of finite-dimensional unitary representations of surface group)
On the other hand, it is known that an $n$-dimensional irreducible unitary representation of a surface group $\Gamma_g$ is one to one corresponds a stable vector bundle of rank $n$ and degree zero on compact surfaces $M = M_g$ of genus $g \geq 2$ by Narasimhan and Seshadri theorem \cite{NarasimhanS} and thus, the former investigations to be reduced to that of the latter moduli space $\mathcal{S}_{g,0}$.

In general, the moduli space $\mathcal{S}_{g,d}$ of stable vector bundles of degree $d$ on compact surfaces $M = M_g$ is smooth provided $g$ and $d$ are relatively prime but $\mathcal{S}_{g,0}$ does not satisfy the above condition and equipp with several singularities. 

Magee establishes the existence of a large $n$ asymptotic expansion, to any fixed order, for the expected value of the trace of any fixed element of $\Gamma_g$ under their random representation into $SU(n)$ with law given by the volume form arising from the Atiyah-Bott-Goldman symplectic form on moduli space $\mathcal{S}_{g,0}$ in \cite{Magee01} and proved the expected value of the trace of a fixed non-identity element of $\Gamma_g$ is bounded as $n \to \infty$\cite{Magee02}. In this work, it is enough to study the smooth part of $\mathcal{S}_{g,0}$.

However, for a more concrete description of $n$-dimensional unitary representations, it seems necessary to consider singularities. For example, it would be helpful to introduce good coordinates using parallel transport by projective connections defined by Faltings \cite{Faltings} or Takhtajan \cite{Takhtajan}. In this case, the behavior around singularities is an important point.

In the case where $\mathcal{S}_{g,d}$ is smooth, Harder and Narasimhan \cite{Harder} and Atiyah and Bott \cite{Atiyah1} compute the ordinary cohomology of $\mathcal{S}_{g,d}$. Recently, C. Felisetti, A. Szenes, and O. Trapeznikova \cite{Felisetti} succeeded in computing the intersection cohomology, as an alternative to the usual one, of $\mathcal{S}_{g,0}$, which is computed using a new method. It would be helpful for further study.

\item[{\rm (c)}] (Conformal Bootstrap)
Although Theorem \ref{MageePvH} demonstrates that the typical behavior of the asymptotics leading to strong convergence, it appears that no explicit example of a sequence of $\phi_n(x)$ satisfying (\ref{strongnormconv}) is known. 
We cannot answer this question at this time. However, is it possible to achieve the structure of explicit examples of the following closely related questions based on discussions in physics, since it would give a hint to the original question?

Is is also one of the primary motivations for Theorem \ref{MageePvH} is answering a question of Buser \cite{Buser} asking the existence or more stronly typical, of sequence of coverings $\{X_n\}$ of compact hyperbolic surfaces $X$ such that the first nonzero eigenvalue $\lambda_1(X_n)$ of the Laplacian $\Delta_{X_i}$ of $X_n$ converges to $1/4$ which is the bottom of the spectrum of the Laplacian $\Delta_X$ of the universal covering $\widetilde{X}$ of $X$. 
Theorem \ref{MageePvH} implies the following result.

\begin{theorem}\label{towerspectrum}
Let $X$ be a closed connected hyperbolic surface. Let $X_n$ denote a uniformly
random degree n covering space of $X$. With probability tending to one as $n \to\infty$
\begin{equation}
\mbox{spec}(\Delta_{X_n}) \cap [0, \frac14-o(1)) = \mbox{spec}(\Delta_{X}) \cap [0, \frac14-o(1))  \label{spectralequality}
\end{equation}
with the same multiplicities on both sides.
In particular, if $\lambda_1(X) \geq \frac14$, then $\lambda_1(X_n) \to \frac14$ in probability as $n \to \infty$.
\end{theorem}

This theorem is a stronger version of a result in \cite{HideMagee} which gives an affirmative answer to Buser's conjecture \cite{Buser}.
However, it also seems that the explicit construction of $X_n$ is not known.

On the other hand,  Kravchuk, Maz\'{a}\v{c} and Pal \cite{KravchukMazac} introduced a new method for constructing Laplacian spectra of hyperbolic surfaces. 
The ideas are a combination of representation theory of $\mbox{PSL}_2(\mathbb{R})$, in particular, the integral of the product of automorphic forms, and semidefinite programming. They are closely inspired by the modern conformal bootstrap in physics (e.g. \cite{HartmanMSDZ}).
They gave several examples such that the upper bound of the first eigenvalues of the Laplacians is nearly sharp.
Currently, the above examples are only applicable to surfaces of low genus. However, if one could develop a method applicable to higher genus and tower of coverings, then it would be a potential explicit solution to Buser's conjecture \cite{Buser} and also provide explicit examples satisfying equality (\ref{spectralequality}). 
Moreover, they would provide an initial condition for parallel transport using projective connections in the above, and further, offer a hint to understand the moduli of the finite-dimensional unitary representations of the surface group.  

There are a few related works \cite{Bonifacio1}, \cite{BonifacioMazac}.


\end{description}

\subsection{ Quasi-morphisms, modified Riemann-Hilbert problems}\label{mRHoper}
\subsubsection{ Knot complements and quasi-morphisms}
Sarnak \cite{Sarnak} (see also Mozzochi \cite{Mozzochi}) generalized homology version of the Chebotarev theorem \ref{abel-geod} to the following form: 

Take a knot complement $Y = S^3 \setminus \tau$ with the trefoil knot $\tau$. It is homeomorphic to the noncompact $3$-dimensional homogeneous quotient space $\mbox{SL}_2(\mathbb{R})/\mbox{SL}_2(\mathbb{Z})$, which is also identified with the unit tangent bundle of the modular surface $X= \mathbb{H}/\Gamma$ with $\Gamma = \mbox{PSL}_2(\mathbb{Z})$. Since $H_1(X, \mathbb{Z}) = \Gamma/[\Gamma,\Gamma ]$ is finite and thus, no non trivial homomorphisms $f : \Gamma \to \mathbb{Z}$. 
There are, however, quasi-morphisms, that is, $f : \Gamma \to \mathbb{Z}$'s satisfying
\[|f(xy)-f(x)-f(y)| < c_f\]
for some constant $c_f < \infty$ depending only on $f$. The method of construction of $f$ will be explained later.

In this situation, he gives an asymptotic formula of prime closed geodesics for the given linking number with $\tau$ as follows:
recall that \[ \pi(x) := \sharp\{\mathfrak{p}|\ell(\mathfrak{p}) \leq x \}
\]
Similarly to $\pi(x,\alpha )$ in Theorem \ref{abel-geod}, let
\[ \pi(x; n) := \sharp\{\mathfrak{p}|\ell(\mathfrak{p}) \leq x, \mbox{link}(\mathfrak{p},\tau) = n \}
\] 
where $\mbox{link}(\mathfrak{p},\tau)$ is the linking number between $\mathfrak{p}$ and the trefoil knot $\tau$. Here, we consider $\mathfrak{p}$ as a closed orbit of the geodesic flow.

Then, we have 
\begin{theorem}[Corollary 2 in \cite{Sarnak}]\label{Corollary2Sarnak}
\[ \pi (x,n) \sim \frac{\pi(x)}{3x}\left(1+\frac{2\left(1-\left(\frac{\pi n}{3}\right)^2\right)}{x^3}\right).
\]
\end{theorem}
\begin{theorem}[Theorem 3 in \cite{Sarnak}]
For $-\infty \leq a \leq b \leq \infty$, as $x \to \infty$
\[ \frac{\sharp\{\mathfrak{p}|\ell(\mathfrak{p}) \leq x, a \leq \frac{\{{\rm link}(\mathfrak{p},\tau)}{\ell (\mathfrak{p})} \leq b \}}{\pi(x)} \sim \frac{{\rm arctan}\left(\frac{\pi b}{3}\right)-{\rm arctan}\left(\frac{\pi a}{3}\right)}{\pi}
\] 
\end{theorem}

These results seem to be inspired by E. Ghys's plenary lecture at ICM 2006  \cite{Ghys}. In this lecture, Ghys described several connections about the following topics:
\begin{center}
closed orbits of a dynamical system, the trefoil knot, the linking number, Dedekind $\eta$ function, Rademacher symbols. 
\end{center}

In particular, the linking number $\mbox{link}(\mathfrak{p},\tau)$ is equal to the value $\psi(A)$ of the Rademacher symbol $\psi$ of an element $A \in \Gamma = \mbox{PSL}_2(\mathbb{Z})$ where the conjugagy class $[A]$ of $A$ in $\Gamma$ is corresponding to the knot $\mathfrak{p}$ in the unit tangent bundle $UX = Y$. Here we identify the above knot $\mathfrak{p}$, which is the closed orbit of the geodesic flow on $UX$, with the closed geodesic in the modular surface $X$. 

Let us compare the homological version of the Chebotarev type theorem, Theorem \ref{abel-geod}, and the quasi-morphism version, Theorem \ref{Corollary2Sarnak}.

Perturbation analysis is used in the spectrum side of the Selberg trace formula in the proof of Theorem \ref{abel-geod}, in which the de Rham-Hodge theorem is crucial.  The $\pi_1$ de Rham theorem plays a similar role in the proof of Theorems \ref{theorem-geod} and \ref{conj-heat}. These theorems can be regarded as the Riemann-Hilbert problem (RH for short) to find a flat connection whose monodromy representation coincides with the given representation of the fundamental group $\pi_1(M)$. 

In this case, we study the situation starting from the existence of the surjective homomorphism $\pi_1(M) \to \Gamma$. Intuitively, the condition of homomorphy that a representation $\rho$ of $\pi_1(M)$ satisfies $\rho(\alpha\beta) = \rho(\alpha)\rho(\beta)$ expresses that the value is independent of the choice of paths of the $\rho(G)$-principal over the Cayley graph of $\pi_1(M)$, which means ``flat'' in discrete sense. Thus, RH can be viewed as a problem of finding a connection that realizes a continuous version of the above.

In this regard, the deviation in quasi-morphisms $f$ seems to be the difference of the holonomies along the path representing $f(\alpha\beta)$ and the concatenation of the paths representing $f(\alpha)+f(\beta)$, namely the ``discrete curvature''' of ``$f(G)$-principal bundle'' over the above Cayley graph. 
Thus, we would like to consider the modified Riemann-Hilbert problem (mRH for short) as follows:
\begin{problem}\label{mRH} Given a quasi-morphism $f$ of $\pi_1(M)$, find a (non-flat) connection $\omega$ whose holonomy represents $f$. 
\end{problem}

In the setting of Ghys \cite{Ghys}, the discrete curvatures are nothing but the Rademacher symbols $\psi$ defined below and the corresponding connection $\omega$ is closely related to $\log \eta$, the logarithm of the Dedekind $\eta$ function on the upper half plane $H = \{\mbox{Im}\;z > 0\}$ defined by
\[\eta(\tau ) = \exp(\pi\sqrt{-1}\tau/12)\prod_{n=1}^\infty(1 - \exp(2\pi\sqrt{-1}n\tau)).\]
In fact, for every $A = \pm\left(\begin{array}{cc}a & b \\ c & d \end{array}\right)\in \Gamma = \mbox{PSL}_2(\mathbb{Z})$, $\eta$ does not vanish. Thus, there is a holomorphic determination of $\log \eta$ defined on the upper half plane. It satisfies the following identity;
\[\eta^{24}\left(\frac{a\tau +b}{c\tau +d}\right) = \eta^{24}(\tau)(c\tau +d)^{12}.
\]
Taking logarithms on both sides, we get
\begin{equation}
24(\log \eta)\left(\frac{a\tau +b}{c\tau +d}\right) = 24(\log \eta) + 6(\log \eta)(-(c\tau +d)^2) + 2\pi\sqrt{-1}\psi(A) \label{defRademacher} 
\end{equation}
for some function $\psi:\mbox{PSL}_2(\mathbb{Z})\to \mathbb{Z}$(the second $\log$ in the right hand side is chosen with imaginary part in $(-\pi,\pi)$). 
This function $\psi$ is referred to as the Rademacher function in \cite{Ghys} and is also known as the Rademacher symbol here.

In addition, Matsuzaka and Ueki recently generalized the trefoil knot in Ghys's result to more general torus knots in \cite{Matsusaka}. In their case, the discrete curvature is also (generalized version of) Rademacher symbols $\psi_{p,q}$, the connection $\omega$ related to the harmonic Maass forms.

Since the above RH and mRH problems fall into the smooth category, which differs from the usual settings for RH problems in the holomorphic category, they may have multiple solutions. 
We need to find a criterion that is the best or a better solution. 
Although solutions such as $\log \eta$, or related objects to the harmonic Maass form, would be good ones, the equations that characterize these solutions are unknown yet.

\subsubsection{ A strategy to solve modified Riemann-Hilbert problems}

Here, we propose a candidate of such an equation based on an idea from the notion of ``oper'', which is a flat vector bundle with possibly non-flat filtrations (cf. \cite{Beilinson1}, \cite{Beilinson2}, \cite{BenZvi}).

In \cite{Matsusaka}, the authors define the (generalized) Rademacher symbol $\psi_{p,q}$ for the triangle group $\Gamma_{p,q}$ (see 2.1 \cite{Matsusaka} for its definition) as a generalization of the above $\psi$ which is written as $\psi = \psi_{2,3}$ for the usual Rademacher symbol $\psi$. 
They also give a characterization as a unique function satisfying $2pqW =-\delta^1 \psi_{p,q}$ where $W$ is the $2$-cocycle corresponding to a central extension of $\mbox{SL}_2\mathbb{R}$ by $\mathbb{Z}$ and $\delta^1$ is the coboundary operator defined in the below. 
The $2$-cocycle $W$ is defined as a map $W:\mbox{SL}_2\mathbb{R} \times \mbox{SL}_2\mathbb{R} \to \mathbb{Z}$ by 

\begin{equation}
W(\gamma_1,\gamma_2) = \frac{1}{2\pi\sqrt{-1}}(\log j(\gamma_1,\gamma_2z)+ \log j(\gamma_2,z))-\log j(\gamma_1\gamma_2,z))
\end{equation}

where $j(\gamma,z) := cz+d$ is the automorphic factor for $\gamma = \left(\begin{array}{cc}a & b \\ c & d\end{array}\right) \in \mbox{SL}_2\mathbb{R}$ and $z \in \mathbb{H}$ of some meromorphic modular forms with assuming $\mbox{arg}j(\gamma,z)= \mbox{Im} \log(j(\gamma,z)) \in [-\pi,\pi)$.

Furthermore the coboundary operator $\delta^1$ acts on a function $f:\Gamma_{p,q}\to\mathbb{Z}$ as
\[ \delta^1f(\gamma_1,\gamma_2) = f(\gamma_1\gamma_2)-f(\gamma_1)-f(\gamma_2)
\]

They also showed that $W$ satisfies the following $2$-cocycle condition:
\[   W(\gamma_1\gamma_2,\gamma_3)+W(\gamma_1,\gamma_2) = W(\gamma_1,\gamma_2\gamma_3)+W(\gamma_2,\gamma_3)
\]
and $\psi_{p,q}$ is uniquely determined by $W$. Asai introduces this $W$ \cite{Asai} and refers to it as the Asai 2-cocycle in \cite{Matsusaka}.

On the other hand, Sunada \cite{Sunada1} defined a magnetic $2$-cocycle $\Theta \in H^2(\Gamma,U(1)) $ for a discrete group $\Gamma$ in his studies of a discrete analog of periodic magnetic Schr\"odinger operators. 
It seems like to be regard $\Theta = \exp (2\pi\sqrt{-1}W)$ for $\Gamma = \Gamma_{p,q}$. 

If we take its one-parameter deformation $\Theta_\lambda = \exp (2\pi\sqrt{-1}\lambda W)$, it would correspond to an irreducible decomposition of the regular representations of the discrete Heisenberg group $\mbox{Heis}_3(\mathbb{Z})$.

As an example of the case $Y = S^3 \setminus \tau \simeq \mbox{SL}_2(\mathbb{R})/\mbox{SL}_2(\mathbb{Z})$, although there are no surjective homomorphism from $\pi_1(Y)$ to $\mbox{Heis}_3(\mathbb{Z})$, the above situation could be realized if we take a $6$-folds covering $Y(2) =  \mbox{SL}_2(\mathbb{R})/\Gamma (2)$ with the principal congruence subgroup.  
\[\Gamma(2) = \left\{\left(\begin{array}{cc}a & b \\ c & d\end{array}\right) \in \mbox{SL}_2\mathbb{R}|a\equiv d \equiv 1, b\equiv c\; (\mbox{mod} 2)\right\}.\] 
Namely, since $Y(2)$ is diffeomorphic to the unit tangent bundle $UX(2)$ of $X(2) = \mathbb{H}/\Gamma(2)$, $X(2)$ has three cusps and its fundamental group $\pi_1(X(2))$ is isomorphic to the free group $F_2$ of rank two, there is a surjective homorphism $\Phi:\pi_1(X(2))\to \mbox{Heis}_3(\mathbb{Z})$.

In this situation, the line bundle $L$ over $X(2)$ induced from above mRH problem on $X$ for a quasi-morphism $f : \Gamma \to \mathbb{Z}$ can be extended to the flat vector bundle $E$ of associated to the irreducible unitary representation $\pi_\lambda$ with $\lambda = 1$ of $\mbox{Heis}_3(\mathbb{Z})$ such that $\pi_\lambda(Z) = \psi(\gamma)$ where $Z = \Phi(\gamma) = (1,0,0)$ and $(1,0,0) = \left( \begin{array}{ccc} 1 & n_3 & n_1 \\ 0 & 1 & n_2 \\ 0 & 0 & 1\end{array}\right)$ as in the notation (\ref{3heis}).

Since $E$ is a flat vector bundle with a non-flat subbundle $L$, it is a kind of oper in the sense explained in the above paragraph. 
Then, we wish to consider the RH problem on $E$. The solutions $\omega$ are given by $\pi_1$-de Rham theorem \ref{pi1derham}. Moreover, the harmonic theory for Chen's iterated integral in Section \ref{harmoniclietheory}, especially in Conditions \ref{harmoniccoexact} (2), the choice $\omega_{12}  = \delta \omega^{(2)}$ would give a candidate of ``good solution'' of the above mRH problem.  
Some parts of the above are taken from an ongoing discussion with Toshiki Matsusaka. 
\begin{remark} 
As problems for quasi-morphisms, the following are naturally arising in a discussion with Morimichi Kawasaki: 
 
\begin{question}\label{quasi1Q}
Which quasi-morphism is the de Rham quasi-morphism?
\end{question}
\begin{question}\label{quasi2Q}
Which de Rham quasi-morphism can be virtually representable as oper?
\end{question}

Here ``virtually'' will mean ``up to passage to a finite covering''.

A prototypical example addressing both questions is the quasi-morphism $f:\Gamma = \mbox{PSL}_2(\mathbb{Z}) \to \mathbb{Z}$ introduced above. In fact, the classical Rademacher symbol entirely determines f, and one recovers it as a de Rham quasi-morphism via the logarithm of the Dedekind eta function (see paragraphs after Problem \ref{mRH}).

By passing to the $6$-folds cover $X(2) = H^2/\Gamma(2)$ of the modular surface $X(1)= H^2/\mbox{PSL}_2(\mathbb{Z})$, one obtains the quotient homomorphism $\pi_1(H^2/\Gamma(2) \to \mbox{Heis}_3(\mathbb{Z})$ which furnishes the realization sought in Question 2.
\end{remark}

\begin{remark}
Given that the above $X(2)$ is homeomorphic to $\mathbb{P}^1(\mathbb{C}) \setminus \{0, 1, \infty\}$, it is natural to anticipate deeper connections with the extensive body of literature on this space within the realms of algebraic and arithmetic geometry.
\end{remark}

\subsubsection{ Beyond the harmonicity}\label{Beyondharmonicity}
In the case of the above RH problems in homology, one criterion is the harmonicity of the closed one forms; we have already used them in connection with the Hodge theory. 

Furthermore, the following facts are noteworthy. The error estimates for the asymptotic expansions in Theorem \ref{Corollary2Sarnak} or its generalizations in \cite{Matsusaka} are much better than the usual homology version, Theorem \ref{abel-geod}. 

This kind of thing similarly appeared in the study of the heat kernel asymptotics of Theorem $6$ in \cite{Kotani3}. 
In this paper, the authors consider both continuous and discrete versions, namely, the latter being random walks on abelian covering graphs. 
They relate the above asymptotics to a certain embedding or realization of an abelian cover, which is a $\mathbb{Z}^d$-periodic manifold or graph in $\mathbb{R}^d$, referred to as the standard embedding by the authors.
This embedding is a special one in the category of harmonic maps, possessing significantly better properties than typical harmonic maps. 
This fact is also closely related to the Abel-Jacobi maps from Riemann surfaces to their Jacobi tori, which can be regarded as the quotient map of the standard embedding by $\mathbb{Z}^d$.
The properties of the above error estimates are translated as follows: the Abel-Jacobi map has the minimum energy among the family of harmonic maps from a fixed Riemann surface to a flat torus carrying the same volume.
There are several other properties of the above map that are better described, and relations to asymptotics are also described in Theorem 6 in \cite{Kotani3}.

Apart from the above, there is a strategy to choose a good candidate in the solutions of RH based on the isomonodromy deformations (cf. e.g. \cite{GuestItsLin1}, \cite{GuestItsLin2}, \cite{GuestItsLin3}). 

\subsubsection{\, Magnetic field and Heisenberg group: A pictorial representation}

Here, we explain why a flat structure appeared from a non-flat situation by adding an extra dimension. It seems known as one of folklore for experts, but we could not find it in the literature.  

Let us consider two homotopic curves $\gamma_1$ and $\gamma_2$ with common initial point $p$ and endpoint $q$ on the Riemann surface $M$. If the connection one form $\eta$ is flat, then the parallel transforms $P^{\gamma_i}(v)$ with respect to $\eta$ from $p$ to $q$ along each curve $\gamma_i$ with the initial vector $v$ coincide, however if $\eta$ is non-flat, corresponding to the constant magnetic field $\omega$, namely, satisfying $d\eta  = \omega$, then the associated parallel transforms $P^{\gamma_i}(v)$ do not coincide in general their difference are proportional to the area of the domain surrounded by the two curves $\gamma_1$ and $\gamma_2$. This fact comes from the Gauss-Bonnet formula.        

Next let us write $\eta= a_1dx_1 + a_2dx_2$ and $\eta_i = a_idx_i,$ $(i=1,2)$. If we add extra direction as in the above construction and denote $\eta_{12}$ one form corresponding to a new direction, then the condition of flatness is expressed as 
\[ d\eta_{12}= -\eta_1\wedge\eta_2 \] 
(cf. (\ref{flatomega}), here we change the notations from $\omega$'s to $\eta$'s.) Since the right-hand side $\eta_1\wedge\eta_2$ is the infinitesimal form of the area. By the effect of extra direction, the ``height'' of the endpoints of newly modified curves $\tilde{\gamma}_i$ from $\gamma_i$ becomes different, which expresses the flatness of the new space.
\begin{figure}[H]
\centering
\includegraphics[width=140mm]{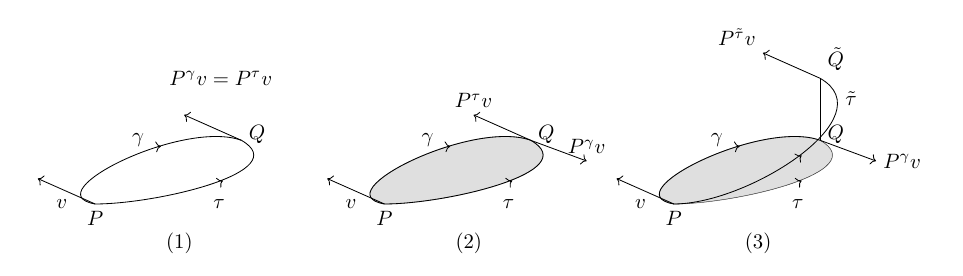}
\caption{{\footnotesize  Joku-ikou no Genri}}
\end{figure}

\begin{remark} At first glance, figure (c) seems to add only one extra dimension to figure (b). However, Chen's iterated integral -- the Lie integral of the Heisenberg algebra -- can be interpreted as an integral over the path space (see Section \ref{generaliteratedintegrals}). Thus, the transition from (b) to (c) may be viewed as a lift from the underlying space to its loop space. 
\end{remark}

\subsection{ Joku-ikou no Genri (Oka's lifting principle)}

We think that the above arguments are considered as a kind of Oka's `Joku-ikou no Genri' or`Joku-iko principle' in Oka's original word. In this subsection, we shall explain seemingly related topics in order. 

\subsubsection{ The statement of Joku-ikou no Genri in the theory of several complex variables} 


This principle is expressed as
\begin{center}
The method of the Oka Extension of holomorphic functions from a complex submanifold of a polydisk to the whole polydisk
\end{center}
In footnote $1$ of the preface of Noguchi's book \cite{Noguchi2}, it is explained as `` `Joku-Iko' is pronounced 'dz\'{o}uku ikou' (not aiko). The Joku-Iko Principle is a term due to K. Oka himself, and the guiding methodological principle of Oka theory throughout his works'' 
In footnote 2 of the same book, a detailed explanation is provided, including the reason why the author chose this terminology over `Oka's lifting principle'.
See also \cite{Noguchi1}. We also notice that it is translated as the hovering principle. 

It should be noted that this is different from the notions ``Oka's principle'' and its deviation ``Oka manifolds'' which are almost a paraphrase of the $h$-principle or homotopy principle, and one of its origins is Oka's third paper.

In the following,  we consider this principle simply as
 ``Reduction of difficulties by increasing the dimensions of spaces.'' 

\subsubsection{ Oka's own explanation for general audience}
Oka is said to illustrate this principle using the example of the Tartaglia-Cardano formula for solving cubic equations in one variable. This formula is derived by introducing a change of variables, where the original variable $x$ is replaced with two new variables, $u$ and $v$.

\subsubsection{ Chern's intrinsic proof of the Gauss-Bonnet formula}
Recall that the generalized Gauss-Bonnet formula says that the Euler-Poincar\'{e} characteristic $\chi(M)$ of $2n$-dimensional closed Riemannian manifold $M$ 
can be expressed as 
\begin{equation*}
\chi(M) = \frac{1}{(2\pi)^n}\int_M\mbox{Pf}(\Omega)
\end{equation*}
where $\mbox{Pf}(\Omega)$ is the Pfaffian of the curvature form $\Omega$ of $M$.
The crucial arguments of the proof in \cite{Chern} is based on the fact the pullback $\pi^\ast(\mbox{Pf}(\Omega))$ of $\mbox{Pf}(\Omega)$ by the canonical projection $\pi:UM \to M$ from the unit tangent bundle $UM$ to $M$ is exact form on $UM$ and application of the Stokes formula to this form. It can be seen as an example of `Joku-ikou no Genri' in the above sense. 

\subsubsection{ Loop space arguments related to KdV hierarchy and oper: Introduction of spectral parameters}
The above arguments would also explain other situations, such as the loop space arguments related to KdV hierarchy \cite{DrinfeldSokolov}, \cite{SegalWilson}, \cite{BenZvi}, in introducing spectral parameters. We quote several sentences from the introduction of \cite{BenZvi} as an explanation:

\begin{quotation}
An integral system naturally expressed in terms of differential data (differential operators, flat connections, etc.) may be characterized using spectral data, on which the flows become linear and group-theoretic, and sometimes geometric, in significance. $\cdots$

Drinfeld and Sokolov \cite{DrinfeldSokolov} also incorporated a spectral parameter into their connections -- thereby providing a generalization of the eigenvalue problem for a differential operator. They showed that the resulting connections (a loop group version of opers, which we refer to as affine opers) can be brought into a canonical gauge, where an infinite-dimensional abelian group of symmetries becomes apparent. Using these symmetries, it is easy to write a commuting hierarchy of flows as "zero-curvature equations" constraints, expressing the flatness of the connection, when extended to new variables using the symmetries.
\end{quotation}

Moreover, the Heisenberg group plays a central role in the core argument of \cite{BenZvi}, appearing as a subgroup of the loop group. Although these subgroups are abelian, the terminology ``Heisenberg'' is justified by the fact that their pullback to the Kac-Moody central extension of the loop group yields a genuine Heisenberg group. Figure 5 illustrates this relationship schematically.

\subsubsection{ Loop groups and harmonic maps:  DPW method}
In the theory of harmonic maps from Riemann surfaces to a compact symmetric space, there are several methods related to loop groups: harmonic maps of finite uniton number, harmonic maps of finite type, and the DPW method. We explain only the last one. For the former method, please refer to Guest's book  \cite{Guest}.

The classical Weierstrass-Enneper representation formula
has been a very useful tool for the construction and investigation of
minimal surfaces in $\mathbb{R}^3$. 
While the differential equation describing these surfaces is highly nonlinear, the data of the formula (Weierstrass data), consisting of a pair consisting of a holomorphic function $f$ and a meromorphic function $g$, is completely unconstrained and solutions of linear $\bar{\partial}$-equations.

In the late 90s, a generalization of these formulas for harmonic maps $\varphi$ from a Riemann surface $M$ to a Riemannian symmetric space of compact type $G/K$ was discovered (\cite{Dorfmeister02}, see also \cite{Dorfmeister01}), which is usually called the DPW formula, named after Dorfmeister, Pedit, and Wu. Rather than tackling the nonlinear PDE of harmonic maps directly, one encodes the map in a loop valued $1$-form (the DPW potential), integrates a linear system, and then recovers the real solution via a loop group factorization (Iwasawa decomposition). 
  
Let us explain this formula briefly from the viewpoint of the construction of the harmonic map $\phi$ from the DPW potential $\eta$.
\begin{description}
\item[Step 1](DPW potential) 
Assume that we have the following meromorphic 1-form $\eta := \eta(z)\,dz$ defined as 
\begin{equation}
\eta := \eta(z)\,dz = \bigl(\lambda^{-1}\eta_{-1}(z)+\eta_0(z)+\lambda\,\eta_1(z)+\cdots\bigr)dz \label{etaDPW}
\end{equation}  
valued in the loop algebra $\Lambda \mathfrak{g}^\mathbb{C}$ of the loop group 
\[\Lambda G^\mathbb{C} =\{A:S^1\to G^\mathbb{C}\}  \] 
over a complexified group $G^\mathbb{C}$ of $G$. 
The single $\lambda^{-1}$ term $\eta_{-1}(z)\,dz$ captures the core geometry (analogous to the Higgs field). 
This $\eta$ is called the DPW potential.

\item[Step 2](Integral of the extended holomorphic frame)  
We can solve the following equation of the parallel transport on the universal cover $\widetilde{M}$ of $M$  
\[
      dC = C\,\eta(z)\,dz, \quad  C(z_0)=I\; (\mbox{identity}).
\]  
so that the solution $C(z,\lambda)$ is holomorphic in $z$ and meromorphic\; in $\lambda$ and called the extended holomorphic frame of $\phi$.
\item[Step 3](Iwasawa Splitting) 
At each $z \in \mathbb{C}$, $C$ can be factored as 
\[
      C(z,\lambda)=F(z,\bar{z},\lambda)\,B(z,\lambda).
\]  
Here $F$ lies in the real loop group
\[\Lambda G := \{A \in \Lambda G^\mathbb{C}| \mbox{a reality condition} A(-\lambda) = A(\lambda)\; \mbox{on}\; |\lambda|=1\}\] 
and $B$ in the positive loop group 
\[\Lambda^+G^\mathbb{C} := \{A \in \Lambda G^\mathbb{C}|\mbox{holomorphic on}\; |\lambda|<1 \}\].

\item[Step 4](Construction of harmonic map)  
To obtain the above harmonic map $\phi$, we evaluate $F$ at $\lambda=1$ and project to $G/K$:  
\[
      \phi(z,\bar{z})=F(z,\bar{z},1)\quad \mbox{mod}\; K.
\]  
\end{description}

As already mentioned in Remark \ref{harmonicrepresentative}, Dorfmeister, Inoguchi and Kobayashi \cite{Dorfmeister1}, \cite{Dorfmeister2} developed this method to apply minimal surfaces into ${\rm Heis}_3(\mathbb{R})$. 

\subsubsection{ DPW vs. Nonabelian Hodge}

In addition to the preceding subsection, the DPW method has the Higgs bundle descriptions in nonabelian Hodge theory.

The nonabelian Hodge theory establishes an equivalence between flat bundles (local systems) and Higgs bundles on a compact K\"ahler manifold. At their intersection, the DPW potential encodes precisely the Higgs field and its adjoint under the harmonic metric. The loop group factorizations in DPW mirror the splitting given by the Hitchin equations and the associated harmonic metric.

\paragraph{Nonabelian Hodge correspondence}

A central result of nonabelian Hodge theory is that for a compact Riemann surface $M$, there is a natural one-to-one correspondence between  
semisimple flat complex vector bundles $(E,\nabla)$, and polystable Higgs bundles $(E,\varphi)$ with vanishing Chern classes,  
mediated by a unique harmonic metric solving Hitchin's self-duality equations. 
Under this correspondence, the flat connection $\nabla$ decomposes as 
 \[ \nabla = D' + D'' + \varphi + \varphi^\ast,
\]  
where $\varphi \in H^0(X,\mbox{End}(E)\otimes K)$ is the Higgs field and $\varphi^\ast$ its adjoint with respect to the harmonic metric.
\paragraph{DPW potentials as Higgs fields}

In the DPW construction for a harmonic map into a symmetric space $G/K$, the leading coefficient $\eta_{-1}(z)$ of the DPW potential $\eta$ in (\ref{etaDPW}) plays the role of the Higgs field $\varphi$. At the same time, the higher-order (nonnegative) terms correspond to the adjoint parts $\varphi^\ast$ and additional gauge corrections required by the harmonic metric. Integrating $\eta$ yields a holomorphic frame whose Iwasawa splitting recovers the unitary Chern connection and the Higgs field exactly as in nonabelian Hodge theory.

\paragraph{Loop group factorization and harmonic Metrics}

The Iwasawa decomposition  
\[
  C(z,\lambda) = F(z,\bar{z},\lambda)\,B(z,\lambda) ,
  \quad
  F \in \Lambda G,\; B \in \Lambda^+G^\mathbb{C}
\]  
realizes, at each point $z$, the separation of the flat connection into its unitary part $F$ (encoding the harmonic metric and Chern connection $D'+D''$) and its nilpotent remainder $B$ (recovering $\varphi,\varphi^\ast$). This mirrors the way solving Hitchin's equations splits a flat connection into Chern plus Higgs parts, and shows DPW as an explicit loop group avatar of the nonabelian Hodge correspondence.

\paragraph{Monodromy, stability, and normalized potentials}
To descend from the universal cover to the surface $M$, one requires that the monodromy of the integrated frame $C$ lies in the real loop group. This condition is equivalent to the polystability of the underlying Higgs bundle and unitarity of the corresponding local system.  
Conversely, given a stable Higgs bundle $(E,\varphi)$ with its harmonic metric, one can produce a normalized DPW potential by expanding the flat connection in powers of $\lambda$ and gauging away all nonnegative powers except the $\lambda^{-1}$ term.

Through these steps, every polystable Higgs bundle gives rise to a unique DPW potential up to gauge, and every suitably normalized DPW potential recovers a solution of the Hitchin equations, thereby cementing the bridge between loop group methods and nonabelian Hodge theory.

\subsubsection{ Steinberg symbol, Milnor $K_2$, the dilogarithm, Heisenberg groups and extensions of Lie algebras}
Below is an overview of how these four structures interweave through algebraic K-theory, regulators, central extensions, and infinite-dimensional Lie theory.
\begin{description}
\item[Steinberg Symbol and Milnor $K_2$]
The Steinberg symbol  
\[\{x,y\}\;\in\;K_2(F)
\]  
for a field $F$ is the generator of the second Milnor $K$-group subject to  
\[
\{x,1-x\}= 1 \quad (\text{Steinberg relation}).
\]  
It encodes a universal $2$-cocycle on $F^\times$. In particular, any bimultiplicative pairing  
\[
c\colon F^\times\times F^\times\;\longrightarrow\;F^\times
\]  
satisfying $c(x,1-x)=1$ factors uniquely through $K_2(F)$.
\item[Dilogarithm as Regulator]
The classical dilogarithm  
\[
\mbox{Li}_2(z)\;=\;\sum_{n=1}^\infty\frac{z^n}{n^2}
\]  
appears as the real-valued regulator on $K_2$ whose five-term functional equation  
\[
\mbox{Li}_2(x)\;-\;\mbox{Li}_2(y)\;+\;\mbox{Li}_2\!\bigl(\tfrac yx\bigr)\;-\;\mbox{Li}_2\!\bigl(\tfrac{1 - x^{-1}}{1 - y^{-1}}\bigr)\;+\;\mbox{Li}_2\!\bigl(\tfrac{1 - x}{1 - y}\bigr)\;=\;0
\]  
mirrors exactly the Steinberg relation in $K_2$. Note that this five-term relation defines the Bloch group, which is closely related to $K_3(F)$ and regulators.
Passing to the Bloch-Wigner version  
\[
D(z)= \mbox{Im}\bigl(\mbox{Li}_2(z)\bigr)+\mbox{arg}(1-z)\,\mbox{ln}|z|
\]  
yields a well-defined map  
$\,K_2(\mathbb{C})\to\mathbb{R}$.

\item[Heisenberg Group (Multiplicative Heisenberg-type group) as a Central Extension]

Let  
\[
\mbox{Heis}_3(F^\times)\;=\;F^\times \times F^\times \times K_2(F)
\]  
with multiplication  
\[
(x_1,y_1,a_1)\,(x_2,y_2,a_2)
\;=\;\bigl(x_1x_2,\;y_1y_2,\;a_1a_2\{x_1,y_2\}\bigr).
\]  
This formula is a multiplicative expression of the Heisenberg group $\mbox{Heis}_3(F)$ over $F$ whose central cocycle is precisely the Steinberg symbol. It generalizes the usual position-momentum Heisenberg group $\mbox{Heis}_3(\mathbb{R})$(extension as additive group $\mathbb{R}\times\mathbb{R}$) in quantum mechanics, replacing $\mathbb{R}$ by $F^\times$ and the integer central charge by $K_2(F)$.

\item[Kac-Moody Lie Algebras and Loop Extensions]

Affinizing a finite-dimensional simple Lie algebra $\mathfrak{g}$ produces the loop algebra  
$\mathfrak{g}[t,t^{-1}]$ and its universal central extension  
\[
\widehat{\mathfrak{g}}
\;=\;\mathfrak{g}[t,t^{-1}]\;\oplus\;\mathbb{C}\,c
\;\oplus\;\mathbb{C}\, d,
\]  
where the central element $c$ arises from an explicit $2$-cocycle on Laurent polynomials.  In that setting, the cocycle can be viewed as taking values in the second Lie algebra cohomology, an avatar of $K_2$ of the function field $\mathbb{C}(t)$. The subalgebra generated by the Cartan loop modes and $c$ is itself a Heisenberg algebra.  

Dilogarithm (and its quantum version) features prominently in constructing intertwiners and $R$-matrices for quantum affine algebras, tying Kac-Moody representation theory back to Steinberg symbol regulators.
\end{description}

This kind of structure similarly appeared in several symbols, such as Modular symbol \cite{JPSerre0}, Tame symbol \cite{Deligne}, \cite{Tamiozzo}, Contou-Carr\'{e}re symbol \cite{ContouCarrere}, and their two-dimensional version \cite{OsipovZhu}. Moreover,  Rademacher symbol, in particular its higher version \cite{Duke1}, \cite{OSullivan} can also use these cohomological descriptions using the Eichler-Shimura cohomology.

\subsubsection{\; Quantum/Affine correspondence: the Peterson-Lam-Shimozono-Kato Framework}

Figure 5 offers a schematic portrayal of the Quantum/Affine correspondence. In the central panel, the effect of a perpendicular magnetic field is emphasized, illustrating the ``quantum'' side of the story -- much like the interplay between the quantum torus and the Hofstadter butterfly. In the rightmost panel, a vertical line marks the loop parameter that underlies affine Lie algebras, thereby capturing the ``affine'' aspect of the correspondence.

This construction suggests potential connections to the following version of this correspondence:

Peterson observed that the $T$-equivariant quantum cohomology ring of a flag variety $G/P$ is isomorphic, as a graded algebra, to the Pontryagin homology of the affine Grassmannian $\mathrm{Gr}_G$. Concretely, quantum multiplication by a Schubert divisor in $QH^\ast_T(G/P)$ matches Pontryagin product by an affine Schubert class in $H_\ast^T(\mathrm{Gr}_G)$. This deep link identifies Gromov-Witten invariants with intersection numbers on the affine Grassmannian, laying the foundation for the quantum/affine correspondence.

Lam and Shimozono made Peterson's map explicit by showing that the quantum Chevalley formula in $QH^\ast_T(G/P)$ admits a direct ``affine analogue'' in $H_\ast^T(\mathrm{Gr}_G)$. Their theorem identifies structure constants in the quantum Schubert basis with those in the affine Schubert basis via an algebra homomorphism constructed using generalized Seidel representations and equivariant Gromov-Witten theory.

The quantum cohomology ring $QH^\ast_T(G/B)$ can be realized as the coordinate ring of the universal centralizer variety of $G$, which is integrable by the (equivariant) Toda lattice. 

Through Peterson's isomorphism, these Toda Hamiltonians act naturally on the equivariant homology of the affine Grassmannian, revealing that the affine side inherits an integrable system structure isomorphic to quantum cohomology.

Moreover, Syu Kato \cite{KatoS} extended the quantum/affine correspondence into the realm of (equivariant) quantum $K$-theory.

\subsubsection{\; Construction of quasicrystal: Cut and project}
There is a vast amount of references for quasicrystals.
Here, we select a few of them, such as \cite{WalterDeloudi} and \cite{Dyson}, and summarize information mainly from Wikipedia, commenting on our viewpoints. 

There are several ways to define quasicrystalline patterns mathematically. 
Here we recall the `cut and project' construction, which is based on the work of Harald Bohr. In this method, a quasicrystal is obtained by taking a (irrational) slice (an intersection with one or more hyperplanes) of a higher-dimensional periodic lattice. The relation slices and higher-dimensional lattice look like 'Joku-iko', and it would be interesting to investigate from this viewpoint.

De Bruijn showed that Penrose tilings can be viewed as two-dimensional slices of five-dimensional hypercubic structures \cite{deBruijn}.
Similarly, icosahedral quasicrystals in three dimensions are projected from a six-dimensional hypercubic lattice, as first described by Peter Kramer and Roberto Neri in 1984 \cite{KramerNeri}. 
Equivalently, the Fourier transform of such a quasicrystal is nonzero only at a dense set of points spanned by integer multiples of a finite set of basis vectors, which are the projections of the primitive reciprocal lattice vectors of the higher-dimensional lattice \cite{SuckSchreiberHaussler}

\subsection{\; Opers, Toda-type equations, Gaiotto conjecture}\label{operGai}

In \cite{Gaiotto1}, Gaiotto studies a four-dimensional $\mathcal{N}=2$ supersymmetric gauge theory of class $S$ compactified on a circle of radius $R$. Its three-dimensional moduli space $\mathcal{M}$ is hyperk\"ahler, and its metric can be encoded in a system of Thermodynamic Bethe Ansatz (TBA)-like integral equations. This space $\mathcal{M}$ roughly coincides with the moduli spaces of solutions to Hitchin's equations, which play a crucial role in both mathematical physics and mathematics. 
He conjectured that if we take a careful limit (``conformal limit'') of the above TBA-like integral equations, then, for theories of the class $S$, this limit the complex symplectic manifold $\mathcal{M}$ collapses to what Gaiotto identifies with a family of opers, i.e. flat connections gauge-equivalent to a single differential operator on the punctured Riemann surface $C$.

In the case of an unpunctured Riemann surface $C$, Dumitrescu et al. \cite{Dumitrescu} give a mathematical proof of this conjecture through the idea of exact WKB analysis and non-abelian Hodge theory. Moreover, the above oper gauge-equivalent to a differential operator is closely related to $tt^\ast$-geometry due to Cecotti and Vafa \cite{CecottiVafa1}, \cite{CecottiVafa2},\cite{CecottiVafa3}, in particular $tt^\ast$-Toda equation as an example. A mathematical theory around this equation and its generalization is developed by Guest et al. \cite{GuestItsLin1}, \cite{GuestItsLin2}, \cite{GuestItsLin3}, and also Mochizuki \cite{Mochizuki1}, \cite{Mochizuki2}.

On the other hand, the Harper equation, appearing in the Hofstadter butterfly, is equivalent to almost the Mathieu equation. Here, there are three types of Toda-like equations
 \begin{itemize}
\item almost Mathieu equation (Harper equation, discrete magnetic Laplacian) 
\begin{equation}  \psi(n+1) + \psi(n-1) + 2\cos (2\pi n\alpha + \nu)\;\psi(n) = 0 \label{Harper}
\end{equation}
\item (classical) Toda equation
\begin{equation}  \psi(x+i\hbar) + \psi(x-i\hbar) + 2\cosh x\;\psi(x) = 0 \label{cToda}
\end{equation}
\item $tt^\ast$-Toda equation
\begin{equation}  \psi(x+i\hbar) + \psi(x-i\hbar) + 2\sinh x\;\psi(x) = 0 \label{ttToda}
\end{equation}
\end{itemize}

Upon reviewing this list, we infer that there may be similar theories for the Harper equation. There seems to be a hope, since the harmonic oscillators are appearting both in $tt^\ast$-Toda and almost Mathieu equation: One in former is the $\mbox{AD}_2$ (Argyres-Douglas $\mathcal{N}=2$ super conformal field theory)-limit associated to the $\mbox{A}_1$ spectral curve with a rank $2$ irregular singularity at infinity in 3.1 in \cite{Gaiotto1} and one in the latter is in the Wilkinson formula in (\ref{introformal}).
Katsushi Ito kindly informed us \cite{Imaizumi} is an example in this direction.

\subsection{ Non-perturbative correction, Resurgence and exact WKB analysis}

One aspect of our analysis of nilpotent Floquet-Bloch theory is grounded in the arguments developed in our alternative mathematical proof of the Wilkinson formula for the Hofstadter butterfly, as presented in Section \ref{Hofstadterbutterfly} and Chapter \ref{Anotherproofwilkinson}.
This formula is expressed as a perturbative series. Recent developments in this area have revealed intriguing progress, while several open problems remain. In particular, non-perturbative corrections to this series have been studied using exact WKB analysis and resurgence theory. A valuable reference on these topics is \cite{GuXu}.

From the perspective of this paper, we propose the following possibilities:

\begin{description}
 \item[{\rm (1)}] Our proof of the Wilkinson formula justifies the formal Taylor expansion (\ref{introformal}) of (\ref{introhtheta}) via approximations using finite-dimensional operators. This approach is based on the decomposition formula (\ref{introdiscretetoHeisenbergLie}) and the structure of finite-dimensional representations (\ref{intromatrixrep}).
Since the decomposition formula is an exact identity, it may also serve as a useful tool for non-perturbative analysis.

\item[{\rm (2)}] As discussed in \cite{Beguin} and further elaborated in this paper, the Hofstadter butterfly is closely related to unitary representations of the discrete Heisenberg group. A natural question arises: can the aforementioned developments be generalized to the analysis of unitary representations of discrete nilpotent groups? \end{description}

\subsection{ Chebyshev bias in geometry} \label{Chebyshevbias}
According to the prime number theorem version of Dirichlet theorem, which is a particular case of the original Chebotarev density theorem, the ratio $\pi(x,\alpha;\ell)$ and $\pi(x,\beta;\ell)$ tends to $1$ as $x \to \infty$ where $\pi(x,\alpha;\ell)$ denote the cardinal of the set of prime numbers $p$ with $p \equiv \alpha$ modulo $\ell$ as usual.
 Then, one of our following interests in asymptotics is the difference. Chebyshev was the first person to consider such a problem. He asked the case $\ell = 4, \alpha =1, \beta = 3$ and conjectured that $\pi(x,3;4)$ is larger than $\pi(x,1;4)$, which is called Chebyshev bias or the prime number race in later.

In \cite{Rubinstein}, Rubinstein and Sarnak showed that roughly $99.59\%$, $\pi(x,3;4) > \pi(x,1;4)$ holds, but at $x\to\infty$ there is an interval where those magnitudes change forever, namely the prime number race cannot be settled, under the Riemann Hypothesis and simple zero conjecture.
 
Recently, Miho Aoki, Ikuya Kaneko, Shinya Koyama, and Nobushige Kurokawa have obtained interesting results related to the deep Riemann hypothesis (\cite{Aoki}, \cite{KoyamaKurokawa},  \cite{KaneKoyama}). However, the method of counting is different from the above.

In \cite{Rubinstein}, the authors noted that the above phenomenon does not occur in the case of counting prime closed orbits $\pi(x,\alpha)$ for infinite abelian extensions, as stated in Theorem \ref{abel-geod}.

Suppose the ``sizes'' of the homology classes are different. In that case, we can verify this by examining the coefficients of the second term in the asymptotic expansion (cf. \cite{Kotani1}). Still, if their ``sizes'' are equal, then it cannot be distinguished by the second term alone and thus needs to see higher order terms. 

Although it may seem complicated to verify that infinitely changing phenomena actually occur, there is a possibility that the winner of the race depends on each element in the moduli $\mathcal{M}$ of hyperbolic metrics on the fixed Riemann surface $M$. This phenomena could be seen by taking ``tropical limit'', which is a graph described as follows: If we take a sequence $\{g_n\}$ of hyperbolic metrics on $M$ which goes to the boundary $\partial\mathcal{M}$ of a canonical compactification $\widehat{\mathcal{M}}$ of $\mathcal{M}$, then diameters of  Riemann surfaces $(M,g_n)$ diverges to infinity which can be seen by the colloring lemma (cf. \cite{Burger}, \cite{Odaka} . Next, we normalize Riemannian metrics $g_n$ to $\bar{g}_n = c_ng_n$ by constant $c_n > 0$ so that diameters of $(M,\bar{g}_n)$ are always to be $1$, which implies the curvature diverges to $-\infty$. Then the Riemannian measures $\mu_n$ on $(M,\bar{g}_n)$ converge to point mass, i.e., the Dirac measures supported on nodes of the limiting objects. These limiting objects can be viewed as a combinatorial graph, and it is sometimes the tropical limit of hyperbolic Riemann surfaces. 

Note that the limiting behavior of the spectral side of the Selberg trace formula is already investigated in \cite{Burger}. 
However, that of the geometric sides remains obscure. It seems to be helpful to see the relations to the arguments in Bismut et al. \cite{BismutLebeau}, \cite{Bismut}, which give a new proof of the Selberg trace formula using deformation arguments of the hypo-elliptic operators, which is ``a quantization of Witten deformation''. Additionally, there are possibilities for making a more direct connection between the Selberg and Ihara zeta functions in this manner.

Concerning a degenerating family of Kleinian Schottky groups, recently Li, Matheus, Pan, and Tao \cite{LiMPT} investigated a similar topic as indicated in the title of their paper ``Selberg, Ihara and Berkovich''.

\subsection{ Topics around the Kashaev invariant and the volume conjecture}
As mentioned in the last paragraph of Section \ref{Hofstadterbutterfly}, we think our method would have an advantage compared with the previous method in \cite{Helffer} and expect it could be applied to other asymptotics. Among several conjectures of asymptotics in knot theory, the volume conjecture \cite{Kashaev2},  \cite{Kashaev1} seems to be the most well-known. It connects the hyperbolic volume of the complement of the knot with the Kashaev invariant or the colored Jones polynomials \cite{Murakami}, which bridges classical and quantum invariants of the knot. This conjecture
concerns a semiclassical expansion of these invariants determined from finite-dimensional representations of a quantum group $U_qsl_2$, and a hyperbolic volume appears in this expansion around the trivial representation. An intuitive geometric interpretation already exists, e.g., \cite{Yokota}; however, it has yet to be rigorously proven, except for several concrete examples. 
Here, we note that the dimensions of the above representations diverge to infinity when they tend to the trivial representation, which is similar to our arguments for the proof of heat kernel asymptotics or the Chebotarev density theorem for nilpotent extensions.
As explained in the simplest exmple of the mathematical justification of the formal proof of Wilkinson formula (\ref{introeigen1}), namely by identifying the Harper opeator $H_\theta: \ell^2(\mathbb{Z}^2) \to \ell^2(\mathbb{Z}^2)$ defined on $\ell^2(\mathbb{Z}^2)$ with the operator 

The following are hints to relate our previous arguments and topics in this subsection:

\begin{description}
\item A mathematical justification for Hikami's argument for the state integral, but instead of ${\rm Heis}_3(\mathbb{Z})$, we need an analysis of the discrete nilpotent group $G_{6,4}$ in the later notation of step $3$, applying the previous study of asymptotic behavior on discrete nilpotent groups.

\item Hofstadter butterfly is investigated through the Bethe Ansatz of $U_q(sl_2)$ \cite{WiegmannZabrodin} and possibly, has a connection to positive representations of $U_q(sl_2)$ \cite{PonsotTeschner}.

\item Quantum cluster algebra by Fock and Goncharov\cite{FockGoncharov1}, \cite{FockGoncharov2} and related research on the latter and the Kashaev invariants by Hikami and Inoue \cite{HikamiInoue1}, \cite{HikamiInoue2}, \cite{HikamiInoue3}. A more structural understanding may be obtained.
\end{description}

We start recalling Hikami's arguments in more detail. In \cite{Hikami}, he defined the partition functions $Z_\gamma(\mathcal{M})$ for the cusped hyperbolic $3$-manifold $\mathcal{M}$ using the state integrals, which he introduced and studied from a formal point of view, giving a formal proof of a volume conjecture through the state integrals which he defined.

First, we point out several problems with the justification of Hikami's formal arguments using several terminologies defined later.
\begin{problem}\label{remained}
\begin{description}
\item[{\rm (a)}] Rigorous definition of the partition function $Z_{\gamma}(\mathcal{M})$:
\begin{itemize}    
\item Convergence of the state integrals,
\item Independence of the choice of triangulations in the definition of the state integrals.
\end{itemize} 
\item[{\rm (b)}] Semiclassical asymptotics (by a formal Taylor expansion):
\[ \lim_{\gamma \to 0} 2\gamma \log (Z_{\gamma}(\mathcal{M})) = \mbox{Vol}(M) + \sqrt{-1}\mbox{CS}(M) \qquad \mbox{(3.19) in \cite{Hikami}}.
\]
\item[{\rm (c)}] Comparison between Hikami's version and Kashaev's original version of volume conjectures.
\end{description}
\end{problem}

An answer to the problem (a) is already given in \cite{Andersen} under some conditions on homologies of the complement of the vertices of the triangulation. Their version of the partition function $Z_{\gamma}(\mathcal{M})$ in the framework of the category of tempered distributions. 
The authors of \cite{Andersen} comment that the main obstacle in constructing such a TQFT comes from the fact that the target category cannot be the category of finite-dimensional vector spaces. A similar impossibility of realization by finite-dimensional objects also appears in the Wilkinson formula (\ref{introeigen1}). 

Our primary concern in this topic is to give some explanations to the problem (b) from our viewpoint, and hopefully, our arguments shed some light on the problem (c).

However, there remained a serious problem for the choice of the contour where the integrals are defined as discussed in e.g. \cite{DimofteGLZ}.
We have learned this point from Yuji Terashima and Kazuhiro Hikami.

\subsubsection{ The Faddeev integral and S-operators}\label{Faddeevintegral}
The partition function $Z_\gamma(\mathcal{M})$ is constructed based on an oriented ideal triangulation by assigning to each tetrahedron the quantum dilogarithm function, introduced by Faddeev, together with so-called $S$-operators.
 
The relation to the original invariant of Kashaev is explained in the Introduction of \cite{Hikami} as follows:

\begin{quotation}
Originally, Kashaev introduced his invariant $J_{\mathcal{K}}(N;e^{2\pi\sqrt{-1}/N})$
for triangulated $3$-manifolds, although the $R$-matrix construction is developed subsequently \cite{Kashaev1}. He studied the Faddeev
quantum dilogarithm function when $q$ is a root of unity \cite{Faddeev1}, \cite{Kashaev0}, (see also \cite{Bazhanov1}), and assigning the quantum
dilogarithm function to the ideal tetrahedron, he defined invariants \cite{Kashaev1}, \cite{Kashaev2}. In this sense, our function $Z_\gamma(\mathcal{M})$ for $\mathcal{M}$ with a complete hyperbolic structure can be regarded as a non-compact $U_q(sl(2,\mathbb{R}))$
analogue of the Kashaev invariant. 
Though we do not know the actual content of our partition function and we have no rigorous proof on convergence, the asymptotic behavior of $Z_\gamma(S^3\setminus\gamma)$ in the limit $\gamma \to 0$ is expected to coincide with that of the Kashaev invariant $J_{\mathcal{K}}(N;e^{2\pi\sqrt{-1}/N})$ in the limit $N \to \infty$, as will be discussed below.

\end{quotation}

Let us recall several notions and terminologies that appeared in \cite{Hikami} in order.

We define a function $\Phi_\gamma(\varphi )$ by an integral form following \cite{Faddeev0}. 
We set $\gamma \in \mathbb{R}$ and for $|\mbox{Im}\; \varphi| < \pi$, 
\begin{equation} \Phi_\gamma(\varphi ) = \exp\left(\int_{\mathbb{R}+\sqrt{-1}0}\frac{e^{-\sqrt{-1}\varphi x}}{4\sinh(\gamma x)\sinh(\pi x) }\frac{dx}{x} \right) \label{qdilog}
\end{equation}
The Faddeev integral (\ref{qdilog}) is also called the quantum dilogarithm function.

The most important property of the Faddeev integral is that it fulfills the pentagon identity \cite{Faddeev1, Faddeev0},
\begin{equation}  \Phi_\gamma(\hat{p} )\Phi_\gamma(\hat{q})=\Phi_\gamma(\hat{q})\Phi_\gamma(\hat{p}+\hat{q})\Phi_\gamma(\hat{p}) \label{pentagon}
\end{equation}
where $\hat{p}$ and $\hat{q}$ are the canonically conjugate operators satisfying the Heisenberg commutation relation,
\[  [\hat{p}, \hat{q}] = \hat{p}\hat{q}-\hat{q}\hat{p} = -2\sqrt{-1}\gamma
\]
By this commuting relation, we call a limit $\gamma \to 0$ a classical limit. Hereafter, we use $V$
as the momentum space $|p\rangle$ with $p \in \mathbb{R}$ which is an eigenstate of the momentum operator:
\[ \hat{p}|p\rangle = p|p\rangle .
\]

A reason for the quantum dilogarithm function is revealed when we take a classical limit $\gamma \to 0$. In this limit, the Faddeev integral reduces to
\[  \Phi_\gamma(\varphi ) \sim \exp\left(\frac{1}{2\sqrt{-1}\gamma}Li_2(-e^\varphi)\right)
\]
Here,
$Li_2(x)$ denotes the Euler dilogarithm function defined by 
\[Li_2(x) =\sum_{n=1}^\infty\frac{x^n}{n^2}
\]
where $|x| < 1$.

To see a relationship between the integral $\Phi_\gamma(\varphi )$
and geometry, we define the S-operator acting on $V\otimes V$ by 
\begin{equation}S^{\varepsilon_1}=S_{1,2}(p_1,p_2,q_1,q_2) = e^{\frac{1}{2\sqrt{-1}\gamma}\hat{p}_1\hat{q}_2}\Phi\gamma(\hat{p}_1-\hat{q}_2 - \hat{p}_2)
\label{Soperator}
\end{equation}
Here the Heisenberg operators $\hat{p}_j$ and $\hat{q}_j$ act on the $j$-th vector space of $V\otimes V$, i.e. $\hat{p}_1 = \hat{p}\otimes 1$,  $\hat{p}_2 = 1\otimes\hat{p}$, and so on.
Moreover, we also express
$\hat{q}_j = \sqrt{2\gamma}p_j$ as a multiplication operator and $\hat{p}_j = \sqrt{2\gamma}\frac{\partial}{\partial p_j}$ as a derivation operator with $[\hat{p}_i, \hat{q}_j] = -2\sqrt{-1}\gamma\delta_{ij}$. 
Then we notice that $\hat{p}_1\hat{q}_2, \hat{p}_1, \hat{p}_2, \hat{q}_1, \hat{q}_2$ generates a nilpotent Lie algebra $\mathfrak{n}$ of step 3 by the commutation formula $[\hat{p}_1\hat{q}_2, \hat{p}_2] = -2\sqrt{-1}\gamma\hat{p}_1, [\hat{p}_1\hat{q}_2, \hat{q}_1] = 2\sqrt{-1}\gamma\hat{q}_2$, which is of six dimension and is denoted by $\mathfrak{g}_{6,4}$ in the list of classification of low dimensional Lie algebras in \cite{Nielsen}. We shall discuss its unitary representations in Example \ref{g64}. 

Then, the pentagon identity (\ref{pentagon}) can be rewritten in a compact form:
\begin{equation}
S_{2,3}S_{1,2}= S_{1,2}S_{1,3}S_{2,3} \label{Spenta}
\end{equation}
where $S_{j,k}$ acts as $S_{1,2}$ on the $j$ and $k$-th spaces of $V\otimes V\otimes V$ and as the identity on the rest. Matrix elements of infinite dimensional operator $S_{1,2}$ can be computed as
{\small \begin{align*}
&{} \langle p_1,p_2|S_{1,2}|p_1',p_2'\rangle \\ &=  \frac{1}{\sqrt{4\pi\gamma}}\delta(p_1+p_2-p_1')\Phi_\gamma\left(p'_2-p_2+\sqrt{-1}\pi+\sqrt{-1}\gamma\right)e^{\frac{1}{2\sqrt{-1}\gamma}\left(-\frac{\pi^2+\gamma^2}{6}-\frac{\gamma\pi}{2}+p_1(p'_2-p_2)\right)} \\
&{} \langle p_1,p_2|S_{1,2}^{-1}|p_1',p_2'\rangle \\ &= \frac{1}{\sqrt{4\pi\gamma}}\delta(p_1-p_1'-p_2')\frac{1}{\Phi_\gamma\left(p_2-p'_2-\sqrt{-1}\pi-\sqrt{-1}\gamma\right)}e^{\frac{1}{2\sqrt{-1}\gamma}\left(\frac{\pi^2+\gamma^2}{6}+\frac{\gamma\pi}{2}-p'_1(p_2-p'_2)\right)}.
\end{align*}}
In the classical limit $\gamma \to 0$, 
\begin{align*}
\langle p_1,p_2|S_{1,2}|p_1',p_2'\rangle &\sim \delta(p_1+p_2-p_1')\exp\left(-\frac{1}{2\sqrt{-1}\gamma}V(p'_2-p_2,p_1)\right) \\
\langle p_1,p_2|S_{1,2}^{-1}|p_1',p_2'\rangle &\sim \delta(p_1-p_1'-p_2')\exp\left(-\frac{1}{\sqrt{-1}\gamma}V(p_2-p'_2,p'_1)\right),
\end{align*}
where we have defined the function $V(x,y)$ by
\[V(x,y) = \frac{\pi^2}{6} + Li_2(e^x)-xy 
\]

In the decomposition of ideal tetrahedra of a hyperbolic cusped 3-manifold $\mathcal{M}$, the S-operator is assigned to each (oriented) tetrahedron, and then the pentagon identity (\ref{Spenta}) is interpreted as the Pancher move, as explained in \cite{Hikami}. 

Moreover, by applying the (formal) saddle point method in a certain integral related to  (\ref{Spenta}), one can find the coincidence between the saddle point equation and the hyperbolic consistency equation.   

\subsubsection{ Partition function of a cusped 3-manifold}

For any hyperbolic cusped 3-manifold $\mathcal{M}$,  take an ideal triangulation, and we know how to glue faces together; we can naturally define the partition function based on the S-operator by
\begin{equation} Z_\gamma(\mathcal{M}) = \int_\mathbb{R}dp\delta_C(p)\delta_G(p)\prod_{i=1}^M\langle p_{2i-1}^{(-)},p_{2i}^{(-)}|S^{\varepsilon_i}| p_{2i-1}^{(+)},p_{2i}^{(+)}\rangle \label{Partition}
\end{equation}
where $p$ denotes a set of variables $ p_1^{(\pm)},\ldots , p_{2M}^{(\pm)}$, and $\varepsilon_i = \pm 1$, depending on an orientation of tetrahedron. 

Here, we follow the same notation as in \cite{Hikami}, which may be a conventional symbolic usage in physics. For example, the domain of integral $\mathbb{R}$ is not always a one-dimensional line, but it may represent the fact that $\mathbb{R}$ is a `real cycle'. As already mentioned in the last paragraph before Subsubsection \ref{Faddeevintegral}, to determine this cocycle precisely is the most serious remaining difficulty.

We set $M$ as the number of ideal tetrahedra.
The condition $\delta_G ( p)$ determines how to glue faces together (``$G$'' stands for ``gluing''). 

Every face with the same momentum has to be glued together, and the fact that in-states can be glued only
to out-states indicates that the gluing condition $\delta_G(p)$ is a product of $\delta( p_j^{(-)}-p_k^{(+)})$ for some $j$ and $k$.

We need another geometrical condition $\delta_C(p)$ to define a partition function in addition to a way to glue faces (``C'' stands for ``completeness'').
 We can draw a developing map from an ideal triangulation of a 3-manifold,
and we need to read off a hyperbolic complete condition. This condition can be written as a constraint for $p$ by
identifications (3.1) and (3.2) in \cite{Hikami}. By construction, the partition function $Z_\gamma(\mathcal{M})$ is invariant under the Pachner move with any orientations.

As was studied by Thurston \cite{Thurston}, a deformation of the hyperbolic structure on a manifold $\mathcal{M}$
can be holomorphically parametrized by a parameter $u$ in a neighborhood of completeness condition $u =0$.
This parameter $u$ is the logarithm of the eigenvalue of the meridian by the holonomy representation, and we set 
\[ m = e^u.\]
Correspondingly, we denote such a manifold
by $\mathcal{M}_u$, which is no more complete, and define the partition function$Z_\gamma(\mathcal{M}_u)$ similarly and its asymptotics when $\gamma \to 0$ is described as follows:
\begin{equation} Z_\gamma(\mathcal{M}_u) \sim \int_\mathbb{R}dx\exp\left(\frac{1}{2\sqrt{-1}\gamma}\Phi_\mathcal{M}(x;u)\right) \label{udeform}
\end{equation} 
Here on the right-hand side, following our convention, we have re-parametrized variables $p$ with $x = (x_1, x_2, \ldots, x_{M-1})$.
After incorporating constraints written in terms of delta functions.

The integral (\ref{udeform}) could be evaluated by the saddle point method, as we have worked on a classical limit
0. The saddle point condition for variables $x$ is
\begin{equation}  \frac{\partial }{\partial x_i}\Phi_\mathcal{M}(x;u) = 0 \label{partialx}
\end{equation} 

These conditions coincide with the hyperbolic consistency conditions around the edges when we glue oriented tetrahedra together, i.e., unity is the product of dihedral angles around each edge \cite{Hikami0}.
By construction, the variable u denotes the meridian of the cusp in this classical limit, and the complete hyperbolic
structure is realized by setting $u =0$. To conclude, the function
$\Phi_\mathcal{M}(x; u)$ defined by a classical limit (\ref{udeform}) of the partition function $Z_\gamma(\mathcal{M}_u)$
under constraints (\ref{partialx}) is nothing but the Neumann-Zagier potential function \cite{Neumann}.

As a result, the differential of the potential function with respect to the deformation parameter $u$ gives
\begin{equation}  \frac{\partial }{\partial u}\Phi_\mathcal{M}(x;u) = v \label{paru}
\end{equation} 
where $v$ is related to the eigenvalue of the longitude by the holonomy representation
\[ \ell = e^v \]
Moreover, the equation (\ref{paru}) can be rewritten as 
\[\frac{\partial }{\partial u}\lim_{\gamma\to 0}\sqrt{-1}\gamma\log Z_\gamma(\mathcal{M}_u) = v \label{partialu}
\]
Combining with
the fact that our partition function $Z_\gamma(\mathcal{M}) = Z_\gamma(\mathcal{M}_{u=0})$
 may be regarded as a quantization of the Bloch invariant, we should have
\[\lim_{\gamma\to 0}2\gamma\log Z_\gamma(\mathcal{M}) = \mbox{Vol}(\mathcal{M}) + \sqrt{-1}\mbox{CS}(\mathcal{M})
\]
as a variant of the volume conjecture \cite{Hikami0} (see \cite{Hikami} in more detail). Here, CS denotes the Chern-Simons invariant defined modulo $\pi^2$
(see \cite{Meyerhoff} for a definition of the Chern-Simons invariant for the case of cusped 3-manifolds).
 However, there remains an ambiguity of the branch in the complex plane in
an actual computation.

In conclusion, it suffices to realize the following for the mathematical justification of the above arguments, which is our aim for the final chapter. 

\begin{quotation}
Give a mathematical explanation of the formal Taylor expansion of the S-operator $S_{1,2}$ with respect to $\gamma$. 
\end{quotation}

The reason why the computation stays in formal is essentially the unboundedness of operators $\hat{p}_1\hat{q}_2, \hat{p}_1, \hat{p}_2, \hat{q}_1, \hat{q}_2$. Thus, we shall mainly focus on remedying these points in the following sections.  

As for Problem \ref{remained} (a), although it seems plausible the authors' comment on \cite{Andersen} that the target category for defining the partition function $Z_\gamma(\mathcal{M})$ cannot be the category of finite dimensional vector spaces, it seems the approximate version of $Z_\gamma(\mathcal{M})$ can realized as a finite dimensional object by the arguments in Chapter 4 and it seems sufficient for the original Kashaev's conjecture.
\subsubsection{ Nilpotent Lie group $G_{6,4}$}
Next, we consider the nilpotent Lie group $G$ related to the S-operator $S_{1,2}$ in (\ref{Soperator}). As already mention after its definition (\ref{Soperator}), $G$ is isomorphic to $G_{6,4}$ in \cite{Nielsen} whose Lie algebra is $\mathfrak{g}_{6,4}$. Let us recall several computations in \cite{Nielsen}. 

\begin{example}\label{g64}$\mathfrak{g}_{6,4} = \mathbb{R}X_1 + \mathbb{R}X_2 + \mathbb{R}X_3 + \mathbb{R}X_4 + \mathbb{R}X_5 + \mathbb{R}X_6$ is a Lie algebra satisfying the following relation:
\[  [X_6,X_5] = X_3, \;  [X_6,X_4] = X_2, \;  [X_5,X_2] = X_1, \;  [X_4,X_3] = X_1.
 \]
Then the following multiplication rule holds:
\begin{align*}
&{} (x_1,x_2,x_3,x_4,x_5,x_6)(y_1,y_2,y_3,y_4,y_5,y_6) \\
&= (x_1+y_1+x_4y_3 +x_5y_2+x_4x_6y_5+x_5x_6y_4+x_6y_4y_5,\\
&{} x_2+y_2+x_6y_4,x_3+y_3+x_6y_5,x_4+y_4,x_5+y_5,x_6+y_6) 
\end{align*}
Then we compute 
\begin{align*}
&{} (0,0,0,s,t,0)(x_1,x_2,x_3,x_4,x_5,x_6) \\
&= (x_1+x_3s +x_2t,x_2,x_3,s+x_4,t+x_5,x_6) \\
&= (x_1+x_3s +x_2t-x_4x_6t-x_5x_6s-x_6st,x_2-x_6s,x_3-x_6t,x_4,x_5,x_6)(0,0,0,s,t,0) \\
\end{align*}
From this computation, we have, for $l \in \mathfrak{g}^\ast$ with $l(\xi_1X_1+\xi_2X_2+\xi_3X_3+\xi_4X_4+\xi_5X_5+\xi_6X_6) = \xi_1$, we put $\rho^{6,4}_l := \mbox{Ind}_M^G(\chi_l)$ with $M = \exp(\mathfrak{m})$ and $\mathfrak{m} = \mathbb{R}X_1+\mathbb{R}X_6$. We have the following:
\begin{align*}
&{} (\rho^{(6,4)}_l(x_1,x_2,x_3,x_4,x_5,x_6)(f))(s,t) \\
&= \tilde{f}((0,0,0,s,t,0)(x_1,x_2,x_3,x_4,x_5,x_6)) \\
&= \tilde{f}((x_1+x_3s +x_2t-x_4x_6t-x_5x_6s-x_6st,x_2-x_6s,x_3-x_6t,x_4,x_5,x_6)(0,0,0,s,t,0))\\
&= \chi_l((x_1+x_3s +x_2t-x_4x_6t-x_5x_6s-x_6st)^{-1})\tilde{f}(0,0,0,s+x_4,t+x_5,0) \\
&= \overline{\chi_l((x_1+x_3s +x_2t-x_4x_6t-x_5x_6s-x_6st))}\tilde{f}(0,0,0,s+x_4,t+x_5,0) \\
&= e^{2\pi\sqrt{-1}(-x_1-x_3s -x_2t +x_4x_6t+x_5x_6s+x_6st)\xi_1}f(s+x_4,t+x_5).
\end{align*}
(Our computation is slightly different from (7)(a) in \cite{Nielsen}. However, likely, the difference does not cause any essential problems here.)

next we take a (partial) Fourier transform $\mathcal{F}^s$ with respect to a variable $s$ to swap the multiplication operator $s$ and the differential operator $\frac{\partial}{\partial s}$. Put $\pi^{(6,4)}_l = \mathcal{F}^s\circ\rho^{(6,4)}_l$. 

Then $\pi^{(6,4)}_l$ is an unitary representation related $S_{1,2}$ in (\ref{Soperator}) by putting
\begin{align*}
\xi_1 \; = \sqrt{2\gamma}, &{} \;  \hat{p}_1\hat{q}_2 = \pi^{(6,4)}_l(X_6), \;  \hat{p}_1 = \pi^{(6,4)}_l(-X_5), \;  \sqrt{-1}\hat{q}_1 = \pi^{(6,4)}_l(X_2)\\
&{}  \; \hat{p}_2 = \pi^{(6,4)}_l(X_3), \; \sqrt{-1}\hat{q}_2 = \pi^{(6,4)}_l(-X_4)
\end{align*}

\end{example}

\subsection{ Condensed matter physics: Moir\'{e} pattern, magic angle, multi-layer graphen}

The Hofstadter butterfly was discovered in the 1970s. However, for a long time, its experimental realization was deemed infeasible due to the requirement of an unrealistically strong magnetic field in its original theoretical setting. Remarkably, a recent experiment has succeeded in realizing this phenomenon despite these challenges \cite{Nuckollsetal}. Prior to this breakthrough, the following developments had taken place:

In the mid-2010s, attention shifted from the original setting to a more experimentally accessible system: bilayer graphene, composed of two superimposed two-dimensional honeycomb lattices. When one layer is rotated relative to the other, a fractal-like spectral structure emerges in the associated operator. This structure was also observed experimentally \cite{Dean}, provided that the rotation angle is appropriately chosen.

This configuration gives rise to a periodic arrangement resembling large molecules, known as the Moir\'{e} pattern. The specific angle at which this pattern exhibits notable electronic properties is referred to as the "magic angle." In this context, the Moir\'{e} pattern corresponds to rational magnetic flux values in the original Hofstadter model.

Nevertheless, a comprehensive understanding of the relationship between the spectral properties of the Harper operator and those of the Moir\'{e} -Laplacian operator, sometimes used to describe the Moir\'{e} system, remains elusive. A more detailed analysis is thus left for future investigation. Additionally, further studies from the perspectives of quantum tori, quantum groups, and quantum cluster algebras are of significant interest.

Moreover, recent developments in physics have extended the scope beyond bilayer graphene to multi-layer configurations. It is hoped that our formulation of Floquet-Bloch analysis will contribute meaningfully to future research in these areas.

\subsection{ Infinite extensions in number theory}\label{InfiniteextensionsNT}

One of the motivations of several density theorems in geometry or dynamical systems is in number theory, the Chebotarev density theorem on finite algebraic extensions. Furthermore, in the case of finite extension, Theorem \ref{finitegeodesic} in the Introduction, the derivation from the analytic properties of L functions to asymptotic results is almost the same as in number theory. However, when we extend to infinite extensions, such as Theorems \ref{abel-geod} and \ref{KatsudaSunadaIHES}, there is no methodological model in number theory. Turning our attention to the number theory again, we become interested in infinite extensions, in this case. 

The basis of all the discussions is the Galois theory. In contrast to the case of finite algebraic extensions, in the case of infinite extensions, the fundamental theorem of Galois theory, the correspondence between subgroups of Galois groups and extension of fields, does not hold as it is. 
For this reason, we need to introduce a topological structure called the Krull topology into the Galois group. The correspondence will then be recovered if we restrict it to closed subgroups of the Galois group only. Because of what this fact represents, the situation is quite different from geometry.
For example, the Galois group  $G$ of infinite extension in number theory has uncountable cardinality and is compact and totally disconnected with respect to the above topology. Furthermore, it is said that there is no universal extension corresponding to universal covering in number theory, at least in a naive sense.


The first version of the Chebotarev density theorem on infinite extensions is the following:
Since $G$ is compact, there exists a bi-invariant finite Haar measure $\mu$ on  $G$. 
If we take a conjugate-invariant subset $C$ of $G$ for an unramified extension of the field of rational numbers (except for a finite number of primes) with $\mu(C) > 0$, namely $C$ has positive density. Let $\pi(x,C)$ denote the cardinal of unramified prime numbers whose Frobenius homomorphism belongs to $C$ and whose norm is less than or equal to $x$. 
Then we have
\[\pi(x,C)\sim \frac{\mu (C)}{\mu (G)}\frac{x}{\log x}.\]
(cf. \cite{Serre0}) 

However, these kinds of results are derived by essentially the same idea as finite extensions, so it is more interesting for sets with density $0$. 

However, naive thinking does not work. For example, the Galois group for the algebraic closure $\overline{\mathbb{Q}}$ of the infinite extension field of the rational number field $\mathbb{Q}$ is called the absolute Galois group $G_\mathbb{Q}$, which is the essential object in number theory. Let us consider the following surjective homomorphism 
\[\Pi:G_\mathbb{Q} \to G_\mathbb{Q}^{\rm ab} := G_\mathbb{Q}/\overline{[G_\mathbb{Q},G_\mathbb{Q}]}\]
where  $G_\mathbb{Q}^{\rm ab}$ is the abelineization of $G_\mathbb{Q}$ and $\overline{[G_\mathbb{Q},G_\mathbb{Q}]}$ is the closure of the commutatar subgroup $[G_\mathbb{Q},G_\mathbb{Q}]$ with respect to the Krull topology. 
The problem of counting primes, such that the Frobenius homomorphisms belonging to the kernel of this homomorphism, corresponds to the problem of the infinite abelian extension in geometry, as seen in Theorems \ref{abel-geod} and \ref{KatsudaSunadaIHES}.  At first glance, it appears to be something natural and meaningful. However, this is a trivial problem; namely, we can see that the Frobenius homomorphism of any prime number does not belong to this kernel for almost trivial reasons. Although the following is well-known to experts, we could not find any literature that explicitly states it; therefore, for the reader's convenience, we explain their reasons here.

First, note that this extension is totally ramified, so the Frobenius conjugate class can not be well defined in the usual sense. Thus, it needs to be modified to account for the inertial groups. However, even if we could think in this way, the following arguments show the above fact. 

In the case of a finite extension, the Galois group is given first. If the prime $p$ is made large enough, its Frobenius homomorphism can be the identity element of the Galois group. However, if we first fix the prime $p$, then the Frobenius homomorphism cannot be the identity element if the degree of the Galois group is larger than $p$. 

In the setting of Dirichlet's density theorem on arithmetic progressions, it is said as follows: If we first fix an integer $\ell$, then there is a prime number $p$ congruent to $1$ modulo $\ell$. However, on the contrary, if we fix $p$ at first and then take $\ell$ larger than $p$, then such a possibility is impossible.

Since the maximal abelian quotient $G_\mathbb{Q}^{\mbox{ab}}$ is known to be isomorphic to $\hat{\mathbb{Z}}^\times = \prod_p\mathbb{Z}_p^\times$ which is called profinite cyclic group. Since this group has a finite group of arbitrarily large orders in its quotient group, the image of the Frobenius homomorphism cannot be the identity element for the above reason.

At present, the most celebrated Chebotarev-type density ``theorem'' for infinite extensions with density $0$ is Lang-Trotter's conjecture; see also \cite{Serre}, \cite{Katz}, \cite{James}, \cite{Wan}. As the ultimate refinement of the celebrated Sato-Tate conjecture, the Lang-Trotter conjecture aims to reveal statistical distributions of Frobenius traces with a fixed integer value instead of a fixed interval.

We describe its detailed form by quoting from the references mentioned above.
For an elliptic curve $E$ over $\mathbb{Q}$ and $r \in \mathbb{Z}$, put
\[\pi_E^r(X) := {}^\sharp\{p<X : a_E(p) = r\}\]
where
\[a_E(p) := p + 1 - {}^\sharp E(F_p)\]
and $E(F_p)$ is the mod $p$ reduction of $E$.

In 1976, S. Lang and H. Trotter \cite{Lang} formulated a precise conjecture on
the number of good primes $p$ such that $a_E(p)$ takes a fixed integer value. Let $N_E$ be the conductor of $E$. Put
\[\pi_{E,r}(x) := {}^\sharp\{p < x| a_E(p) = r, p\not|\; N_E\} \]
for all $r \in \mathbb{Z}$ and $x > 2$.

\begin{conjecture}[Lang--Trotter Conjecture] Let $E$ be an elliptic curve over $\mathbb{Q}$ and assume additionally that $r \neq 0$ if $E$ has complex multiplication. Then
\[ \pi_{E,r}(x)\sim C_{E,r}\frac{\sqrt{X}}{\log X}, \]
where $C_{E,r} \geq 0$ is a constant that can be described in terms of the image of the associated Galois representation.
\end{conjecture}

An explicit form of $C_{E,r}$ is written in \cite{James} for the non-CM (non-complex multiplication) case.
For some historical context on the conjecture, see the Introduction of \cite{Wan}. The CM case is somewhat better understood than the non-CM case, as in \cite{Wan}.
Here, we concentrate only on the case that $E$ is non-CM and $r=0$, namely $E$ is supersingular. According from \cite{Serre}, it is reformulated as follows:
Let $l$ be a prime number, and let 
\begin{equation}
\rho_l:G_\mathbb{Q}\to \mbox{GL}_2(\mathbb{Z}_l)
\end{equation}
the $l$-adic representation defined by the $l^m$-division points of the elliptic curve $E$ (m = 0,1,\ldots,). This representation is unramified outside the set $\{11,l\}$. The group $G_l := \mbox{Im}(\rho_l)$ is an open subgroup of $\mbox{GL}_2(\mathbb{Z}_l)$ (\cite{Serre0}, \cite{Serre1}); $l$-adic Lie group of dimension $4$. If $l \neq 11,l$, the Frobenius conjugacy class $s_p$ of $p$ in $G_l$ is well-defined. By Eichler, Shimura, and Igusa, we have
\[\mbox{Tr}(s_p) = a_p\]
The relation $a_p =0$ is equivalent to the fact that $p$ is in the subvariety $C_l$ of $G_l$, which consists of elements of the trace $0$.

Concerning this problem, the best-known result for the upper bound is due to Elkies \cite{elkies} based on the results of Kaneko \cite{kaneko}, which is 
\[ \pi_{E,0}(x) = O(x^{3/4})\]
Previously, the same bound was obtained in \cite{Serre} under the generalized Riemann hypothesis. Moreover, Elkies also gives some estimates on the lower bound in the same paper, but it is far from the conjecture. 
Among the above results, the strategy in \cite{Serre} is based on an approximation by finite extensions and uses the effective Chebotarev theorem in this case.
Several works demonstrate that similar methods are also applicable to other non-supersingular cases. This method contrasts with the case of geometry, which deals directly with infinite groups. It should be noted that the asymptotics in the Lang-Trotter conjecture differ from those of Theorems \ref{abel-geod} and \ref{theorem-geod}, somewhat similar to the Large deviation in geometry.

Here, as a small proposal, we would like to suggest a heuristic hybrid method based on an approximation of the tower of finite extensions and geometric ones. We do not know whether this proposal is meaningful for the conjecture.

First, note that the target group $\mbox{GL}_2(\mathbb{Z}_l)$ of $\rho_l$ is pro-nilpotent, namely a projective limit of finite nilpotent groups. In our analysis of the geometric version of the nilpotent group, the leading term can be dominated by the abelian quotient, although the difference to abelian extensions appears in the Plancherel measure.

In the case of number theory, it corresponds to the analysis of the Dirichlet L-functions. However, there is a clear difference in the geometric case. Namely, their results come from our analysis of the behavior of the poles of the L-function associated to the characters in a vicinity of the trivial character $\textbf{1}$. In Chapter \ref{Asymptoticsclosedgeodesics}, we have used the trace formula to simplify the explanation. However, it is well known that we could reprove them using L functions such as \cite{Katsuda1}, \cite{Katsuda2}, \cite{Pollicott1}. 

On the other hand, the Dirichlet  L function $L(s,\chi)$ is the entire function; thus, the poles disappear provided $\chi$ is not principal. Here, we do not know how to remedy this situation in the context of Hilbert-Polya's strategy for the Riemann hypothesis.

Even though, in order to persuade the analogy with the geometric case, assuming that $L(s,\chi)$ had a ``virtual pole'' $s(\chi) \in \mathbb{R}$, namely 
\begin{equation} \mbox{``}L(s,\chi) \sim \frac{1}{s-s(\chi)}\mbox{''} \label{virtualpole} \end{equation}
In the limit $\chi \to \textbf{1}$, we would like to explore possibilities for applying our previous analysis, such as Chapter \ref{Asymptoticsclosedgeodesics}, \cite{Katsuda1}, \cite{Katsuda2}, or \cite{Pollicott1}. 
Under the assumption \ref{virtualpole}, the position of $s(\chi)$ can be read
from the value $L(1,\chi)$. Thus, we have a connection to the following conjecture:
\begin{conjecture}[Conjecture 1.1 in \cite{Lamzouri}]
Let $Q > 0$ be large. Then
\[
\max_{q \leq Q}\max_{\substack{\chi \neq \chi_0,(\mbox{\footnotesize{mod $q$}})\\ \mbox{\footnotesize{$\chi$ primitive}} }}|L(1,\chi)| = (e^\gamma + o(1))\log\log Q,
\]
where $\chi_0$ is the trivial character and $\gamma$ is the Euler constant. 
\end{conjecture}
Masanobu Kaneko informs this conjecture. He was kind enough to answer our questions and track down this document.

Although no progress has been made in this direction to date, we would like to consider our suggestion as a possibility for future research.

\begin{remark}
We have asked number theory experts about the other possibilities of the Chebotarev density theorem of infinite extension with density $0$ and have received the following comments.
\begin{description}
\item[{\rm (1)}] Kurokawa informed about the Bateman-Horn conjecture (cf. \cite{alet}), a comprehensive conjecture that derives from the twin prime conjecture and Green-Tao's theorem.
\item[{\rm (2)}] Sarnak introduced the talk of Serre at Minerva's lecture series \cite{SerreMinerva} (in particular, Talk 1). He suggested taking a look at this particular Serre's "motivated" versus "non-motivated" problems regarding primes. The Lang-Trotter conjecture is a ``motivated one'' where one has a natural feature that the traces are integers and one can look for primes in such small sets. There are many such problems and variations.
\end{description}
\end{remark}

\section*{Appendices}
\addcontentsline{toc}{section}{Appendices}
 
\appendix
In Appendices A and B, we summarize the unitary representation theory of nilpotent Lie groups. Of course, these materials are standard, basic knowledge of representation theory, and thus, there is no originality on our part.
Thus, these are only for some expected readers unfamiliar with the above. All materials are included in \cite{Corwin} in more detail.
In Appendix C, we recall the harmonic theory of Chen's iterated integral for general torsion-free nilpotent groups, mainly for future reference, whose
main contents are quoted from Kohno's book \cite{Kohno} in Japanese with our translations to English.
\section{Induced representations}\label{Inducedrep}
Here, we recall the definition of induced representations. Although it is a fundamental notion in representation theory, we write here that it is essential in the orbit method a l\`{a} Kirillov \cite{Kirillov1, Kirillov2}. However,  it seems to be difficult to understand this notion at first glance for a beginner in representation theory. Kirillov also described his personal experience \cite{Kirillov3} in the memorial volume of George Mackey.

In Wikipedia, induced representations are summarized as follows:
\vspace{12pt}

In group theory, the induced representation represents a group, $G$, constructed using a known representation of a subgroup $H$. Given a representation of $H$, the induced representation is, in a sense, the ``most general'' representation of $G$ that extends the given one. 
\vspace{12pt}

When $H$ is a trivial group $\{e\}$, then a representation of $G$ induced from the trivial representation $\textbf{1}$ is nothing but the (right) regular representation $R$ of $G$ defined by
\[     (R(g)f)(x) := f(xg) \quad \mbox{for} \quad g \in G, f \in L^2(G)
\]
which realizes the above sentences.

A rigorous definition for an unimodular group $G$,  i.e., it admits a bi-invariant Haar measure on a locally compact topological group $G$, is given as follows:

Consider a closed subgroup $H$ and its unitary representation $(\pi, W)$, where $ W$ is a Hilbert space. Then the representation space $V$ of the representation $\rho =\mbox{Ind}_{H}^{G}\pi$ of $G$ induced from $(\pi, W)$ is defined by 
\begin{align*}
V = \left\{\phi \colon G\to W\ | \phi (h^{-1}g)=\pi (h)\phi (g) (\; \mbox{for all}\; h \in H),\;g \in G,  \phi \in L^{2}(H\backslash G)\right\}
\end{align*}
Here $\phi \in L^2(H\backslash G)$ means that the coset space $H\backslash G$ carries a suitable invariant measure, and since the norm of $\phi(g)$ is constant on each left coset of $H$, we can integrate the square of these norms over $H\backslash G$ and obtain a finite result. 
Finally, the induced representation $\rho = \mbox{Ind}_{H}^{G}\pi$ is defined by 
\[\rho(g)\phi(x)=\phi(xg)\]
for $g,x \in G$ and $\phi \in V$

Furthermore, the induced representations are geometrically described as follows:
Consider the $H$-principal bundle $P = \Pi :G \to H\backslash G$. Then the above representation space $V$ is identified with the space $\Gamma(E)$ of sections of the associated vector bundle $E$ to $P$ defined by 
\begin{equation}      E := W {}_{\pi}\times P = W \times P/\sim ,  \qquad (w,x) \sim (\pi(h)w, h^{-1}x) \quad \mbox{for} \quad h \in H \label{inducedrepresentationbundle}
\end{equation} 
and the representation $\rho(g)$ acts on $\Gamma(E)$ as
\[(\rho(g)\psi)(\Pi(x)) = \psi(\Pi(x)g) \quad \mbox{for} \quad g \in G, \psi \in \Gamma(E).
\]
As mentioned in Step 1 (F1) of Section \ref{Infiniteabelian}, this construction is similar to that of the flat vector bundle $E_R$ over $M$ associated to the regular representation $R$.

\section{Unitary representation of nilpotent Lie groups}

\subsection{Construction of irreducible unitary representations}\label{B.1}
This section is also provided for the convenience of readers unfamiliar with the terminology of representation theories and can be viewed as an ideal model for the orbit method.

Let $G$ be a nilpotent Lie group. It is well known that irreducible unitary representations are described by the orbit method,  a la Kirillov \cite{Kirillov1}, \cite{Kirillov2}, \cite{Corwin}.  
All the materials in this section are quoted from \cite{Corwin} essentially with some minor modifications.

Let $\mathfrak{g}^{\ast}$ be the vector space dual of the Lie algebra $\mathfrak{g}$ of $G$. $G$ acts on $\mathfrak{g}^\ast$ by co-adjoint action $\mbox{Ad}^\ast(G)$. For $l \in \mathfrak{g}^\ast$ and $x \in G$, we sometimes write $l\cdot x$ for $\mbox{Ad}^\ast(x^{-1})l$; note that $l\cdot (xy)=(l\cdot x)\cdot y$.
Given $l \in \mathfrak{g}^\ast$, let $B_l$ be the bilinear form $B_l(X,Y)= l([X, Y])$ and let $\mathfrak{r}_l$ be its radical defined by 
\[\mathfrak{r}_l = \{X \in \mathfrak{g} | B_l(X,Y) = 0,\; \mbox{for all} \; Y \in \mathfrak{g}\}.
\]

Choose a maximal subordinate subalgebra (a.k.a. polarized algebra or polarization) $\mathfrak{m}$ for $l$, which is maximal in the set of subalgebras $\mathfrak{h}$ satisfying $B_l(X,Y)= 0$ for all $X,Y \in \mathfrak{h}$,  and let $M =\exp \mathfrak{m}$. Then the map of $M \to S^1 \cong U(1)$ defined by 
\[ \chi_{l,M}(\exp Y) = e^{2\pi\sqrt{-1}l(Y)}, \quad Y \in \mathfrak{m},\]
is a one-dimensional representation (i.e. unitary character) of $M$, since $B_l(\mathfrak{m}, \mathfrak{m})= 0$. 
We may therefore form the induced representation $\pi_{l,M}  = \mbox{Ind}_M^G\chi_{l,M}$.
The following results describe the unitary dual $\widehat{G}$, the set of (unitary) equivalence classes of irreducible unitary representations of $G$, in terms of these induced representations. 

\begin{theorem}\label{nilpLieirredrep} Let $l \in \mathfrak{g}^\ast$. Then we have:
\begin{description}
\item[{\rm (1)}]{\rm (2.2.1 Theorem in \cite{Corwin})} There exists a maximal subordinate subalgebra $\mathfrak{m}$ for $l$ such that $\pi_{l,M}$ is irreducible. 
\item[{\rm (2)}]{\rm (2.2.2 Theorem in \cite{Corwin})} Let $\mathfrak{m}, \mathfrak{m}'$  be two maximal subordinate subalgebras for $l$. Then $\pi_{l,M} \cong \pi_{l,M'}$. $($unitary equivalent, hence we may write $\pi_l$ for $\pi_{l,M}$ if we are interested only in equivalence classes of unitary representations$)$.
 In particular, $\pi_{l,M}$ is irreducible whenever $\mathfrak{m}$ is maximal subordinate for $l$.
\item[{\rm (3)}]{\rm (2.2.3 Theorem in \cite{Corwin})} Let $\pi$ be any irreducible unitary representation of $G$. Then there is an $l \in \mathfrak{g}^\ast$ such that $\pi_l \cong \pi$.
\item[{\rm (4)}]{\rm (2.2.4 Theorem in \cite{Corwin})} Let $l, l' \in \mathfrak{g}^\ast$ . Then $\pi_l \cong \pi_{l'}$ (unitary eqivalent) if and only if $l$ and $l'$ are in the same $\mbox{Ad}^\ast(G)$-orbit in $\mathfrak{g}^\ast$.
\end{description}
\end{theorem}

We may summarize these results as follows: the map $l \mapsto \pi_{l,M}$ is independent of $\mathfrak{m}$ and gives a bijection between the orbits $\mathfrak{g}^\ast /\mbox{Ad}^\ast(G)$ and $G$.

\subsection{Coadjoint orbits}

Concerning the Fourier inversion formula, Theorem \ref{FourierinversionLie}, we need only the information of the generic orbits. The following are the first instances describing them. 

\begin{theorem}[{\rm 3.1.6 Theorem in \cite{Corwin}}]\label{corwin316} Let $G$ act unipotently on $V$ and let $\{e_1,\ldots,e_m\}$ be a Jordan-H\"older basis for this action, namely, $V_j := \mathbb{R}$\mbox{-span}$\{e_{j+1},\ldots,e_m\}$ is $G$-stable for all $j$. Then there are disjoint sets of indices $S$, $T$ with $S \cup T=\{1,2,\ldots,m\}$, a Zariski open set $U \subset V$, and rational functions $Q_1(x, t),\ldots, Q_m(x, t)$ of the variables $(x, t) = (x_1 ,\ldots, x_m, t_1 ,\ldots, t_k)$ where $k = \mbox{card}(S)$, with the following properties: if $S = \{j_1 < \cdots <j_k\}$ and if we identify $x \in \mathbb{R}^m$ with $v = \sum_{i= 1}^m x_ie_i$, then 

\begin{description}
\item[{\rm (i)}]	The functions $Q_i(x, t)$ are rational nonsingular on $U \times \mathbb{R}^k$. For fixed $x$, they are polynomials in $t$.
\item[{\rm (ii)}] For each $v = \sum_{i= 1}^m x_ie_i$ in $U$, $Q(x, t) = \sum_{i= 1}^m Q_i(x, t)e_i$ maps $\mathbb{R}^k$ diffeomorphically onto the orbit $G \cdot v$.
\item[{\rm (iii)}] For fixed $x$, the function $Q_j(x, t)$ depends only on those $t_i$ such that $j_i \leq j$.
\item[{\rm (iv)}] If $j \not\in S$ then $Q_j(x, t) =x_j+ R(x_1,\ldots, x_{j-1}, t_1,\ldots, t_i)$ where $i$ is the largest index such that $j_i <j$ and $R$ is rational. Moreover, $Q_1(x, t) = x_1$.
\item[{\rm (v)}] $Q_{j_i}(x,t)= t_i+x_{j_i}+R(x_1 ,\ldots,,x_{j_i-1}, t_1,\ldots,t_{i-1} )$ where $R$ is rational.
\end{description}
\end{theorem}

Next, by changing the variables $(x, t) = (x_1,\ldots, x_m, t_1,\ldots, t_k)$ to $(x, u) = (x_1,\ldots, x_m, u_ 1,\ldots,u_k)$, we can simplify the item (v) in this theorem to the following item (v) in following theorem.

\begin{theorem}[{\rm 3.1.8 Corollary in \cite{Corwin}}]\label{corwin318}
Let $G$ act unipotently in $V$, let $\{ e_1,\cdots, e_m \}$ be a Jordan-H\"older basis, and $d_j= \mbox{generic dimension of G-orbits}$ in $V/V_j$ where $V_j= \mathbb{R}\mbox{-span}\{e_{j+1},\ldots, e_m \}$. Let $k = d_m =$ generic dimension of orbits in $V$, $U = \{ v \in V| \; \mbox{{\rm dim}}G \cdot \bar{v}\; \mbox{in}\; V/V_j\; \mbox{is equal to}\; d_j, 1 \leq j \leq m \}$, $S = \{j_1 < \cdots < j_k \}$ the indices $j$ such that $d_j \neq d_{j-1}$, and $T=\{1,2,\ldots,m\} \sim S$. lf we identify $v= \sum_{i=1}^m x_ie_i$ with $x \in \mathbb{R}^m$, there are rational functions $P_1,\ldots, P_m$ in $m + k$ variables $(x, u) = (x_1,\ldots, x_m, u_ 1,\ldots,u_k)$ such that the following hold. 
\begin{description}
\item[{\rm (i)}] The $P_1,\ldots,P_m$ are rational, nonsingular on $U \times \mathbb{R}^k$. If $x$ is fixed, they are polynomials in $u \in \mathbb{R}^k$.
	
\item[{\rm (ii)}] If $v= \sum_{i=1}^m x_ie_i$, the map of $\mathbb{R}^k$ into $V$ given by $P(x, u) = \sum_{i=1}^m P_j(x, u)e_j$ is a diffeomorphism onto the closed submanifold $G \cdot v$.
\item[{\rm (iii)}] For fixed $x$, the function $P_j(x, u)$ depends only on the $u_j$ such that $j_i\leq j$.
\item[{\rm (iv)}] If $j \not\in S$, $P_j(x,u) = x_j + R(x_1,\ldots, x_{j-1},u_1,\ldots,u_i)$ where $i$ is the largest index such that $j_i < j$ and $R$ is rational. Moreover $P_1(x, t) = x_1$.
\item[{\rm (v)}] $P_{j_i}(x,u)=u_i$, $1\leq i\leq k$.
\end{description}
Finally, $U$ is $G$-invariant and if $u \in \mathbb{R}^k$ is fixed, each $P_j(x, u)$ is a rational nonsingular function on $U$, constant on $G$-orbits. 
\end{theorem}

The above two theorems are also expressed in different forms: 

\begin{theorem}[{\rm 3.1.9 Theorem in \cite{Corwin}}]\label{corwin319}
Given a Jordan-H\"older basis $\{e_1,\ldots, e_m \}$ for a unipotent action of $G$ on vector space $V$, define the generic dimensions $d_i$ $(1\leq i\leq m)$, a partition $S\cup T = \{1, 2,\ldots, m \}$, and the $G$-invariant set $U$ of generic orbits as in Theorems \ref{corwin316} and \ref{corwin318}.
Let $V_S =\mathbb{R}\mbox{-span}\{e_i: i \in S\}$, $V_T = \mathbb{R}\mbox{-span}\{e_i: i \in T\}$ and let $p_S, p_T$ be the projections of $V$ to $V_S, V_T$. Then
\begin{description}
\item[{\rm (i)}] Every $G$-orbit in $U$ meets $V_T$ in a unique point. In particular, $U\cap V_T$ is nonempty and Zariski open in $V_T$.
There is a map $\psi: (U\cap V_T) \times V_S \to U$ such that

\item[{\rm (ii)}] $\psi$ is a rational, nonsingular bijection with rational, nonsingular inverse,
\item[{\rm (iii)}] For each $v \in V_T\cap U$, the map $P_v(\cdot)= p_T(\psi(v,\cdot ))$ from $V_S$ into $V_T$ is a polynomial, and the orbit $G \cdot v$ is its graph,
\item[{\rm (iv)}] The Jacobian determinant of $\psi$ is identically $1$.
\end{description}
\end{theorem}

\subsection{Decomposition of the right regular representation}\label{B.3}
We recall the Fourier inversion formula and the Kirillov character formula.
The former describes the content of a decomposition of the right regular representation into a direct integral of irreducible unitary representations of the nilpotent Lie group $G$. To define this formula, we need to compute the trace of the Fourier transform $\pi_l(\phi)$ defined by 
\[  \pi_l(\phi) := \int_G\pi_l(\sigma )\phi(\sigma)d\sigma \quad \phi \in \mathcal(S)(G)\]
where $l \in \mathfrak{g}^\ast$ and $\mathcal{S}(G)$ is the space of rapidly decreasing functions on $G$ (a.k.a. Schwartz class).

The Kirillov character formula is a tool for the computation of the trace $\mbox{Tr}(\pi_l(\phi))$ of $\pi_l(\phi)$. We first recall several results to compute them.  

First, we recall notions of weak or strong Malcev basis of a nilpotent Lie group $G$ and its Lie algebra $\mathfrak{g}$.

\begin{definitiontheorem}[1.1.13 Theorem and Note in \cite{Corwin}]\label{malcevbasis}
Let $\mathfrak{g}$ be a nilpotent Lie algebra, and let $\mathfrak{g}_1 \supset \mathfrak{g}_2 \supset \cdots \supset \mathfrak{g}_k$ be subalgebras, with ${\rm dim}\; \mathfrak{g}_i = m_i$ and ${\rm dim}\; \mathfrak{g} = n$.
\begin{description}
\item[{\rm (a)}] $\mathfrak{g}$ has a basis $\{X_1, \ldots, X_n\}$ such that
\begin{description}
\item[{\rm (i)}] for each $m$, $\mathfrak{h}_m = \mathbb{R}\mbox{-span}\{ X_1, \ldots, X_m\}$ is a subalgebra of $\mathfrak{g}$,
\item[{\rm (ii)}] for $1 \geq j	\geq k$, $\mathfrak{h}_{m_j} = \mathfrak{g}_j$.
\end{description}
This basis is called a weak Malcev basis for $\mathfrak{g}$ (through $\mathfrak{g}_1,\ldots, \mathfrak{g}_k$).

\item[{\rm (b)}] If the $\mathfrak{g}_i$ are ideals of $\mathfrak{g}$, then one can pick the $X_i$ so that (i) is replaced by
\begin{description}
\item[{\rm (iii)}] for each $m$, $\mathfrak{h}_m = \mathbb{R}\mbox{-span}\{ X_1, \ldots, X_m\}$ is an ideal of $\mathfrak{g}$.
\end{description}
This basis is called a strong Malcev basis for $\mathfrak{g}$ (through $\mathfrak{g}_1,\ldots, \mathfrak{g}_k$).
\end{description}
\end{definitiontheorem}

\begin{theorem}[4.2.1 Theorem in \cite{Corwin}]\label{nilpotenttrace} Let $\pi =\pi_l$ be an irreducible representation of a nilpotent Lie group $G$, let $\mathfrak{m}$ be a polarization for $l$, and model $\pi$ in $L^2(\mathbb{R}^k)$ {\rm (}i.e. representation space of $\pi${\rm )} using any weak
Malcev basis through $\mathfrak{m}$. If $\phi \in \mathcal{S}(G)$, then $\pi_l(\phi)$ is trace class and
\begin{equation*}
\pi_l(\phi) f(s) = \int_{\mathbb{R}^k}K_\phi (s, t)f(t) dt, \quad \mbox{for all} \quad f \in L^2(\mathbb{R}^k )
\end{equation*}
where $K_\phi \in \mathcal{S}(\mathbb{R}^k \times \mathbb{R}^k)$. Furthermore, $\mbox{Tr}\;\pi_l(\phi)$ is given by
\begin{equation*}
\mbox{Tr}\;\pi_l(\phi) = \int_{\mathbb{R}^k}K_\phi (s, s)ds. \quad \mbox{(absolutely convergent)} 
\end{equation*}
\end{theorem}

A detailed expression of the kernel function $K_\phi$ is given as follows:
Let $\{X_1,\ldots ,X_n\}$ is weak Malcev basis taken in Theorem \ref{nilpotenttrace} and put $p = n-k = \mbox{dim}\;\mathfrak{m}$. Define polynomial maps $\gamma: \mathbb{R}^n \to G$, $\alpha:\mathbb{R}^p\to M =\exp(\mathfrak{m})$, $\beta:\mathbb{R}^k \to G$ by
\begin{align*}
&{} \gamma(s,t) = \exp s_1X_1\cdot\cdots\cdot\exp s_pX_p\cdot\exp t_1X_{p+1}\cdot\cdots\cdot\exp t_kX_n \\  
&{} \alpha(s) = \gamma(s,0),\quad \beta(t) = \gamma(0,t),
\end{align*}
and let $dg$, $dm$, $d\dot{g}$ be the invariant measures on $G$, $M$, $M\backslash G$ determined by Lebesgue measures $dsdt$, $ds$, $dt$ through maps. $\gamma$, $\alpha$, $\beta$ in the above. 

Then, we have the following:
\begin{proposition}[4.2.2 Proposition in \cite{Corwin}]\label{nilpLietrace}
If we take the standard basis realization of $\pi =\pi_l$ in $L^2(\mathbb{R}^k 
)$ 
relative to the given Malcev basis, the kernel $K_\phi$ has the form 
\begin{equation}
K_\phi(t',t) = \int_M\chi_l(m)\phi(\beta(t')^{-1}m\beta(t))dm \quad \mbox{{\rm (} absolutely convergent {\rm )}}  
\end{equation}
where $\chi_l(\exp Y) = e^{2\pi\sqrt{-1}l(Y)}$ and $\beta$ is the map defined above. 
\end{proposition}

\begin{remark} If we directly relate representations $G$ to $\Gamma$, then we need to identify a function $\phi:\Gamma \to \mathbb{C}$ with that $\widetilde{\phi}:G \to \mathbb{C}$. 
Although there seem to be several methods to perform it, if we take an easy way to define $\widetilde{\phi}(g) = \phi(\sigma)$ if $x \in \sigma\mathcal{D}$ with a fundamental domain $\mathcal{D}$ of the canonical projection $G \to G/\Gamma$, then $\widetilde{\phi} \in L^1(G)$ but $\not\in \mathcal{S}(G)$ and thus, $\pi(\widetilde{\phi})$ is a compact operator but not in trace class in general. (see Section \ref{Step0-9}).
\end{remark}

Once we proceed with the computation further, we arrive at the following formula.
Given any Euclidean measure $dX$ on $\mathfrak{g}$, define the Euclidean Fourier transforms $\hat{f}$ of functions $f$ on $G$ to be
\[ \hat{f}(l) = \int_\mathfrak{g}e^{2\pi\sqrt{-1}l(X)}f(\exp X)dX \quad \mbox{for all}\quad l \in \mathfrak{g}^\ast.
\]

Each coadjoint orbit $\mathcal{O}_l=\mbox{Ad}^\ast G(l)$ of $l$, being a closed set, 
carries an invariant measure $\mu$ that is unique up to a scalar multiple because $\mathcal{O}_l \simeq R_l\backslash G$, 
where $R_l = \mbox{Stab}_G(l) = \{x \in G|(\mbox{Ad}^\ast x)l = l\}$. 

\begin{theorem}[4.2.4 Theorem in \cite{Corwin}, Kirillov character formula]\label{nilpLiecharcter} If $\pi$ is an irreducible unitary representation of a nilpotent Lie group $G$, corresponding to the coadjoint orbit $\mathcal{O}_l=\mbox{Ad}^\ast G(l) \subset \mathfrak{g}^\ast$, there is a unique 
choice of the invariant measure $\mu$ on $\mathcal{O}_l$ such that
\[ \mbox{Tr}\;\pi_\phi = \int_{\mathfrak{g}^\ast}\hat{\phi}(l)\mu(dl) \quad \mbox{for all}\quad \phi \in \mathcal{S}(G),
\]
The integral is absolutely convergent. 
\end{theorem}

Under the above preparation, we can state the following formula. 

\begin{theorem}[4.3.9 Theorem in \cite{Corwin}, Fourier Inversion formula]\label{FourierinversionLie}
Let $\{X_1,\ldots,X_n\}$ be a strong Malcev basis for a nilpotent Lie algebra $\mathfrak{g}$ and 
let $\{l_1,\ldots, l_n\}$ be the dual basis for $\mathfrak{g}^\ast$. 
Take a set $U$ of generic coadjoint orbits, index sets 
\[ S = \{i_1 < \cdots < i_{2k}\} \] 
and $T$ which will appear in Theorem \ref{corwin319}, and (absolute value of) Pfaffian $|\mbox{Pf}(l)|$ satisfies
\[ |\mbox{Pf}(l)|^2 = \det \;B, \quad B= (B_{jl}) = (B_l(X_{i_j},X_{i_l})) := (l([X_{i_j},X_{i_l}])) \quad j,l =1,\ldots, 2k \] 
Then for $\phi \in \mathcal{S}( G)$, $\phi(e)$ is given by an absolutely convergent integral 
\begin{equation}
\phi(e) = \int_{U\cap V_T}|\mbox{Pf}(l)|\mbox{Tr}\;\pi_l(\phi) dl \label{B2}
\end{equation}
where $dl$ is the Euclidean measure on $V_T = \mathbb{R}\mbox{-span}\{l_i| i \in T\}$ such that the cube determined by $\{l_i| i \in T\}$ has mass $1$. 
\end{theorem}

\begin{remark}
In the above theorem, if we cannot have a strong Malcev basis, then the suitably chosen weak Malcev basis has the same role as in the induction step of the argument of Subsubsection \ref{choicepolarization}.
\end{remark}

\subsection{Three examples} 

For nilpotent Lie group $G$, the exponential map $\exp: \mathfrak{g} \to G$ is a diffeomorphism, therefore we can define $X \ast Y \in \mathfrak{g}$ for all $X,Y \in \mathfrak{g}$ by the relation
\[\exp(X\ast Y) = \exp X\cdot\exp Y
\]
If $\mathfrak{g}$ is equipped with coordinates associated with a linear basis, the corresponding coordinates in $G$ will be called exponential coordinates, or canonical coordinates of the first kind.

\subsubsection{$(2n+1)$-dimensional real Heisenberg group $H_n$}\label{2n+1dimensionalHeisenberg}
We summurizes Examples 1.1.2, 1.2.4, 1.3.9, 2.2.6, 4.3.11 in \cite{Corwin}. Note that $H_n$ in \cite{Corwin} is the same as the Heisenberg-Lie group $\mbox{Heis}_n(\mathbb{R})$ in the notation in the body of this paper.
 
\paragraph{Definition}
We define $\mathfrak{h}_n$, the $(2n+ 1)$-dimensional Heisenberg algebra,
to be the Lie algebra with basis $\{Z, Y_1,\ldots, Y_n, X_1,\ldots, X_n\}$, whose pairwise Lie brackets are equal to zero except for
$[X_i,Y_i]=Z$, $1 \geq j \geq n$. It is a two-step nilpotent Lie algebra. One way to realize it as matrix algebra is to let $zZ + \sum_{i=1}^n (x_iX_i +y_iY_i)$ correspond to the $(n + 2) \times (n + 2)$ matrix
\[ \left(\begin{array}{ccccc} 0 & x_1 & \ldots & x_n & z \\ 
0 &  0   &  \ldots   &  0   & y_1 \\				
{}  &  {}   &  {}     &  {}   & \vdots \\
0 & {}	&   {}     &  0  & y_n \\
0 & 0	&   \ldots  &  0  & 0
\end{array}\right)
\]				
Note that $\{Z, Y, X\} = \{Z, Y_1,\ldots, Y_n, X_1,\ldots, X_n\}$ is a strong Malcev basis for $\mathfrak{h}_n$ through $\mathfrak{g}_1, \mathfrak{g}_2, \mathfrak{g}_3$ 
where $\mathfrak{g}_1$ is the center of $\mathfrak{h}_n$, which is generated by $Z$, $\mathfrak{g}_2$ is generated by $Z,Y_1,\ldots, Y_n$ and $\mathfrak{g}_3 = \mathfrak{h}_n$. These are ideals of $\mathfrak{h}_n$.

We denote a typical element of $\mathfrak{h}_n$ by $zZ + \sum_{i=1}^n (x_iX_i+y_iY_i) = (z,y,x)$, with $z\in \mathbb{R}$ and $x, y \in \mathbb{R}^n$. Using exponential coordinates for $H_n$, we get
\[
(z, y, x) \ast (z', y', x') = (z + z' + x\cdot y'-y\cdot x', y + y, x +x'),
\]
where $x\cdot y$ is the usual inner product on $\mathbb{R}^n$. Similarly,
\[
({\rm Ad}(\exp(z, y, x)))(z', y', x')=(z' + x\cdot y' -y\cdot x',y',x').
\]
If we can use the matrix representaion of $\mathfrak{h}_n$, then
\begin{equation}
\exp(z, y, x) = \left(\begin{array}{ccccc}
1 & x_1 & \ldots & x_n & z^\ast \\ 
0  & 1  & \ldots &  0   & y_1 \\				
  &     &.       &     & \vdots \\
0 &	&        &  1   & y_n \\
0 & 0	&  \ldots &  0  & 1
\end{array}\right) =: [z^\ast,y,x] \quad {\rm where}\quad
 z^\ast = z + x\cdot y. \label{expnentiltomatrix}
\end{equation}

\paragraph{Coadjoint orbits} 
Next, we consider coadjoint orbits in the dual space $\mathfrak{h}_n^{\ast}$ and associated irreducible unitary representations of $H_n$. We denote by $\{Z^\ast, Y_1^\ast,\ldots, Y_n^\ast, X_1^\ast,\ldots, X_n^\ast\}$ the dual basis of $\{Z, Y_1,\ldots, Y_n, X_1,\ldots, X_n\}$. An element $l \in \mathfrak{g}^\ast$ is written as 
\[ l = l_{\alpha,\beta,\gamma} := \gamma Z^\ast + \sum_{i=1}^n (\beta Y_i^\ast +\alpha_i X_i^\ast)
\]
For $w = (z,y,x) \in H_n$ and $W = cZ + \sum_{i=1}^n b_iY_i + a_iX_i \in \mathfrak{h}_n$, we can compute as 
\[ {\rm Ad}^\ast(w)(l_{\alpha,\beta,\gamma})(W) = {\rm Ad}^\ast(w)(l)(W) =  l({\rm Ad}(w^{-1})W) = l_{\alpha+\gamma y,\beta-\gamma x, \gamma}  
\]
Thus, adjoint orbits are classified into the following two categories:
\begin{description}
\item[{\rm (i)}]($2n$)-dimensional orbits with $\gamma \neq 0$)
\[ \gamma Z^\ast + \mathfrak{z}^{\perp} =\{l_{\alpha+\gamma y,\beta-\gamma x, \gamma}| x,y \in \mathbb{R}^n\}
\]
where $\mathfrak{z}^{\perp} = \{l \in \mathfrak{g}^\ast | l(Z) = 0\} = \mathbb{R}Y^\ast + \mathbb{R}X^\ast$. 

\item[{\rm (ii)}]($0$-dimensinal orbits with $\gamma = 0$) The $0$-dimensinal orbit is each point in $\mathfrak{z}^{\perp}$.
\end{description}

\paragraph{Irreducible representations associated with coadjoint orbits}
For the point orbits, $\mathfrak{g} = \mathfrak{h}_n$ is the radical $\mathfrak{r}_l$ and the only porlarizing subalgebra $\mathfrak{m}$. If $l(Z) = \gamma \neq 0$, then $\mathfrak{r}_l = \mathbb{R}Z$ and $\mathfrak{m}$ are many, for example 
\[ \mathfrak{m} = \mathbb{R}Z\oplus\mathbb{R}\mbox{-span}\{Y_1,\ldots, Y_n\} \quad {\rm or}\quad \mathfrak{m} = \mathbb{R}Z\oplus\mathbb{R}\mbox{-span}\{X_1,\ldots, X_n\} 
\]
Since unitary equivalence class of $\pi_{l,M}$ is independent of the choice of orbit representative $l$ and polaring subalgebra $\mathfrak{m}$, it is convenient to take an element $l \in \gamma Z^\ast+ \mathfrak{z}^{\perp}$ as $l=\lambda Z^\ast (\lambda = \gamma \neq 0)$ and $\mathfrak{m}$ as above. 
Then we have
\[ \chi_l\left(\exp\left(zZ+\sum_{j=1}^n y_jY_j\right)\right) =e^{2\pi\sqrt{-1}\lambda z}
\]
is a character on a normal subgroup $M = \exp \mathfrak{m}$ of $G$, which induces to a representation $\pi_l = \mbox{Ind}_M^G(\chi_l)$ on $G$. A description of $\pi_l$ as an action on $L^2(M\backslash G, \mathbb{C})$ is more enlightening than the standard model as an action on functions on $G$ varying like $\chi_l$ along $M$-cosets. The calculation of this new model is typical of induced representations and will often arise. Let
\[
A= \Big\{\exp\left(\sum_{j=1}^n t_jX_j\right)\Big| t_j \in \mathbb{R}\Big\}.
\]
Then $S$ is  a transversal for $M\backslash G$ since $M$ is normal, and Lebesgue measure $dt = dt_1\ldots dt_n$ transfers to  a  right-invariant measure 
$d\dot{g}$ on $M\backslash G$. Using $dt$ and $d\dot{g}$ to define $L^2$-norms, and identifying $\mathbb{R}^n$ with $A$ in the obvious way, 
the restriction map $f \to f|A$ becomes  an  isometry  from  the representation space $\mathcal{H}_{\pi_l}$ of $\pi_l$ to $L^2(\mathbb{R}^n)$. To compute  the  action  of  $\pi_l$ modeled in $L^2(\mathbb{R}^n)$. 
The fundamental problem in computing this action is to split products $(0, 0, t)(z, y, x)$ in the form $h\cdot (0, 0, t')$ where $h \in M$. From the multiplication law
\[ (z,y,x)\cdot (z',y',x') = \left(z+z'-\frac{y\cdot x'}{2}+\frac{x\cdot y'}{2}, y+y',x+x'\right),
\]
we see that
\begin{equation}
(0,0,t)\cdot (z,y,x) = \left(z+\frac{t\cdot y}{2}, y,x+t\right) =  \left(z+ t\cdot y+\frac{x\cdot y}{2}, y,0\right)(0,0,t+x).  \label{exponetialmultiplication}
\end{equation}
Hence if $f \in \mathcal{H}_{\pi_l}$, we have
\[
\pi_l(z, y, x)f(0, 0, t)=f\left( \left(z+ t\cdot y+\frac{x\cdot y}{2},y,0\right)\cdot (0, 0, t+ x)\right) = e^{2\pi\sqrt{-1}\lambda (z+t\cdot y+(1/2)x\cdot y)}f(0, 0, t+x).\]
Thus the action on $\tilde{f} \in L^2(\mathbb{R}^n)$ is given by
\begin{equation}
\pi_l(z, y, x)\tilde{f}(t)= e^{2\pi\sqrt{-1}\lambda (z+t\cdot y+(1/2)x\cdot y)}\tilde{f}(t+x) \quad (\lambda \neq 0).\label{2n+1heisenbergirred}
\end{equation}
Note again that this is different from the formula (\ref{schrep}) in the case when $n=1$ for the reason explained there. In fact, in the marix expression (\ref{expnentiltomatrix}), the equlity (\ref{exponetialmultiplication}) is changed to 
\begin{equation}
[0,0,t]\cdot [z^\ast,y,x] = \left[z^\ast+t\cdot y, y,x+t\right] =  \left[z^\ast+ t\cdot y, y,0\right][0,0,t+x]. \label{matrixmultiplication}
\end{equation}
Kirillov's theory says that any irreducible representations of $H_n$ are given as above. 

\paragraph{Fourier inversion formula}
Finally, we recall the Fourier inversion formula. 
Recall that Malcev basis $\{Z, Y_1,\ldots, Y_n, X_1,\ldots,X_n\}$ satisfies
\[  [X_i,Y_j] = \delta_{ij}Z,\quad 1 \leq i,j \leq n.
\] 
If $l_1,\ldots, l_{2n+1}$ is the basis for $\mathfrak{g}^\ast$, the indices $\{1, 2,\ldots, 2n + 1\}$ partition as $S = \{2,3,\ldots ,2n+1\}, T = \{l_1\}; V_T = \mathbb{R}l_1$ is  a  cross-section  for  the  generic orbits $U = \{l | l(Z) \neq 0\}$ and the Pfaffian is a polynomial on $\mathfrak{g}^\ast$ such
that
\[ 
|\mbox{Pf}(l)|^2 = \mbox{det}\left(\begin{array}{cccccc}
{}&{}&{}&l(Z)&{}&{}\\{}&O&{}&{}&\ddots&{} \\{}&{}&{}&{}&{}&l(Z) \\-l(Z)&{}&{}&{}&{}&{} \\{}&\ddots&{}&{}&O&{} \\{}&{}&-l(Z)&{}&{}&{}  \end{array}\right)
\]

Hence $|\mbox{Pf}(l)|= |l(Z)|^n$. If we identify $V_T= \{\alpha l_1 | \alpha \in \mathbb{R}\}$ and let $d\alpha$ be the normalized Lebesgue on $\mathbb{R}$, then the Fourier inversion formula becomes
\begin{equation} \phi(\sigma ) = \int_\mathbb{R} \mbox{Tr}(\pi_{\alpha l_1}(\sigma^{-1})\pi_{\alpha l_1}(\phi )|\alpha|^nd\alpha,\quad \mbox{for} \quad \phi \in \mathcal{S}(\mathbb{R})\label{2n+1heisenbergplancherel}
\end{equation}
t

\subsubsection{Engel group $E_4$}
We summurizes Examples 1.1.3, 1.2.5, 1.3.10, 2.2.7, 4.3.12  in \cite{Corwin}.
\paragraph{Definition}
We define the Engel Lie algebra $\mathfrak{e}_4$ to be the $4$-dimensional Lie algebra spanned by $W,X,Y,Z$ with
\[
[W, X]=Y, , [W,Y]= Z, [X,Y] = [W,Z] = [X,Z] = [Y,Z] =O.
\]
This one is a $3$ step nilpotent Lie algebra. One realization of matrix algebra is obtained by letting $wW+xX+yY+zZ$ correspond to the $4 \times 4$ matrix.

\[ 
\left(\begin{array}{cccc}
0 & w & 0 & z  \\ 
0 & 0 & w & y  \\				
0 & 0 & 0 & x  \\
0 & 0 & 0 & 0  \end{array}\right)
\]

\paragraph{Coadjoint orbits}
The basis $\{Z,Y,X,W\}$ is a strong Malcev, but we shall use the exponential coordinates on $E_4:= \exp \mathfrak{e}_4$,
\[
(z,y,x,w) = \exp(zZ + xX+ yY+ wW)
\]
below. The Campbell-Baker-Hausdorff formula yields the multiplication law in these coordinates:
\begin{align}
&{} (z, y, x, w) \ast (z', y', x', w') \notag \\
&=  (z + z' + \frac12(wy' -way)+ \frac{1}{12}(w^2x' -ww'x -ww'x') \label{engelmultiplication} \\ &{} + x(w')^2,y+y'+\frac12(wx'-w'x),x+x',w+w') \notag 
\end{align}
The adjoint action is given by
\[ {\rm Ad}_{(z,x,y,z)}(z', y', x', w') =(z'+ wy'-w'y +\frac12(w2x' - ww'x),
y' +wx' -w'x,x',we)
\]

Relative to the given basis in $\mathfrak{e}=4$ and the dual basis $\{Z^\ast = l_1,Y^\ast = l_2, X^\ast =l_3, W^\ast =l_4\}$ in $\mathfrak{e}_4^\ast$. If $l= \delta l_1 + \gamma l_2 + \beta l_3 + \alpha l_4$, the coadjoint orbits and convenient representatives are
\begin{description}
\item[{\rm (i)}]  $\mathcal{O}_{\delta,\beta} = {\rm Ad}^\ast (G)(\delta l_1 +  \beta l_3) = \{ \delta l_1 + tl_2 + ( \beta + t^2/2\delta )l_3+ sl_4 | s, t \in \mathbb{R} \}$, if $\delta \neq 0$,
\item[{\rm (ii)}] $\mathcal{O}_\gamma = {\rm Ad}^\ast(G)(\gamma l_2)= \gamma l_2+ \mathbb{R}l_3+ \mathbb{R}l_4$, where $\gamma \neq 0$,
\item[{\rm (iii)}] $\mathcal{O}_{\alpha,\beta} = {\rm Ad}^\ast (G)(\beta l_3 + \alpha l_4)= \{\beta l_3+ \alpha l_4\}$, where $\alpha, \beta \in \mathbb{R}$ (one point orbit)
\end{description}

\paragraph{Irreducible representations associated with generic orbits}
The generic orbits are parabolic cylinders $\mathcal{O}_{\delta,\beta}$ with $\delta \neq 0$. For the representatives $l_{\delta,\beta} = \delta l_1 + \beta l_3$, the radical $\mathfrak{r}_l = \mathbb{R}Z + \mathbb{R}X$; hence polarizing subalgebra $\mathfrak{m}$ are three-dimensional. 
Since $\mathfrak{m} = \mathbb{R}Z +\mathbb{R}Y+ \mathbb{R}X$ is an abelian ideal in $\mathfrak{e}_4$ of the correct dimension, it is a polarizing subalgebra for each $l_{\delta,\beta}$.

Write $(z,y,x,w) = \exp(zZ+yY+xX+wW)$. Obviously $M = \exp(\mathfrak{m}) = \{(z,y,x,0)| z, y, x \in \mathbb{R}\}$ and the character $\chi_{\delta,\beta} := \chi_{l_{\delta,\beta}}$ on $M$ is
\[
\chi_{\delta,\beta}(z,y,x,0)= e^{2\pi\sqrt{-1}(\beta x +\delta z)}
\]
Also $M\backslash G$ with $G=E_4$ has $\Sigma = \exp{\mathbb{R}W}$ as a cross section, which we identify with $\mathbb{R}$; 
Lebesgue measure $dt$ on $\mathbb{R}$ gives a right-invariant measure $d\dot{g}$ on $M\backslash G$.
Using these measures to define $L^2$-norms we have an isometry between the representation space $\mathcal{H}_{\delta,\beta}$ of $\pi_{\delta,\beta} = {\rm Ind}_M^G(\chi_{\delta,\beta})$ and $L^2(\mathbb{R})$. 
We compute the action of $\pi_{\delta,\beta}$ modeled in $L^2(\mathbb{R})$ from the multiplication law given in (\ref{engelmultiplication}); first note that
\begin{align*}
&{}(0, 0, 0, t)\cdot(z, y, x, w) \\
&{} \quad  =(z +\frac12 ty+ \frac{1}{12}[t^2 x - twx ], y+ \frac12 tx, x, t + w) \\
&{} \quad =(z + ty + \frac12 t^2x +\frac12 twx + \frac12 wy + \frac16 w^2x, y+ tx +\frac12 wx, x, 0)\cdot(0, 0, 0, t + w).
\end{align*}
Thus if $f \in \mathcal{H}_{\delta,\beta}$
\begin{align*}
&{}[\pi_{\delta,\beta}(z,y,x,w)f](0, 0, 0, t) = f((0, 0, 0, t)\cdot (z, y, x, w)) \\
&{}= e^{2\pi\sqrt{-1}\beta x}e^{2\pi\sqrt{-1}\delta(z+ty+(1/2)t^2x +(1/2)twx +(1/2)wy + (1/6)w^2x)}f(0,0,0,t+w)
\end{align*}
and the action  modeled  in $L^2(\mathbb{R})$ is
\begin{align}
&{}[\pi_{\delta,\beta}(z,y,x,w)\tilde{f}](t) \notag \\
&{}= e^{2\pi\sqrt{-1}\beta x}e^{2\pi\sqrt{-1}\delta(z+ty+(1/2)t^2x +(1/2)twx +(1/2)wy + (1/6)w^2x)}\tilde{f}(t+ w). \label{engelirred}	
\end{align}
This action is easier to visualize if one examines the action of individual one-parameter subgroups. For instance,
\begin{align*}
&{}[\pi_{\delta,\beta}(\exp xX)\tilde{f}](t) = e^{2\pi\sqrt{-1}\beta x}e^{2\pi\sqrt{-1}\delta(1/2)t^2x}\tilde{f}(t),\\
&{}  [\pi_{\delta,\beta}(\exp wW)\tilde{f}](t) = \tilde{f}(t+w).
\end{align*}

We omit the description of the other orbits and irreducible representations since they do not appear in the Fourier inversion formula below (cf. 2.2.7 Example in \cite{Corwin}).

\paragraph{Fourier inversion formula}
Recall that $\{l_1,l_2,l_3,l_4 \}$ is the dual basis for $\mathfrak{g}^\ast$ of strong Malcev basis $\{Z, Y, X, W\}$ then, we have 
\[
S =\{2,4\},T =\{1, 3\},
\]
for the notions in Theorem \ref{corwin316}.
\begin{align*}
\mbox{Generic orbits}\quad  U &= \left\{\sum \alpha_il_i | \alpha_1 \neq 0 \right\}, \\
\mbox{Cross section} V_T \cap U &= \{\delta l_1 + \beta l_3 | \delta \neq 0, \beta \in \mathbb{R} \}.
\end{align*}
The Pfaffian is a polynomial on $\mathfrak{g}$ such that 
${\rm Pf}(l)^2 = l(Z )^2$, so ${\rm Pf}(Z)=l(Z)$. 
Identifying $V_T= \{\delta l_1+\beta l_3 | \delta, \beta \in \mathbb{R}\}$ and letting $d\delta d\beta$ be normalized Lebesgue measure, 
we have $|{\rm Pf}(l)|=|\delta |$; if $\pi_{\delta,\beta}$ is the representation corresponding to $l_{\delta,\beta} = \delta l_1 + \beta l_3$, 
the Fourier inversion formula is
\begin{equation} \phi(\sigma ) = \int_{\mathbb{R}^2} \mbox{Tr}(\pi_{\delta,\beta}(\sigma^{-1})\pi_{\delta,\beta}(\phi )|\delta|d\delta d\beta,\quad \mbox{for} \quad \phi \in \mathcal{S}(\mathbb{R})\label{engelplancherel}
\end{equation}

\subsubsection{The group of $4 \times 4$-upper trianglular matrices $N_4$}
We summurizes Examples 1.1.4, 1.2.6, 1.3.11, 2.2.8, 4.3.13  in \cite{Corwin}.

\paragraph{Definition}
Let $\mathfrak{g} = \mathfrak{n}_4$ be the Lie algebra of strictly upper triangular $4 \times 4$  matrices; it is a $3$ step nilpotent algebra, of dimension $6$, and its center is one-dimensional. A typical element $X$ can be written as
\begin{align} X &= (z,y_1,y_2,x_1,x_2,x_3)= zZ +y_1Y_1 +y_2Y_2+x_1X_1+x_2X_2+x_3X_3 \notag	\\
&= \left( \begin{array}{cccc}
0 & x_1 & y_1 & z \\
0 & 0   & x_2 & y_2 \\
0 & 0 & 0 & x_3 \\
0 & 0 & 0 & 0
\end{array}\right) \label{liealgebran4}
\end{align}
The calculation of $W_1 \ast W_2$ ($W_1, W_2 \in \mathfrak{g}$) will not be useful for us, and we omit it. We shall, however, compute ${\rm Ad}$:
\begin{align*}
&{}({\rm Ad}( \exp(z, y_1, y_2, x_1, x_2, x_3)))(z', y'_1, y'_2, x'_1, x'_2, x'_3) \\
&= (z'+x_1y'_2-x'_1y_2+y_1x'_3 -  y'_1x_3 + \frac12(x_1x_2x'_3-2x_1x'_2x_3 + x'_1x_2x_3),\\
&{}  y'_1+x_1x'_2 -x_2x'_1, y'_2 + x_2x'_3-x_3x'_2,x'_1,x'_2,x'_3 )
\end{align*}

In the above, $Z,\ldots, X_3$ is a basis of $\mathfrak{n}_4$ corresponding to the matrix entries. The nontrivial commutators are easily computed:
\[
[X_1,X_2] = Y_1 ,\; [X_3,X_2] = -Y_2,\;	[X_1,Y_2] = Z,\; [X_3,Y_1] = - Z.
\]
Since $\mathfrak{a}=\mathbb{R}\mbox{-span}\{Z, Y_1, Y_2, X_2\}$ is the largest abelian ideal in $\mathfrak{n}_4$, $\{Z,Y_1,Y_2,X_2,X_1,X_3\}$ is a strong Malcev basis for $\mathfrak{n}_4$.  

Similarly as (\ref{liealgebran4}), elements $l$ of $\mathfrak{n}_4^\ast$
are written as 
\begin{align} l &= (\alpha,\beta_1,\beta_2,\gamma_1,\gamma_2,\gamma_3) \notag \\ &= \alpha l_1+\beta_1l_2+\beta_2l_3+\gamma_1l_4+\gamma_2l_5+\gamma_3l_6 \notag \\
&= \left( \begin{array}{cccc}
0 & 0 & 0 & 0 \\
\gamma_1 & 0 & 0 & 0 \\
\beta_1 & \gamma_2 & 0 & 0 \\
\alpha & \beta_2 & \gamma_3 & 0
\end{array}\right) \label{duallien4} 
\end{align}
where $\{l_1,\ldots, l_6\}$ is the dual basis of $\{Z,\ldots, X_3\}$.

\paragraph{Coadjoint orbits}
In matrix form, the duality is given by 
\[l(X) = \alpha z + \cdots + \gamma_3x_3.\]
 By the computation of the coadjoint action of $N_4$ on $\mathfrak{n}_4$ in Example 1.3.11 in \cite{Corwin}, we see that the orbits have representatives and parametrizations as follows.
\begin{align*} {\rm (i)}\quad \mathcal{O}_{\alpha',\gamma'_2} &= {\rm Ad}^\ast G(\alpha', 0, 0, 0, \gamma'_2, 0) \\
&= \left\{\left.\left(\alpha', t_1, t_2, s_1, \gamma'_2 + \frac{t_1t_2}{\alpha'},s_2\right)\right| t_1,t_2,s_1,s_2 \in \mathbb{R}\right\} \\
&= \{l| \alpha = \alpha',\alpha\gamma'_2 - \beta_1\beta_2 = \alpha'\gamma_2\}
\end{align*}
where $\alpha' \neq 0$, $\gamma'_2 \in \mathbb{R}$. (These are the orbits of $l$ such that $\alpha \neq 0$).

\begin{align*}{\rm (ii)}\quad \mathcal{O}_{\beta'_1,\beta'_2,\gamma'_3} &=
\mbox{Ad}^\ast G(0, \beta'_1, \beta'_2,0, 0,\gamma'_3) \\
&= \left\{\left.\left(0,\beta'_1,\beta'_2,t_1,t_2, \frac{\beta'_1\gamma'_3-t_1\beta'_2}{\beta'_1}\right) \right| t_1,t_2 \in \mathbb{R}\right\} \\
&= \{l| \alpha=0, \beta_1= \beta'_1, \beta_2 = \beta'_2,\gamma_1\beta_2 +\gamma_3\beta_1 = \beta'_1\gamma'_3 \}
\end{align*}
where $\beta'_1 \neq 0$, $\beta'_2, \gamma'_3 \in \mathbb{R}$. (These are the orbits of $l$ such that $\alpha=0, \beta_1 \neq 0$.) 

\begin{align*}{\rm (iii)}\quad \mathcal{O}_{\beta'_2,\gamma'_1} &=
\mbox{Ad}^\ast G(0, 0, \beta'_2,\gamma'_1, 0,0) \\
&= \{(0,0,\beta'_2,\gamma'_1,t_1,t_2) | t_1,t_2 \in \mathbb{R}\} \\
&= \{l|\alpha= \beta_1 = 0, \beta_2 = \beta'_2, \gamma_1 = \gamma'_1\}
\end{align*}
where $\beta'_2 \neq 0$, $\gamma'_1 \in \mathbb{R}$. (These are the orbits of $l$ such that $\alpha= \beta_1 = 0, \beta_2 \neq 0$.)

\[{\rm (iv)}\quad \mathcal{O}_{\gamma'_1,\gamma'_2,\gamma'_3} =
\mbox{Ad}^\ast G(0, 0, 0,\gamma'_1, \gamma'_2, \gamma'_3) = \{(0, 0, 0,\gamma'_1, \gamma'_2, \gamma'_3)\} 
\]
where $\gamma'_1,\gamma'_2,\gamma'_3 \in \mathbb{R}$. (These are the orbits of $l$ such that $\alpha= \beta_1 = \beta_2 = 0$.)
\paragraph{Irreducible representations associated with generic orbits}
Generic orbits (i) are $4$ $4$-dimensional hypersurfaces; those in (ii) and (iii) are
$2$ $2$-dimensional planes and those in (iv) are points. 
Taking the indicated orbit representatives $l=l_{\alpha',\gamma'_2}$, in Case (i), $\mathfrak{r}_l = \{(z, 0, 0, 0, x_2 , 0)| z, x_2 \in \mathbb{R}\}$ and the abelian ideal $\mathfrak{m} =\{(z, y_1, y_2, 0, x_2, 0)| z, y_1,y_2, x_2 \in \mathbb{R}\}$ is a polarizing subalgebra for each such $l$. 
Since $\mathfrak{m}$ is an ideal, we may use the exponential coordinates
$(z,\ldots, x_3) = \exp X$ in $G = N_4$ and obtain a cross-section for $M\backslash G $,
$\Sigma = \{(0,0,0,t_1,0,t_3)| t_1,t_3 \in \mathbb{R}\}$. The action of $\pi_{\alpha',\gamma'_2}$ can then  be modeled in $L^2(\mathbb{R}^2)$. 
The necessary splitting of the product, computed  using the Campbell-Hausdorff formula, is
\[
(0,0,0,t_1,0,t_3)\cdot (z,\ldots, x_3)=(z',y'_1, y'_2,0,x_2,0)\cdot (0,0,0,t_1+x_1,0,t_3 +x_3),
\]
where
\begin{align*}
z' &= z + \frac{x_1y_2}{2}-\frac{x_3y_1}{2} -t_1x_2t_3-\frac12x_1x_2t_3 -\frac12t_1x_2x_3+\frac13x_1x_2x_3 \\
y'_1 &= y_1 + t_1x_2+\frac{x_1x_2}{2} \\
y'_2 &= y_2 - t_3x_2+\frac{x_2x_3}{2}.
\end{align*}
Thus, the action on $\tilde{f} \in L^2(\mathbb{R}^2)$ is
\begin{align}
&{}[\pi_{\alpha',\gamma'_2}\tilde{f}](t_l,t_3) \notag \\
&= e^{2\pi\sqrt{-1}\gamma'_2x_2}e^{2\pi\sqrt{-1}\alpha'[z + (1/2)x_1y_2 -x_3y_1 -t_1x_2t_3-(1/2)x_1x_2t_3 -(1/2)t_1x_2x_3+(1/3)x_1x_2x_3]} \notag \\
&{} \cdot\tilde{f}(t_1+x_1,t_3 +x_3). \label{4upperirred}
\end{align}
The above computation is enough for our later arguments, and thus, we proceed to the next step. 
(The calculations of the nongeneric representations are left to the reader in Example 2.2.8. in \cite{Corwin}).

If we choose the dual basis in (\ref{duallien4}), we have
\[ 
S=\{2,3,5,6\},\quad T=\{1,4\}
\]
and
\begin{align*}
\mbox{Generic orbits}\quad  U &= \{\alpha l_1+\beta_1l_2+\beta_2l_3+\gamma_1l_4+\gamma_2l_5+\gamma_3l_6 | \alpha \neq 0 \},\\
\mbox{Cross section} V_T \cap U &= \{\alpha l_1 + \gamma_1l_4 | \alpha \neq 0 \}.
\end{align*}

The Pfaffian $\mbox{Pf}(l)$ is expressed as 
\[
\mbox{Pf}(l)^2 = \mbox{det} B_l(X_{i_j}, X_{i_k}) = \mbox{det}\left(\begin{array}{cccc}0 & 0 & 0 & l(Z) \\
0 & 0 & l(Z) & 0 \\
0 & l(Z) & 0 & 0 \\
l(Z) & 0 & 0 & 0
\end{array}\right)
\]

\paragraph{Fourier inversion formula}
If we identify $V_T \cong \mathbb{R}^2$ and let $d\alpha d\gamma_1 =$ normalized Lebesgue measure on $\mathbb{R}^2$, the Fourier inversion formula and the Plancherel formula become
\begin{align}
 \tilde{f}(\sigma ) &= \int_{\mathbb{R}^2}\mbox{Tr}(\pi_{\alpha,\gamma_1}(\sigma^{-1})\pi_{\alpha,\gamma_1}(\tilde{f}))|\alpha|^2 d\alpha d\gamma_1, \; \tilde{f} \in \mathcal{S}(\mathbb{R}^2) \notag \\
\|f\|_2^2 &= \int_{\mathbb{R}^2}\|(\pi_{\alpha,\gamma_1}(f)\|_{\rm HS}^2|\alpha|^2 d\alpha d\gamma_1,  \label{4upperplancherel}\; f \in \mathcal{S}(\mathbb{R}^2).
\end{align}
where $\|\cdot \|_{\rm HS}$ denotes the Hilbert-Schmidt norm.

\section{Harmonic Theory of Chen's iterated integral: Homology connection}\label{Homologyconnection}

We have discussed the harmonic theory of iterated integral in the case $\Gamma = {\rm Heis}_3(\mathbb{Z})$ in Section \ref{harmoniclietheory}.

We separate the discussion for the case where $\Gamma$ is a general torsion-free nilpotent group, as it is more convenient to reformulate in terms of Chen's power series connection and the homology connection to discuss the harmonic theory of iterated integrals.

This strategy helps search for reasonable solutions to several (modified) Riemann-Hilbert problems discussed in Section \ref{mRHoper}. For future reference, we quote several related notions from \cite{Chen}, \cite{Kohno} here.
Main contents here are quoted from Kohno's book \cite{Kohno} with English translation.

\subsection{Formulation of general iterated integrals}\label{generaliteratedintegrals}
Let us introduce the iterated integral for differential $p$-forms on the differentiable manifold $M$. It is formulated as differential forms on the piecewise smooth space of paths $PM$ and thus, it gives a reformulation of the iterated integrals on differential $1$-forms given in Section \ref{Lieintegral}.

First, we will show that from a differential $p$-form on the differentiable manifold $M$, we can construct a differential $(p-1)$-form on the space of paths $PM$ through integration.

Let $M$ be a differentiable manifold and let $x_0, x_1 \in M$. We denote by
\[
P(M; x_0, x_1)
=
\bigl\{\gamma \colon I \to M \,\big|\, \gamma\in C^\infty,\ \gamma(0)=x_0,\ \gamma(1)=x_1 \bigr\}
\]
 We call $P(M; x_0, x_1)$ the path space of $M$ from $x_0$ to $x_1$.

Moreover, we denote by
\[
PM = \bigl\{\gamma \colon I \to M \,\big|\, \gamma\in C^\infty\bigr\}
\]
 We call $PM$ the (free) path space of $M$. We endow both $P(M; x_0, x_1)$ and $PM$ with the compact-open topology and regard them as topological spaces.

To define the differential forms on $PM$, we define a map.
\[
  \varphi\colon I \times PM \;\longrightarrow\; M,\qquad
  \varphi(t,\gamma)=\gamma(t).
\]
For an openset $U$ of $\mathbb{R}^n$, take a local chart
\[
  \phi\colon U \;\longrightarrow\; PM.
\]
defined by a local parameter system
\[
  \varphi_U\colon I \times U \;\longrightarrow\; M,\qquad
  \varphi_U(t,u)=\phi(u)(t).
\]

The pullback $\varphi_U^*\omega$ is considered a $p$-form on the product $I \times U$ and decomposed into the component involving $dt$ and the complementary part $\beta$ as
\[
  \varphi_U^*\omega\;=\;dt\wedge\alpha\;+\;\beta.
  \label{eq:decompose}
\]
Using local coordinates $(x^1,\dots,x^n)$ on $U$, write
\[
  \alpha
  \;=\;
  \sum_{1\le i_1<\cdots<i_{p-1}\le n}
    \alpha_{i_1\cdots i_{p-1}}(t,x^1,\dots,x^n)\;
    dx^{i_1}\wedge\cdots\wedge dx^{\,i_{p-1}}.
  \label{eq:alpha}
\]

The fiber integral of $\varphi_U^\ast\omega$ over the unit interval $I=[0,1]$ is then defined by
\[
  \int_I \varphi_U^\ast\omega
  \;=\;
  \sum_{1\le i_1<\cdots<i_{p-1}\le n}
  \Bigl(\int_0^1 \alpha_{i_1\cdots i_{p-1}}(t,x)\,dt\Bigr)\;
  dx^{i_1}\wedge\cdots\wedge dx^{\,i_{p-1}}.
  \label{eq:fiberIntegral}
\]

This produces a $(p-1)$-form on $U$.  Since for each local chart 
$\phi\colon U\to PM$, the result $\int_I\varphi_U^*\omega$ is a well-defined $(p-1)$-form on $U$, these local forms patch together to give a global 
$(p-1)$-form on the path space $PM$.  

Equivalently, if $p\colon I \times PM \;\longrightarrow\; PM$ defined by $p(t,\gamma)=\gamma$ is the projection, then one writes
  $p_\ast(\varphi^*\omega)\;=\;\int_I\varphi^\ast\omega$
and calls this the fiber integral of the $p$-form $\omega$ along the interval $I$.

By composing the pullback $\phi^\ast$ with the fiber integration, one can construct, from a $p$-form $\omega$ on the manifold $M$, a $(p-1)$-form on the whole path space $PM$.  
From now on, we will abbreviate this process by writing.
\[
\int_I \phi^* \omega \;=\;\int \omega.
\]

Finally, we define iterated integrals of differential forms on $M$ as a differential form on the space of paths $PM$ through integration.
Starting from differential forms 
$\omega_1,\dots,\omega_k$ on $M$ of degrees $p_1,\dots,p_k$, the iterated integral produces a form of total degree $p_1+\cdots+p_k$ on the path space $PM$:
\[
  \int_{\Delta_k}\varphi^*\bigl(\omega_1\wedge\cdots\wedge\omega_k\bigr).
\]
where $\Delta_k$ denotes the standerd $k$-simplex.
We denote this form by
\[
  \int_{\Delta_k}\omega_1\wedge\cdots\wedge\omega_k,
\]
and refer to it as the iterated integral of $\omega_1,\dots,\omega_k$.

In addition, this construction implies a multilinear map
\[
  I\colon
  \mathcal{A}^{p_1}(M)\,\wedge\cdots\wedge\,\mathcal{A}^{p_k}(M)
  \;\longrightarrow\;
  \mathcal{A}^{p_1+\cdots+p_k}(PM),
\]
where $\mathcal{A}^p(M)$ is the space of $p$-forms on $M$.

Let 
\[
  p\colon \Delta_k\times PM \;\longrightarrow\; PM
\]
be the projection and 
\[
  \varphi\colon \Delta_k\times PM \;\longrightarrow\; M^k
\]
the evaluation map $\varphi(t_1,\dots,t_k;\gamma)=(\gamma(t_1),\dots,\gamma(t_k))$.  Then, the iterated integral is reconstructed by composing the pullback with the fiber integration.

\subsection{Chen's homology connection and the Hodge decomposition}

First, let us recall the (noncommutative) formal power series connection. Let $M$ be a smooth manifold with a finite-dimensional homology group $H_\ast(M,\mathbb{R})$. 
Consider a dirert sum $H_+(M,\mathbb{R})$ of its positive part
\[ H_+(M,\mathbb{R}) := \bigoplus_{q>0}H_q(M,\mathbb{R}), 
\]
and the tensor algebra $TH_+(M,\mathbb{R})$ generated by  $H_+(M,\mathbb{R})$
\[ TH_+(M,\mathbb{R}) := \bigoplus_{k>0}\left(\bigotimes_{q>0}^k H_q(M,\mathbb{R})\right), 
\]
Take a linear basis $z_1,z_2,\dots, z_m$ be a basis of $H_+(M,\mathbb{R})$ such that $z_j \in H_{q_j}(M,\mathbb{R})$, $1 \leq q_j \leq m$. 
Corresponding to $z_1,\ldots, z_m$, take indeterminates $X_1,X_2,\dots, X_m$. Then, we can identify $TH_+(M,\mathbb{R})$ with non-commutative polynomial ring $\mathbb{R}\langle X_1, \ldots, X_m\rangle$. Next, we define a 
degree $\mbox{deg}\; X_j$ of $X_j$ by 
\[ \mbox{deg}\; X_j = q_j - 1,
\]
and extending the degree to each element in $TH_+(M,\mathbb{R})$, 
which implies this algebra to be graded. 

Next, we define the augmentation map $\varepsilon :\mathbb{R}\langle X_1, \ldots, X_m\rangle \to \mathbb{R}$ as the algebra-homomophism satisfying
\[ \varepsilon(X_j)= 0 \quad \mbox{for} \quad 1 \leq q_j \leq m\quad \mbox{and}\quad \varepsilon(c)= c \quad \mbox{for} \quad c \in \mathbb{R}.  \]
and the augmentation ideal $J$ by $J:= \mbox{Ker}\;\varepsilon$.

To define Chen's homology connection, we need the notion of noncommutative power series $\mathbb{R}\langle\langle X_1, \ldots, X_m\rangle\rangle$ with indeterminates $X_1, X_2, \dots, X_m$, which is obtained by a completion of $\mathbb{R}\langle X_1, \ldots, X_m\rangle$ by the topology determined by the power $J^k, k \geq 1$ of the augmentation ideal $J$. Namely, it is defined as the projective limit
\[ \mathbb{R}\langle\langle X_1, \ldots, X_m\rangle\rangle = \lim_{\leftarrow}\mathbb{R}\langle X_1, \ldots, X_m\rangle/J^k. 
\]
Let $\hat{J}$ be an ideal of $\mathbb{R}\langle\langle X_1, \ldots, X_m\rangle\rangle$ consisting of elements with vanishing constant terms.
Namely, an element of $\hat{J}$ can be expressed as 
\[ \sum_{i=1}^ma_iX_i+\cdots + \sum_{i_1,\ldots,i_k =1}^m a_{i_1,\ldots,i_k}X_{i_1}\cdots X_{i_k}+\cdots.
\]
It can be obtained upon completing the augmentation ideal $J$.
The degree of its elements is defined similarly to that of $J$. 
We denote $\mathbb{R}\langle X_1, \ldots, X_m\rangle_q$ the sub-module of $\mathbb{R}\langle\langle X_1, \ldots, X_m\rangle\rangle$ generated by the monomials of degree $q$.

For a smooth manifold $M$, let $\mathcal{A}^\ast(M)$ be the de Rham complex of $M$. We assume that the dual basis $z^1,\ldots, z^m$ of $z_1,\ldots, z_m$ are idetified with a basis $\omega_1, \ldots, 
\omega_m$ consisting by closed forms in $H^\ast(M,\mathbb{R})$.
 Namely, if $z_i, (1 \leq i \leq m)$ is represented by a cycle $c_i, (1 \leq i \leq m)$, then the integral $\int_{c_i}\omega_j$ of $\omega_j$on $c_i$ is equal to $\delta_{ij} = \left\{\begin{array}{cc}1 & (i=j)\\
0 & (i\neq j)\end{array}\right.$.

We define (non-commutative) formal power series connection $\omega$ by an element of $\mathcal{A}^\ast(M)\langle\langle X_1, \ldots, X_m\rangle\rangle$ the following form:
\begin{equation}   \omega = \sum_{i=1}^m\omega_iX_i + \cdots + \sum_{i_1,\ldots 
,i_k =1}^m \omega_{i_1,\ldots,i_k}X_{i_1}\cdots X_{i_k}+\cdots
\label{powerseriesconnection} \end{equation}
where $\omega_1,\ldots,\omega_{i_1,\ldots,i_k},\ldots$ are forms of positive degree on $M$.  
For a chain $\sigma$ in $M$, define integral $\int_\sigma \omega$
by 
\begin{equation*}   \int_\sigma\omega = \sum_{i=1}^m\int_\sigma\omega_iX_i + \cdots + \sum_{i_1,\ldots 
,i_k =1}^m \int_\sigma\omega_{i_1,\ldots,i_k}X_{i_1}\cdots X_{i_k}+\cdots.
\end{equation*}

Furthermore we define the exterior derivative $d$, the wedge product $\wedge$ on $\mathcal{A}^\ast(M)\langle\langle X_1, \ldots, X_m\rangle\rangle$ as (bi-)linear extensions of the following formula:
\begin{description}
\item[{\rm (1)}]{(the exterior derivative)} For a differential form $\tau$ on $M$ and a monomial $Z \in \mathbb{R}\langle\langle X_1, \ldots, X_m\rangle\rangle$,
\[   d(\tau Z) := (d\tau )Z
\]
\item[{\rm (2)}]{(the wedge product)} For differential forms $\tau_1, \tau_2$ on $M$ and monomials $Z_1, Z_2 \in \mathbb{R}\langle\langle X_1, \ldots, X_m\rangle\rangle$,
\[   (\tau_1 Z_1)\wedge(\tau_2 Z_2) := (\tau_1\wedge \tau_2)Z_1Z_2
\]
\end{description}
Moreover, if both differential form $\tau_1, \tau_2$ on $M$ has
positive degree, the iterated integral $\int(\tau_1Z_1)(\tau_2Z_2)$ of $\tau_1Z_1$ and $\tau_2Z_2$ is defined as
\[  \int(\tau_1Z_1)(\tau_2Z_2) := \left(\int\tau_1\tau_2\right)Z_1Z_2
\]
and taking its linear extensions, we can define iterated integrals 
\[   \int\omega\omega, \int\omega\omega\omega, \cdots
\]
for $\omega \in \mathcal{A}^\ast(M)\langle\langle X_1, \ldots, X_m\rangle\rangle$.

Next, we extend the definition of the augmentation map $\varepsilon$ to a map 
\[\varepsilon: \mathcal{A}^\ast(M)\langle\langle X_1, \ldots, X_m\rangle\rangle \to \mathcal{A}^\ast(M)\langle\langle X_1, \ldots, X_m\rangle\rangle,
\] 
by linear extension of the following property:
\[ \varepsilon(\tau Z) = \varepsilon(\tau)Z = \left\{\begin{array}{cc}\tau Z & \mbox{if deg $\tau$ is even}\\
-\tau Z & \mbox{if deg $\tau$ is odd}\end{array}\right.
\]
for a differential form $\tau$ on $M$ and a monomial $Z \in \mathbb{R}\langle\langle X_1, \ldots, X_m\rangle\rangle$.

For a (non-commutative) formal power series connection $\omega$, we define its curvature form $\kappa$ of $\omega$ by
\begin{equation} \kappa = d\omega - \varepsilon(\omega)\wedge\omega, \label{curvatureformal}
\end{equation} 
This notion is an extension of the usual curvature form for a connection one-form $\omega$.

Moreover, a linear transform $\delta$ defined on a non commutative formal power series ring $\mathbb{R}\langle\langle X_1,X_2,\ldots,X_m \rangle\rangle$ is called derivation if for  
\[ f(X_1,\ldots,X_m) = a_0 + \sum_{i=1}^m a_iX_i + \cdots + \sum_{i_1,\ldots 
,i_k =1}^m a_{i_1,\ldots,i_k}X_{i_1}\cdots X_{i_k}+\cdots.
\]
It is defined as 
\[ \delta f= \sum_{i=1}^m a_i\delta X_i + \cdots + \sum_{i_1,\ldots 
,i_k =1}^m a_{i_1,\ldots,i_k}\delta (X_{i_1}\cdots X_{i_k})+\cdots
\]
and satisfies the following conditions:
\begin{description}
\item[{\rm (1)}] $\delta$ decreases degree $1$, namely
\[ \delta:\mathbb{R}\langle\langle X_1,X_2,\ldots,X_m \rangle\rangle_q \to \mathbb{R}\langle\langle X_1,X_2,\ldots,X_m \rangle\rangle_{q-1}
\]
where $\mathbb{R}\langle\langle X_1,X_2,\ldots,X_m \rangle\rangle_q$ is a subspace spanned by degree $q$ monomimals. 

\item[{\rm (2)}](Leibnitz rule) For monomials $u,v$,
\[   \delta (uv) = (\delta u)v + u(\delta v)
\] 
holds.
\item[{\rm (3)}] For $X_j, 1 \leq j \leq m$, $\delta X_j \in \widehat{J}^2$. 
\end{description}

Here we recall that indeterminates $X_1,X_2,\dots, X_m$ are correspond to elements  in linear basis $z_1,z_2,\dots, z_m$ of $H_+(M,\mathbb{R})$. 
In contrast to the case for the exterior derivative $d$, we extend $\delta$ as an operator defined on $\mathcal{A}^\ast(M)\langle\langle X_1, \ldots, X_m\rangle\rangle$ as linear extensions of the following formula:
\[   \delta(\tau Z) := \tau\delta(Z),
\]
for a differential form $\tau$ on $M$ and a monomial $Z \in \mathbb{R}\langle\langle X_1, \ldots, X_m\rangle\rangle$,

We call this formal power series connection $\omega$ homology connection if the following two conditions are fulfilled:
\begin{description}[leftmargin=0.0cm] 
\item[{\rm (1)}](flatness condition)
\begin{equation} \delta\omega + \kappa = 0 \label{flatcondition}
\end{equation}

Note that this condition means that each $k$-component satisfies $(\delta\omega + \kappa)_k = 0 $ for each $k = 0,1,2$
 \item[{\rm (2)}] (degree condition) The degree of $\omega_{i_1,\ldots,i_k}$ is $q_{i_1}+\cdots + q_{i_k}-k+1$. namely 
\[ \mbox{deg}\; \omega_{i_1,\ldots,i_k}= \mbox{deg}\; X_{i_1}\cdots X_{i_k}+1
\]
\end{description}

If $M$ is a compact Riemannian manifold, then the following Hodge decomposition formula holds:
\begin{equation} \mathcal{A}^\ast(M) = \mathcal{H}\oplus B_d\oplus B_{d^{\ast}}\label{decomposition1}
\end{equation} 
where $\mathcal{H}$ is the space of harmonic forms, $B_d, B_{d^{\ast}}$ are the images of the exterior operator $d$ and its adjoint operator $d^{\ast}$ respectively. Moreover if we put $\mathcal{B} := B_{d^{\ast}}$, then $B_d = d\mathcal{B}$ and thus, the above decomposition (\ref{decomposition1}) can be rewritten as
\begin{equation} \mathcal{A}^\ast(M) = \mathcal{H}\oplus d\mathcal{B}\oplus \mathcal{B} \label{decomposition2}
\end{equation} 

Then the following theorem can be viewed as one can take a homology connection which is analogous to the harmonic representative described in Remark \ref{harmonicrepresentative}. 

\begin{proposition} If there is a decomposition in the form as (\ref{decomposition2}) of $\mathcal{A}^\ast(M)$ satisfying
\begin{description}
\item[{\rm (1)}] all elements in $\mathcal{H}$ are closed forms, and the inclusion $i:\mathcal{H} \to \mathcal{A}^\ast(M)$ induces an isomorphism between $\mathcal{H}$ and de Rham cohomology $H_{\mbox{DR}}^\ast(M)$ of $M$.
\item[{\rm (2)}] $\mathcal{B}$ is a graded subspace of $\mathcal{A}^\ast(M)$ and any closed form in $\mathcal{B}$ is only vanishing form $0$.
\end{description}
then, there is a unique homology connection $(\omega,\delta)$ such that 
\begin{equation}
\omega_i \in \mathcal{H},\quad 1 \leq  i \leq m, \quad \omega_{i_1,\ldots i_k} \in \mathcal{B},\quad k \geq 2.
\end{equation}
\end{proposition} 

This proposition can be seen as a generalization of Condition \ref{harmoniccoexact}.

\subsection{Holonomy representation of the fundamental group}
Here, we explain a universal presentation of the holonomy representation of the fundamental group using the homology connection.

Take a degree $0$-part $\omega_0$ of a homology connection $\omega$. Then, $\omega_0$ is an element in $\mathcal{A}^1(M)\otimes\mathbb{R}\langle \langle X_1,\ldots, X_m\rangle \rangle$, namely a closed one form on $M$ with noncommutative formal power series coefficients, thus in general, it is an infinite sum. Moreover, we put its curvature $\kappa_0$ as
\[     \kappa_0 = d\omega_0 + \omega_0\wedge\omega_0.
\]  
It is the usual curvature form of a connection one form $\omega_0$.

We define the transport $T$ using the iterated integrals on a loop with base point $x_0 \in M$ as
\[ T = 1 + \sum_{k=1}^\infty \int \underbrace{\omega\cdots\omega}_k
\]
then $T$ can be viewed as an element in $\mathcal{A}(\Omega M)\langle \langle X_1,\ldots,X_m\rangle\rangle$. In this case, it coincides with the holonomy of the connection one form $\omega_0$, but in a more general context of homology connections, it is its generalization.  
From $T$, the holonomy map $\Theta$ is introduced as follows:
\begin{align}
 \Theta:C_\ast(\Omega M)\otimes \mathbb{R} &\to \mathbb{R}\langle \langle X_1,\ldots,X_m\rangle\rangle \notag \\
\Theta(c) &= \langle T,c\rangle,\quad C_\ast(\Omega M) \label{holonomymap}
\end{align}
where $\Omega M$ is the loop space of $M$ based at $x_0$. Then, it is a chain map. Let us notice the map
\[ H_0(\Omega M;\mathbb{R}) \to H_0(\mathbb{R}\langle \langle X_1,\ldots,X_m\rangle\rangle,\delta )
\]
which is induced from $\Theta$. 
The homology $H_0(\Omega M;\mathbb{R})$  has a product defined by the composition of loops and then, it is isomorphic to the the group algebra $\mathbb{R}\pi_1(M,x_0)$ of the fundamental group $\pi_1(M,x_0)$ over $\mathbb{R}$. 

We denote by $\mathcal{N}$ the image of the boundary operator
$\delta:\mathbb{R}\langle \langle X_1,\ldots,X_m\rangle\rangle_1 \to \mathbb{R}\langle \langle X_1,\ldots,X_m\rangle\rangle_0$. 
Then $\mathcal{N}$ is an ideal. The $0$-th homology of the chain complex $\mathbb{R}\langle \langle X_1,\ldots,X_m\rangle\rangle,\delta )$ can be expressed as 
\[ H_0(\mathbb{R}\langle \langle X_1,\ldots,X_m\rangle\rangle,\delta ) = \mathbb{R}\langle \langle X_1,\ldots,X_m\rangle\rangle,\delta )/\mathcal{N}.
\]

By this procedure, we derive from $\Theta$ the homology homomorphism
\[ \Theta_0:\mathbb{R}\pi_1(M,x_0) \to \mathbb{R}\langle \langle X_1,\ldots,X_m\rangle\rangle,\delta )/\mathcal{N}.\] 
This map is given by an infinite sum of iterated integrals of the first order differential from $\omega_0$, such as 
\[ \Theta_0(\gamma) = 1 + \sum_{k=1}^\infty \int_\gamma \underbrace{\omega_0\cdots\omega_0}_k.
\]
As before, the augmentation ideal $J$ is defined as the kernel of the augmentation map $\varepsilon:\mathbb{R}\pi_1(M,x_0)\to \mathbb{R}$ and similarly, $\widehat{J}$ as the augmentation ideal of $\mathbb{R}\langle \langle X_1,\ldots, X_m\rangle\rangle$ whose element is a formal power series with vanishing constant term. Then the following Theorem holds:
\begin{theorem}[{\rm Theorem 5.3.1 in \cite{Kohno}}]\label{chenpi1derham}
The holonomy homomorphiam $\Theta_0$ induces the following $\mathbb{R}$-algebra isomorphiam
\[ \mathbb{R}\pi_1(M,x_0)/J^{k+1} \cong \mathbb{R}\langle \langle X_1,\ldots,X_m\rangle\rangle/(\mathcal{N}+\widehat{J}^{k+1} )
\]
\end{theorem}

 We reproduce the proof in \cite{Kohno} since it is written in Japanese. The following text is our translation. The content is essentially contained in \cite{Chen1}.
This proof is performed through a connection to another object, known as the cobar complex.

Before proceeding to the actual proof, we introduce several key notions.

First, we recall Adams's cobar construction of the last object in the above.
Let $M$ be a connected smooth manifold and take a base point $x_0 \in M$. 
We denote by $C_q(M)_{x_0}$ free module generated by $q$-dimensional singular simpleces $\sigma:\Delta_q \to M$ such that any vertex $p$ of $q$-dimensional standard simplex $\Delta_q$ maps to $x_0$. 
Put $C_\ast := \bigoplus_qC_q(M)_{x_0}$. 
Let $F(C_\ast)$ be the free associative unital algebra over $\mathbb{Z}$ generated by all of the positive-dimensional singular simplexes. 
Denote by $1$ its unit element and by $[\sigma]$ the genrator of $F(C_\ast)$ corresponeding singular simplex $\sigma$ whose degree $\mbox{deg}\;[\sigma]$
is defined as $\mbox{dim}\; \sigma -1$. 
Moreover, for generators $[\sigma_1],\ldots,[\sigma_k] \in F(C_\ast)$, we define \[ \mbox{deg} ([\sigma_1],\ldots,[\sigma_k] ) = \mbox{deg}[\sigma_1]+ \cdots +\mbox{deg} [\sigma_k], \]
which makes $F(C_\ast)$ graded algebra.

On the other hand, for a loop space $\Omega M$ with base point $x_0$, we shall define its cube complex $(C_\ast(\Omega M), \partial)$.
Let $I_n$ be $n$ times direct products of unit interval $I = [0,1]$. Here we consider members of $\Omega M$ are piecewise smooth loops. A map
$\psi:I_n \to \Omega M$ is called $n$-dimensional cube ($n$-cube) of $\Omega M$ if the map $\varphi:I \times I_n\to M$ defined by
\[   \varphi(t,\xi_1,\ldots,\xi_n) := \psi(\xi_1,\ldots,\xi_n)(t)
\]   
is piecewise smooth. Note that $\varphi(0,\xi_1,\ldots,\xi_n) = \varphi(1,\xi_1,\ldots,\xi_n) =x_0$. We denote by $C_n(\Omega M)$ free module generated by such $n$-cubes and put $C_\ast(\Omega M) = \bigoplus_{n \geq 0}C_n(\Omega M)$.

Next we define the boundary operator $\partial:C_n(\Omega M)\to C_{n-1}(\Omega M)$ as follows;
Take verteces $v,w$ in the standard $q$-simplex $\Delta_q$ in the Euclidean space and let $\mathcal{P}(\Delta_q;v,w)$ be the set of smooth paths from $v$ to $w$.
We denote $j$-th face operator of $\Delta_q$ by
\[   \partial_j:\Delta_{q-1} \to \Delta_q
\]
and $j$-th face operator of $n$-cube $I_n$ by
\begin{align*}   \lambda_j^\epsilon:I_{n-1}\to I_n,& \quad 1 \leq j \leq, \quad \epsilon = 0,1 \\
\lambda_j^\epsilon(x_1,\ldots, x_{n-1}) &= (y_1,\ldots, y_n)
\end{align*} 
where
\[     y_j = \left\{\begin{array}{cc} x_j, & 1 \leq j \leq i-1 \\
\epsilon, & j=i \\
x_{j-1}, & i+1 \leq j \leq n \end{array} \right. .
\]

In the sequel, we put numbering of the vertices of the $q$-dimensional simplex $\Delta_q$ by $v_0, v_1, \ldots, v_q$. 
For vertices $w_0, w_1, \ldots, w_i$ of $\Delta_i$, we define the simplicial map $f_i:\Delta_i\to \Delta_q$ by $f_i(w_k) = v_k$, $0 \leq k \leq i$ and for vertices $w_0, w_1, \ldots, w_{q-i}$ of $\Delta_{q-i}$, we define the simplicial map $\ell_{q-i}:\Delta_{q-i}\to \Delta_q$ by $\ell_{q-i}(w_k) = v_{k+i}$, $0 \leq k \leq q-i$.  
Moreover associated to these maps $f_i, \ell_{q-i}$, we define natural inclusion maps
\begin{align*}   
 P(f_i) &:\mathcal{P}(\Delta_i;w_0,w_i) \to \mathcal{P}(\Delta_q;v_0,v_i) \\ 
 P(\ell_{q-i}) &:\mathcal{P}(\Delta_{q-i};w_0,w_{q-i}) \to \mathcal{P}(\Delta_q;v_i,v_q). 
\end{align*}   
Similarly, we define natural inclusion maps 
\[ P(\partial_j) :\mathcal{P}(\Delta_{q-1};w_0,w_{q-1}) \to \mathcal{P}(\Delta_q;v_0,v_q) \quad 0 \leq j \leq q-1 
\]
associated to $\partial_j$.
Then, the following lemma holds:
\begin{lemma} There exists $q-1$-cube
\[ \theta_q:I_{q-1}\to \mathcal{P}(\Delta_q;v_i,v_q)
\]  
such that for $1 \leq i \leq q-1$, the following equality holds:
\[   \theta_q\circ\lambda_i^0 = P(\partial_i)\circ\theta_{q-1}.
\]
\end{lemma}   
\begin{proof}
If $q=1$, then $I_0$ consists of one element and the image $\theta_1(I_0)$ is $1$-simplex $\Delta_1$ defined as a curve connecting verteces $v_0$ and $v_1$. 
Assume that the map $\theta_q$ satisfying the above conditions is given for $q \leq k-1$. Then, by the condition, the following map
\[ \phi:\partial I_{k-1}\to \mathcal{P}(\partial\Delta_k;v_0,v_k)\] 
is already defined and defines the corresponding map
\[ \varphi:I\times \partial I_{k-1}\to \partial\Delta_k
\]
by $\varphi(t,x) = \phi(x)(t)$. The map $\phi$ represents the covering of $\partial \Delta_k$ by one layer consisting of a family of paths from $v_0$ to $v_k$ on $\partial\Delta_k$. 
Since boundaries $\partial I_{k-1}, \partial \Delta_k$ are homeomorphic to their same dimensional spheres $S^{k-2}, S^{k-1}$ respectively. 
Topologically, it corresponds to the family of paths from the North Pole to the South Pole, parametrized by the equator $S^{k-2}$. Thus, by extending this family to the $k-1$-dimensional ball from its boundary $S^{k-2}$, the map $\phi$ can be naturally extended to a piecewise ssmooth map
\[ \widetilde{\varphi}:I\times I_{k-1}\to \Delta_k\] 
satisfying $\widetilde{\varphi}(0,x) = v_0, \widetilde{\varphi}(1,x) = v_k, x \in I_{k-1}$. Then its corresponding $\theta_k: I_{k-1}\to \mathcal{P}(\partial\Delta_k;v_0,v_k)$  gives our objective map for $q = k$ and thus, by induction, we get the conclusion. 
\end{proof}

Next, we shall construct a chain map
\[ \mu:F(C_\ast) \to C_\ast(\Omega M)
\]
For a simplex $\sigma:\Delta_q\to M$ with base point $x_0 \in M$, we consider a map
\[ P(\sigma):\mathcal{P}(\partial\Delta_q;v_0,v_q)\to \Omega M
\]
defined by $P(\sigma)(\gamma) = \sigma\circ\gamma$.

When $\mbox{dim}\;\sigma = q >1$, we define as
\[ \mu([\sigma]) = P(\sigma)\circ\theta_q.
\] 
When $\mbox{dim}\;\sigma = 1$, we define as
\[ \mu([\sigma]) = P(\sigma)\circ\theta_1-[x_0]
\] 
where $[x_0]$ is $0$-dimensional chain representing the constant path $x_0$.
From the above, $\mu$ is defined as an algebra homomorphism.
To be $\mu$ is also the chain map, we define the boundary operator
$\partial$ on $F(C_\ast)$.
When $\mbox{dim}\;\sigma = 1$, it is defined as 
\[  \partial([\sigma]) = \sum_{i=1}^{q-1}(-1)^i\{[\partial_i\sigma]-[f_i\sigma][\ell_{q-i}\sigma]\}
\]
where $\partial_i\sigma, f_i\sigma, \ell_{q-i}\sigma$ represent $\partial_i\circ\sigma, f_i\circ\sigma, \ell_{q-i}\circ\sigma$ respectively.

When $\mbox{dim}\;\sigma = 1$, we defined as $\partial([\sigma]) = 0$.

Finally, we define the linear operator we define $\partial: F(C_\ast)\to F(C_\ast)$ compatible with the product such that, for the product $xy$ of $x, y \in F(C_\ast)$ with
$\mbox{dim}\;x = k, \mbox{dim}\;y = \ell$, 
\[\partial(xy) = \partial(x)y + (-1)^kx\partial(y).\]

Then,
 we have
\begin{proposition}[a chain map]
The map
\[ \mu:F(C_\ast) \to C_\ast(\Omega M)
\]
is a chain map, namely it satisfies $\partial\circ\mu = \mu\circ\partial$.
\end{proposition}

The straightforward calculation is the proof of this proposition. We omit the detail.

Furthermore, we define a filtration of $F(C_\ast)$ 
\[ F(C_\ast) = F_0(C_\ast) \supset F_{-1}(C_\ast) \supset \cdots \supset F_{-k}(C_\ast) \supset \cdots 
\]
by the submodule $F_{-k}(C_\ast)$ consists of elements in $F(C_\ast)$ with the form of product
\[ [\sigma_1][\sigma_2]\cdots[\sigma_r], \quad r\geq k
\]
for positive dimensional simpleces $\sigma_1, \sigma_2, \ldots, \sigma_r$.

\begin{proof}[Proof of Theorem \ref{chenpi1derham}]
Put $F(C_\ast)_{\mathbb{R}} = F(C_\ast)\times\mathbb{R}$. Then, the filtration preserving homomorphism
\[ \Theta\circ\mu:F(C_\ast)_{\mathbb{R}}\to \mathbb{R}\langle\langle X_1, \ldots, X_m\rangle\rangle.
\]
Here we recall that the filtration of $\mathbb{R}\langle\langle X_1, \ldots, X_m\rangle\rangle$ is given by its degree of the product of indeterminants $X_i$.

Let us consider the map between the spectral sequences of both sides induced from $\Theta\circ\mu$.
From its $E^0$-term, an induced linear map
\[ \bigotimes^k\overline{C}_0(M,\mathbb{R})\to \widehat{J}^k/\widehat{J}^{k+1}
\]
is obtained. Here $\overline{C}_0(M,\mathbb{R})$ is $0$-th part of the reduced complex $\overline{C}_\ast(M,\mathbb{R})$ which is induced from complex $C_\ast(M,\mathbb{R})$ as 
\[ \overline{C}_q(M,\mathbb{R}) = \left\{\begin{array}{cc} 0 & q<0 \\ C_1(M,\mathbb{R})/\partial(C_0(M,\mathbb{R})) & q =0 \\ C_{q+1}(M,\mathbb{R}) & q > 0
\end{array}\right..
\]
Moreover, from its $E^1$-term, an induced linear isomorphism
\[ \bigotimes^kH_1
(M,\mathbb{R})\to \widehat{J}^k/\widehat{J}^{k+1}
\]
is also obtained.
Using the fact that isomorphisms
\[ H_0(F(C_\ast)_{\mathbb{R}}) \cong H_0(\Omega M,\mathbb{R}) \cong H_0(\mathbb{R}\langle\langle X_1, \ldots, X_m\rangle\rangle, \delta)
\]
preserving the filtration, we get the conclusion by the isomorphism between their $E^\infty$-terms.
\end{proof}

\begin{example}
If we take $G = {\rm Heis}_3(\mathbb{R})$ and $\Gamma = {\rm Heis}_3(\mathbb{Z})$, then the quotient $M = G/\Gamma$ is a $3$-dimensional nil-manifold with
\begin{equation*}
  H_q(M,\mathbb{Z}) = \left\{ \begin{array}{cc} \mathbb{Z}, & q= 0, 3 \\
\mathbb{Z}\oplus\mathbb{Z}, & q = 1, 2 \\ 0, & \mbox{otherwise}
\end{array}\right.
\end{equation*}
Then each component of the matrix \ref{formonHeisenberg} is invariant one form with respect to the left action of $\Gamma$.

Then, the first de Rham cohomology $H_{\mbox{{\tiny DR}}}^1(M)$ is generated by $\omega_1 = dx$ and $\omega_2 = dy$.
If we put $\omega_{12} = -xdy + dz$, then
\begin{equation} \omega_1\wedge\omega_2 = -d\omega_{12} ,\label{vanishingtwo}
\end{equation} 
moreover, thus, the second cohomology class representing $\omega_1\wedge\omega_2$ is zero.
Then, the second de Rham cohomology $H_{\mbox{{\tiny DR}}}^1(M)$ is generated by $\omega_1\wedge\omega_{12}$ and $\omega_2\wedge\omega_{12}$. Thus, we take indeterminacies $X_1, X_2, Y_1, Y_2$ which are dual to differential forms $\omega_1, \omega_2,  \omega_1\wedge\omega_{12}, \omega_2\wedge\omega_{12}$ respectively.

Now we compute the homology connection $\omega$ described as 
\[  \omega = \omega_1X_1 + \omega_2X_2 + \omega_1\wedge\omega_{12}Y_1 + \omega_1\wedge\omega_{12}Y_2 + \cdots
\]
satisfying
\begin{align}
\delta\omega + d\omega &= \varepsilon(\omega)\wedge\omega \notag \\
&= -\omega_1\wedge\omega_2 + \cdots \label{kappavanish}
\end{align}
From the relation (\ref{vanishingtwo}), the quadratic term of $\omega$ with respect to $X_1, X_2$ is
\[    -\omega_{12}(X_1X_2 - X_2X_1) = -\omega_{12}[X_1,X_2]
\]
Considering the $0$-th order part of the flatness condition (\ref{kappavanish}),
we have 
\[ \delta X_1 = \delta X_2 = 0.\]
Using the flatness condition (\ref{kappavanish}) again, we have
\[ \delta Y_1 = [[X_1,X_2],X_1], \quad  \delta Y_2 = [[X_1,X_2],X_2].
\]
Consequently, the homology connection $\omega$ is expressed as
\[ \omega = \omega_1X_1 + \omega_2X_2 + \omega_1\wedge\omega_{12}Y_1 + \omega_2\wedge\omega_{12}Y_2 + \omega_{12}[X_1,X_2]. 
\]

From the above arguments, we see that the completion of the group ring of the fundamental group $\pi_1(M) = \pi_1(G/\Gamma)$ is 
\[ \mathbb{R}\widehat{\pi}_1(M) = \mathbb{R}\langle\langle X_1,X_2\rangle\rangle/\mathcal{N}
\]
where $\mathcal{N}$ is an ideal generated by  $[[X_1,X_2],X_1]$ and $[[X_1,X_2],X_2]$.
\end{example}

Note that homology connections for general stratified nilpotent groups,
$k$-th order part $(\omega)_k = \sum \omega_{i_1,\ldots, i_k}X_{i_1}\ldots X_{i_k}$ can be rearranged as a sum of multiples of one-forms and polynomials of iterated Lie bracket of ${X_i}$'s of order $k$ as in the above example.



\bigskip
\address{ 
Research and Education Center for Natural Science \\
Hiyoshi Campus, Keio University \\
4-1-1, Hiyoshi, Kohoku-ku \\
 Yokohama 223-8521 \\ 
Japan \\
\smallskip \\
Faculty of Mathematics \\
Kyushu University \\
744 Motooka, Nishi-ku,\\
Fukuoka 819-0395 \\
Japan
}
{katsuda@math.kyushu-u.ac.jp}

\end{document}